\documentclass[11pt]{article}
\usepackage{macros, xcolor}
\usepackage{geometry}
\title{On the Kottwitz conjecture for local shtuka spaces}
\author{David Hansen, Tasho Kaletha, and Jared Weinstein}

\begin{document}
\maketitle

\begin{abstract}  Kottwitz's conjecture describes the contribution of a supercuspidal represention to the cohomology of a local Shimura variety in terms of the local Langlands correspondence.  A natural extension of this conjecture concerns Scholze's more general spaces of local shtukas.  
Using a new Lefschetz-Verdier trace formula for v-stacks, we prove the extended conjecture, disregarding the action of the Weil group, and modulo a virtual representation whose character vanishes on the locus of elliptic elements.  As an application, we show that for an irreducible smooth representation of an inner form of $\GL_n$, the $L$-parameter constructed by Fargues-Scholze agrees with the usual semisimplified parameter arising from local Langlands.
\end{abstract}
{\let\thefootnote\relax\footnotetext{T.K. was supported in part by NSF grants DMS-1161489, DMS-1801687, and a Sloan Fellowship.  J.W. was supported by NSF grants DMS-1303312,  DMS-1902148, and a Sloan Fellowship.}}

\tableofcontents

%!TEX root = FixedPoints2021.tex
\section{Introduction}
\label{SectionIntroduction}

Let $F$ be a finite extension of the field $\Q_p$ of $p$-adic numbers, and let $G$ be a connected reductive group defined over $F$.  Scholze \cite[\S23]{ScholzeLectures} has introduced a tower of moduli spaces of mixed-characteristic shtukas
\[ \Sht_{G,b,\mu}=\varprojlim_K \Sht_{G,b,\mu,K} \]
depending on a $\sigma$-conjugacy class of $b\in G(\breve F)$ (where $\breve F$ is the completion of the maximal unramified extension of $F$) and on a conjugacy class of cocharacters $\mu\from \Gm\to G$ defined over $\overline{F}$.   Here $K$ ranges over open compact subgroups of $G(F)$.   Each $\Sht_{G,b,\mu,K}$ is a locally spatial diamond defined over $\Spd \breve{E}$, where $E$ is the field of definition of the conjugacy class of $\mu$.  

When $\mu$ is minuscule, $\Sht_{G,b,\mu,K}$ is the diamond associated to a rigid-analytic variety $\mc{M}_{G,b,\mu,K}$ \cite[\S24]{ScholzeLectures}.  The latter is a {\em local Shimura variety}, whose general existence was conjectured in \cite{RapoportViehmann}.  The theory of Rapoport-Zink spaces \cite{RapoportZink} provides instances of $\mc{M}_{G,b,\mu,K}$ admitting a moduli interpretation, as the generic fiber of a deformation space of $p$-divisible groups.  

%Let $F$ be a finite extension of the field $\Q_p$ of $p$-adic numbers, and let $\breve F$ be the completion of the maximal unramified extension of $F$, relative to a fixed algebraic closure $\bar F$. Let $\sigma \in \tx{Aut}(\breve F/F)$ be the arithmetic Frobenius element. Let $G$ be a connected reductive group defined over $F$, $[b] \in B(G)$ a $\sigma$-conjugacy class of elements of $G(\breve F)$, and $\{\mu\}$ a conjugacy class of cocharacters $\Gm \to G$ defined over $\bar F$. Assume that $\{\mu\}$ is minuscule and that $[b] \in B(G,\{\mu\})$.  The triple $(G,[b],\{\mu\})$ is called a local Shimura datum \cite[\S5]{RapoportViehmann}. In loc. cit. the authors conjecture the existence of an associated tower $\mc{M}_{G,b,\mu,K}$ of rigid analytic spaces over $\breve E$, indexed by open compact subgroups $K \subset G(F)$.  Here $E$ is the field of definition of the conjugacy class $\{\mu\}$, a finite extension of $F$.  The isomorphism class of the tower should only depend on the classes $[b]$ and $\{\mu\}$.  The theory of Rapoport-Zink spaces \cite{RapoportZink} provides instances of such a tower together with a moduli interpretation, as the generic fiber of a deformation space of $p$-divisible groups. In general these towers were constructed in \cite[\S24]{ScholzeLectures}.

The {\em Kottwitz conjecture} \cite[Conjecture 5.1]{RapoportNonArchimedean}, \cite[Conjecture 7.3]{RapoportViehmann} relates the cohomology of $\mc{M}_{G,b,\mu,K}$ to the local Langlands correspondence, in the case that $b$ lies in the unique basic class in $B(G,\mu)$.  There is a natural generalization of this conjecture for $\Sht_{G,b,\mu,K}$, as we now explain.  

%Let us review the precise statement.
%Let $B(G)_\tx{bas}\subset B(G)$ be the set of basic $\sigma$-conjugacy classes. Assume that $[b]\in B(G)_{\tx{bas}}$ and choose a representative $b\in [b]$. Let $G_b$ be the associated inner form of $G$.  Note that since $B(G,\set{\mu})$ contains a unique basic element, $[b]$ is uniquely determined by $\set{\mu}$.

%The tower $\mc{M}_{G,b,\mu,K}$ receives commuting actions of $G_b(F)$ and $G(F)$. The action of $G_b(F)$ preserves each $\mc{M}_{G,b,\mu,K}$, while the action of $g \in G(F)$ sends $\mc{M}_{G,b,\mu,K}$ to $\mc{M}_{G,b,\mu,gKg^{-1}}$.  There is furthermore a {\em Weil descent datum} on this tower from $\breve{E}$ down to $E$.  It need not be effective.

Let $G_b$ the inner form of $G$ associated to $b$.  The tower $\Sht_{G,b,\mu,K}$ admits commuting actions of $G_b(F)$ and $G(F)$.  The action of $G_b(F)$ preserves each $\Sht_{G,b,\mu,K}$, whereas the action of $g\in G(F)$ sends $\Sht_{G,b,\mu,K}$ to $\Sht_{G,b,\mu,gKg^{-1}}$.  There is furthermore a (not necessarily effective) Weil descent datum on this tower from $\breve{E}$ down to $E$.  

Let $\ell$ be a prime distinct from $p$.  The geometric Satake equivalence produces an object $\mc{S}_\mu$ in the derived category of \'etale $\Z_\ell$-sheaves on $\Sht_{G,b,\mu,K}$;  this is compatible with the actions of $G(F)$ and $G_b(F)$ on the tower.  Let $C$ be the completion of an algebraic closure of $\breve{E}$.  For a smooth representation $\rho$ of $G_b(F)$ with coefficients in $\overline{\Q}_\ell$, we define:
\[ R\Gamma(G,b,\mu)[\rho]=\varinjlim_K R\Hom_{G_b(F)}(R\Gamma_c(\Sht_{G,b,\mu,K,C},\mc{S}_{\mu}),\rho). \]
Then $R\Gamma(G,b,\mu)[\rho]$ lies in the derived category of smooth representations of $G(F)\times W_E$ with coefficients in $\overline{\Q}_\ell$, where $W_E$ is the Weil group.  Informally, this is the $\rho$-isotypic component of the cohomology of the tower $\Sht_{G,b,\mu}$. 

 A recent result of Fargues-Scholze \cite[Corollary I.7.3]{FarguesScholze} states that if $\rho$ is finite length and admissible, then $R\Gamma(G,b,\mu)[\rho]$ is a complex of finite length admissible representations of $G(F)$ admitting a continuous action of $W_E$.
 
Let $\Groth(G_b(F))$ be the Grothendieck group of the category of finite length admissible representations of $G(F)$ with $\overline{\Q}_\ell$ coefficients.  Also, let $\Groth(G(F)\times W_E)$ be the Grothendieck group of the category of finite length admissible representations of $G(F)$ with $\overline{\Q}_\ell$ coefficients, which come equipped with a continuous action of $W_E$ commuting with the $G(F)$-action.  Following \cite{ShinGaloisRepresentations} we define a map
\[ \Mant_{b,\mu}\from \Groth(G_b(F))\to \Groth(G(F)\times W_E) \]
(for ``Mantovan'', referencing \cite{Mantovan}) sending $\rho$ to the Euler characteristic of $R\Gamma(G,b,\mu)[\rho]$.

%
%\[ H^i_c(\mc{M}_{G,b,\mu,K}\times_{\breve E}C, \bar\Q_\ell),\] 
%a $\bar\Q_\ell$-vector spaces equipped with an action of $G_b(F)$ as well as an action of $I_E$, which extends to an action of $W_E$ due to the Weil descent datum. The actions of $G_b(F)$ and $W_E$ commute. Given an irreducible smooth admissible representation $\rho$ of $G_b(F)$ we have the $\bar \Q_\ell$-vector space
%\[H^{i,j}(G,b,\mu)[\rho]=\varinjlim_K \tx{Ext}^j_{G_b(F)}(H^i_c(\mc{M}_{G,b,\mu,K}\times_{\breve E}C, \bar\Q_\ell),\rho) \]
%with a smooth action of  $G(F)\times W_E$, leading to the virtual $G(F)\times W_E$-representation
%\[ H^*(G,b,\mu)[\rho] := \sum_{i\in\Z} (-1)^{i+j} H^{i,j}(G,b,\mu)[\rho](-d), \]
%where $d= \dim \mc{M}_{G,b,\mu} = \<2\rho_G,\mu\>$, and $\rho_G$ is half the sum of the positive roots of $G$.  The isomorphism class of $H^*(G,b,\mu)[\rho]$ only depends on $(G,[b],\set{\mu})$ and $\rho$.

The Kottwitz conjecture (appropriately generalized) describes $\Mant_{b,\mu}(\rho)$ in terms of the local Langlands correspondence, when $\rho$ lies in a supercuspidal $L$-packet. The complex dual groups of $G$ and $G_b$ are canonically identified, and we write $\hat{G}$ for either.   Let $^LG=\hat G \rtimes W_F$ be the $L$-group. The basic form of the local Langlands conjecture predicts that the set of isomorphism classes of essentially square-integrable representations of $G(F)$ (resp., $G_b(F)$) is partitioned into $L$-packets $\Pi_\phi(G)$ (resp., $\Pi_\phi(G_b)$), and that each such packet is indexed by a discrete Langlands parameter $\phi : W_F \times \tx{SL}_2(\C) \to {^LG}$. When $\phi$ is discrete and trivial on $\tx{SL}_2(\C)$, we say $\phi$ is supercuspidal; in this case it is expected that the packets $\Pi_\phi(G)$ and $\Pi_\phi(G_b)$ consist entirely of supercuspidal representations. 

Our generalized Kottwitz conjecture is conditional on the refined local Langlands correspondence for supercuspidal $L$-parameters, in the formulation of \cite[Conjecture G]{KalethaLocalLanglands}. In particular, it relies crucially on the endoscopic character identities satisfied by $L$-packets. These are reviewed in Appendix \ref{LLCReview}. Note that we do not assume any compatibility between the validity of \cite[Conjecture G]{KalethaLocalLanglands} and the construction of \cite{FarguesScholze}, i.e. we do not require that the construction of \cite{FarguesScholze} satisfy any portion of \cite[Conjecture G]{KalethaLocalLanglands}.

We take this opportunity to give a brief summary of the status of \cite[Conjecture G]{KalethaLocalLanglands}. In short, the full conjecture is known for regular supercuspidal parameters \cite[Definition 5.2.3]{KalRSP} provided $G$ splits over a tame extension of $F$, $F$ has characteristic zero, and $p$ is sufficiently large (at least $(e+2)n$, where $e$ is the ramification index of $F/\Q_p$ and $n$ is the smallest size of a faithful algebraic representation of $G$). The proof is contained in \cite[\S5.3]{KalRSP} and \cite[\S4.4]{FKS}. However, various parts of that conjecture are known under less restrictive assumptions. To describe this, we remind the reader that \cite[Conjecture G]{KalethaLocalLanglands} consists of the following assertions:
\begin{enumerate}
    \item The existence of a finite set $\Pi_\phi$ of representations of rigid inner forms of $G$ for each tempered $L$-parameter $\phi$.
    \item The existence and uniqueness of a generic constituent of $\Pi_\phi$ for a fixed Whittaker datum.
    \item A bijection between $\Pi_\phi$ and the set $\tx{Irr}(\pi_0(S_\phi^+))$ of irreducible representations of the refined centralizer component group associated to $\phi$.
    \item The character identities of ordinary endoscopy, as recalled in Appendix \ref{LLCReview}.
\end{enumerate}
At the moment a set $\Pi_\phi$ has been constructed in \cite[\S\S4.1,4.2]{KalSLP} for every supercuspidal parameter $\phi$ provided $G$ splits over a tame extension of $F$ and $p$ does not divide the order of the Weyl group of $G$ (this assumption on $p$ implies that any supercuspidal parameter maps wild inertia into a torus of $\hat G$; under weaker assumptions on $p$ this is not automatically true, but for parameters $\phi$ that do have this property the construction of \cite{KalSLP} works under weaker assumptions on $p$). A bijection between $\Pi_\phi$ and $\tx{Irr}(\pi_0(S_\phi^+))$ has been constructed in \cite[\S\S4.3-4.5]{KalSLP} for any supercuspidal parameter $\phi$. Assuming $F$ has characteristic zero and $p \geq (e+2)n$, the existence and uniqueness of generic constituent in $\Pi_\phi(G)$, as well as the character identities of ordinary endoscopy, are proved in \cite[\S4.4]{FKS} for all regular supercuspidal parameters $\phi$. They are also proved for non-regular supercuspidal parameters $\phi$ but only for certain endoscopic elements.

Returning to the subject of this paper, let $S_\phi=\tx{Cent}(\phi,\hat{G})$. For any $\pi \in \Pi_\phi(G)$ and $\rho \in \Pi_\phi(G_b)$ the refined form of the local Langlands conjecture implies the existence of an algebraic representation $\delta_{\pi,\rho}$ of $S_\phi$, which can be thought of as measuring the relative position of $\pi$ and $\rho$. (The representation $\delta_{\pi,\rho}$ also depends on $b$, but we suppress this from the notation.) The conjugacy class of $\mu$ determines by duality a conjugacy class of weights of $\hat G$;  we denote by $r_{\mu}$ the irreducible representation of $\hat G$ of highest weight $\mu$. There is a natural extension of $r_{\mu}$ to $^LG_E$, the $L$-group of the base change of $G$ to $E$ \cite[Lemma 2.1.2]{Kot84t}. Write $r_{\mu} \circ \phi_E$ for the representation of $S_\phi \times W_E$ given by
\[ r_{\mu} \circ \phi_E(s,w) = r_{\mu}(s\cdot \phi(w)). \]

\begin{cnj} \label{cnj:kottwitz}
\label{Kottwitz} Let $\phi : W_F \to {^LG}$ be a supercuspidal Langlands parameter. Given $\rho \in \Pi_\phi(G_b)$, we have the following equality in $\Groth(G(F)\times W_E)$:
\begin{equation}
\label{KottwitzEquation}
\Mant_{b,\mu}(\rho) = \sum_{\pi \in \Pi_\phi(G)} \pi \boxtimes \tx{Hom}_{S_\phi}(\delta_{\pi,\rho},r_{\mu} \circ \phi_E).
\end{equation}
\end{cnj}

%\begin{cnj}[Kottwitz] \label{cnj:kottwitz}
%\label{Kottwitz} Let $\phi : W_F \to {^LG}$ be a supercuspidal Langlands parameter. Given $\rho \in \Pi_\phi(G_b)$, each $H^i(G,b,\mu)[\rho]$ is admissible, and we have the following equality\footnote{\cite[Conjecture 7.3]{RapoportViehmann} omits the sign $(-1)^d$.} in $\Groth(G(F)\times W_E)$:
%\begin{equation}
%\label{KottwitzEquation}
%H^*(G,b,\mu)[\rho] = (-1)^{d}\sum_{\pi \in \Pi_\phi(G)} \pi \boxtimes \tx{Hom}_{S_\phi}(\delta_{\pi,\rho},r_{\{\mu\}} \circ \phi_E)(-\frac{d}{2}).
%\end{equation}
%\end{cnj}
%It is not obvious that $H^*_c(G,b,\mu)[\rho]$ is admissible as a virtual $G(F)$-representation;  this is an assertion of the conjecture. In fact, as stated in \cite{RapoportViehmann}, the expectation is that each $H^{i,j}_c(\mc{M}_{G,b,\mu})[\rho]$ is an admissible representation of $G(F)$.

This conjecture is more general than the formulation of Kottwitz's conjecture in \cite{RapoportNonArchimedean} and \cite{RapoportViehmann}, in that two conditions are removed.   The first is that we are allowing the cocharacter $\mu$ to be non-minuscule -- this is what requires passage from the local Shimura varieties $\mc{M}_{G,b,\mu}$ to the local shtuka spaces $\Sht_{G,b,\mu}$.  The second is that we do not require $G$ to be a $B$-inner form of its quasi-split inner form $G^*$.  This condition, reviewed in \S\ref{sub:binner}, has the effect of making the definition of $\delta_{\pi,\rho}$ straightforward. To remove it, we use the formulation of the refined local Langlands correspondence \cite[Conjecture G]{KalethaLocalLanglands} based on the cohomology sets $H^1(u \to W,Z \to G)$ of \cite{KalRI}. The definition of $\delta_{\pi,\rho}$ in this setting is a bit more involved and is given in \S\ref{sub:geninner}, see Definition \ref{DeltaPiRho}.

We now present our main theorem.  

\begin{thm}  \label{TheoremMain}
Assume the refined local Langlands correspondence \cite[Conjecture G]{KalethaLocalLanglands}. Let $\phi\from W_F \times \tx{SL}_2 \to {^LG}$ be a discrete Langlands parameter with coefficients in $\overline{\Q}_\ell$, and let $\rho\in \Pi_\phi(G_b)$ be a member of its $L$-packet.  After ignoring the action of $W_E$, we have an equality in $\Groth(G(F))$:
 \[\Mant_{b,\mu}(\rho)=\sum_{\pi\in\Pi_{\phi}(G)} \left[\dim \Hom_{S_\phi}(\delta_{\pi,\rho},r_\mu)\right]\pi + \mathrm{err}, \]
where $\mathrm{err} \in \Groth(G(F))$ is a virtual representation whose character vanishes on the locus of elliptic elements of $G(F)$.

If the packet $\Pi_{\phi}(G)$ consists entirely of supercuspidal representations and the semisimple $L$-parameter $\varphi_{\rho}$ associated with $\rho$ as in \cite[\S I.9.6]{FarguesScholze} is supercuspidal, then in fact $\mathrm{err}=0$.
\end{thm}

Of course we expect that $\varphi_{\rho}=\phi^{\mathrm{ss}}$, so that if $\phi$ is supercuspidal, the error term should vanish.  In that case we obtain Conjecture \ref{cnj:kottwitz} modulo ignoring the action of $W_E$.  For a discrete but non-supercuspidal parameter $\phi$, the error term in Theorem \ref{TheoremMain} is often provably nonzero, cf. \cite{Imai} for some examples. However, for applications to the local Langlands correspondence, it is crucial to have Theorem \ref{TheoremMain} in this extra generality.

The shtukas appearing in our work have only one ``leg''. Scholze defines moduli spaces of mixed-characteristic shtukas $\Sht_{G,b,\set{\mu_i}}$ with arbitarily many legs, fibered over a product $\prod_{i=1}^r \Spd \breve{E}_i$.  It is straightforward to extend Conjecture \ref{cnj:kottwitz} and Theorem \ref{TheoremMain} to this setting as well. In fact, Theorem \ref{TheoremMain} in this extended level of generality follows immediately from the results already proved in this paper, by allowing the legs to coalesce and using the fact that cohomology of shtuka spaces forms a local system over $(\mathrm{Div}^1)^I$. We leave the details to the interested reader.

Theorem \ref{TheoremMain} has an application to the local Langlands correspondence.

\begin{thm} \label{TheoremLLCInnerForm} Let $G$ be any inner form of $\mathrm{GL}_n/F$, and let $\pi$ be an irreducible smooth representation of $G(F)$. Then the $L$-parameter $\varphi_\pi$ associated with $\pi$ by the construction of Fargues-Scholze \cite[\S I.9]{FarguesScholze} agrees with the usual semisimplified $L$-parameter attached to $\pi$.
\end{thm}

\subsection{Remarks on the proof, and relation with prior work}
Ultimately,  Theorem \ref{TheoremMain} is proved by an application of a Lefschetz-Verdier trace formula.   Let us illustrate the idea in the Lubin-Tate case:  say $F=\Q_p$,  $G=\GL_n$, $\mu=(1,0,\dots,0)$, and $b$ is basic of slope $1/n$.   Let $H_0$ be the $p$-divisible group over $\overline{\F}_p$ with isocrystal $b$, so that $H_0$ has dimension $1$ and height $n$.  In this case $G_b(F)=\Aut^0 H_0=D^\times$, where $D/\Q_p$ is the division algebra of invariant $1/n$.  The spaces $\M_K=\M_{G,b,\mu,K}$ are known as the Lubin-Tate tower;  we consider these as rigid-analytic spaces over $C$, where $C/\Q_p$ is a complete algebraically closed field. 

Atop the tower sits the infinite-level Lubin-Tate space $\M=\varprojlim_K\M_K$ as described in \cite{ScholzeWeinstein}.  This is a perfectoid space admitting an action of $G(\Q_p)\times G_b(\Q_p)$.  
The $C$-points of $\M$ classify equivalence classes of triples $(H,\alpha,\iota)$, where $H/\OO_{C}$ is a $p$-divisible group, $\alpha\from \Q_p^n\to VH$ is a trivialization of the rational Tate module, and $\iota\from H_0\otimes_{\overline{\F}_p} \OO_C/p\to H\otimes_{\OO_{C}} \OO_C/p$ is an isomorphism in the isogeny category.   (Equivalence between two such triples is a quasi-isogeny between $p$-divisible groups which makes both diagrams commute.)  Then $\M$ admits an action of $G(\Q_p)\times G_b(\Q_p)$, via composition with $\alpha$ and $\iota$, respectively.

The Hodge-Tate period map exhibits $\M$ as a pro-\'etale $D^\times$-torsor over Drinfeld's upper half-space $\Omega^{n-1}$ (the complement in $\mathbf{P}^{n-1}$ of all $\Q_p$-rational hyperplanes).  This map $\M\to \Omega^{n-1}$ is equivariant for the action of $G(\Q_p)$.  

Now suppose $g\in G(\Q_p)$ is a regular elliptic element (that is, an element with irreducible characteristic polynomial).  Then $g$ has exactly $n$ fixed points on $\Omega^{n-1}$.  For each such fixed point $x\in (\Omega^{n-1})^g$, the element $g$ acts on the fiber $\M_x$.  Because $\M\to \Omega^{n-1}$ is a $G_b(F)$-torsor, there must exist $g'\in G_b(\Q_p)$ such that $(g,g')$ fixes a point in the fiber $\M_x$.

\bigskip
\noindent {\textbf{Key observation}}.  The elements $g\in G(\Q_p)$ and $g'\in G_b(\Q_p)$ are related, meaning they become conjugate over $\overline{\Q}_p$.
\bigskip

We sketch the proof of this claim.   Suppose $y$ corresponds to the triple $(H,\alpha,\iota)$.  This means there exists an automorphism $\gamma$ of $H$ (in the isogeny category) which corresponds to $g$ on the Tate module and $g'$ on the special fiber, respectively.  We verify now that $g$ and $g'$ are related.  Let $B_{\cris}=B_{\cris}(C)$ be the crystalline period ring.  There are isomorphisms
\[ B_{\cris}^n {\to} VH \otimes_{\Q_p} B_{\cris} {\to} M(H_0) \otimes B_{\cris}, \]
where the first map is induced from $\alpha$, and the second map comes from the comparison isomorphism between \'etale and crystalline cohomology of $H$ (using $\iota$ to identify the latter with $M(H_0)$).  
The composite map carries the action of $g$ onto that of $g'$, which is to say that $g$ and $g'$ become conjugate over $B_{\cris}$.   This implies that $g$ and $g'$ are related.

Suppose that $\rho$ is an admissible representation of $D^\times$ with coefficients in $\overline{\Q}_\ell$.  There is a corresponding $\overline{\Q}_\ell$-local system $\mc{L}_{\rho}$ on $\Omega^{n-1}_{C,\et}$.

Let $g\in G(F)$ be elliptic.  A na\"ive form of the Lefschetz trace formula would predict that:
\[
\tr\left(g\vert R\Gamma_c(\Omega^{n-1},\mc{L}_\rho)\right) = \sum_{x\in (\Omega^{n-1})^g} \tr(g\vert \mathcal{L}_{\rho,x}).\]
For each fixed point $x$, the key observation above gives $\tr(g\vert \mathcal{L}_{\rho,x})=\tr \rho(g')$, where $g$ and $g'$ are related.  By the Jacquet-Langlands correspondence, there exists a discrete series representation $\pi$ of $G(\Q_p)$ satisfying $\tr \pi(g)=(-1)^{n-1}\tr \rho(g')$ (here $\tr \pi(g)$ is interpreted as a Harish-Chandra character).  Thus the Euler characteristic of $R\Gamma_c(\Omega^{n-1},\mc{L}_\rho)$  equals $(-1)^{n-1}n\pi$ up to a virtual representation with trace zero on the elliptic locus.  

In this situation $\mc{S}_\mu = \Z_\ell[n-1]$ (up to a Tate twist), and we find that $R\Gamma(G,b,\mu)[\rho]$ is the shift by $n-1$ of the dual of $R\Gamma_c(\Omega^{n-1},\mc{L}_{\rho^\vee})$.  
 Therefore in $\Groth(\GL_n(\Q_p))$ we have 
\[ \Mant_{b,\mu}(\rho) = n\pi+\rm{err}, \]
where the character of err vanishes on the locus of elliptic elements.  This is in 
accord with Theorem \ref{TheoremMain}.

This argument goes back at least to the 1990s, as discussed in \cite[Chap. 9]{HarrisPortrait}, and as far as we know first appears in \cite{Fal94}.  The present article is our attempt to push this argument as far as it will go.  If a suitable Lefschetz formula is valid, then the equality in Theorem \ref{TheoremMain} can be reduced to an endoscopic character identity relating representations of $G(F)$ and $G_b(F)$ (Theorem \ref{TheoremEndoscopicTraceRelation}), which we prove in \S\ref{SectionTransferOfFunctions}.

Therefore the difficulty in Theorem \ref{TheoremMain} lies in proving the validity of the Lefschetz formula.  Prior work of Strauch and Mieda proved Theorem \ref{TheoremMain} in the case of the Lubin-Tate tower  \cite{StrauchJL}, \cite{StrauchDeformationSpaces}, \cite{MiedaLubinTateTower}, \cite{MiedaGeometricApproach} and also in the case of a basic Rapoport-Zink space for GSp(4) \cite{MiedaGSp4}.  

In applying a Lefschetz formula to a non-proper rigid space, care must be taken to treat the boundary.  For instance, if $X$ is the affinoid unit disc $\set{\abs{T}\leq 1}$ in the adic space $\mathbf{A}^1$, then the automorphism $T\mapsto T+1$ has Euler characteristic 1 on $X$, despite having no fixed points.  The culprit is that this automorphism fixes the single boundary point in $\overline{X}\backslash X$.  Mieda \cite{MiedaLefschetz} proves a Lefschetz formula for an operator on a rigid space, under an assumption that the operator has no topological fixed points on a compactification.
Now, in all of the above cases, $\M_{G,b,\mu,K}$ admits a {\em cellular decomposition}.  This means (approximately) that $\M_{G,b,\mu,K}$ contains a compact open subset, whose translates by Hecke operators cover all of $\M_{G,b,\mu,K}$.  This is enough to establish the ``topological fixed point'' hypothesis necessary to apply Mieda's Lefschetz formula.   Shen \cite{ShenUnitary} constructs a cellular decomposition for a basic Rapoport-Zink space attached to the group $U(1,n-1)$, which paves the way for a similar proof of Theorem \ref{TheoremMain} in this case as well.
For general $(G,b,\mu)$, however, the $\mc{M}_{G,b,\mu,K}$ do not admit a cellular decomposition, and so there is probably no hope of applying the methods of \cite{MiedaLefschetz}.  

We had no idea how to proceed, until we learned of the shift of perspective offered by Fargues' program on the geometrization of local Langlands \cite{FarguesGeometrization}, followed by the work \cite{FarguesScholze}.  At the center of that program is the stack $\Bun_G$ of $G$-bundles on the Fargues-Fontaine curve.  This is a geometrization of the Kottwitz set $B(G)$:  There is a bijection $b\mapsto \E^b$ between $B(G)$ and points of the underlying topological space of $\Bun_G$.  For basic $b$ there is an open substack $\Bun_G^b\subset \Bun_G$ classifing $G$-bundles which are everywhere isomorphic to $\E^b$;  in this situation $\Aut \E^b=G_b(F)$ and so we have an isomorphism $\Bun_G^b\isom [\ast/G_b(F)]$.  

Let $\mu$ be a cocharacter of $G$.  As in geometric Langlands, there is a stack $\Hecke_{G,\leq\mu}$ lying over the product $\Bun_G\times\Bun_G$, which parametrizes $\mu$-bounded modifications of $G$-bundles at one point of the curve.  For each $\mu$, one uses $\Hecke_{G,\leq\mu}$ to define a Hecke operator $T_\mu$ on a suitable derived category $D(\Bun_G,\Z_\ell)$ of \'etale $\Z_\ell$-sheaves on $\Bun_G$.  If $b\in B(G,\mu)$, then the moduli space of local shtukas $\Sht_{G,b,\mu}$ appears as the fiber of $\Hecke_{G,\leq\mu}$ over the point $(\E^b,\E^1)$ of $\Bun_G\times \Bun_G$.  Consequently there is an expression for $R\Gamma(G,b,\mu)[\rho]$ in terms of the Hecke operators $T_\mu$, see Proposition \ref{PropHInTermsOfHecke}.

Heavy use is made in \cite{FarguesScholze} of the notion of {\em universal local acyclity} (ULA) as a property of objects $A\in D(X,\Z_\ell)$ for Artin v-stacks $X$. When $X=[\point/G_b(F)]$, a ULA object is an admissible complex of representations of $G_b(F)$.   It is proved in \cite{FarguesScholze} that the Hecke operators $T_\mu$ preserve ULA objects;  the admissibility of $R\Gamma(G,b,\mu)[\rho]$ is deduced from this.

We learned from \cite{LuZheng} that the ULA condition is precisely the right hypothesis necessary to prove a Lefschetz-Verdier trace formula applicable to the cohomology of $A$. This explains the counterexample above: $j_!\Z_\ell$ fails to be ULA, where $j$ is the inclusion of the affinoid disc $X$ into its compactification $\overline{X}$. In fact \cite{LuZheng} is written in the context of schemes, but their formalism applies equally well in the context of rigid-analytic spaces and diamonds. Indeed, some interesting new phenomena occur in the diamond context.  For instance, if $H$ is a locally profinite group acting continuously on a proper diamond $X$, and $A\in D(X,\Z_\ell)$ is a ULA object which is $H$-equivariant, then $R\Gamma(X,A)$ is an admissible $H$-module.   One gets a formula for the {\em trace distribution} of $H$ acting on $R\Gamma(X,A)$, in terms of local terms living on the fixed-point locus in $H\times X$.  We explain the Lefschetz-Verdier trace formula for diamonds in \S\ref{SectionLefschetzVerdier}.

In \S\ref{SectionLocalTerms}, we study the Lefschetz-Verdier trace formula as it pertains to the mixed-characteristic affine Grassmannian $\Gr_{G,\leq \mu}$.  The object $\mc{S}_\mu$ is ULA on $\Gr_{G,\leq\mu}$ and $G(F)$-equivariant, so it makes sense to ask for its local term $\loc_g(x,A)$ at a fixed point $x$ of a regular element $g\in G(F)$.  (Such fixed points are all isolated.)  We found quickly that that result we needed for Theorem \ref{TheoremMain} would follow if we knew that $\loc_g(x,A)$ agreed with the na\"ive local term $\tr(g\vert A_x)$.  We asked Varshavsky, who devised a method for proving this agreement in the scheme setting.  We show how to deduce the required statement for $\Gr_{G,\leq\mu}$, using the Witt vector affine Grassmannian as a bridge between diamonds and schemes.  (We thank the referee for pointing out that an earlier argument we had here was incorrect.)

Finally, in \S\ref{SectionProofOfMainTheorem} we prove Theorem \ref{TheoremMain} by applying our trace formula to the Hecke stack $\Hecke_{G,b,\leq\mu}$.  An important step is to show that fixed points of elliptic elements $g\in G(F)$ acting on $\Gr_{G,\leq\mu}$ are admissible, as we observed above in the Lubin-Tate case.

\subsection*{Acknowledgments}
We are grateful to Peter Scholze for explaining to us some of the material on inertia stacks that appears in \S4.  We also thank Jean-Fran\c{c}ois Dat, Laurent Fargues, Martin Olsson, Jack Thorne, and Yakov Varshavsky for many helpful conversations.  Additionally, DH is grateful to Marie-France Vign\'eras for sharing a scanned copy of her book \cite{VignerasBook}. Finally, we are very grateful to the referee for their detailed comments and feedback on earlier versions of this paper.

%!TEX root = FixedPoints2022.tex
\section{Review of the objects appearing in Kottwitz's conjecture} \label{sec:stm}

\subsection{Basic notions}

Let $\breve{F}$ be the completion of the maximal unramified extension of $F$, and let $\sigma\in \Aut \breve{F}$ be the Frobenius automorphism. Let $G$ be a connected reductive group defined over $F$. Fix a quasi-split group $G^*$ and a $G^*(\overline{F})$-conjugacy class $\Psi$ of inner twists $G^*\to G$; thus elements $\psi \in \Psi$ are isomorphisms $G^*_{\overline{F}} \to G_{\overline{F}}$ such that for each $\tau \in \Gamma$ the automorphism $\psi^{-1}\circ\tau(\psi)$ of $G^*_{\overline{F}}$ is inner. Given an element $b\in G(\breve{F})$, there is an associated inner form $G_b$ of a Levi subgroup of $G^*$ as described in \cite[\S3.3,\S3.4]{Kot97}. Its group of $F$-points is given by
\[ G_b(F) \isom \set{ g \in G(\breve F)\; \big\vert \; \tx{Ad}(b)\sigma(g)=g }. \]
Up to isomorphism the group $G_b$ depends only on the $\sigma$-conjugacy class $[b]$. It will be convenient to choose $b$ to be decent \cite[Definition 1.8]{RapoportZink}. Then there exists a finite unramified extension $F'/F$ such that $b \in G(F')$. This allows us to replace $\breve F$ by $F'$ in the above formula. The slope morphism $\nu : \D \to G_{\breve F}$ of $b$, \cite[\S4]{KottwitzIsocrystals}, is also defined over $F'$. The centralizer $G_{F',\nu}$ of $\nu$ in $G_{F'}$ is a Levi subgroup of $G_{F'}$. The $G(F')$-conjugacy class of $\nu$ is defined over $F$, and then so is the $G(F')$-conjugacy class of $G_{F',\nu}$. There is a Levi subgroup $M^*$ of $G^*$ defined over $F$ and  $\psi \in \Psi$ that restricts to an inner twist $\psi : M^* \to G_b$, see \cite[\S4.3]{Kot97}.

From now on assume that $b$ is basic.  This is equivalent to $M^*=G^*$, so that $G_b$ is in fact an inner form of $G^*$ and of $G$. Furthermore, $\Psi$ is an equivalence class of inner twists $G^* \to G$ as well as $G^* \to G_b$. This identifies the dual groups of $G^*$, $G$, and $G_b$, and we write $\hat G$ for either of them. 

 Let $\phi : W_F \times \tx{SL}_2(\C) \to {^LG}$ be a discrete Langlands parameter and let $S_\phi=\tx{Cent}(\phi,\hat G)$. For $\lambda\in X^*(Z(\hat{G})^\Gamma)$ write $\tx{Rep}(S_\phi,\lambda)$ for the set of isomorphism classes of algebraic representations of the algebraic group $S_\phi$ whose restriction to $Z(\hat{G})^\Gamma$ is $\lambda$-isotypic, and write $\tx{Irr}(S_\phi,\lambda)$ for the subset of irreducible such representations. The class of $b$ corresponds to a character $\lambda_b : Z(\hat G)^\Gamma \to \C^\times$ via the isomorphism $B(G)_\tx{bas} \to X^*(Z(\hat G)^\Gamma)$ of \cite[Proposition 5.6]{KottwitzIsocrystals}. Assuming the validity of the refined local Langlands conjecture \cite[Conjecture G]{KalethaLocalLanglands} we will construct in the following two subsections for any $\pi \in \Pi_\phi(G)$ and $\rho \in \Pi_\phi(G_b)$ an element $\delta_{\pi,\rho} \in \tx{Rep}(S_\phi,\lambda_b)$ that measures the relative position of $\pi$ and $\rho$.

\subsection{Construction of $\delta_{\pi,\rho}$ in a special case} \label{sub:binner}

The statements of the Kottwitz conjecture given in \cite[Conjecture 5.1]{RapoportNonArchimedean} and \cite[Conjecture 7.3]{RapoportViehmann} make the assumption that $G$ is a $B$-inner form of $G^*$. In that case, the construction of $\delta_{\pi,\rho}$ is straightforward and we shall now recall it.

The assumption on $G$ means that some $\psi \in \Psi$ can be equipped with a decent basic $b^* \in G^*(F^\tx{nr})$ such that $\psi$ is an isomorphism $G^*_{F^\tx{nr}} \to G_{F^\tx{nr}}$ satisfying $\psi^{-1}\sigma(\psi)=\tx{Ad}(b^*)$. In other words, $\psi$ becomes an isomorphism over $F$ from the group $G^*_{b^*}$ to $G$. Under this assumption, and after choosing a Whittaker datum $\mf{w}$ for $G^*$, the isocrystal formulation of the refined local Langlands correspondence \cite[Conjecture F]{KalethaLocalLanglands}, which is implied by the rigid formulation \cite[Conjecture G]{KalethaLocalLanglands} according to \cite{KalRIBG}, predicts the existence of bijections
\begin{eqnarray*}
\Pi_{\phi}(G) &\isom & \Irr(S_{\phi},\lambda_{b*})\\
\Pi_{\phi}(G_b) &\isom & \Irr(S_{\phi},\lambda_{b^*}+\lambda_b)
\end{eqnarray*}
where we have used the isomorphisms $B(G)_{\bas}\isom X^*(Z(\hat{G})^{\Gamma}) \cong B(G^*)_\tx{bas}$ of \cite[Proposition 5.6]{KottwitzIsocrystals} to obtain from $[b] \in B(G)_\tx{bas}$ and $[b^*] \in B(G^*)_\tx{bas}$ characters $\lambda_{b}$ and $\lambda_{b^*}$ of $Z(\hat G)^\Gamma$.

These bijections are uniquely characterized by the endoscopic character identities which are part of \cite[Conjecture F]{KalethaLocalLanglands}.  Write $\pi\mapsto \tau_{b^*,\mf{w},\pi}$, $\rho\mapsto \tau_{b^*,\mf{w},\rho}$ for these bijections and define
\begin{equation} \delta_{\pi,\rho} := \check\tau_{b^*,\mf{w},\pi} \otimes\tau_{b^*,\mf{w},\rho}. \end{equation}
While these bijections depend on the choice of Whittaker datum $\mf{w}$ and the choice of $b^*$, we will argue in Subsection \ref{sub:geninner} that for any pair $\pi$ and $\rho$ the representation $\delta_{\pi,\rho}$ is independent of these choices. Of course it does depend on $b$, but this we take as part of the given data.

\subsection{Construction of $\delta_{\pi,\rho}$ in the general case} \label{sub:geninner}

We now drop the assumption that $G$ is a $B$-inner form of $G^*$. Because of this, we no longer have the isocrystal formulation of the refined local Langlands correspondence. However, we do have the formulation based on rigid inner twists \cite[Conjecture G]{KalethaLocalLanglands}. What this means with regards to the Kottwitz conjecture is that neither $\pi$ nor $\rho$ correspond to representations of $S_\phi$. Rather, they correspond to representations $\tau_\pi$ and $\tau_\rho$ of a different group $\pi_0(S_\phi^+)$. Nonetheless it will turn out that $\check\tau_\pi \otimes \tau_\rho$ provides in a natural way a representation $\delta_{\pi,\rho}$ of $S_\phi$.

In order to make this precise we will need the material of \cite{KalRI} and \cite{KalRIBG}, some of which is summarized in \cite{KalethaLocalLanglands}. First, we will need the cohomology set $H^1(u \to W,Z \to G^*)$ defined in \cite[\S3]{KalRI} for any finite central subgroup $Z \subset G^*$ defined over $F$. As in \cite[\S3.2]{KalRIBG} it will be convenient to package these sets for varying $Z$ into the single set
\[ H^1(u \to W,Z(G^*) \to G^*) := \varinjlim H^1(u \to W,Z \to G^*). \]
The transition maps on the right are injective, so the colimit can be seen as an increasing union.

Next, we will need the reinterpretation, given in \cite{KotBG}, of $B(G)$ as the set of cohomology classes of algebraic 1-cocycles of a certain Galois gerbe $1 \to \D(\bar F) \to \mc{E} \to \Gamma \to 1$. This reinterpretation is also reviewed in \cite[\S3.1]{KalRIBG}. For this, we recall that inflation along $W_F \to \Z$ induces an isomorphism between $B(G)=H^1(\<\sigma\>,G(L))$ and $H^1(W_F,G(\bar L))$, where we have written $L=\breve F$ to ease typestting. In \cite[App B]{Kot97} Kottwitz constructs a continuous homomorphism $W_F \to \mc{E}$ whose composition with the natural projection $\mc{E} \to \Gamma$ is the natural map $W_F \to \Gamma$. He proves in \cite[\S8 and App B]{Kot97} that pulling back along this homomorphism and pushing along the inclusion $G(\bar F) \to G(\bar L)$ gives an isomorphism $H^1_\tx{alg}(\mc{E},G(\bar F)) \to B(G)$, and in particular $H^1_\tx{bas}(\mc{E},G(\bar F)) \to B_\tx{bas}(G)$. While the section $W_F \to \mc{E}$ is not completely canonical, the induced isomorphism on cohomology is independent of the choice of section. Strictly speaking, Kottwitz gives the proof only in the case of tori, but the general case is immediate from that.

Finally, we will need the comparison map
\[ H^1_\tx{bas}(\mc{E},G(\bar F)) \to H^1(u \to W,Z(G) \to G) \]
of \cite[\S3.3]{KalRIBG}. 

After this short review we turn to the construction of $\delta_{\pi,\rho} \in \Rep(S_\phi,\lambda_b)$. For this, it is not enough to work with the cohomology class of $b$, because $\delta_{\pi,\rho}$ is an invariant of the equivalence class of the triple $(b,\pi,\rho)$, and changing $b$ within its cohomology class must be accompanied with a corresponding change in $\rho$. Therefore we must work with cocycles. 

To that end, fix the section $W_F \to \mc{E}$. If $z_b \in Z^1_\tx{bas}(\mc{E},G(\bar F))$ denotes a representative of the element of $H^1_\tx{bas}(\mc{E},G(\bar F))$ corresponding to the class of $b$, there exists $g \in G(\bar L)$, unique up to right multiplication by elements of $G_b(F)$, such that
\begin{equation} \label{eq:zb}
g^{-1} z_b(w)w(g) = b \cdot \sigma(b) \cdots \sigma^{|w|-1}(b) \qquad \forall w \in W_F \to \mc{E},    
\end{equation}
where $|w|$ is the image of $w$ under $W_F \to \Z$. Note that the image of $g$ in $G_\tx{ad}(\bar L)$ lies in $G_\tx{ad}(\bar F)$ and that $\tx{Ad}(g)$ induces an $F$-isomorphism $G_{z_b} \to G_b$. Therefore $\rho \circ \tx{Ad}(g)$ is an irreducible representation of $G_{z_b}(F)$ whose isomorphism class does not depend on the choice of $g$.

Choose any inner twist $\psi \in \Psi$ and let $\bar z_\sigma := \psi^{-1}\sigma(\psi) \in G^*_\tx{ad}(\ol{F})$. Then $\bar z \in Z^1(F,G^*_\tx{ad})$ and the surjectivity of the natural map $H^1(u \to W,Z(G^*) \to G^*) \to H^1(F,G^*_\tx{ad})$ asserted in \cite[Corollary 3.8]{KalRI} allows us to choose $z \in Z^1(u \to W,Z(G^*) \to G^*)$ lifting $\bar z$. Then $(\psi,z) : G^* \to G$ is a rigid inner twist, and $(\psi,\psi^{-1}(z)\cdot z_b) : G^* \to G_{z_b}$ is also a rigid inner twist.

The $L$-packets $\Pi_\phi(G)$ and $\Pi_\phi(G_{z_b})$ are now parameterized by representations of a certain cover $S_\phi^+$ of $S_\phi$. While \cite[Conjecture G]{KalethaLocalLanglands} is formulated in terms of a finite cover depending on an auxiliary choice of a finite central subgroup $Z \subset G^*$, we will adopt here the point of view of \cite{KalRIBG} and work with a canonical infinite cover, namely the preimage of $S_\phi$ in the universal cover of $\hat G$. Following \cite[\S3.3]{KalRIBG} we can present this universal cover as follows. Let $Z_n \subset Z(G)$ be the subgroup of those elements whose image in $Z(G)/Z(G_\tx{der})$ is $n$-torsion, and let $G_n=G/Z_n$. Then $G_n$ has adjoint derived subgroup and connected center. More precisely, $G_n=G_\tx{ad} \times C_n$, where $C_n=C_1/C_1[n]$ and $C_1=Z(G)/Z(G_\tx{der})$. It is convenient to identify $C_n=C_1$ as algebraic tori and take the $m/n$-power map $C_1 \to C_1$ as the transition map $C_n \to C_m$ for $n|m$. The isogeny $G \to G_n$ dualizes to $\hat G_n \to \hat G$ and we have $\hat G_n = \hat G_\tx{sc} \times \hat C_1$. Note that $\hat C_1=Z(\hat G)^\circ$. The transition map $\hat G_m \to \hat G_n$ is then the identity on $\hat G_\tx{sc}$ and the $m/n$-power map on $\hat C_1$. Set $\hat{\bar G} = \varprojlim \hat G_n = \hat G_\tx{sc} \times \hat C_\infty$, where $\hat C_\infty = \varprojlim \hat C_n$. Then $\hat{\bar G}$ is the universal cover of $\hat G$. Elements of $\hat{\bar G}$ can be written as $(a,(b_n)_n)$, where $a \in \hat G_\tx{sc}$ and $(b_n)_n$ is a sequence of elements $b_n \in \hat C_1$ satisfying $b_n=(b_m)^{\frac{m}{n}}$ for $n|m$. In this presentation, the natural map $\hat{\bar G} \to \hat G$ sends $(a,(b_n))$ to $a_\tx{der} \cdot b_1$, where $a_\tx{der} \in \hat G_\tx{der}$ is the image of $a \in \hat G_\tx{sc}$ under the natural map $\hat G_\tx{sc} \to \hat G_\tx{der}$.

\begin{dfn}\label{DefinitionOfSphiplus} Let $Z(\hat{\bar G})^+ \subset S_\phi^+ \subset \hat{\bar G}$ be the preimages of $Z(\hat G)^\Gamma \subset S_\phi \subset \hat G$ under $\hat{\bar G}\to \hat G$.
\end{dfn}

Given a character $\lambda : \pi_0(Z(\hat{\bar G})^+) \to \C^\times$ (which we will always assume trivial on the kernel of $Z(\hat{\bar G})^+ \to \hat G_n$ for some $n$) let $\Rep(\pi_0(S_\phi^+),\lambda)$ denote the set of isomorphism classes of representations of $\pi_0(S_\phi^+)$ whose pull-back to $\pi_0(Z(\hat{\bar G})^+)$ is $\lambda$-isotypic, and let $\Irr(\pi_0(S_\phi^+),\lambda)$ be the (finite) subset of irreducible representations. Let $\lambda_z$ be the character corresponding to the class of $z$ under the Tate-Nakayama isomorphism
\[ H^1(u \to W,Z(G^*) \to G^*) \to \pi_0(Z(\hat{\bar G})^+)^* \]
of \cite[Corollary 5.4]{KalRI}, and let $\lambda_{z_b}$ be the character corresponding to the class of $z_b$ in $H^1(u \to W,Z(G) \to G)$. Then according to \cite[Conjecture G]{KalethaLocalLanglands}, upon fixing a Whittaker datum $\mf{w}$ for $G^*$ there are bijections
\begin{eqnarray*}
\Pi_{\phi}(G) &\isom & \Irr(\pi_0(S_{\phi}^+),\lambda_{z})\\
\Pi_{\phi}(G_{z_b}) &\isom & \Irr(\pi_0(S_{\phi}^+),\lambda_{z}+\lambda_{z_b})
\end{eqnarray*}
again uniquely determined by the endoscopic character identities. We write $\pi\mapsto \tau_{z,\mf{w},\pi}$, $\rho\mapsto \tau_{z,\mf{w},\rho}$ for these bijections, and $\tau\mapsto \pi_{z,\mf{w},\tau}$, $\tau\mapsto \rho_{z,\mf{w},\tau}$ for their inverses. We form the representation $\check\tau_{z,\mf{w},\pi} \otimes\tau_{z,\mf{w},\rho} \in \Rep(\pi_0(S_\phi^+),\lambda_{z_b})$, where we are identifying $\rho$ with the representation $\rho\circ\tx{Ad}(g)$ of $G_{z_b}(F)$.

Recall the map \cite[(4.7)]{KalRIBG}
\begin{equation} \label{eq:rivsbgllc} S_\phi^+ \to S_\phi,\qquad (a,(b_n)) \mapsto \frac{a_\tx{der} \cdot b_1}{N_{E/F}(b_{[E:F]})}. \end{equation}
Here $a_\tx{der} \in \hat G_\tx{der}$ is the image of $a \in \hat G_\tx{sc}$ under the natural map $\hat G_\tx{sc} \to \hat G_\tx{der}$ and $E/F$ is a sufficiently large finite Galois extension. This map is independent of the choice of $E/F$. According to \cite[Lemma 4.1]{KalRIBG} pulling back along this map defines a natural bijection $\Irr(\pi_0(S_\phi^+),\lambda_{z_b}) \cong \Irr(S_\phi,\lambda_b)$. Note that since $\phi$ is discrete the group $S_\phi^\natural$ defined in loc. cit. is equal to $S_\phi$. The lemma remains valid, with the same proof, if we remove the requirement of the representations being irreducible, and we obtain the bijection $\Rep(\pi_0(S_\phi^+),\lambda_{z_b}) \to \Rep(S_\phi,\lambda_b)$.
\begin{dfn}
\label{DeltaPiRho}
Let $\delta_{\pi,\rho}$ be the image of $\check\tau_{z,\mf{w},\pi} \otimes\tau_{z,\mf{w},\rho}$ under the bijection  $\Rep(\pi_0(S_\phi^+),\lambda_{z_b}) \to \Rep(S_\phi,\lambda_b)$.
\end{dfn}
 In the situation when $G$ is a $B$-inner form of $G^*$, this definition of $\delta_{\pi,\rho}$ agrees with the one of Subsection \ref{sub:binner}, because then we can obtain $z$ from $b^*$ just like we obtained $z_b$ from $b$, and then $\tau_{z,\mf{w},\pi}$ and $\tau_{b^*,\mf{w},\pi}$ are related via \eqref{eq:rivsbgllc}, and so are $\tau_{z,\mf{w},\rho}$ and $\tau_{b^*,\mf{w},\rho}$, see \cite[\S4.2]{KalRIBG}.

\begin{lem} \label{lem:deltaindep}  Assume \cite[Conjecture G]{KalethaLocalLanglands}.
The representation $\delta_{\pi,\rho}$ is independent of the choices of Whittaker datum $\mf{w}$ and of a rigidifying 1-cocycle $z \in Z^1(u \to W,Z(G^*) \to G^*)$.
\end{lem}
\begin{proof}
Both of these statements follow from \cite[Conjecture G]{KalethaLocalLanglands}. For the independence of Whittaker datum, one can prove that the validity of this conjecture implies that if $\mf{w}$ is replaced by another choice $\mf{w'}$ then there is an explicitly constructed character $(\mf{w},\mf{w'})$ of $\pi_0(S_\phi/Z(\hat G)^\Gamma)$ whose inflation to $\pi_0(S_\phi^+)$ satisfies $\tau_{z,\mf{w},\sigma}=\tau_{z,\mf{w'},\sigma} \otimes (\mf{w},\mf{w'})$ for any $\sigma \in \Pi_\phi(G) \cup \Pi_\phi(G_b)$. See \S4 and in particular Theorem 4.3 of \cite{KalGen}, the proof of which is valid for a general $G$ that satisfies \cite[Conjecture G]{KalethaLocalLanglands}, bearing in mind that the transfer factor we use here is related to the one used there by $s \mapsto s^{-1}$. The independence of $z$ follows from the same type of argument, but now using \cite[Lemma 6.2]{KalRIBG}.
\end{proof}

\subsection{Spaces of local shtukas and their cohomology} \label{sub:shtcoh}

We recall here some material from \cite{ScholzeLectures} and \cite{FarguesGeometrization} regarding the Fargues-Fontaine curve and moduli spaces of local shtukas. 

Let $k$ be the residue field of $F$.  For a perfectoid space $S$ over $k$, we have the Fargues-Fontaine curve $X_S$ \cite{FarguesFontaineCurve}, an adic space over $F$.   For $S=\Spa(R,R^+)$ affinoid with pseudouniformizer $\varpi$, the adic space $X_S$ is defined as follows:
\begin{eqnarray*}
 Y_S &=& (\Spa W_{\OO_F}(R^+))\backslash\set{p[\varpi]=0} \\
 X_S &=& Y_S/\Frob^\Z.
 \end{eqnarray*}
Here $\Frob$ is the $q$th power Frobenius on $S$. 

For an affinoid perfectoid space $S$ lying over the residue field of $F$, the following sets are in bijection:
\begin{enumerate}
\item $S$-points of $\Spd F$
\item Untilts $S^\sharp$ of $S$ over $F$,
\item Cartier divisors of $Y_S$ of degree 1.
\end{enumerate}
Given an untilt $S^\sharp$, we let $D_{S^\sharp}\subset Y_S$ be the corresponding divisor.  If $S^\sharp=\Spa(R^\sharp,R^{\sharp +})$ is affinoid, then the completion of $Y_S$ along $D_{S^\sharp}$ is $\Spf B_{\dR}^+(R^\sharp)$, where $B_{\dR}^+(R^\sharp)$ is the de Rham period ring attached to the perfectoid algebra $R^\sharp$.  The untilt $S^\sharp$ determines a Cartier divisor on $X_S$, which we still refer to as $D_{S^\sharp}$.

There is a functor $b \mapsto \mc{E}^b$ from the category of isocrystals with $G$-structure to the category of $G$-bundles on $X_S$ (for any $S$). When $S$ is a geometric point this functor induces a bijection between the sets of isomorphism classes \cite{FarguesGBundles}.

We now recall Scholze's definition of the local shtuka space. It is a set-valued functor on the pro-\'etale site of perfectoid spaces over $\F_p$ and is equipped with a morphism to $\Spd C$. Thus it can be described equivalently as a set-valued functor on the pro-\'etale site of perfectoid spaces over $C$.

\begin{dfn}
\label{DefinitionMGBMu} The local shtuka space $\Sht_{G,b,\mu}$ inputs a perfectoid $C$-algebra $(R,R^+)$, and outputs the set of isomorphisms
\[ \gamma\from \E^1\vert_{X_{R^\flat}\backslash D_R} \isom \E^b\vert_{X_{R^\flat}\backslash D_R} \]
of $G$-torsors that are meromorphic along $D_R$ and bounded by $\mu$ pointwise on $\Spa R$. 
\end{dfn}

Let us briefly recall the condition of being pointwise bounded by $\mu$. If $\Spa(C,O_C) \to \Spa R$ is a geometric point we obtain via pull-back $\gamma : \mc{E}^1|_{X_{C^\flat} \sm \{x_C\}} \to \mc{E}^b|_{X_{C^\flat} \sm \{x_C\}}$, where we have written $x_C$ in place of $D_C$ to emphasize that this a point on $X_{C^\flat}$. The completed local ring of $X_{C^\flat}$ at $x_C$ is Fointaine's ring $B_{\dR}^+(C)$. A trivialization of both bundles $\mc{E}^1$ and $\mc{E}^b$ on a formal neighborhood of $x_C$, together with $\gamma$, leads to an element of $G(B_{\dR}(C))$, well defined up to left and right multiplication by elements of $G(B_{\dR}^+(C))$. The corresponding element of the double coset space $G(B_{\dR}^+(C)) \lmod G(B_{\dR}(C)) / G(B_{\dR}^+(C))$ is indexed by a conjugacy class of co-characters of $G/C$ according to the Cartan decomposition, and we demand that this conjugacy class is dominated by $\mu$ in the usual order (given by the simple roots of the universal Borel pair).

The space $\Sht_{G,b,\mu}$ is a locally spatial diamond \cite[\S23]{ScholzeLectures}. Since the automorphism groups of $\mc{E}^1$ and $\mc{E}^b$ are the constant group diamonds $\ul{G(\Q_p)}$ and $\ul{G_b(\Q_p)}$ respectively, the space $\Sht_{G,b,\mu}$ is equipped with commuting actions of $G(\Q_p)$ and $G_b(\Q_p)$, acting by pre- and post-composition on $\gamma$.

\begin{rmk}
According to \cite[Corollary 23.2.2]{ScholzeLectures} the above definition recovers the moduli space of local shtukas with one leg and infinite level structure.  We have dropped the subscript $\infty$ used in \cite{ScholzeLectures} to denote the infinite level structure.
\end{rmk}

We will use the cohomology theory developed in \cite{ScholzeEtaleCohomology}. For any compact open subgroup $K \subset G(F)$ the quotient $\Sht_{G,b,\mu,K}=\Sht_{G,b,\mu}/K$ is again a locally spatial diamond \cite[\S23]{ScholzeLectures}.   For each $n=1,2,\dots$, let $V_{\mu,n}\in\Rep(\hat{G},\Z/\ell^n\Z)$ be the Weyl module associated to $\mu$.  By the geometric Satake equivalence (Theorem \ref{ThmSatakeEquivalence}), there is a corresponding object $\mc{S}_{\mu,n}$ of $D_{\et}(\Gr_{G,b,\leq\mu},\Z/\ell^n\Z[\sqrt{q}])$.  Define
\[ R\Gamma_c(\Sht_{G,b,\mu}/K,\mc{S}_{\mu}) = \varinjlim_U R\Gamma_c(U,\mc{S}_{\mu}), \]
where $U\subset \Sht_{G,b,\mu}/K$ runs over quasicompact open subsets, and where we have put
\[ R\Gamma_c(U,\mc{S}_{\mu})=\varprojlim_n R\Gamma_c(U,\mc{S}_{\mu,n}). \]
Then $R\Gamma_c(\Sht_{G,b,\mu}/K,\mc{S}_{\mu})$ is a complex of $\Z_\ell[\sqrt{q}]$-modules carrying an action of $G_b(F)\times W_E$.  

\begin{dfn} \label{dfn:stcoh}
Let $\rho$ be a finite-length admissible representation of $G_b(F)$ with coefficients in $\overline{\mathbf{Q}_{\ell}}$. Then we define
\[ R\Gamma(G,b,\mu)[\rho]=
\varinjlim_{K\subset G(F)} R\Hom_{G_b(F)}(R\Gamma_c(\Sht_{G,b,\mu}/K,\mc{S}_\mu) \otimes \overline{\mathbf{Q}_{\ell}}, \rho),\]
where $K$ runs over the set of open compact subgroups of $G(F)$. 
\end{dfn}

By Proposition \ref{PropHInTermsOfHecke} below, this defines a finite-length $W_E$-equivariant object in the derived category of smooth representations of $G(F)$ with coefficients in $\overline{\mathbf{Q}_{\ell}}$, and we write $\Mant_{b,\mu}(\rho)$ for the image of $R\Gamma(G,b,\mu)[\rho]$ in $\Groth(G(F) \times W_E)$.

\begin{rmk} We now discuss the relationship between our definition of $\Mant_{b,\mu}(\rho)$ and the virtual representation $H^*(G,b,\mu)[\rho]$ defined in \cite{RapoportViehmann}.

When $\mu$ is minuscule, $\Sht_{G,b,\mu,K}$ is the diamond $\M_{G,b, \mu,K}^{\diamond}$ associated to the local Shimura variety $\M_{G,b,\mu,K}$ \cite[\S 24.1]{ScholzeLectures}.   The latter is a rigid-analytic variety of dimension $d=\class{\mu,2\rho_G}$, where $2\rho_G$ is the sum of the positive roots.  Moreover, in that case, $\mc{S}_\mu=\mathbf{Z}_{\ell}[\sqrt{q}][d](\tfrac{d}{2})$ is a shift and twist of the constant sheaf. In \cite{RapoportViehmann}, $H^*(G,b,\mu)[\rho]$ is defined as the alternating sum
\[ \sum_{i,j\in\Z} (-1)^{i+j} H^{i,j}(G,b,\mu)[\rho](-d),\]
where
\[ H^{i,j}(G,b,\mu)[\rho]=\varinjlim_K \tx{Ext}^i_{G_b(F)}(H^{j}_{c}(\M_{G,b,\mu,K},\Z_\ell) \otimes \overline{\mathbf{Q}_{\ell}},\rho)
\]
Note that $H^{i,j}(G,b,\mu)[\rho]$ vanishes for all but finitely many $(i,j)$ and each $H^{i,j}(G,b,\mu)[\rho]$ is an admissible representation of $G_b(F)$, by the analysis in \cite{FarguesScholze}. On the other hand, unwinding definitions, we see that there is a spectral sequence $H^{i,j-d}(G,b,\mu)[\rho](-\tfrac{d}{2}) \implies H^{i+j}\left( R\Gamma(G,b,\mu)[\rho] \right)$. 

Putting these observations together, we get the equality \[ \Mant_{b,\mu}(\rho) = (-1)^d H^*(G,b,\mu)[\rho](\tfrac{d}{2}).\] Note that in our formulation, the Tate twist appearing in \cite[Conjecture 7.3]{RapoportViehmann} has been absorbed into the normalization of $\Mant_{b,\mu}$.
\end{rmk}

%!TEX root = FixedPoints2022.tex
\section{Transfer of conjugation-invariant functions from $G(F)$ to $G_b(F)$}
\label{SectionTransferOfFunctions}
Throughout, $F/\Q_p$ is a finite extension, and $G/F$ is a connected reductive group.  

\subsection{The space of strongly regular conjugacy classes in $G(F)$}

The following definitions are important for our work.

\begin{itemize}
	\item $G_{\rs}\subset G$ is the open subvariety of regular semisimple elements, meaning those whose connected centralizer is a maximal torus.
	\item $G_{\sr}\subset G$ is the open subvariety of strongly regular semi-simple elements, meaning those regular semisimple elements whose centralizer is connected, i.e. a maximal torus;
	\item $G(F)_{\elli}\subset G(F)$ is the open subset of strongly regular elliptic elements, meaning those strongly regular semisimple elements in $G(F)$ whose centralizer is an elliptic maximal torus.
\end{itemize}
We put $G(F)_{\sr}=G_{\sr}(F)$ and $G(F)_{\rs}=G_{\rs}(F)$.  Note that $G(F)_{\elli}\subset G(F)_{\sr}\subset G(F)_{\rs}$.   The inclusion $G(F)_{\sr}\subset G(F)_{\rs}$ is dense.

If $g$ is regular semisimple, then it is necessarily contained in a unique maximal torus $T$, namely the neutral component $\Cent(g,G)^\circ$, but this is not necessarily all of $\Cent(g,G)$.  If $G_\tx{der}$ is simply connected, then $\Cent(g,G)$ is connected;  thus in such a group, regular semisimple and strongly regular semisimple mean the same thing.

Observe that if $g$ is regular semisimple, then $\alpha(g)\neq 1$ for all roots $\alpha$ relative to the action of $T$.  Indeed, if $\alpha(g)=1$, then the root subgroup of $\alpha$ would commute with $g$, and then it would have dimension strictly greater than $\dim T$.

All of the sets $G(F)_{\sr}$, $G(F)_{\rs}$, $G(F)_{\elli}$ are conjugacy-invariant, so we may for instance consider the quotient
$G(F)_{\sr}\sslash G(F)$, considered as a topological space.

\begin{lem} \label{LemLocallyProfinite} $G(F)_{\rs}\sslash G(F)$ is locally profinite, in fact equal to the disjoint union of the locally profinite sets $T(F)_{\rs}/N(T,G)(F)$, where $T$ runs over the set of $G(F)$-conjugacy classes of $F$-rational maximal tori in $G$ and $N(T,G)$ is the normalizer of $T$ in $G$. The same is true with ``rs'' replaced by ``sr''.
\end{lem}

\begin{proof}  Let $T\subset G$ be a $F$-rational maximal torus.  The set $H^1(F,N(T,G))$ classifies conjugacy classes of $F$-rational tori, as follows:  given a $F$-rational torus $T'$, we must have $T'=xTx^{-1}$ for some $x\in G(\overline{F})$.  Then for all $\sigma\in \Gal(\overline{F}/F)$, $x^{-1} x^\sigma$ normalizes $T$.  We associate to $T'$ the class of $\sigma\mapsto x^{-1}x^{\sigma}$ in $H^1(F,N(T,G))$, and it is a simple matter to see that this defines a bijection as claimed.  (In fact $H^1(F,N(T,G))$ is finite.)

There is a map $G(F)_{\rs}\sslash G(F)\to H^1(F,N(T,G))$, sending the conjugacy class of $g\in G(F)_{\rs}$ to the conjugacy class of the unique $F$-rational torus containing it, namely $\Cent(g,G)^\circ$.  We claim that this map is locally constant.

To prove the claim, we consider
\[ \varphi : G(F) \times T_{\rs}(F) \to G_{\rs}(F),  \;\;\; (g,t)\mapsto gtg^{-1}, \]
a morphism of $p$-adic analytic varieties.  We would like to show that $\varphi$ is open.  To do this, we will compute its differential at the point $(g,t)$, by means of a change of variable.  Consider the map
\[ \psi = L_{gtg^{-1}}^{-1}\circ\varphi \circ (L_g \times L_t).\]
Explicitly, for $(z,w) \in G(F) \times T(F)$ we have  $\psi(z,w)=gt^{-1}ztwz^{-1}g^{-1}$.

Let $\mf{g}=\Lie G$, $\mf{t}=\Lie T$.  The derivative $d\psi(1,1) : \mf{g} \times \mf{t} \to \mf{g}$ is given by the formula
\[d\psi(1,1)(Z,W)=\tx{Ad}(g)[ (\tx{Ad}(t^{-1})-\tx{id})Z+W].\]
We would like to check that $d\psi(1,1)$ is surjective.  We may decompose $\mf{g}=\mf{t} \oplus \mf{t}^\perp$, where $\mf{t}^\perp$ is the descent to $F$ of the direct sum of all root subspaces of $\mf{g}_{\overline{F}}$ for the action of $T$.

The element $t$ is regular, hence $\alpha(t)\neq 1$ for all roots of $\mf{g}$ for the action of $T$.  Therefore $\tx{Ad}(t^{-1})-\tx{id} : \mf{g}/\mf{t} \to \mf{t}^\perp$ is an isomorphism.   It follows that $d\psi$ is surjective.  The derivative of $\varphi$ at $(g,t)$ is
\[ d\varphi(g,t)=dL_{gtg^{-1}}(gtg^{-1})\circ d\psi(1,1)\circ (dL_g(1) \times dL_t(1)).\]

All terms $dL$ are isomorphisms, so $d\varphi(g,t)$ is also surjective. Thus $\varphi$ is a submersion in the sense of Bourbaki VAR \S5.9.1, hence it is open by loc. cit. \S5.9.4.

Therefore if $g\in T(F)_{\rs}$ and $g'$ is sufficiently close to $g$ in $G(F)$, then $g'$ is conjugate in $G(F)$ to an element of $T(F)$, which proves the claim about the local constancy of $G(F)_{\rs}\sslash G(F)\to H^1(F,N(T,G))$.

The fiber of this map over $T'$ is $T'(F)_{\rs}$ modulo the action of the finite group $N(T',G)(F)/T'(F)$.   Since $T'(F)_{\rs}$ is locally profinite, so is its quotient by the action of a finite group.
\end{proof}

\subsection{Hecke transfer maps}

Suppose that $b\in G(\breve{F})$ is basic. The goal of this section is to define a family of explicit maps, which input a conjugation-invariant function on $G(F)_{\sr}$ and output a conjugation-invariant function on $G_b(F)_{\sr}$. We shall call them Hecke transfer maps, as a way of foreshadowing their relation to the Hecke operators defined on the stack $\Bun_G$.

Given a sufficiently strong version of the local Langlands conjectures, we will show that the Hecke transfer maps act predictably on the trace characters attached to irreducible admissible representations.

We begin by recalling the concept of related elements and the definition of their invariant in the isocrystal setting from \cite{KalIso}.

\begin{lem} 
\label{LemmaSteinberg}
Suppose $g\in G(F)$ and $g'\in G_b(F)$ are strongly regular elements which are conjugate over an algebraic closure of $\breve{F}$. Then they are conjugate over $\breve F$.
\end{lem}

\begin{proof}  Let $K$ be an algebraic closure of $\breve{F}$.  Say $g'=zgz^{-1}$ with $z\in G(K)$.  Let $T=\Cent(g,G)$;  then for all $\tau$ in the inertia group $ \Gal(\overline{F}/F^{\nr})$, $z^{-\tau} z$ commutes with $g$ and therefore lies in $T(K)$.   Then $\tau\mapsto z^{-\tau} z$ is a cocycle in $H^1(\breve{F},T)$.  Since $T$ is a connected algebraic group, $H^1(\breve{F},T)=0$ \cite[Theorem 1.9]{SteinbergRegularElements}.  If $x\in T(K)$ splits the cocycle, then $y=zx^{-1}\in G(\breve{F})$, and $g'=ygy^{-1}$, so that $g$ and $g'$ are related.
\end{proof}

It is customary to call elements $g,g'$ as in the above lemma \emph{stably conjugate}, or \emph{related}. Suppose we have strongly regular elements $g\in G(F)_{\sr}$ and $g'\in G_b(F)_{\sr}$ which are related.  Let $T=\Cent(g,G)$, and suppose $y\in G(\breve{F})$ with $g'=ygy^{-1}$.  The rationality of $g$ means that $g^\sigma=g$, whereas the rationality of $g'$ in $G_b$ means that $(g')^\sigma = b^{-1}g' b$.  Combining these statements shows that $b_0:=y^{-1}by^{\sigma}$ commutes with $g$ and therefore lies in $T(\breve{F})$.
%% Steps:\begin{eqnarray*}
%ygy^{-1} &=& j \\
%&=& b(j)^{\sigma}b^{-1} \\
%&=& b(y g y^{-1})^\sigma b^{-1} \\
%&=& by^\sigma g y^{-\sigma}b^{-1}.
%\end{eqnarray*}

\begin{dfn} \label{dfn:inv}
For strongly regular related elements $g\in G(F)_{\sr}$ and $g'\in G_b(F)_{\sr}$, the invariant $\inv[b](g,g')$ is the class of $y^{-1}by^{\sigma}$ in $B(T)$, where $y\in G(\breve{F})$ satisfies $g'=ygy^{-1}$. 
\end{dfn}

\begin{fct} \label{fct:inv}
The invariant $\inv[b](g,g')\in B(T)$ only depends on $b$, $g$, and $g'$ and not on the element $y$ which conjugates $g$ into $g'$.  It depends on the rational conjugacy classes of $g$ and $g'$ as follows:
\begin{itemize}
\item For $z\in G(F)$ we have $\inv[b]((\ad z)(g),g') = (\ad z)(\inv[b](g,g'))$, a class in $B((\ad z)(T))$.
\item For $z\in G_b(F)$ we have $\inv[b](g,(\ad z)(g')) = \inv[b](g,g')$.
\end{itemize}
The image of $\tx{inv}[b](g,g')$ under the composition of $B(T) \to B(G)$ and $\kappa\from B(G) \to \pi_1(G)_\Gamma$ equals $\kappa(b)$.
\end{fct}

\begin{dfn} \label{DfnRelB}
We define a diagram of topological spaces
\begin{equation}
\label{Relation}
\xymatrix{
& \Rel_b \ar[dl] \ar[dr] & \\
G(F)_{\sr}\stslash G(F) && G_b(F)_{\sr}\stslash G_b(F).
}
\end{equation}
as follows.  The space $\Rel_b$ is the set of conjugacy classes of triples $(g,g',\lambda)$, where $g\in  G(F)_{\sr}$  and $g'\in G_b(F)_{\sr}$ are related, and $\lambda\in X_*(T)$, where $T=\Cent(g,G)$.  It is required that $\kappa(\inv[b](g,g'))$ agrees with the image of $\lambda$ in $X_*(T)_\Gamma$.  We consider $(g,g',\lambda)$ conjugate to $((\ad z)(g),(\ad z')(g'),(\ad z)(\lambda))$ whenever $z\in G(F)$ and $z'\in G_b(F)$.  We give $\Rel_b\subset (G(F) \times G_b(F)\times X_*(G))/(G(F)\times G_b(F))$ the subspace topology, where $X_*(G)$ is taken to be discrete.
\end{dfn}

\begin{rmk}\label{RmkAtMostOne}
 Given $g\in G(F)_{\sr}$ and $\lambda$ a cocharacter of its torus, there is at most one conjugacy class of $g'\in G_b(F)$ with $(g,g',\lambda)\in \Rel_b$.  In other words, $g$ and $\inv[b](g,g')$ determine the conjugacy class of $g'$.   Indeed, suppose $(g,g',\lambda)$ and $(g,g'',\lambda)$ are both in $\Rel_b$.  Then $g'=ygy^{-1}$ and $g''=zgz^{-1}$ for some $y,z\in G(F^{\nr})$, and $y^{-1}by^{\sigma}$ and $z^{-1}bz^\sigma$ are $\sigma$-conjugate in $T(\breve{F})$.  This means there exists $t\in T(\breve{F})$ such that $y^{-1}by^{\sigma}=(zt)^{-1}b (zt)^{\sigma}$.   We see that $x=zty^{-1}\in G_b(F)$, and that $x$ conjugates $g'$ onto $g''$.
\end{rmk}

\begin{lem} \label{LemEtale} The map $\Rel_{b}\to G(F)_{\sr}\sslash G(F)$ is a homeomorphism locally on the source.  Its image consists of those classes that transfer to $G_b$.  In particular, the image is open and closed.

The analogous statement is true for $\Rel_b\to G_b(F)_{\sr}\sslash G_b(F)$.
\end{lem}

\begin{proof}

The proof of Lemma \ref{LemLocallyProfinite} shows that $G(F)_{\sr}\sslash G(F)$ is the disjoint union of spaces $T(F)_{\sr}/W_T$, as $T\subset G$ runs through the finitely many conjugacy classes of $F$-rational maximal tori, and $W_T=N(T,G)(F)/T(F)$ is a finite group.  By the above remark, $\Rel_b$ injects into the disjoint union of the spaces $T(F)_{\sr}/W_T\times X_*(T)$, with the map to $G(F)_{\sr}\stslash G(F)$ corresponding to the projection $T(F)_{\sr}/W_T\times X_*(T)\to T(F)_{\sr}/W_T$.  Since $X_*(T)$ is discrete, this map is a homeomorphism locally on the source.  The other statements are evident from the definitions.  
\end{proof}

The definition of $\Rel_b$ already suggests a means for transferring functions from $G(F)_{\sr}\stslash G(F)$ to $G_b(F)_{\sr}\stslash G_b(F)$, namely, by pulling back from $G(F)_{\sr}\stslash G(F)$ to $\Rel_b$, multiplying by a compactly supported kernel function, and then pushing forward to $G_b(F)_{\sr}\stslash G_b(F)$.  We will define one such kernel function for each geometric conjugacy class of cocharacters $\mu\from \Gm\to G_{\overline{F}}$.  

Let $\hat{G}$ be the Langlands dual group. It comes equipped with a splitting, in particular with a torus and Borel $\hat T \subset \hat B \subset \hat G$. Given a conjugacy class of cocharacters $\mu$ for $G$ as above, we obtain a character $\hat\mu : \hat T \to \Gm$ which is $\hat B$-dominant. Let $r_\mu$ be the Weyl module of the dual group $\hat G$ whose highest weight with respect to $(\hat T,\hat B)$ is $\hat\mu$.  

A cocharacter $\lambda\in X_*(T)$ corresponds to a character $\hat{\lambda}\in X^*(\hat{T})$.  Let $r_\mu[\lambda]$ be the $\hat{\lambda}$-weight space of $r_\mu$.   The quantity $\dim r_\mu[\lambda]$ will give us our kernel function. While we will not need it here, we note that there is an explicit formula for $\dim r_\mu[\lambda]$ coming from the Weyl character formula.

%\begin{lem} \label{LemEtale}  Given a class $j\in G(F)_{\rs}\stslash G(F)$, there are only finitely many triples $(g,j,\lambda)\in \Rel_b$ with 
%\end{lem}

%\begin{proof} The first statement follows from the definition of the topology on $\Rel_{b,\mu}$. The second statement is immediate from the fact that $r_\mu$ is finite-dimensional and thus has only finitely many weights, each having finite multiplicity. For the third statement, let $g \in G(F)_{\sr}$ be any element whose stable class transfers to $G_b$ and let $T$ be its centralizer. As $j$ runs over the elements of $G_b(F)$ in the corresponding stable class, $\tx{inv}[b](g,j)$ runs over the full fiber of $\pi_1(T)_\Gamma \to \pi_1(G)_\Gamma$ over the image of $[b]$ under $B(G) \to \pi_1(G)_\Gamma$. Choose any $\lambda \in X_*(T)$ that is a weight of $r_\mu$. Then $\lambda$ has the same image in $\pi_1(G)_\Gamma$ as $\mu$ and hence there is $j \in G_b(F)$ such that the image of $\lambda$ in $\pi_1(T)_\Gamma$ equals $\tx{inv}[b](g,j)$. Thus $\dim r_\mu[\lambda] \cdot (g,j,\lambda) \in \Rel_{b,\mu}$.

%The fact that the image is open and closed follows from the proof of Lemma \ref{LemLocallyProfinite}.
%\end{proof}

We now fix a commutative ring $\Lambda$ in which $p$ is invertible.  For a topological space $X$, we let $C(X,\Lambda)$ be the space of continuous $\Lambda$-valued functions on $X$, where $\Lambda$ is given the discrete topology.

\begin{dfn}\label{DefinitionTbmu} Let $d=\class{\mu,2\rho_G}$, where $2\rho_G$ is the sum of the positive roots of $G$.  We define the Hecke transfer map
\[ T_{b,\mu}^{G\to G_b}\from C(G(F)_{\sr}\stslash G(F),\Lambda)\to C(G_b(F)_{\sr}\stslash G_b(F),\Lambda) \]
by
% pulling back $f$ to $\Rel_{b,\mu}$ and then integrating along the fiber of $\Rel_{b,\mu} \to G_b(F)_{\sr}\stslash G_b(F)$. Explicitly,
\[ [T_{b,\mu}^{G\to G_b} f](g') = (-1)^d\sum_{(g,g',\lambda) \in \Rel_{b}} f(g) \dim r_\mu[\lambda]. \]
Analogously, we define 
\[ T_{b,\mu}^{G_b\to G}\from C(G_b(F)_{\sr}\stslash G_b(F),\Lambda)\to C(G(F)_{\sr}\stslash G(F),\Lambda) \]  
by
\[ [T_{b,\mu}^{G_b\to G} f'](g) =(-1)^d \sum_{(g,g',\lambda) \in \Rel_{b}} f'(g') \dim r_\mu[\lambda]. \]
\end{dfn}

Since $r_\mu$ is finite-dimensional, the sum is finite. If $f'$ has compact support, then so does its image.

\begin{lem}
The Hecke transfer map $T_{b,\mu}^{G\to G_b}$ is zero unless $[b]$ is the unique basic class in $B(G,\mu)$.  
\end{lem}
\begin{proof}
Suppose there exists an $F$-rational maximal torus $T\subset G$ and a cocharacter $\lambda\in X_*(T)$ such that $r_\mu[\lambda]\neq 0$.  Then $\hat{\mu}$ and $\hat{\lambda}$ must agree when restricted to the center $Z(\hat{G})$, which is to say that $\hat{\mu}$ and $\hat{\lambda}$ have the same image in $X^*(Z(\hat{G}))$.  Equivalently, if we conjugate $\mu$ so as to assume it is a cocharacter of $T$, then $\mu$ and $\lambda$ have the same image under $X_*(T)\isom \pi_1(T)\to \pi_1(G)$. By Fact \ref{fct:inv} and the functoriality of $\kappa$ the image of $\lambda$ in $\pi_1(G)_\Gamma$ equals $\kappa(b)$.  We conclude that $\kappa([b])$ equals the image of $\mu$ in $\pi_1(G)_\Gamma$.  This means that $[b]$ is the unique basic class in $B(G,\mu)$.	
\end{proof}

%Every triple $(\hat T,\hat B,\hat\mu)$ gives the same restriction $\hat\mu|_{Z(\hat G)}$. The resulting element of $X^*(Z(\hat G))=\pi_1(G)$ can also be obtained by choosing any maximal torus $T \subset G$, an element $\mu \in X_*(T)=\pi_1(T)$ in the conjugacy class $\{\mu\}$, and mapping $\mu$ under $\pi_1(T) \to \pi_1(G)$. We assume that the image of $[b]$ under $\kappa : B(G) \to \pi_1(G)_\Gamma = X^*(Z(\hat G)^\Gamma)$ is equal to $\hat\mu|_{Z(\hat G)^\Gamma}$.

%We consider the following variation $\Rel_{b,\mu}$ of $\Rel_b$: We define $\Rel_{b,\mu}$ to be the multiset whose elements are the triples $(g,j,\lambda) \in \Rel_b$ and each element occurs with multiplicity equal to the multiplicity of $\lambda$ as a weight of the representation $r_\mu$. Thus, if $\lambda$ is not a weight of $r_\mu$ then the element $(g,j,\lambda) \in \Rel_b$ does not belong to $\Rel_{b,\mu}$, and if $\lambda$ occurs in $r_\mu$ with positive multiplicity, then $(g,j,\lambda)$ occurs with that multiplicity as an element of $\Rel_{b,\mu}$. We denote by $r_\mu[\lambda]$ the $\lambda$-weight space of $\mu$, so that $\dim r_\mu[\lambda]$ is the multiplicity of $\lambda$ in $r_\mu$.

Assume therefore that $[b]$ is the unique basic class in $B(G,\mu)$.  We may define a ``truncation'' $\Rel_{b,\mu}\subset \Rel_b$, consisting of conjugacy classes of triples $(g,g',\lambda)$ for which $\lambda\leq \mu$ .   Then the kernel function $(g,g',\lambda)\mapsto \dim r_\mu[\lambda]$ is supported on $\Rel_{b,\mu}$.  In the diagram
\begin{equation}
\label{Relbmu}
\xymatrix{
& \Rel_{b,\mu} \ar[dl] \ar[dr] & \\
G(F)_{\sr}\stslash G(F) && G_b(F)_{\sr}\stslash G_b(F),
}
\end{equation}
both maps are finite \'etale over their respective images.

The following theorem is proved in \S\ref{SectionProofOfEndoscopicTraceRelation}.  It relates the Hecke transfer map $T_{b,\mu}^{G\to G_b}$ to the local Jacquet-Langlands correspondence for $G$.

\begin{thm} \label{TheoremEndoscopicTraceRelation} Assume that $b\in B(G,\mu)$ is basic, and that $\Lambda$ is an algebraically closed field of characteristic 0 abstractly isomorphic to $\mathbf{C}$.  Let $\phi\from W_F \times \mathrm{SL}_2 \to \;^LG$ be a discrete $L$-parameter with coefficients in $\Lambda$, and 
let $\rho\in \Pi_\phi(G_b)$.  Let $\Theta_\rho\in C(G_b(F)_{\sr}\stslash G_b(F),\Lambda)$ be its Harish-Chandra character.   Then for any $g\in G(F)_{\sr}$ that transfers to $G_b(F)$, we have
\begin{equation}
\label{EqEndoscopicTraceRelation}
 \left[T_{b,\mu}^{G_b\to G}\Theta_\rho\right](g)=\sum_{\pi\in \Pi_\phi(G)} \dim \Hom_{S_\phi}(\delta_{\pi,\rho},r_\mu)\Theta_{\pi}(g), 
\end{equation}
assuming the validity of the refined local Langlands conjecture, i.e. \cite[Conjecture G]{KalethaLocalLanglands}.
\end{thm}

%\begin{rmk} When $g$ does not transfer to $G_b(F)$, the left-hand side is zero by definition, but the right-hand side need not be zero, as the following example shows.
%\end{rmk}

\begin{exm}
Let $G=\tx{GL}_2$, and let $\mu\from \bf{G}_m \to G$ the co-character sending $x$ to the diagonal matrix with entries $(x,1)$. We have $\pi_1(G)=\Z$ as a trivial $\Gamma$-module.  Let $b \in B(G,\mu)$ be the basic class.  Then $b$ corresponds to the isocrystal of slope $1/2$,  and $G_b(F)$ is the multiplicative group of the nonsplit quaternion algebra over $F$.   Let $\phi$ be a discrete Langlands parameter.  The $L$-packets $\Pi_\phi(G)=\set{\pi}$ and $\Pi_\phi(G_b)=\set{\rho}$ are singletons.   We have $S_\phi=Z(\hat G)=\C^\times$, and $\delta_{\pi,\rho}$ is the identity character of $S_\phi$. The representation $r_\mu$ is the standard representation of $\hat G=\tx{GL}_2(\C)$,  and $\dim\Hom_{S_\phi}(\delta_{\pi,\rho},r_\mu)=2$.   Therefore the right-hand side of \eqref{EqEndoscopicTraceRelation} equals $2\Theta_\pi(g)$.

Let $w\mu$ be the cocharacter sending $x$ to the diagonal matrix with entries $(1,x)$. The map $\lambda \mapsto r_\mu[\lambda]$ sends $\mu$ and $w\mu$ to $1$ and all other co-characters to $0$. For any strongly regular $g' \in G_b(F)$ there is a unique $G(F)$-conjugacy class of strongly regular $g \in G(F)$ related to $g'$. Let $S \subset G$ be the centralizer of one such $g$. Then $X_*(S) \cong \Z[\Gamma_{E/F}]$ for a quadratic extension $E/F$ and the map $\pi_1(S)_\Gamma \to \pi_1(G)_\Gamma$ is an isomorphism. There are exactly two elements $\lambda,w\lambda \in X_*(S)$ that map to $\tx{inv}[b](g,g')$.  Finally,  $d=1$.   Therefore $T_{b,\mu}^{G \to G_b}f(g')=-2f(g)$.  Setting $f=\Theta_\rho$, we find that Theorem \ref{TheoremEndoscopicTraceRelation} reduces to the Jacquet-Langlands character identity
\[ \Theta_\rho(g')=-\Theta_\pi(g). \]
%In this situation the functions $T_{b,\mu}^{G\to G_b}\Theta_\pi$ and $-2\Theta_\rho$ are equal.  

% (referee objected to this paragraph) We can switch the roles of the two groups and let $G$ be the non-trivial inner form of $\tx{GL}_2$. We can still choose the same $\mu$ and $b$. Then $G_b$ is isomorphic to $\tx{GL}_2$.   Now,  not all $g' \in G_b(F)$ transfer to $G(F)$; only the elliptic ones do. Therefore the functions $T_{b,\mu}^{G \to G_b}\Theta_\pi$ and $-2\Theta_\rho$ are not equal, because the former is zero on all non-elliptic $g'$, while the latter is not.

\end{exm}

\subsection{Proof of Theorem \ref{TheoremEndoscopicTraceRelation}}
\label{SectionProofOfEndoscopicTraceRelation}
%Note I have switched the roles of G and G_b to be consistent with section 2 -JW
We now give the proof of Theorem \ref{TheoremEndoscopicTraceRelation}. We will use the notation and results of \S\ref{CharacterRelations}.

%We are assuming that $j$ transfers to an element of $G(F)$. We have the bijection $g \mapsto \tx{inv}[b](j,g)$ from the set of $G(F)$-conjugacy classes in the unique stable class in $G(F)$ to which $j$ transfers and the fiber of $\pi_1(T)_\Gamma \to \pi_1(G)_\Gamma$ over $\kappa([b])^{-1}$, where $T$ is the centralizer of $j$. Given $\nu$ in that fiber let $\iota_{b,\nu}(j)$ be the corresponding element. 
We are given a discrete $L$-parameter $\phi$, a representation $\rho\in \Pi_{\phi}(G_b)$ in its $L$-packet, and an element $g\in G(F)_{\sr}$.  We assume that $g$ is related to an element of $G_b(F)$.  This means there exists a triple in $\Rel_b$ of the form $(g,g',\lambda)$.   For the moment we fix such a triple $(g,g',\lambda)$.  

Let $s\in S_\phi$ be a semi-simple element, and let $\dot s \in S_\phi^+$ be a lift of it.  Then we have the refined endoscopic datum $\mf{\dot e}=(H,\mc{H},\dot s,\eta)$ defined in \eqref{eq:red};  we choose as in that section a $z$-pair $\mf{z}=(H_1,\eta_1)$. Then 
\begin{eqnarray*}
%e(G)\sum_{\pi'\in \Pi_\phi(G)}
%\tr\tau_{z,\mf{w},\pi'}(\dot s)\Theta_{\pi'}(\iota_{b,\nu}(j))\\
%&\stackrel{\eqref{FormulaForThetaRho}}{=}&
%\sum_{h_1\in H_1(F)/\tx{st}} \Delta(h_1,\iota_{b,\nu}(j))S\Theta_{\phi^
%s}(h_1) \\
%&\stackrel{\eqref{RelationBetweenTransferFactors}}{=}&
%\sum_{h_1\in H_1(F)/\tx{st}} \Delta(h_1,j)\class{\inv[b](j,\iota_{b,\nu}(g)),s_{h,j}}S\Theta_{\phi^s}(h_1) \\
%&\stackrel{\eqref{ModificationOfE}(3)}{=}&
%\sum_{h_1\in H_1(F)/\tx{st}} \Delta(h_1,j)\nu(s_{h,j})S\Theta_{\phi^s}(h_1).
% Invariants get inverted if we switch arguments, see llc1, Fact 2.1.3. 
&\;&e(G_b)\sum_{\rho'\in \Pi_\phi(G_b)}
\tr\tau_{z,\mf{w},\rho'}(\dot s)\Theta_{\rho'}(g')\\
&\stackrel{\eqref{FormulaForThetaPi}}{=}&
\sum_{h_1\in H_1(F)/\tx{st}}
\Delta(h_1,g')S\Theta_{\phi^s}(h_1) \\
&\stackrel{\eqref{RelationBetweenTransferFactors}}{=}&
\sum_{h_1\in H_1(F)/\tx{st}}
\Delta(h_1,g)\class{\inv[b](g,g'),s^\natural_{h,g}}S\Theta_{\phi^s}(h_1) \\
&=&
\sum_{h_1\in H_1(F)/\tx{st}} \Delta(h_1,g)\lambda(s^\natural_{h,g})S\Theta_{\phi^s}(h_1).
\end{eqnarray*}
We now multiply this expression by the kernel function $\dim r_\mu[\lambda]$, and then sum over all $G_b(F)$-conjugacy classes of elements $g' \in G_b(F)$ and all $\lambda \in X_*(T_g)$ such that $(g,g',\lambda)$ lies in $\Rel_b$. We obtain
\begin{eqnarray*}
&&e(G_b)
\sum_{(g',\lambda)}\;
\sum_{\rho'\in \Pi_{\phi}(G_b)}
\tr\tau_{z,\mf{w},\rho'}(\dot s)\Theta_{\rho'}(g')\dim r_\mu[\lambda]\\
&=& 
\sum_{h_1\in H_1(F)/\tx{st}}\Delta(h_1,g)
S\Theta_{\phi^s}(h_1)\sum_{(g',\lambda)}\lambda(s^\natural_{h,g})\dim r_\mu[\lambda]\\
&\stackrel{(*)}{=}&\sum_{h_1\in H_1(F)/\tx{st}}\Delta(h_1,g)S\Theta_{\phi^s}(h_1)\tr r_{\mu}(s^\natural_{h,g}) \\
&\stackrel{(**)}{=}&\tr r_{\mu}(s^\natural)\sum_{h_1\in H_1(F)/\tx{st}}\Delta(h_1,g)S\Theta_{\phi^s}(h_1)\\
&\stackrel{\eqref{FormulaForThetaPi}}{=}&
\tr r_{\mu}(s^\natural)e(G)\sum_{\pi\in \Pi_{\phi}(G)} \tr\tau_{z,\mf{w},\pi}(\dot s)\Theta_{\pi}(g).\\
\end{eqnarray*}

%the old one
%\sum_{h_1\in H_1(F)/\tx{st}}\Delta(h_1,j)
%\sum_{\lambda\in X^*(\hat{T})}\lambda(s_{h,g}^{-1})\dim r_\mu[\lambda]S\Theta_{\phi^s}(h_1)\\
%&=&\sum_{h_1\in H_1(F)/\tx{st}}\Delta(h_1,j)\tr r_{\mu}(s_{h,j}^{-1}) S\Theta_{\phi^s}(h_1)\\
%&\stackrel{(*)}{=}&\tr \check r_{\mu}(s^\natural)\sum_{h_1\in H_1(F)/\tx{st}}\Delta(h_1,j)S\Theta_{\phi^s}(h_1)\\
%&\stackrel{\eqref{FormulaForThetaPi}}{=}&
%\tr \check r_{\mu}(s^\natural)e(G_b)\sum_{\rho\in \Pi_{\phi}(G_b)} \tr\tau_{z,\mf{w},\rho}(\dot s)\Theta_{\rho}(j),
%\end{eqnarray*}
We justify $(**)$: Let $T \subset G$ be the centralizer of $g$. The image of $s^\natural_{h,g}$ under any admissible embedding $\hat T \to \hat G$ is conjugate to $s^\natural$ in $\hat{G}$ and $\tr r_\mu$ is conjugation-invariant. Recall here that $s^\natural \in S_\phi$ is the image of $\dot s$ under \eqref{eq:rivsbgllc}.

We justify $(*)$: $\lambda \in X_*(T)$ determines the $G_b(F)$-conjugacy class of $g'$, since $\tx{inv}[b](g,g') \in B(T)$ determines it. Therefore the sum over $(g',\lambda)$ is in reality a sum only over $\lambda$. There exists $g' \in G_b(F)$ with $\kappa(\tx{inv}[b](g,g'))$ being the image of $\lambda$ in $X_*(T)_\Gamma$ if and only if the image of $\lambda$ under $X_*(T) \to X_*(T)_\Gamma \to \pi_1(G)_\Gamma$ equals $\kappa(b)$. Since the image of $\mu$ in $\pi_1(G)_\Gamma$ also equals $\kappa(b)$,  the sum over $(g',\lambda)$ is in fact the sum over $\lambda \in X_*(T)$ having the same image as $\mu$ in $\pi_1(G)_\Gamma$. In terms of the dual torus $\hat T$ this is the sum over $\lambda \in X^*(\hat T)$ whose restriction to $Z(\hat G)^\Gamma$ equals that of $\mu$. Since for $\lambda$ not satisfying this condition the number $\tx{dim}r_\mu[\lambda]$ is zero, we may extend the sum to be over all $\lambda \in X_*(T)=X^*(\hat T)$.

We now continue with the equation. Multiply both sides of the above equation by $\tr\check\tau_{z,\mf{w},\rho}(\dot s)$. As functions of $\dot s \in S_\phi^+$, both sides then become invariant under $Z(\hat{\bar G})^+$ and thus become functions of the finite quotient $\bar S_\phi = S_\phi^+/Z(\hat{\bar G})^+=S_\phi/Z(\hat G)^\Gamma$. Now apply $\abs{\bar S_\phi}^{-1}\sum_{\bar s \in \bar S_\phi}$ to both sides to obtain an equality between
\begin{equation}
\label{EquationFirstExpression}
\abs{\bar S_\phi}^{-1} e(G_b)\sum_{\bar s \in \bar S_\phi} \sum_{(g',\lambda)}
\sum_{\rho'\in \Pi_{\phi}(G_b)}
\tr\check\tau_{z,\mf{w},\rho}(\dot s)\tr\tau_{z,\mf{w},\rho'}(\dot s)\Theta_{\rho'}(g')\dim r_\mu[\lambda]
\end{equation}
%the old one
%\[ |\bar S_\phi|^{-1} e(G)\sum_{\bar s \in \bar S_\phi} \sum_{\nu\in X^*(\hat{T})}
%\sum_{\pi'\in \Pi_{\phi}(G)}
%\tr\check\tau_{z,\mf{w},\pi}(\dot s)\tr\tau_{z,\mf{w},\pi'}(\dot s)\Theta_{\pi'}(\iota_{b,-\lambda}(j))\dim r_\mu[\lambda]
%\]
and
\begin{equation}
\label{EquationSecondExpression}
\abs{\bar S_\phi}^{-1}e(G) \sum_{\bar s \in \bar S_\phi}\tr r_{\mu}(s^\natural)\sum_{\pi\in \Pi_{\phi}(G)} \tr\check\tau_{z,\mf{w},\rho}(\dot s) \tr\tau_{z,\mf{w},\pi}(\dot s)\Theta_{\pi}(g), 
\end{equation}
%the old one
%\[ \abs{\bar S_\phi}^{-1} \sum_{\bar s \in \bar S_\phi}\tr \check r_{\mu}(s^\natural)e(G_b)\sum_{\rho\in \Pi_{\phi}(G_b)} \tr\check\tau_{z,\mf{w},\pi}(\dot s) \tr\tau_{z,\mf{w},\rho}(\dot s)\Theta_{\rho}(j), \]
where in both formulas $\dot s$ is an arbitrary lift of $\bar s$ and $s^\natural \in S_\phi$ is the image of $\dot s$ under \eqref{eq:rivsbgllc}. Executing the sum over $\bar s$ in \eqref{EquationFirstExpression}  gives
\[ e(G_b)\sum_{(g',\lambda)} \Theta_{\rho}(g')\dim r_\mu[\lambda] = e(G_b)[T^{G_b\to G}_{b,\mu}\Theta_\rho](g).
\]
%\[ e(G_b)\sum_{} \Theta_{\pi}(\iota_{b,-\lambda}(j))\dim r_\mu[\lambda] = e(G)[T_{b,\mu}\Theta_\pi](j),
%\]
%which is the left side of the equation in Theorem \ref{TheoremEndoscopicTraceRelation}, multiplied by $e(G)$.

To treat \eqref{EquationSecondExpression} note that $\check\tau_{z,\mf{w},\rho}\otimes\tau_{z,\mf{w},\pi}(\dot s) = \check{\delta}_{\pi,\rho}(s^\natural)$. Furthermore, the composition of the map \eqref{eq:rivsbgllc} with the natural projection $S_\phi \to S_\phi/Z(\hat G)^\Gamma$ is equal to the natural projection $S_\phi^+ \to S_\phi^+/Z(\hat{\bar G})^+ = S_\phi/Z(\hat G)^\Gamma = \bar S_\phi$. Thus $s^\natural$ is simply a lift of $\bar s$ to $S_\phi$.  We find that \eqref{EquationSecondExpression} equals
\[ e(G)\abs{\bar S_\phi}^{-1} \sum_{\bar s \in \bar S_\phi}\tr r_{\mu}(s^\natural) \tr\check{\delta}_{\pi,\rho}(s^\natural) = e(G)\dim \Hom_{S_\phi}(\delta_{\pi,\rho}, r_\mu).\]
We have now reduced Theorem \ref{TheoremEndoscopicTraceRelation} to the identity
\begin{equation}
\label{EquationSignIdentity}
e(G)e(G_b)=(-1)^{\class{2\rho_G,\mu}},
\end{equation}
where $\rho_G$ is the sum of the positive roots. Recall that $G^*$ is a quasi-split inner form of $G$.  Let $\mu_1,\mu_2 \in X^*(Z(\hat G_\tx{sc})^\Gamma)$ be the elements corresponding to the inner twists $G^* \to G$ and $G^* \to G_b$ by Kottwitz's homomorphism \cite[Theorem 1.2]{Kot86}. By Lemma \ref{lem:kotsign} we have $e(G_b)e(G)=(-1)^{\<2\rho,\mu_2-\mu_1\>}$. But since $G_b$ is obtained from $G$ by twisting by $b$, the difference $\mu_2-\mu_1$ is equal to the image of $\kappa(b) \in X^*(Z(\hat G)^\Gamma)$ under the map $X^*(Z(\hat G)^\Gamma) \to X^*(Z(\hat G_\tx{sc})^\Gamma)$ dual to the natural map $Z(\hat G_\tx{sc}) \to Z(\hat G)$. Since $b \in B(G,\mu)$ we see that $\mu_2-\mu_1=\mu$ and \eqref{EquationSignIdentity} follows. The proof of Theorem \ref{TheoremEndoscopicTraceRelation} is complete.

\subsection{An adjointness property} \label{SectionInteractionWithOrbitalIntegrals}

In this section we will discuss an adjointness property of the Hecke transfer maps $T_{b,\mu}^{G \to G_b}$. This will be used in \S\ref{sub:6.3}.

Let $\Lambda$ be an algebraically closed field of characteristic zero.  For a topological space $X$, we let $C_c(X,\Lambda)$ be the space of compactly supported locally constant $\Lambda$-valued functions. The space of distributions $\Dist(G(F),\Lambda)$ is the $\Lambda$-linear dual of $C_c(G(F),\Lambda)$. The subspace of invariant distributions $\Dist(G(F),\Lambda)^{G(F)}$ is the linear dual of the space of coinvariants $C_c(G(F),\Lambda)_{G(F)}$. 

Given a $\Lambda$-valued Haar measure $dx$ on $G(F)$, integration against a function $f \in C(G(F)\stslash G(F),\Lambda)$ is a $G(F)$-invariant distribution on $G(F)$. Due to the functions in $C_c(G(F),\Lambda)$ being locally constant and having compact support, the ``integral'' is in reality a finite sum. For our purposes we will work with functions $f \in C(G(F)_\tx{sr} \stslash G(F),\Lambda)$ and integrate them against test functions in $C_c(G(F)_\tx{sr})$.

The Weyl integration formula can be used to compute this distribution in terms of orbital integrals. In fact, we will need a ``stable'' variant of this formula. Before we can explain this, we need to discuss choices of measures.

Choose a $\Lambda$-valued Haar measure on $F$. Then a choice of an element $\eta \in \bigwedge^{\dim(G)}(\tx{Lie}(G)(F)^*)=\bigwedge^{\dim(G)}(\tx{Lie}(G)^*)(F)$ leads to a $\Lambda$-valued Haar measure $dx_\eta$ on $G(F)$; note that multiplying $\eta$ by an element of $\OO_F^\times$ doesn't affect the measure $dx_\eta$. More generally, any element of $\bigwedge^{\dim(G)}(\tx{Lie}(G)^*)(\breve F)$ leads to a $\Lambda$-valued Haar measure $dx_\eta$ on $G(F)$ by choosing $a \in \OO_{\breve F}^\times$ with the property that $a\eta \in \bigwedge^{\dim(G)}(\tx{Lie}(G)^*)(F)$ and defining $dx_\eta := dx_{a\eta}$, noting that this does not depend on the choice of $a$.
% since any two choices of $a$ differ by an element of $O_F^\times$
In fact, this procedure allows us to even attach a measure to an element $\eta \in \bigwedge^{\dim(G)}(\tx{Lie}(G)^*)(\ol{F})$, by taking $a \in \ol{F}^\times$ such that $a\eta \in \bigwedge^{\dim(G)}(\tx{Lie}(G)^*)(F)$ and letting $dx_\eta := |a|_\Lambda^{-1}dx_{a\eta}$. But for this we need to make sense of $|a|_\Lambda$, which requires choosing a compatible system of roots of $p$ in $\Lambda$. For us, elements of $\bigwedge^{\dim(G)}(\tx{Lie}(G)^*)(\breve F)$ will suffice, so we will not make such a choice.

 % To do this, we fix a compatible system $p_n \in \Lambda^\times$, $n=1,2,3,\dots$, of roots of $p$, i.e. $p_1=p$ and $p_n^{(n/m)}=p_m$ for all $m|n$. This produces a map $\Q \to \Lambda$, $r/s \mapsto p_s^r$, which we denote by $t \mapsto p^t$. For any $x \in \ol{F}^\times$ we then define $|x|_\Lambda \in \Lambda$ by $|x|_\Lambda=q^{-\tx{ord}_F(x)}$, where $\tx{ord}_F : \ol{F}^\times \to \Q$ is the unique extension of the normalized valuation $F^\times \to \Z$, and $q$ is the size of the residue field of $F$, an integral power of $p$. Now, given $\eta \in \bigwedge^{\dim(G)}(\tx{Lie}(G)^*)(\ol{F})$, there exists $a \in \ol{F}^\times$ such that $a\eta \in \Lambda^{\dim(G)}(\tx{Lie}(G)^*)(F)$. The Haar measure $dx_\eta := |a|_\Lambda^{-1}dx_{a\eta}$ of $G(F)$ depends only on $\eta$, and not on $a$.

This procedure allows us to choose Haar measures compatibly in the following two situations. First, consider the inner forms $G$ and $G_b$. They are canonically identified over $\breve F$, which gives an identification $\bigwedge^{\dim(G)}(\tx{Lie}(G)^*)(\breve F)= \bigwedge^{\dim(G)}(\tx{Lie}(G_b)^*)(\breve F)$. Haar measures on $G(F)$ and $G_b(F)$ corresponding to the same $\eta$ will be called compatible. Second, consider two maximal $F$-rational tori $T_1$ and $T_2$, each either in $G$ or $G_b$. They are called related if there exists $g \in G(\ol{F})$, or equivalently (cf. Lemma \ref{LemmaSteinberg}) $g \in G(\breve F)$, such that $gT_1g^{-1}=T_2$ and the isomorphism $\tx{Ad}(g) :T_1 \to T_2$ is $F$-rational; we are using here the identification $G_{\breve F} = (G_b)_{\breve F}$. We obtain an isomorphism $\tx{Ad}(g) : \bigwedge^{\dim(G)}(\tx{Lie}(T_1)^*)(\breve F)= \bigwedge^{\dim(G)}(\tx{Lie}(T_2)^*)(\breve F)$, which leads again to the notion of compatible measures on $T_1(F)$ and $T_2(F)$. The choice of $g \in G(\breve F)$ is unique up to multiplication by $N(T_1,G)(\breve F)$, and since this group acts on $\bigwedge^{\dim(G)}(\tx{Lie}(T_1)^*)(\breve F)$ via a $\OO_{\breve F}^\times$-valued character of the Weyl group, the notion of compatible measures does not depend on the choice of $g$.

% $T \subset G$ and $T_b \subset G_b$. Over $\bar F$, the torus $T$ is conjugate to $T_b$ in $G_{\ol{F}}=G_{b,\ol{F}}$. The resulting isomorphism $T_{\ol{F}} \to T_{b,\ol{F}}$ is well-defined up to the action of the Weyl group. The action of the Weyl group on $\bigwedge^{\dim(G)}(\tx{Lie}(T)^*)(\ol{F})$ is through an $\OO_{\ol{F}}^\times$-valued character. Therefore, an element of $\bigwedge^{\dim(G)}(\tx{Lie}(T)^*)(\ol{F})$ produces the same Haar measure on $T(F)$ as any of its Weyl conjugates. Since the Weyl orbits in $\bigwedge^{\dim(G)}(\tx{Lie}(T)^*)(\ol{F})$ are naturally identified with those in $\bigwedge^{\dim(G)}(\tx{Lie}(T_b)^*)(\ol{F})$, we can speak of compatibility of measures on $T(F)$ and $T_b(F)$. In fact, choosing a Haar measure on one maximal torus of $G$ provides compatible choices on all maximal tori of all inner forms.

From now on we assume that the Haar measures on $G(F)$ and $G_b(F)$ have been chosen compatibly, and the Haar measures on all tori of $G$ and $G_b$ that are related to each other have been chosen compatibly.

We now return to the discussion of distributions. For $\phi\in C_c(G(F)_\tx{sr},\Lambda)$, let $\phi_G\in C_c(G(F)_{\sr}\stslash G(F),\Lambda)$ be the orbital integral function,
\[ \phi_G(y) = \int_{x\in G(F)/G(F)_y} \phi(xyx^{-1})\; dx.\]
As remarked by the referee, the map $\phi \to \phi_G$ is induces an isomorphism
\begin{equation} \label{eq:orbint}
C_c(G(F)_\tx{sr},\Lambda)_{G(F)} \to C_c(G(F)_{\sr}\stslash G(F),\Lambda),	
\end{equation}
cf. Lemma \ref{LemLocallyProfinite}. The stable Weyl integration formula states 
\begin{equation}
\label{EqWeylIntegration}
 \int_{G(F)} f(x)\phi(x)\; dx = \class{f,\phi_G}_G.
 \end{equation}
We explain now the notation $\class{f,\phi_G}_G$.  For a function $h\in C_c(G(F)_{\sr}\stslash G(F),\Lambda)$, we define
\[ \class{f,h}_G = \sum_{T} \abs{W(T,G)(F)}^{-1}\int_{t \in T(F)_{\sr}} \abs{D(t)} \sum_{t_0 \sim t} f(t_0)h(t_0)dt, \]
where 
\begin{itemize}
\item $T$ runs over a set of representatives for the stable classes of maximal tori, 
\item $W(T,G)=N(T,G)/T$ is the absolute Weyl group, 
\item $D(t)=\det\left(\Ad(t)-1\biggm\vert\Lie G/\Lie T\right)$ is the usual Weyl discriminant, and
\item $t_0$ runs over the $G(F)$-conjugacy classes inside of the stable class of $t$.
\end{itemize}
Note that the integral does not depend on the chosen representative, since any two are isomorphic over $F$ by definition of stable conjugacy, and the isomorphism is canonical up to the action of the Weyl group $W(T,G)(F)$, which is irrelevant given the sum $t_0 \sim t$.

The following lemma shows that the Hecke transfer maps $T_{b,\mu}^{G\to G_b}$ and $T_{b,\mu}^{G_b\to G}$ are adjoint with respect to the pairing $\class{\cdot,\cdot}_G$ and its analogue $\class{\cdot,\cdot}_{G_b}$, defined similarly.

\begin{lem} \label{LemAdjointTransfers} Given $f' \in C(G_b(F)_{\sr}\stslash G_b(F),\Lambda)$ and $f \in C(G(F)_{\sr}\stslash G(F),\Lambda)$, one of which has compact support, we have
\[ \<T_{b,\mu}^{G_b \to G}f',f\>_G = \<f',T_{b,\mu}^{G \to G_b}f\>_{G_b}. \]
\end{lem}
\begin{proof}
By definition $\<T_{b,\mu}^{G_b \to G}f',f\>_G$ equals
\begin{equation}
\label{EqHeckePairingStep1}
(-1)^d\sum_T |W(T,G)(F)|^{-1}\int_{t \in T(F)_{\sr}} |D(t)| \sum_{t_0 \sim t} \sum_{(t_0,t_0',\lambda)}r_{\mu,\lambda}f'(t_0')f(t_0)dt. 
\end{equation}
The first sum runs over a set of representatives for the stable classes of maximal tori in $G$. The second sum runs over the set $t_0$ of $G(F)$-conjugacy classes of elements that are stably conjugate to $t$. Let $T_{t_0}$ denote the centralizer of $t_0$. The third sum runs over triples $(t_0,t_0',\lambda)$, where $t_0'$ is a $G_b(F)$-conjugacy class that is stably conjugate to $t_0$, and $\lambda \in X_*(T_{t_0})$ maps to $\tx{inv}(t_0,t_0') \in X_*(T_{t_0})_\Gamma$. Note that if $T$ does not transfer to $G_b$ then it does not contribute to the sum, because the sum over $(t_0,t_0',\lambda)$ is empty. Let $\mc{X}$ be a set of representatives for those stable classes of maximal tori in $G$ that transfer to $G_b$. The above expression becomes
\[(-1)^d\sum_{T \in \mc{X}} |W(T,G)(F)|^{-1}\int_{t \in T(F)_{\sr}} |D(t)| \sum_{(t_0,t_0',\lambda)}r_{\mu,\lambda}f'(t_0')f(t_0)dt, \]
where now the second sum runs over triples $(t_0,t_0',\lambda)$ with $t_0$ a $G(F)$-conjugacy class and $t_0'$ a $G_b(F)$-conjugacy class, both stably conjugate to $t$, and $\lambda \in X_*(T_{t_0})$ mapping to $\tx{inv}(t_0,t_0') \in X_*(T_{t_0})_\Gamma$.

Let $\mc{X}'$ be a set of representatives for those stable classes of maximal tori of $G_b$ that transfer to $G$. We have a bijection $\mc{X} \leftrightarrow \mc{X}'$. Fix arbitrarily an admissible isomorphism $T \to T'$ for any $T \in \mc{X}$ and $T' \in \mc{X}'$ that correspond under this bijection. It induces an isomorphism $W(T,G) \to W(T',G_b)$ of finite algebraic groups, as well as an isomorphism $T(F) \to T'(F)$ of toplogical groups that preserves the chosen measures (since we have arrange the measures to be compatible). Given $t \in T(F)_{\sr}$ let $t' \in T'(F)_{\sr}$ be its image under the admissible isomorphism. Then $|D(t)|=|D(t')|$, and \eqref{EqHeckePairingStep1} becomes
\begin{equation}
\label{EqHeckePairingStep2}
(-1)^d\sum_{T' \in \mc{X}'} |W(T',G_b)(F)|^{-1}\int_{t' \in T'_{\sr}(F)} |D(t')| \sum_{(t_0,t_0',\lambda)}r_{\mu,\lambda}f'(t_0')f(t_0)dt, 
\end{equation}
where now the second sum runs over triples $(t_0,t_0',\lambda)$, where $t_0$ is a $G(F)$-conjugacy class, $t_0'$ is a $G_b(F)$-conjugacy class, both are stably conjugate to $t'$, and $\lambda \in X_*(T_{t_0})$ maps to $\tx{inv}(t_0,t_0') \in X_*(T_{t_0})_\Gamma$. Reversing the arguments from the beginning of this proof we see that this expression equals $\<f',T_{b,\mu}^{G \to G_b}f\>_{G_b}$.
\end{proof}

In \S\ref{sub:6.3} we will define by geometric means an operator 
\[ \tilde T_{b,\mu}^{G \to G_b} : C_c(G(F)_{\elli},\Lambda)_{G(F)} \to C_c(G_b(F)_{\elli},\Lambda)_{G(F)} \] and show (Proposition \ref{PropositionLiftOfPhi}) that $T_{b,\mu}^{G \to G_b}$ corresponds to $\tilde T_{b,\mu}^{G \to G_b}$ under \eqref{eq:orbint}.

\section{The Lefschetz-Verdier trace formula for v-stacks}
\label{SectionLefschetzVerdier}

The goal of this section is to build up some machinery related to the Lefschetz-Verdier trace formula.

We briefly review the setup in the context of a separated finite type morphism of schemes $p\from X\to \Spec k$, where $k$ is an algebraically closed field.   Let $\ell$ be a prime unequal to the characteristic of $k$, and let $A$ be an object of $D(X_{\et},\overline{\Q}_\ell)$, the derived category of \'etale $\overline{\Q}_\ell$-sheaves on $X$.  Suppose $f\from X\to X$ is a $k$-linear endomorphism.  If we are given the additional datum of a morphism\footnotemark $Rf_!A\to A$,  we obtain an operator $Rp_!A\isom Rp_!Rf_!A\to Rp_!A$ on the compactly supported cohomology $Rp_!A=R\Gamma_c(X,A)$.
\footnotetext{Equivalently, a morphism $A\to Rf^!A$.  A special case occurs when $f$ is an automorphism, $A$ is an honest sheaf on $X_{\et}$,  and $A\to f^*A$ is a morphism, such as the identity morphism on the constant sheaf $A=\overline{\Q}_\ell$.  More generally, we may replace $f$ with an algebraic correspondence $c=(c_1,c_2)\from C\to X\times_k X$.  In that setting, the required extra datum is a morphism $c_1^*\mc{F} \to c_2^!\mc{F}$:  this is the notion of a {\em cohomological correspondence} lying over $c$.}

In the special case that $X$ is proper over $k$,  so that $R\Gamma_c(X,A)=R\Gamma(X,A)$, the Lefschetz-Verdier trace formula \cite{SGA5}, \cite{VarshavskyLefschetz} expresses $\tr (f\vert R\Gamma(X,A))$ in terms of data living on the fixed point locus $\Fix(f)$ of $f$.   In particular, at isolated fixed points $x\in \Fix(f)$, there are local terms $\loc_x(f,A) \in \overline{\Q}_\ell$, and if all fixed points are isolated, then $\tr (f\vert R\Gamma(X,A))$ is the sum of the $\loc_x(f,A)$.  

In order to apply the Lefschetz-Verdier trace formula, we need to assume that $A$ satisfies a suitable finiteness hypothesis (constructible of bounded amplitude).   Under this hypothesis one establishes an isomorphism \cite[Expos\'e III,  (3.1.1)]{SGA5}
\begin{equation}
\label{EqCrucialIsomorphismSchemes}
 \D A \boxtimes_k^{\LL} A \isom \RHom(\pr_1^*A, \pr_2^!A) 
 \end{equation}
 in $D((X\times_k X)_{\et},\overline{\Q}_\ell)$,  where $\pr_1,\pr_2\from X\times_k X\to X$ are the projection maps, and $\D$ is Verdier duality relative to $k$.  Once \eqref{EqCrucialIsomorphismSchemes} is established, the definition of local terms and the validity of the Lefschetz-Verdier trace formula can be derived by applying Grothendieck's six functor formalism.    The special case where $X=\Spec k$ is instructive;  the finiteness condition on $A$ is that it be a perfect complex of $\overline{\Q}_\ell$-vector spaces,  and then \eqref{EqCrucialIsomorphismSchemes} reduces to the fact that $A^\vee\otimes^{\LL} A\to \operatorname{RHom}(A,A)$ is an isomorphism.   This allows us to express the trace of an endomorphism $f\in \End A$ as the image of $f$ under the evaluation map $A^\vee\otimes^{\LL} A\to \overline{\Q}_\ell$.  

In this section we extend the formalism of the Lefschetz-Verdier trace formula to the setting of perfectoid spaces, diamonds, and v-stacks.   The main result is Theorem \ref{ThmLefschetzVerdier} and its Corollary \ref{CorLefschetzVerdier}.  We very closely follow the approach of \cite{LuZheng}, putting a suitable \emph{ symmetric monoidal 2-category of cohomological correspondences} at center stage.   In both the schematic and perfectoid settings, the finiteness condition required of the object $A$ can be stated in terms of the property of {\em universal local acyclicity} (ULA);  as noted in \cite[Theorem IV.2.23]{FarguesScholze}, this is precisely the hypothesis necessary to obtain the isomorphism in \eqref{EqCrucialIsomorphismSchemes}.  

The statement of Lefschetz-Verdier is formally identical in the schematic and perfectoid settings. However, in the perfectoid setting, there arises the possibility that the fixed point locus $\Fix(f)$ has the structure of a locally profinite set, in which case the local terms appearing in Lefschetz-Verdier are not a function on $\Fix(f)$, but rather a {\em distribution} on $\Fix(f)$.  This observation is critical to our applications.  

For our applications, we have included two additional theorems concerning local terms in the perfectoid setting, which could also have been stated in the schematic setting and may be of independent interest.  Theorem \ref{ThmKunnethForCharacteristicClass} is a sort of K\"unneth isomorphism for local terms on a fiber product of stacks.   Theorem \ref{ThmCharClassOnXModG} states that in the situation of a smooth group $G$ acting on a diamond $X$,  the local terms corresponding to individual elements $g\in G$ agree with local terms computed on the quotient stack $[X/G]$.

\subsection{Decent v-stacks and the six-functor formalism}

We recall here some material from \cite{ScholzeEtaleCohomology} and \cite{EnhancedSixFunctors} on the main classes of geometric objects we deal with - perfectoid spaces, diamonds, and v-stacks - and their associated \'etale cohomology formalism. 

Let $\Perf$ be the category of perfectoid spaces in characteristic $p$.  There are four topologies we consider on $\Perf$, which we list from coarsest to finest:  the analytic topology, the \'etale topology, the pro-\'etale topology, and the v-topology.  The v-topology is a rough analogue of the fpqc topology on schemes.  All representable presheaves on $\Perf$ are sheaves for the v-topology \cite[Theorem 1.2]{ScholzeEtaleCohomology}.

A {\em diamond} is a pro-\'etale sheaf on $\Perf$ of the form $X/R$, where $X$ is a perfectoid space and $R\subset X \times X$ is a pro-\'etale equivalence relation.  Diamonds are automatically v-sheaves \cite[Proposition 11.9]{ScholzeEtaleCohomology}.  A particularly well-behaved class of diamonds are the locally spatial diamonds \cite[Definition 1.4] {ScholzeEtaleCohomology}.  There is a natural functor $X\mapsto X^\diamond$ from analytic adic spaces over $\Z_p$ to locally spatial diamonds.  

 A v-sheaf $Y$ on $\Perf$ is {\em small} if there exists a surjective map $X\to Y$ from a perfectoid space $X$.  A {\em v-stack} is a stack over $\Perf$ with its $v$-topology.   A {\em small v-stack} \cite[Definition 12.4]{ScholzeEtaleCohomology} is a v-stack $Y$ on $\Perf$ such that there exists a surjective map $X\to Y$ from a perfectoid space $X$, such that $X\times_Y X$ is a small v-sheaf.  

As with any category of stacks, v-stacks form a strict $(2,1)$-category.  The objects of this category are v-stacks $X$, which are themselves categories fibered in groupoids over $\Perf$.  The morphisms between v-stacks $X\to Y$ are functors between fibered categories.  Given two morphisms $f_1,f_2\from X\to Y$, a 
2-morphism $\alpha\from f_1\Rightarrow f_2$ is an invertible natural transformation between functors.  

\begin{exm} \label{ExmBG} Let $S$ be a diamond, and let $G\to S$ be a group diamond.  The stack $BG=[S/G]$ classifying $G$-torsors is a small v-stack, as $S\times_{[S/G]} S\isom G$ is already a diamond.  Let $H\to S$ be another group diamond.  The morphisms $[S/G]\to [S/H]$ correspond to $S$-homomorphisms $G\to H$.  Suppose we are given two homomorphisms $f_1,f_2\from G\to H$, inducing morphisms $\phi_1,\phi_2\from [S/G]\to [S/H]$.  The set of 2-morphisms $\phi_1\Rightarrow \phi_2$ may be identified with the set of $h\in H(S)$ satisfying $f_1 = (\ad h)\circ f_2$.  
\end{exm}

We use the notation $\ast$ to indicate ``horizontal'' composition between 2-morphisms.  Thus if $X,Y,Z$ are v-stacks, $f_1,f_2\from X\to Y$ and $g_1,g_2\from Y\to Z$ are morphisms, and $\alpha\from f_1\Rightarrow f_2$ and $\beta\from g_1\Rightarrow g_2$ are 2-morphisms, then $\beta\ast \alpha\from g_1\circ f_1\Rightarrow g_2\circ f_2$ is another 2-morphism.

Let $\Lambda$ be a ring which is $n$-torsion for some $n$ prime to $p$.  For every small v-stack $X$ there is a triangulated category $D_{\et}(X,\Lambda)$ \cite[Definition 1.7]{ScholzeEtaleCohomology}.  If $X$ is a locally spatial diamond, then $D_{\et}(X,\Lambda)$ is equivalent to the left-completion of the derived category of sheaves of $\Lambda$-modules on the \'etale topology of $X$ \cite[Proposition 14.15]{ScholzeEtaleCohomology}.

The familiar six functors of Grothendieck have analogues in the world of small v-stacks \cite[Definition 1.7]{ScholzeEtaleCohomology}.  There is a derived tensor product $\otimes_\Lambda^\LL$ and a derived internal hom $\RHom_{\Lambda}$.  For any morphism $f\from Y\to X$ of small v-stacks, there is a pair of adjoint functors $f^*$ and $Rf_*$.  

\begin{rmk}\label{RmkAdjointnessCompatibility} The adjointness between $f^*$ and $Rf_*$ is compatible with 2-morphisms, in the following sense.  Suppose $\alpha\from f\Rightarrow g$ is a 2-morphism between $f,g\from Y\to X$.  Then there are natural isomorphisms $\alpha_*\from f_*\to g_*$ and $\alpha^*\from f^*\to g^*$, such that the following diagrams commute:
\[
\xymatrix{
\id_{D_{\et}(X,\Lambda)} \ar[r]^{\text{unit}} \ar[dr]_{\text{unit}} & f_*f^* \ar[d]^{\alpha_*\alpha^*} \\
& g_*g^*
}
\;\;
\xymatrix{
f^*f_* \ar[rr]^{\text{counit}} \ar[d]_{\alpha_*\alpha^*} && \id_{D_{\et}(Y,\Lambda)} \\
g^*g_* \ar[urr]_{\text{counit}} &&
}
\]
\end{rmk} 

We propose for convenience the following definition. 

\begin{dfn} A morphism $f\from Y\to X$ is {\em representable in nice diamonds} or simply {\em nice} if is compactifiable \cite[Definition 22.2]{ScholzeEtaleCohomology}, representable in locally spatial diamonds \cite[Definition 13.3]{ScholzeEtaleCohomology}, and locally of finite geometric transcendence degree \cite[Definition 21.7]{ScholzeEtaleCohomology}.  
\end{dfn}

If $f\from Y\to X$ is representable in nice diamonds, then there are an adjoint pair of functors $Rf_!$ and $Rf^!$ \cite[Sections 22 and 23]{ScholzeEtaleCohomology}.  

\begin{thm}[{\cite[Theorem 1.8]{ScholzeEtaleCohomology}}]
\label{ThmSixOperations} The six operations $\otimes_\Lambda$, $\RHom_\Lambda$, $f^*$, $Rf_*$, $Rf_!$, and $Rf^!$ obey the rules:
\begin{enumerate}
\item[(P1.)]$f^*A \otimes_\Lambda^\LL f^*B \isom f^*(A\otimes_\Lambda^\LL B)$,
\item[(P2.)] $Rf_*\RHom_\Lambda(f^*A,B)\isom \RHom_\Lambda(A,Rf_*B)$,
\item[(P3.)] $Rf_!(A\otimes_\Lambda^\LL f^*B)\isom Rf_!A\otimes_\Lambda^\LL B$ (the projection formula),
\item[(P4.)] $\RHom_\Lambda(Rf_!A,B)\isom Rf_*\RHom_\Lambda(A,Rf^!B)$ (local Verdier duality),
\item[(P5.)] $Rf^!\RHom_\Lambda(A,B)\isom \RHom_\Lambda(f^*A,Rf^!B)$
\end{enumerate}
\end{thm}

We also need the following base change results.

\begin{thm}[{\cite[Theorem 1.9]{ScholzeEtaleCohomology}}]
\label{ThmBaseChange}
Let 
\begin{equation}
\label{EqCartesianDiagram}
\xymatrix{
Y' \ar[r]^{\tilde{g}} \ar[d]_{f'} & Y \ar[d]^f \\
X' \ar[r]_g & X 
}
\end{equation}
be a cartesian diagram of small v-stacks.
\begin{enumerate}
\item[(BC1.)] If $f$ is representable in nice diamonds, then $g^*Rf_! \isom Rf'_!\tilde{g}^*$.
\item[(BC2.)] If $g$ is representable in nice diamonds, then $Rg^!Rf_*\isom Rf_*'R\tilde{g}^!$. 
\end{enumerate}
\end{thm}

There is a notion of {\em cohomological smoothness} \cite[Definition 23.8]{ScholzeEtaleCohomology} for morphisms between small v-stacks which are representable in nice diamonds.  Let $\Lambda$ be an $n$-torsion ring for some $n$ not divisible by $p$.  For a morphism $f\from Y\to X$ of small v-stacks which is representable in nice diamonds, there is a natural map of functors
\begin{equation}
\label{EqPullbackToUpperShriek}
Rf^!\Lambda_X\otimes_\Lambda^{\LL} f^*\to Rf^!,
\end{equation}
adjoint to 
\[Rf_!(Rf^!\Lambda_X \otimes^{\LL}_{\Lambda} f^*A) \stackrel{(P3)}{\isomto} Rf_!Rf^!\Lambda_X \otimes^{\LL}_{\Lambda} A \stackrel{\text{counit}}{\to} A.\]  If $f$ is cohomologically smooth, then \eqref{EqPullbackToUpperShriek} is an equivalence \cite[Theorem 1.10]{ScholzeEtaleCohomology}.  Furthermore, the object $Rf^!\Lambda_X$ is invertible in the monoidal category $D_{\et}(Y,\Lambda)$.  (For $X$ a small v-stack, an object $A$ of $D_{\et}(X,\Lambda)$ is invertible if and only if \'etale-locally on $X$ there is an isomorphism $A\isom L[n]$ for some invertible $\Lambda$-module $L$.)

There is also the following base change theorem for cohomologically smooth morphisms.

\begin{thm}[{\cite[Theorem 1.10]{ScholzeEtaleCohomology}}] \label{ThmSmoothBaseChange}
In the cartesian diagram \eqref{EqCartesianDiagram}, assume that $f$ is cohomologically smooth.  Then $\tilde{g}^*Rf^!\isom R(f')^!g^*$ and $(f')^*Rg^!\isom R\tilde{g}^! f^*$.  
\end{thm}

In our applications, we will crucially need to deal with \emph{stacky} morphisms $f\from Y\to X$ between v-stacks. These morphisms are never representable in nice diamonds, and the hoped-for functors $Rf_!$ and $Rf^!$ were not constructed in \cite{ScholzeEtaleCohomology}. In the companion paper \cite{EnhancedSixFunctors}, we have extended the $!$-functor formalism to certain stacky maps between certain small v-stacks, using the $\infty$-categorical machinery of \cite{LiuZheng}. Here we briefly recall the main results from \cite{EnhancedSixFunctors}, referring the reader to that paper for a more detailed discussion.

\begin{dfn}[{\cite[Definition 1.1]{EnhancedSixFunctors}}]  A {\em decent v-stack} is a small v-stack $X$ such that the diagonal $\Delta_X \from X\to X \times X$ is representable in locally separated locally spatial diamonds \cite[Definition 4.3]{EnhancedSixFunctors}, and such that there is a locally separated locally spatial diamond $U$ with a morphism $ U\to X$ which is strictly surjective \cite[Definition 4.1]{EnhancedSixFunctors}, representable in locally spatial diamonds, and which locally on $U$ is compactifiable of finite dim.trg and cohomologically smooth. Any such morphism $U \to X$ is called a \emph{chart} for $X$.

A morphism $f:X\to Y$ between decent v-stacks is \emph{fine} if there exists a commutative diagram
\[
\xymatrix{W\ar[d]_{b}\ar[r]^{g} & V\ar[d]^{a}\\
X\ar[r]^{f} & Y
}
\]
where the vertical maps are charts and $g$ is locally on $W$ compactifiable
of finite dim.trg.
\end{dfn}

Note that these definitions rely on the notion of cohomological smoothness for morphisms representable in nice diamonds. 

In our applications, we will often need to deal with decent v-stacks equipped with a structure map to a fixed v-stack $S$. We refer to such objects as decent $S$-v-stacks.

In \cite{EnhancedSixFunctors}, we showed that decent v-stacks and fine morphisms between them are very reasonable notions: 

\begin{itemize}
\item Any locally separated locally spatial diamond is a decent v-stack. In particular, if $X$ is any analytic adic space over $\Spa \mathbf{Z}_p$, the associated diamond $X^\lozenge$ is a decent v-stack.
\item Decent v-stacks are Artin v-stacks in the sense of \cite{FarguesScholze}. 
\item Any absolute product or fiber product of decent v-stacks is decent.
\item Fine morphisms are stable under composition and (decent) base change. 
\item Any morphism of decent v-stacks which is representable in nice diamonds is fine.
\end{itemize}

The key motivation for singling out fine morphisms of decent v-stacks is the following result.
\begin{thm}[{\cite[Theorem 1.4]{EnhancedSixFunctors}}] If $f\from Y\to X$ is any fine map of decent v-stacks, there exist functors $Rf_!$ and $Rf^!$ satisfying the properties listed in Theorems \ref{ThmSixOperations} and \ref{ThmBaseChange}, and agreeing with the constructions in \cite{ScholzeEtaleCohomology} when $f$ is representable in nice diamonds. Moreover, the associations $f \rightsquigarrow Rf_!$ and $f \rightsquigarrow Rf^!$ naturally have the structure of pseudo-functors, and on the class of proper morphisms there is a pseudo-natural isomorphism $Rf_! \to Rf_*$.

Finally, there is a notion of cohomological smoothness for fine maps between decent v-stacks, which can be defined extrinsically in terms of charts or intrinsically in terms of the $!$-functors \cite[Proposition 4.17]{EnhancedSixFunctors}, agreeing with the notion discussed above for morphisms representable in nice diamonds, and with the same formal properties as in the representable case. In particular, the map (\ref{EqPullbackToUpperShriek}) is an isomorphism for cohomologically smooth morphisms, and the evident analogue of Theorem \ref{ThmSmoothBaseChange} holds.
\end{thm}

Again, we refer the reader to \cite{EnhancedSixFunctors} for a complete discussion.

Finally, we need the notion of the relative dualizing complex for v-stacks.

\begin{dfn}[The dualizing complex]  Let $f\from X\to S$ be a fine morphism of decent v-stacks.  We define $K_{X/S}=Rf^!\Lambda$, an object in $D_{\et}(X,\Lambda)$.  

Suppose that $S$ is connected, and that $f$ is proper.  Then $Rf_*=Rf_!$, and so there is a morphism $Rf_*K_{X/S}=Rf_!Rf^!\Lambda \overset{\mathrm{counit}}{\longrightarrow} \Lambda$, which induces a morphism on the level of global sections\footnote{Here and elsewhere, we write $H^0(X,A)$ as shorthand for $\Hom(\Lambda_X,A)$ whenever $A\in D_{\et}(X,\Lambda)$.}
\[ H^0(X,K_{X/S})=H^0(S,Rf_*K_{X/S})\to H^0(S,\Lambda)=\Lambda, \]
which we notate as $\omega\mapsto \int_X \omega$.  
\end{dfn}

We record the following lemmas for convenience.
\begin{lem} 
\label{LemInvertibleSheaf}
Let $f\from Y\to X$ be a fine morphism of decent v-stacks, and let $A,I\in D_{\et}(X,\Lambda)$ be any objects with $I$ invertible.  The natural map $Rf^!A\otimes^{\LL}_{\Lambda} f^*I \to Rf^!(A\otimes^{\LL}_{\Lambda} I)$ of \eqref{EqPullbackToUpperShriek} is an isomorphism.
\end{lem}

\begin{proof} Let $w_{A,I}$ be this morphism.  We also get a map $w_{A\otimes^{\LL}_{\Lambda} I,I^{-1}}\from Rf^!(A \otimes^{\LL}_{\Lambda} I) \otimes^{\LL}_{\Lambda} f^*I^{-1} \to Rf^!A$, which induces $Rf^!(A\otimes^{\LL}_{\Lambda} I)\to Rf^!A \otimes^{\LL}_{\Lambda} f^*I$.  This is the inverse to $w_{A,I}$.
\end{proof}

\begin{lem}  \label{LemInverseDualizingComplexes} Let $f\from Y\to X$ be a fine morphism of decent v-stacks which is cohomologically smooth.  Then there is a canonical isomorphism
\begin{equation}
\label{EqInverseDualizingComplexes}
R\Delta_{Y/X}^!\Lambda_{Y\times_X Y} \otimes^{\LL}_{\Lambda} Rf^!\Lambda_X \isom \Lambda_Y,
\end{equation}
so that $K_{Y/Y\times_X Y} \isom K_{Y/X}^{-1}$ is invertible.  
\end{lem}

\begin{proof} Let $\pr_1,\pr_2\from Y\times_X Y\to Y$ be the projection morphisms; each is cohomologically smooth.  We have 
\[\Lambda_Y\isom R\id_Y^!\Lambda_Y\isom R\Delta_{Y/X}^!R\pr_1^!\Lambda_Y \isom R\Delta_{Y/X}^!\Lambda_{Y\times_X Y} \otimes^{\LL}_{\Lambda} \Delta_{Y/X}^*R\pr_1^!\Lambda_Y,\]
where in the last isomorphism we used Lemma \ref{LemInvertibleSheaf} combined with the cohomological smoothness of $\pr_1$.  Now use Theorem \ref{ThmSmoothBaseChange} to obtain an isomorphism $\Delta_{Y/X}^*R\pr_1^!f^*\Lambda_X \isom \Delta_{Y/X}^*\pr_2^*Rf^!\Lambda_X\isom Rf^!\Lambda_X$.
\end{proof}

For the remainder of the section,  we fix $\Lambda$, an $n$-torsion ring for some $n$ not divisible by $p$.  We will now start writing $f_!$ for $Rf_!$ and $\otimes$ for $\otimes^{\LL}_{\Lambda}$, etc.

\subsection{Examples}

We wish to illustrate the behavior of the functors $f_!$ and $f^!$ through a long list of examples. In the following, assume $S=\Spd C$ for an algebraically closed perfectoid field $C$, or else $S=\Spd k$ for an algebraically closed discrete field of characteristic $p$.  In both cases, we freely identify $D_{\et}(S,\Lambda)$ with the derived category of $\Lambda$-modules. Observe that in both cases, $S$ is decent: this is trivial for $S=\Spd C$, while for $S=\Spd k$ the condition on the diagonal is easy to check, and one can show that $U = S \times \Spd \mathbf{F}_p((t^{1/p^{\infty}}))\to S$ is a chart.

%\begin{exm} An object $A\in D_{\et}(S,\Lambda)$ is ULA over $S$ if and only if (when considered as a complex of $\Lambda$-modules) $A$ is perfect.  
%\end{exm}

\begin{exm} \label{ExmLocallyProfiniteSet} Let $T$ be a locally profinite set, and let $T_S=\underline{T} \times S$ be the associated constant diamond over $S$.  Then $f\from T_S\to S$ is representable in nice diamonds.  Let $C(T,\Lambda)$ be the ring of continuous functions $T\to \Lambda$, for the discrete topology on $\Lambda$.  We may naturally identify $D_{\et}(T_S,\Lambda)$ with the derived category of the abelian category of smooth $C(T,\Lambda)$-modules in the sense of \S \ref{Sheaveslpfsets}. Indeed, $D_{\et}(T_S,\Lambda) \cong D(T_{S,\et},\Lambda)$ since $T_S$ locally has cohomological dimension zero, and then the site $T_{S,\et}$ agrees with the site associated with the topological space $T$.   Finally,  Lemma \ref{lem:sheaflps} identifies $\Sh(T,\Lambda)$ with the category of smooth $C(T,\Lambda)$-modules.   This identification matches the constant sheaf $\Lambda$ with the smooth $C(T,\Lambda)$-module $C_c(T,\Lambda)$. Under this identification, we have concrete descriptions of the four operations associated with $f\from T_S\to S$.   Here we freely use some language and notation from \S \ref{Sheaveslpfsets}.
\begin{itemize}
\item $f^*$ sends a $\Lambda$-module $M$ to the smooth $C(T,\Lambda)$-module $C_c(T,\Lambda)\otimes_\Lambda M$,
\item $f_*$ sends a smooth $C(T,\Lambda)$-module $M$ to $M^c$ regarded as a $\Lambda$-module.
\item $f_!$ sends a smooth $C(T,\Lambda)$-module $M$ to its underlying $\Lambda$-module.
\item $f^!$ sends a $\Lambda$-module $M$ to $\Rhom_\Lambda(C_c(T,\Lambda),M)^s$.  In particular $H^0(T_S, f^!\Lambda_S) \isom \Dist(T,\Lambda)$, the module of $\Lambda$-valued distributions on $T$.
\end{itemize}
%It will not generally be true that $\Lambda_{T_S}$ is ULA over $S$, unless $T$ is discrete.
\end{exm}

\begin{exm} \label{ExLocallyProPGroup} Suppose $G$ is a locally pro-$p$ group.  Let $[S/G_S]$ be the classifying v-stack of $G_S$-torsors.  Assume that there is a separated locally spatial diamond $X$ together with a strictly surjective cohomologically smooth map $X\to S$, and admitting a free $G_S$-action.\footnote{In particular, these hypotheses are satisfied when $G$ is a closed subgroup of $\mathrm{GL}_n(E)$ for $E$ a finite extension of $\mathbf{Q}_p$ or $\mathbf{F}_p((t))$, which will cover all the cases we need. In this situation, we may simply take $X = \mathrm{GL}_{n,E}^{\lozenge} \times S$ with its evident $G_S$-action, where $\mathrm{GL}_{n,E}$ is regarded as a rigid analytic group over $E$. See \cite[Example IV.1.9.iv]{FarguesScholze} for some additional discussion.} Then $[S/G_S]$ is a decent v-stack. Indeed, the condition on the diagonal is easy to check, and one can also check that the natural map $X/G_S \to [S/G_S]$ is a chart.
%We assume that there is a free action of $G_S$ on a cohomologically smooth diamond $X\to S$.   
%In that case $[S/G_S]$ has a cohomologically smooth uniformization by the cohomologically smooth diamond $X/G_S$;  thus $[S/G_S]$ is a smooth $S$-Artin v-stack. 

Let $q\from S\to [S/G_S]$ be the quotient map.   Then $q$ is representable in nice diamonds.   The functor $q^*$ is an equivalence of monoidal categories between $D_{\et}([S/G_S],\Lambda)$ and the derived category of $\Lambda$-modules with a smooth $G$-action \cite[Theorem V.1.1]{FarguesScholze}.   (Strictly speaking, if $M$ is an object of $D_{\et}([S/G_S],\Lambda)$, then $q^*M$ is a bare $\Lambda$-module, but then for each $g\in G$ there is a 2-morphism $\alpha_g\from q\implies q$ as in Example \ref{ExmBG} inducing an automorphism of $q^*M$.)   

With respect to this equivalence of categories, the functors $q^*,q^!$ (resp., $q_*,q_!$) take the following values on a $\Lambda$-module $M$ (resp., a $\Lambda$-module $M$ with smooth $G$-action).
\begin{itemize}
\item $q^*M=M$ with its $G$-action forgotten.
\item $q_*M=C(G,M)^{G\text{-sm}}$, the module of continuous $M$-valued functions on $G$ which are smooth with respect to the action of $G$ by right translation.
\item $q_!M=C_c(G,M)=C_c(G,\Lambda)\otimes M$.
\item $q^!M=\Hom_G(C_c(G,\Lambda),M)$.  In particular, $q^!\Lambda=\mathrm{Haar}(G,\Lambda)$ is the module of left-invariant Haar measures on $G$.
\end{itemize}

We give justifications for these expressions in the next example, which is more general.
\end{exm}

\begin{exm} \label{ExLocallyProPGroupQuotient}  This example generalizes the previous one.  Let $G$ be a locally pro-$p$ group as in Example \ref{ExLocallyProPGroup}.   Suppose $H\subset G$ is a closed subgroup.  Then $[S/H_S]$ is also a decent v-stack.  Let
\[ q\from [S/H_S]\to [S/G_S] \]
be the quotient map, so $q$ is representable in nice diamonds.   The functors $q^*,q^!$ (resp., $q_*,q_!$) take the following values on a $\Lambda$-module $M$ with smooth $G$-action (resp., smooth $H$-action):
\begin{itemize}
\item $q^*M$ is the restriction of $M$ from $G$ to $H$.
\item $q_*M=\mathrm{Ind}_H^G M$ is the smooth induction of $M$ from $H$ to $G$.
\item $q_!M=\mathrm{cInd}_H^G M$ is the compact induction of $M$ from $H$ to $G$.
\item $q^!M$ seems difficult to describe explicitly in general, but there are two special cases:
\begin{enumerate}
\item If $H\subset G$ is open, then $q^!M\isom q^*M$ is the restriction of $M$ from $G$ to $H$.
\item If $H$ is a direct factor of $G$, so that $G=H\times H'$, then
\[ q^!M= \Hom_{H'}(C_c(H',\Lambda),M)^{H\mathrm{-sm}}. \]
\end{enumerate}
\end{itemize}

For the claims regarding $q^*$ and $q_*$:  Let $q_H\from S\to [S/H_S]$ be the quotient map for $H$, and similarly for $q_G\from S\to [S/G_S]$.  Then the underlying module of $q^*M$ is $q_H^*q^*M=q_G^*M$,  which we have identified with $M$ itself.   For $h\in H$ we have a 2-morphism $\beta_h\from q_H\implies q_H$, which induces an action of $h\in H$ on $q_G^*M$, as well as the 2-morphism $\alpha_h\from q_G\implies q_G$ inducing the action of $h\in G$ on $q_G^*M$ as discussed in Example \ref{ExLocallyProPGroup};  these actions agree because $\alpha_h = \id_q\ast \beta_h$.  Since $q^*$ is restriction,  $q_*$ must be its right adjoint, which is smooth induction.

For the claim about $q_!$,  consider the cartesian diagram:
\[
\xymatrix{
(G/H)_S \ar[d]_{\tilde{q}} \ar[r]^{\tilde{q}_G} & [S/H_S] \ar[d]^q \\
S \ar[r]_{q_G} & [S/G_S]
}
\]
The underlying module of $q_!M$ is $q_G^*q_!M$.  By base change property (BC1) we have $q_G^*q_!M \isom \tilde{q}_!\tilde{q}_G^*M$, which by Example \ref{ExmLocallyProfiniteSet} is identified with the underlying $\Lambda$-module of $\tilde{q}_G^*M$.  The latter is the descent of $M$ along the $H$-torsor in topological spaces $G\to G/H$.   In our dictionary between sheaves on $G/H$ and smooth $C(G/H,\Lambda)$-modules, $\tilde{q}_G^*M$ is the module of smooth $H$-equivariant functions $G\to M$ which are compactly supported modulo $H$.  This is none other than the compact induction of $M$ from $H$ to $G$.  (To show that the action of $G$ is by right translation on such functions,  one has to appeal to the compatibility of base change with the 2-isomorphisms $\alpha_g\from q_G\implies q_G$.)

We now turn to the claims for $q^!$.  In the case that $H\subset G$ is open, $q$ is \'etale, and so $q^!\isom q^*$.  In the case that $G=H\times H'$, suppose $M$ is a smooth $H$-module;  we have an isomorphism of smooth $G$-modules
\[ q_!M = \mathrm{cInd}_H^G M \isom M\boxtimes C_c(H',\Lambda),\]
from which it is easy to see that the right adjoint to $q_!M$ is as claimed.  
\end{exm}

\begin{exm}\label{ExLocallyProPStructureMap} Suppose $G$ is a locally pro-$p$-group as in Example \ref{ExLocallyProPGroup}.   Let $f\from [S/G_S]\to S$ be the structure morphism.   Then $f$ is fine, and in fact is cohomologically smooth.  The functors associated with $f$ have the following descriptions:
\begin{itemize}
\item $f^*M=M$ with trivial $G$-action.
\item $f_*M=M^G$ is the (derived) $G$-invariants of $M$, that is, the group cohomology.  
\item $f_!M = (M\otimes\mathrm{Haar}(G,\Lambda))_G$ is the group homology of $M$ twisted by the module of Haar measures.
\item $f^!M = \mathrm{Haar}(G,\Lambda)^* \otimes M$ is $M$ (with trivial $G$-action) twisted by the dual of the module of Haar measures.  In particular $K_{[S/G_S]/S}=\mathrm{Haar}(G,\Lambda)^*$.
\end{itemize}

The claim about $f^*$ is clear from the definitions, and $f_*$ is the right adjoint to $f^*$.  Next we consider $f^!$.  Since $f$ is cohomologically smooth, we have $f^!M\isom f^*M\otimes f^!\Lambda$.  Let $q\from S\to [S/G_S]$ be as in Example \ref{ExLocallyProPGroup}, so that $f\circ q = \id_S$.  We have $M=q^!f^!M=\mathrm{Haar}(G,\Lambda)\otimes f^!M$, so that $f^!M=\mathrm{Haar}(G,\Lambda)^* \otimes M$ as claimed. From here it is easy to compute $f_!$ as the left adjoint of $f^!$.  
\end{exm}

%\begin{exm}\label{ExLocallyProPStackyMaps} This is a generalization of Example \ref{ExLocallyProPStructureMap}.  Let $G$ be a locally pro-$p$ group as in Example \ref{ExLocallyProPGroup}.  Let $H\subset G$ be a closed normal subgroup, and let $f\from [S/G_S] \to [S/(G/H)_S]$ be the morphism sending a $G_S$-torsor $Y\to X$ to the $(G/H)_S$-torsor $Y/H_S\to X$.   Then $f$ is cohomologically smooth.   In terms of smooth $G$- and $(G/H)$-modules we have:
%\begin{itemize}
%\item $f^*M$ is the inflation of $M$ from $G/H$ to $G$.
%\item $f_*M=M^H$ is the (derived) $H$-invariants of $M$.
%\item $f_!M = (M\otimes \mathrm{Haar}(H,\Lambda))_H$.  
%\item $f^!M=\mathrm{Haar}(H,\Lambda)^* \otimes M$.
%\end{itemize}
%The justifications for $f^*$ and $f_*$ being similar to the previous example, we turn to $f_!$ and $f^!$.  Since $f$ is cohomologically smooth, we have $f^!M\isom f^*M\otimes f^!\Lambda$.   Thus we must determine the invertible object $f^!\Lambda$.   This can be accomplished by applying Theorem \ref{ThmSmoothBaseChange} to the diagram:
%\[
%\xymatrix{
%[S/H_S] \ar[r]^{\tilde{q}} \ar[d]_{\tilde{f}} & [S/G_S] \ar[d]^f \\
%S \ar[r]_{q\;\;} & [S/(G/H)_S] 
%}
%\]
%We get that the module underlying $f^!\Lambda$ is $\tilde{q}^*f^!\Lambda\isom \tilde{f}^!q^*\Lambda=\mathrm{Haar}(H,\Lambda)^\ast$ by Example \ref{ExLocallyProPStructureMap}.  
%\end{exm} 

\begin{exm} \label{ExTModG} This example combines Examples \ref{ExmLocallyProfiniteSet} with \ref{ExLocallyProPStructureMap}.  Let $G$ be a locally pro-$p$ group as in Example \ref{ExLocallyProPGroup}.   Let $T$ be a locally profinite set equipped with a continuous action of $G$, and let $T_S$ be the constant diamond over $S$.  Then the stacky quotient $[T_S/G_S]$ is a decent v-stack, and the structure map $[T_S/G_S] \to S$ is fine. Indeed, we have already seen that $[S/G_S]$ is a decent v-stack fine over $S$, and the evident morphism $[T_S/G_S] \to [S/G_S]$ is representable in nice diamonds, so the claim immediately follows from \cite[Proposition 4.11]{EnhancedSixFunctors}.  The stack $[T_S/G_S]$ is not in general cohomologically smooth over $S$.  The category $D_{\et}([T_S/G_S],\Lambda)$ may be identified with the derived category of $G$-equivariant smooth $C(T,\Lambda)$-modules. Using this identification, we get a natural isomorphism 
\begin{equation}
\label{EqDualizingSheafForTModG}
H^0([T_S / G_S],K_{[T_S/G_S]/S}) \isom
\Hom_G(C_c(T,\Lambda)\otimes \mathrm{Haar}(G,\Lambda), \Lambda) 
\end{equation}
which we can think of as the space of $G$-invariant distributions on $T$ with values in $\mathrm{Haar}(G,\Lambda)^*$.   If $G$ is unimodular, and if we choose a Haar measure on $G$, then $H^0([T_S / G_S],K_{[T_S/G_S]/S})$ becomes isomorphic to $\Dist(T,\Lambda)^G$, the module of $G$-invariant distributions on $T$.   
\end{exm}

\subsection{The category of cohomological correspondences}

The Lefschetz-Verdier trace formula was expressed elegantly by Lu and Zheng \cite{LuZheng} in the language of symmetric monoidal 2-categories.  In brief, \cite{LuZheng} constructs such a category of {\em cohomological correspondences}, where the objects are pairs $(X,A)$, where $X$ is a scheme over a fixed base scheme $S$ and $A$ is an object of $D(X_{\et},\Lambda)$, and a morphism $(X,A)\to (X',A')$ is a correspondence $c=(c_1,c_2)\from C\to X\times_S X'$ together with a morphism $c_1^*A\to c_2^!A'$.  An endomorphism of a dualizable object $(X,A)$ has a categorical trace, which lives over the fixed point locus of $c$.  In the special case that $X=X'=C=S$ and $A$ is a perfect complex of $\Lambda$-modules, the categorical trace is just the Euler characteristic of an endomorphism of $A$. The trace formula is interpreted as the statement that the categorical trace is compatible with proper pushforwards.

We adapt here \cite{LuZheng} to the setting of v-stacks, but the same language could be used in the world of stacks in the scheme setting.  The main point of departure from \cite{LuZheng} is that stacks form a 2-category, and so one must keep track of the 2-morphisms witnessing commutativity of diagrams of stacks.  This means that the definition of cohomological correspondences we give (Definition \ref{DefCoCorr}) is a little more delicate than its analogue in \cite{LuZheng}.

First we recall some definitions and constructions concerning the categorical trace.

\begin{dfn}[{\cite[Definition 1.1, Construction 1.6]{LuZheng}}]
An object $X$ of a symmetric monoidal 2-category $(\mc{C},\otimes,1_{\mc{C}})$ is {\em dualizable} if there exists an object $X^\vee$ together with morphisms $\text{ev}_X\from X^\vee\otimes X \to 1_{\mc{C}}$ and $\text{coev}_X\from 1_{\mc{C}} \to X\otimes X^\vee$, such that the compositions

\[
\xymatrix{
X \ar[rr]^{\text{coev}_X\otimes \id_X\hspace{1cm}}
&&
X\otimes X^\vee \otimes X \ar[rr]^{\hspace{1cm}\id_X \otimes \text{ev}_X}
&& X
}
\]
and
\[
\xymatrix{
X^\vee \ar[rr]^{\id_{X^\vee}\otimes\text{coev}_X\hspace{1cm}}
&&
X^\vee \otimes X \otimes X^\vee \ar[rr]^{\hspace{1cm}\text{ev}_X\otimes \id_X}
&& X^\vee
}
\]
are isomorphic to the identities on $X$ and $X^\vee$, respectively.  Consequently, the functor $Y\mapsto X\otimes Y$ has right adjoint $Y\mapsto X^\vee\otimes Y$.  If $X$ is dualizable, then $X^\vee\otimes Y$ serves as an internal mapping object $\iHom(X,Y)$.   

Let $\Omega\mc{C}=\End(1_{\mc{C}})$ be the (1-)category of endomorphisms of the unit object of $\mc{C}$.  Let $f\in \End X$ be an endomorphism of a dualizable object $X$.  Define the categorical trace $\tr(f)$ as the composite:
\[
\xymatrix{
1_{\mc{C}} \ar[rr]^{\text{coev}_X} &&
X\otimes X^\vee \ar[rr]^{f\otimes \id_{X^\vee}} &&
X\otimes X^\vee \ar[rr]^{\text{ev}_X} &&
1_{\mc{C}}
}
\]
so that $\tr(f)$ is an object of $\Omega\mc{C}$.
\end{dfn}

\begin{exm} Let $\Lambda$ be an arbitrary ring, and let $D(\Lambda)$ be the derived category of $\Lambda$-modules.  An object $A$ of $D(\Lambda)$ is dualizable if and only if it is a perfect complex, in which case $\D A=\RHom(A,\Lambda[0])$ is a dual object.  (See Lemma \ref{lem:reflexivemodules} in the Appendix for a proof of this claim and related conditions.)  If $f$ is an endomorphism of the perfect complex of $A$, then the categorical trace $\tr(f)$ agrees with the Euler characteristic $\tr(f\vert A)$ of $f$.  
\end{exm}

Next we define the symmetric monoidal 2-category $\Corr_S$ of correspondences of v-stacks, and its cohomological enhancement $\CoCorr_S\to \Corr_S$. In the following discussion, we fix a decent v-stack $S$. 

\begin{dfn}[The category of correspondences] \label{DefCorr} We define a symmetric monoidal 2-category $\Corr_S$ as follows:
\begin{itemize}
\item The objects of $\Corr_S$ are decent $S$-v-stacks $X$ whose structure map $X \to S$ is fine.  
\item Given objects $X$ and $X'$, the category $\Hom_{\Corr_S}(X,X')$ has for its objects the correspondences:
\begin{equation}
\label{EqCorrespondence}
\xymatrix{
& C \ar[dl]_{c_1} \ar[dr]^{c_2} & \\
X && X'
}
\end{equation}
where each $c_i$ a morphism of decent v-stacks, and $c_2$ is assumed to be fine.\footnote{By \cite[Proposition 4.10]{EnhancedSixFunctors}, it is then automatic that the composition $C \to X' \to S$ is fine, and then also that $c_1:C \to X$ is fine.}  The composition of $(c_1,c_2)\from C\to X\times_S X'$ with $(d_1,d_2)\from D\to X'\times_S X''$ is the correspondence $(c_1d_1',d_2c_2')$ defined by the diagram:
\begin{equation}
\label{EqCompositionOfCorrespondences}
\xymatrix{
& & C\times_{X'} D \ar[dl]_{d_1'} \ar[dr]^{c_2'} && \\
& C \ar[dl]_{c_1} \ar[dr]^{c_2} && D \ar[dl]_{d_1} \ar[dr]^{d_2}  & \\
X && X' && X'' 
}
\end{equation}

\item If $c=(c_1,c_2)\from C\to X\times_S X'$ and $d=(d_1,d_2)\from D\to X\times_S X'$ represent two objects in $\Hom_{\Corr_S}(X,X')$, a 2-morphism $c\Rightarrow d$ is an equivalence class of 2-commutative diagrams
\begin{equation}
\label{EqTwoMorphism}
\xymatrix{
& & C \ar[dll]_{c_1} \ar[drr]^{c_2} \ar[dd]_p & & \\
X & \Downarrow \alpha_1 && \alpha_2 \Downarrow & X' \\
& & D \ar[ull]^{d_1} \ar[urr]_{d_2} & &
}
\end{equation}
where $p$ is proper and $\alpha_i\from c_i\Rightarrow d_i\circ p$ (for $i=1,2$) is a 2-isomorphism witnessing the 2-commutativity of the appropriate triangle.  We write $(p,\alpha_1,\alpha_2)$ as shorthand for the datum of such a diagram, or just $p$ if the 2-isomorphisms are clear from context.

We declare two such diagrams $(p,\alpha_1,\alpha_2)$ and $(q,\beta_1,\beta_2)$ equivalent if there is a 2-isomorphism $\gamma\from p\Rightarrow q$ such that $\beta_i=(\id_{d_i}\ast \gamma)\circ \alpha_i$ for $i=1,2$.

\item  The monoidal structure is defined by $X\otimes Y = X\times_S Y$, with unit object $S$.  Given $c=(c_1,c_2)\from C\to X\times_S X'$ and $d=(d_1,d_2)\from D\to Y\times_S Y'$ representing objects in $\Hom_{\Corr_S}(X,X')$ and $\Hom_{\Corr_S}(Y,Y')$ respectively, we define the object $c\otimes d$ of $\Hom_{\Corr_S}(X\times_S Y,X'\times_S Y')$ as the correspondence: 
\begin{equation}
\label{EqProductCorrespondence}
\xymatrix{
& C\times_S D \ar[dl]_{c_1\times_S d_1} \ar[dr]^{c_2\times_S d_2} &\\
X\times_S Y && X'\times_S Y'
}
\end{equation}
\end{itemize}
\end{dfn}

The next thing to do is to construct a symmetric monoidal 2-category $\CoCorr_S$ of cohomological correspondences, which lies over $\Corr_S$.

%The Grothendieck construction then produces a symmetric monoidal 2-category $\CoCorr_S$, whose object 

\begin{dfn} [The category of cohomological correspondences] \label{DefCoCorr} We define a symmetric monoidal 2-category $\CoCorr_S$, which comes equipped with a functor to $\Corr_S$.  

\begin{itemize}
\item An object of $\CoCorr_S$ is a pair $\mf{X}=(X,A)$, where $X$ is a decent $S$-v-stack whose structure map $X \to S$ is fine, and $A\in D_{\et}(X,\Lambda)$ is arbitrary. 
\item Given objects $\mf{X}=(X,A)$ and $\mf{X}'=(X',A')$ of $\CoCorr_S$,  the category $\Hom_{\CoCorr_S}(\mf{X},\mf{X}')$ consists of pairs $\mf{c}=(c,u)$, where $c=(c_1,c_2)$ is a correspondence as in \eqref{EqCorrespondence}, and $u\from c_1^*A\to c_2^!A'$ is a morphism in $D_{\et}(C,\Lambda)$.  The composition of $\mf{c}=(c,u)\from (X,A)\to (X',A')$ with $\mf{d}=(d,v)\from (X',A')\to (X'',A'')$ is $\mf{d}\circ\mf{c}=(e,w)$, where $e\from C\times_{X'} D\to X\times_S X''$ is the correspondence in \eqref{EqCompositionOfCorrespondences}, and $w$ is the composition
\[ (d_1')^*c_1^*A\stackrel{u}{\to} (d_1')^*c_2^! A' \stackrel{a}{\to} (c_2')^!d_1^*A' \stackrel{v}{\to} (c_2')^! d_2^! A'', \]
where the map labeled $a$ is adjoint to the base change isomorphism $(c_2')_!(d_1')^* \isom d_1^* (c_2)_!$. 
\item  Let $\mf{X}=(X,A)$ and $\mf{X}'=(X',A')$ be two objects of $\CoCorr_S$.  Let $\mf{c}=(c,u)$ and $\mf{d}=(d,v)$ be two objects in $\Hom_{\CoCorr_S}(\mf{X},\mf{X}')$, where $c=(c_1,c_2)\from C\to X\times_S X'$ and $d=(d_1,d_2)\from D\to X\times_S X'$ are correspondences, and $u\from c_1^*A\to c_2^!A'$ and $v\from d_1^*A\to d_2^!A'$ are morphisms.  A 2-morphism $\mf{p}\from \mf{c}\implies \mf{d}$ is an equivalence class of triples $(p,\alpha_1,\alpha_2)$ as in the diagram \eqref{EqTwoMorphism}, such that the composition
\begin{eqnarray*}
d_1^* A & \stackrel{\text{unit}}{\longrightarrow} & p_*p^*d_1^* A \\
&\stackrel{(\alpha_1^*)^{-1}}{\longrightarrow} & p_*c_1^* A \\
&\stackrel{u}{\to}&  p_* c_2^! A' \\
&\stackrel{\alpha_2^!}{\longrightarrow}& p_*p^! d_2^! A' \isom p_!p^!d_2^! A'\\
&\stackrel{\text{counit}}{\longrightarrow} & d_2^!A'
\end{eqnarray*}
agrees with $v\from d_1^*A\to d_2^!A'$.  
Here $\alpha_1^*$ and $\alpha_2^!$ are the natural isomorphisms $c_1^*\isomto p^*d_1^*$ and $c_2^!\isomto p^!d_2^!$, respectively.  

We need to check that the condition on the aforementioned composition depends only on the equivalence class of the triple $(p,\alpha_1,\alpha_2)$. To check this, let $(q,\beta_1,\beta_2)$ be another triple which is equivalent to $(p,\alpha_1,\alpha_2)$ by a 2-isomorphism $\gamma\from p\Rightarrow q$ in the sense of Definition \ref{DefCorr}, so that assume that $\beta_i=(\id_{d_i}\ast \gamma)\circ \alpha_i$ for $i=1,2$.  Consider the diagram in $D_{\et}(D,\Lambda)$:
\[
\xymatrix{
d_1^*A \ar[d]_{\mathrm{unit}} \ar[r]^{=} & d_1^* A \ar[d]^{\mathrm{unit}} \\
p_*p^*d_1^*A \ar[r]_{\gamma_*\gamma^*} \ar[d]_{(\alpha_1^*)^{-1}} &
q_*q^*d_1^*A \ar[d]^{(\beta_1^*)^{-1}} \\
p_*c_1^* A \ar[r]_{\gamma_*} \ar[d]_{u} & q_*c_1^* A \ar[d]^{u} \\
p_!c_2^! A' \ar[r]_{\gamma_!} \ar[d]_{\alpha_2^!} &
q_!c_2^! A' \ar[d]^{\beta_2^!} \\
p_!p^!d_2^! A'\ar[r]_{\gamma_!\gamma^!} \ar[d]_{\mathrm{counit}} &
q_!q^!d_2^! A' \ar[d]^{\mathrm{counit}} \\
d_2^!A' \ar[r]_{=} & d_2^!A'
}
\]
The first and fifth squares commute by the compatibility described in Remark \ref{RmkAdjointnessCompatibility}, the second and fourth squares commute because of the condition $\beta_i=(\id_{d_i}\ast \gamma)\circ \alpha_i$, taking into account the pseudo-functor structures on the four non-binary operations, and the third square commutes because of the equalities $p_*=p_!$, $q_*=q_!$, and $\gamma_*=\gamma_!$.   The commutativity of the outside rectangle says that if the composition along the left vertical arrow is $v$, then so is the composition along the right vertical arrow.  This shows that our notion of 2-morphsm in $\CoCorr_S$ is well-defined.

\item The symmetric monoidal structure on $\CoCorr_S$ is given by $(X,A)\otimes (X',A')= (X\times_S X', A\boxtimes_S A')$.  The unit object is $1_{\CoCorr_S}=(S,\Lambda_S)$.  Finally, given $(c,u)\from (X_1,A_1)\to (X'_1,A'_1)$ and $(d,v)\from (X_2,A_2)\to (X'_2,A'_2)$, the tensor product $(c,u)\otimes (d,v)$ is $(c\otimes d, w)$, where $c\otimes d$ is the correspondence in \eqref{EqProductCorrespondence}, and $w$ is the composition
\[ (c_1\times_S d_1)^*(A_1\boxtimes_S A_2)\isom c_1^*A_1 \boxtimes d_1^*A_2 \stackrel{u\boxtimes_S v}{\to} c_2^!A_1' \boxtimes_S d_2^!A_2' \stackrel{\kappa}{\to} (c_2\times_S d_2)^! (A_2\boxtimes A_2'), \]
where $\kappa$ is adjoint to the K\"unneth isomorphism $(c_2\times d_2)_!(B_1\boxtimes_S B_2)\isom (c_2)_!B_1 \boxtimes_S (d_2)_!B_2$.  
\end{itemize}
\end{dfn}

The category $\CoCorr_S$ has internal mapping objects:  if $\mf{X}_1=(X_1,A_1)$ and $\mf{X}_2=(X_2,A_2)$, then $\iHom(\mf{X}_1,\mf{X}_2)=(X_1\times_S X_2,\RHom(\pr_1^*A_1,\pr_2^!A_2))$, where $\pr_i\from X_1\times_S X_2\to X_i$ is the projection.

%Let $F\from \CoCorr_S\to\Corr_S$ be the forgetful functor.   An important feature of this construction is that if $f\from \mf{X}\to \mf{X}'$ is a morphism, and $F(f)\implies g$ is a 2-morphism, then there exists a unique 

We have the following characterization of dualizable objects in $\CoCorr_S$.

\begin{pro} Let $X$ be a decent $S$-v-stack whose structure map $\pi\from X\to S$ is fine, and let $A\in D_{\et}(X,\Lambda)$ be any object.  The following are equivalent.
\begin{enumerate}
\item The object $(X,A)$ is dualizable in $\CoCorr_S$.
\item The natural map $m\from \D_{X/S} A\boxtimes_S A \to \RHom(\pr_1^*A,\pr_2^!A)$ (see proof for construction) is an isomorphism.
\end{enumerate}
In this situation, the dual of $(X,A)$ is $(X,\D_{X/S}A)$.
\end{pro} 

\begin{proof} Let $\mf{X}=(X,A)$, and let $\mf{X}'=(X,\D_{X/S}A)$.  There is a morphism $\mf{e}\from \mf{X}'\otimes \mf{X}\to 1_{\CoCorr_S}$, defined as the pair $(c,u)$, where $c=(\Delta_{X/S},\pi)$ and $u$ is the composition 
\[\Delta_{X/S}^*(\D_{X/S} A \boxtimes_S A) \isomto \D_{X/S} A \otimes A \to K_{X/S}=\pi^!\Lambda_S.\]
Then $\mf{e}$ induces a morphism $\mf{X}'\otimes \mf{Y}\to \iHom(\mf{X},\mf{Y})$ for any object $\mf{Y}$ of $\CoCorr_S$.  For $\mf{Y}=\mf{X}$, the map $u$ becomes the map $m$ in (2).

Suppose $\mf{X}$ is dualizable, with witnesses $\mf{X}^\vee$, $\ev_{\mf{X}}$, and $\coev_{\mf{X}}$.   Then $\mf{X}^\vee\otimes \mf{Y}\to \iHom(\mf{X},\mf{Y})$ is an isomorphism for all objects $\mf{Y}$.   Setting $\mf{Y}=1_{\CoCorr_S}$, we find an isomorphism $\mf{X}^\vee\isom \mf{X}'$ which identifies $\ev_{\mf{X}}$ with $\mf{e}$.  Setting $\mf{Y}=\mf{X}$, we find that $m$ is an isomorphism.  

Conversely, if $m$ is an isomorphism, let $\coev_{\mf{X}}=(d,w)$, where $d=(\pi,\Delta_{X/S})$ and $w$ is the composition
\[
\pi^*\Lambda_S \isom \Lambda_X \stackrel{\epsilon}{\to} \RHom_\Lambda(A,A) 
\stackrel{(P5)}{\isomto}  \Delta_{X/S}^!\RHom_\Lambda(\pr_1^*A, \pr_2^!A) 
\]
followed by $m^{-1}\from \Delta_{X/S}^!\RHom_\Lambda(\pr_1^*A, \pr_2^!A)\to \Delta_{X/S}^!(A\boxtimes_S \D_{X/S}A)$.  Here $\epsilon$ is adjoint to $\id_A\from A\to A$.  A diagram chase now shows that $\coev_{\mf{X}}$ and $\ev_{\mf{X}}$ witness the dualizability of $\mf{X}$.
\end{proof}  

In the scheme setting, a pair $(X,A)$ is dualizable if and only if $A$ is locally acyclic over $S$ \cite[Theorem 2.16]{LuZheng}, under some mild assumptions.  Similarly, if $f:X\to S$ is a morphism of v-stacks which is representable in nice diamonds, then $(X,A)$ is dualizable in $\CoCorr_S$ if and only if $A$ is $f$-universally locally acyclic \cite[Theorem IV.2.24]{FarguesScholze}.  This result extends immediately to the situation of fine morphisms between decent v-stacks, using that universal local acyclicity is cohomologically-smooth-local on the source.

If $\mf{X}$ is a dualizable object of $\CoCorr_S$, and $\mf{f}\from \mf{X}\to \mf{X}$ is an endomorphism, we may define the categorical trace $\tr(\mf{f})$, an object of $\Omega\CoCorr_S=\End 1_{\CoCorr_S}$.  Let us make this explicit.  The category $\Omega\CoCorr_S$ has objects $(X,\omega)$, where $X$ is a decent $S$-v-stack with fine structure map $X\to S$, and $\omega\in H^0(X,K_{X/S})$ is arbitrary.  A morphism $(X,\omega)\to (X',\omega')$ is a diagram
\[
\xymatrix{
X\ar[rr]^p \ar[dr]_\pi && X' \ar[dl]^{\pi'} \\
& S & }
\]
with $p$ proper, such that $\omega'=p_*\omega$.  Here \[ p_*\from H^0(X,K_{X/S})\to H^0(X',K_{X'/S}) \]
is induced from $\pi_*K_{X/S}\isom \pi'_*p_*p^!K_{X'/S} \isom \pi'_*p_!p^!K_{X'/S} \stackrel{\text{counit}}{\to} \pi'_*K_{X'/S}$.  

Now let $\mf{X}=(X,A)$ be a dualizable object in $\CoCorr_S$, and let $\mf{f}=(c,u)\from \mf{X}\to \mf{X}$ be an endomorphism, with $c=(c_1,c_2)\from C\to X\times_S X$ and $u\from c_1^*A\to c_2^!A$.  By definition, $\tr(\mf{f}) = \ev_{\mf{X}}\circ (f\otimes \id_{\mf{X}^\vee}) \circ \coev_{\mf{X}}$.   The object $\tr(\mf{f})$ is represented by a pair $(\tr(c),\tr(u))$.  Here $\tr(c)\in \Omega 1_{\Corr_S}$ is the correspondence $\Fix(c)\to S\times_S S = S$, where $\Fix(c)$ is the fixed-point locus of the correspondence $c$, as in the cartesian diagram:
\[
\xymatrix{
\Fix(c) \ar[r]^{c'} \ar[d]_{\Delta'_{X/S}} & X \ar[d]^{\Delta_{X/S}} \\
C \ar[r]_c & X\times_S X
}
\]
For its part, the element $\tr(u)$ is an element of $H^0(\Fix(c),K_{\Fix(c)/S})$. \footnote{We will sometime notate this element as $\tr_c(u,A)$ if we wish to emphasize the roles of $c$ and $A$.}   It is the image of $u\in \Hom(c_1^*A,c_2^!A)$ under
\begin{eqnarray*}
H^0(C,\RHom(c_1^*A, c_2^!A)) 
&\stackrel{(P5)}{\isomto}& H^0(C,c^!\RHom(\pr_1^*A, \pr_2^!A)) \\
&\stackrel{\eqref{EqCrucialIsomorphismSchemes}}{\isomto}& H^0(C,c^!(\D_{X/S}A \boxtimes_S A)) \\
&\stackrel{\alpha}{\to}& H^0(C,c^!(\Delta_{X/S})_* (\D_{X/S}A \otimes A) ) \\
&\stackrel{\ev_A}{\to}& H^0(C,c^!(\Delta_{X/S})_* K_{X/S} )\\
&\stackrel{{\text{(BC2)}}}{\isomto}& H^0(C,(\Delta'_{X/S})_* (c')^! K_{X/S}) \\
&\isom&  H^0(\Fix(c),K_{\Fix(c)/S}).
\end{eqnarray*}
Here, the map labeled $\alpha$ is adjoint to $(\Delta_{X/S})^*(\D_{X/S}A \boxtimes_S A)\isomto \D_{X/S} A \otimes A$.   

\begin{dfn}[Inertia stack, characteristic class] \label{DefCharClass}
In the special case that $\mf{f}=\id_{\mf{X}}=(\Delta_{X/S},\id_A)$, the object $\tr(\Delta_{X/S})=X\times_{X\times_S X} X$ is the {\em inertia stack} of $X$, which we notate as $\In_S(X)$.  Its objects are pairs $(x,g)$, where $x\in X$ and $g\in \Aut x$.  Then $\tr(\id_A)$ is an element of $H^0(\In_S(X),K_{\In_S(X)})$, which we call the {\em characteristic class} of $A$.   We notate this element as $\cc_{X/S}(A)$. 
\end{dfn}

We record one lemma here for later reference.

\begin{lem}\label{LemRestrictionOfCC} Let $i\from U\to X$ be an open immersion of decent $S$-v-stacks fine over $S$.  Then $\In_S(i)\from \In_S(U)\to \In_S(X)$ is also an open immersion.  If $A\in D_{\et}(X,\Lambda)$ is ULA over $S$, then so is $i^*A$, and then 
\[ \cc_{U/S}(i^*A)=\In_S(i)^*\cc_{X/S}(A). \]
\end{lem}

\begin{proof} All constructions are local on $X$.
\end{proof}

\begin{thm}[Relative Lefschetz-Verdier trace formula]  \label{ThmLefschetzVerdier} Let $\mf{X}=(X,A)$ be a dualizable object in $\CoCorr_S$, and let $\mf{f}\in \End \mf{X}$ lie over the correspondence $c\from C\to X\times_S X$.  Suppose we are given a diagram
\[ 
\xymatrix{
X \ar[d]_{q} & C\ar[l] \ar[r] \ar[d]_p & X \ar[d]_{q} \\
X' & C'\ar[l] \ar[r] & X' 
}
\]
with $p,q$ proper.  Then $\mf{X}'=(X',q_*A)$ is also dualizable.   Let $\mf{q}\from \mf{X}\to \mf{X}'$ be the evident 1-morphism lying over $q$.  There is a unique morphism $\mf{f}'\in \End\mf{X}'$ lying over $C'\to X'\times_S X'$, such that $p$ defines a 2-morphism $\tilde{p}$ filling in the square
\[
\xymatrix{
 \mf{X} \ar[r]^{\mf{f}} \ar[d]_{\mf{q}} & \mf{X} \ar[d]^{\mf{q}} \\
\mf{X}' \ar[r]_{\mf{f}'} & \mf{X}'
} 
\]
Finally, there exists a (necessarily unique) 2-morphism $\tr(\tilde{p})\from \tr(\mf{f})\implies \tr(\mf{f}')$ lying over $\tr(p)\from \tr(c)\implies \tr(c')$.  
\end{thm}

\begin{proof}  This is formally the same as the proof of (a special case of) \cite[Theorem 2.21]{LuZheng}, so we give a brief sketch.  (In \cite{LuZheng} one gets a statement about the more general Lefschetz-Verdier pairing, which we also could have established.)   The existence and uniqueness of $\mf{f}'\in \End\mf{X}'$ follows from the definition of 2-morphisms in $\CoCorr_S$.   The dualizability of $\mf{X}'$ is \cite[Proposition 2.23]{LuZheng}.  
The dual of $\mf{X}'$ is $(\mf{X}')^\vee=(X',q_*\D_{X/S}A)$;  there is another natural map $\mf{q}^\vee\from \mf{X}^\vee\to (\mf{X}')^\vee$ defined similarly to $\mf{q}$.  

Consider the diagram:
\[
\xymatrix{
1_{\CoCorr_S} \ar[r]^{\coev_{\mf{X}}} \ar[dr]_{\coev_{\mf{X}'}} & \mf{X}\otimes \mf{X}^\vee \ar[d]^{\mf{q}\otimes \mf{q}^\vee} \ar[r]^{\mf{f}\otimes\id_{\mf{X}^\vee}} & \mf{X} \otimes \mf{X}^\vee \ar[d]_{\mf{q}\otimes \mf{q}^\vee} \ar[dr]^{\ev_{\mf{X}}} & \\
& \mf{X}'\otimes (\mf{X}')^\vee \ar[r]_{\mf{f}'\otimes\id_{\mf{X}'}} & \mf{X}'\otimes (\mf{X}')^{\vee} \ar[r]_{\ev_{\mf{X}'}} & 1_{\CoCorr_S} 
}
\]
The two outer triangles can be filled in with a 2-morphism, as can the inner square (via $p\otimes\id_{\mf{q}^{\vee}}$).  See \cite[Construction 1.7]{LuZheng} for details.  Composing, we find the required 2-morphism $\tr(\tilde{p})\from \tr(\mf{f})\implies \tr(\mf{f}')$.  
\end{proof}

\begin{cor}\label{CorLefschetzVerdier} Let $(X,A)$ be a dualizable object, and let $p\from X\to X'$ be proper.  Then $(X',p_!A)$ is dualizable.  The morphism $\In_S(p)\from \In_S(X)\to \In_S(X')$ is proper, and 
\[ \In_S(p)_*\cc_{X/S}(A) = \cc_{X'/S}(p_!A).\]
\end{cor}

\begin{exm} \label{ExLefschetzForDiamonds}
We can immediately deduce a familiar-looking trace formula from Theorem \ref{ThmLefschetzVerdier} in the case of proper diamonds.   Suppose $S=\Spd C$ for an algebraically closed perfectoid field $C$, and suppose $q\from X\to S$ is a nice diamond.  Let $A\in D_{\et}(X,\Lambda)$ be ULA over $S$.  Then $\mf{X}=(X,A)$ is a dualizable object of $\CoCorr_S$.   Let $\mf{f}=(c,u)$ be an endomorphism of $\mf{X}$ lying over a correspondence $c\from Y\to X\times_S X$.  The categorical trace $\tr(\mf{f})$ is an endomorphism of $1_{\CoCorr_S}$ consisting of the pair $(\Fix(c),\omega)$, where $\Fix(c)=Y\times_{c,X\times_S X,\Delta_{X/S}} X$ is the fixed point locus of the correspondence, and $\omega$ is a global section of $K_{\Fix(c)/S}$.  

Now suppose that $X$ is proper over $S$.  In the setting of Theorem \ref{ThmLefschetzVerdier}, we put $\mf{X}'=(S,q_*A)$ and $C'=S$.   The dualizability of $\mf{X}'$ means that $q_*A=R\Gamma(X,A)$ is a perfect complex.  The morphism $\mf{f}'\from \mf{X}'\to \mf{X}'$ supplied by Theorem \ref{ThmLefschetzVerdier} is the endomorphism $q_*(\mf{f})\from q_*A \to q_*A$.  Finally, the existence of $\tr(\tilde{p})\from \tr(\mf{f})\implies \tr(\mf{f}')$ lying over $\tr(p)\from \tr(c)\implies \tr(\id_S)$ implies that 
\begin{equation}
\label{EqLefschetzProperDiamonds}
\tr\left(q_*(\mf{f})\biggm\vert R\Gamma(X,A)\right) = \int_{\Fix(c)} \omega.
\end{equation}
\end{exm}

\begin{dfn} Let $x\in \Fix(c)$ be an isolated point, such that $x\to S$ an isomorphism.   The {\em local term} $\loc_x(\mf{f})$ is the restriction of $\omega$ to $x$, considered as an element of $\Lambda$.  

If $\mf{f}$ arises from an automorphism $g\from X\to X$ along with a morphism $u\from g^*A\to A$, we write the local term as $\loc_x(g,A)$ (the dependence on $u$ being implicit).
\end{dfn}

In the latter situation, if it so happens that $\Fix(g)$ consists of finitely many isolated $S$-points $x_1,\dots,x_n$, then \eqref{EqLefschetzProperDiamonds} reduces to
\[ \tr\left(q_*(g)\biggm\vert R\Gamma(X,A)\right) = \sum_{i=1}^n \loc_x(g,A).\]

\subsection{The trace distribution as a characteristic class}

Let $S$ be a geometric point, and let $G$ be a locally pro-$p$ group as in Example \ref{ExLocallyProPGroup}, so that $[S/G_S]$ is a decent v-stack and $f\from [S/G_S]\to S$ is fine and cohomologically smooth.   As in that example, we freely identify $D_{\et}([S/G_S],\Lambda)$ with the derived category of $\Lambda$-modules with a smooth $G$-action.   For an object $M\in D_{\et}([S/G_S],\Lambda)$, we let $M^\vee = \RHom(M,\Lambda)$; by \cite[Corollary V.1.4]{FarguesScholze}, this is just the usual (derived) smooth dual.

For a compact open subgroup $K\subset G$, we let $M^K$ be the complex of derived $K$-invariants.  If $K$ is pro-$p$, the map of complexes $M^K\to M$ admits a section, namely, averaging over $K$ with respect to a normalized Haar measure.  Thus $M^K$ is naturally a summand of $M$.  

\begin{pro} \label{ProAdmissible} Let $M$ be an object of $D_{\et}([S/G_S],\Lambda)$.  The following are equivalent:
\begin{enumerate}
\item The object $([S/G_S],M)$ of $\CoCorr_S$ is dualizable, with dual $([S/G_S],\D M)$, where $\D M=\mathrm{Haar}(G,\Lambda)^* \otimes M^\vee$.  
\item The object $M$ is ULA over $S$.
\item For all compact open pro-$p$ subgroups $K\subset G$, the (derived) $K$-invariants $M^K$ are a perfect complex of $\Lambda$-modules.
\end{enumerate}
\end{pro}

\begin{proof}  The equivalence between (1) dualizability in $\CoCorr_S$ and (2) the ULA property is \cite[Theorem IV.2.23]{FarguesScholze}.  For the equivalence between (2) and (3), see \cite[V.7.1]{FarguesScholze}.  (There, the authors work in the more general context of $\Bun_G$, but the method of proof can be used in our situation of $[S/G_S]$.)   The Verdier dual of $M$ is $\RHom(M,f^!\Lambda)$, and $f^!\Lambda = \mathrm{Haar}(G,\Lambda)^*$ by Example \ref{ExLocallyProPStructureMap}.

We note that is possible to give a direct proof of the implication (3)$\implies$(1).  Let $\mf{X}=([S/G_S],M)$ and $\mf{X}^\vee=([S/G_S],\D M)$.  The evaluation map $\mf{X}^\vee\otimes \mf{X} \to 1_{\CoCorr_S}$ lies over the correspondence $\Delta_f\times f\from [S/G_S]\to [S/G_S]^2 \times S$;  on the level of sheaves it is a twist of the evaluation map $M^\vee \otimes M \to \Lambda$.  The coevaluation map $1_{\CoCorr_S} \to \mf{X}\otimes \mf{X}^\vee$ lies over the correspondence $f\times \Delta_f\from [S/G_S] \to S\times [S/G_S]^2$.  Note that the diagonal map presents $G$ as a direct factor of $G^2$, so that Example \ref{ExLocallyProPGroupQuotient} applies to give an explicit description of $\Delta_f^!$.  The result is that the $\Lambda$-module of cohomological correspondences lying over $f\times \Delta_f$:
 \[f^*\Lambda\to \Delta_f^!(M\boxtimes \D M)\]
may be identified with the $\Lambda$-module 
\[ \Hom_{G\times G}(C_c(G,\Lambda)\otimes \mathrm{Haar}(G,\Lambda), M \boxtimes M^\vee). \]
In the latter expression, $G\times G$ acts on $C_c(G,\Lambda)$ by left and right translation.  We describe the coevaluation map as a $G\times G$-equivariant function 
\[ I\from C_c(G,\Lambda)\otimes \mathrm{Haar}(G,\Lambda) \to H^0(M\otimes M^\vee).\]
Let $h\in C_c(G,\Lambda)$ and $\mu\in \mathrm{Haar}(G,\Lambda)$;  then integration against $h\;d\mu$ describes an endomorphism $I_{h,\mu}\in \End M$:
\[ I_{h,\mu}(v) = \int_{g\in G} h(g)gv \; d\mu(g). \]
The function $h$ is left and right $K$-invariant for some sufficiently small pro-$p$ open subgroup $K\subset G$, in which case $I_{h,\mu}$ factors through a map $M^K\to M^K$.  Since $M^K$ is perfect by hypothesis, we have described an element of 
\[ \Hom(M^K,M^K) \isom H^0(\mathrm{RHom}(M^K,M^K))\isom H^0(M^K \otimes (M^K)^\vee).\]  Then $I(h\otimes \mu)$ is the image of $I_{h,\mu}$ in $H^0(M\otimes M^\vee)$.
\end{proof}

\begin{dfn}  The object $M$ of $D_{\et}([S/G_S],\Lambda)$ is {\em admissible} if it satisfies the equivalent conditions of Proposition \ref{ProAdmissible}.
\end{dfn}

Suppose $M$ is an admissible object of $D_{\et}([S/G_S],\Lambda)$.   The {\em trace distribution} of $M$ is a canonical element
\begin{equation}
\label{EqTrDist}
\trdist(M)\in \Hom_G(C_c(G,\Lambda)\otimes \mathrm{Haar}(G,\Lambda), \Lambda),
\end{equation}
where $G$ is meant to act on $C_c(G,\Lambda)$ by conjugation and on $\mathrm{Haar}(G,\Lambda)$ by the modular character.  Namely, $\trdist(M)$ sends $h\otimes \mu$ (where $h\in C_c(G,\Lambda)$ and $\mu\in \mathrm{Haar}(G,\Lambda)$ to the Euler characteristic of the operator $I_{h,\mu}$ described in the proof of Proposition \ref{ProAdmissible}.  

On the other hand, we have the inertia stack $\In_S([S/G_S])$ and the characteristic class $\cc_{[S/G_S]/S}(M)$ as in Definition \ref{DefCharClass}.  We have 
\[ \In_S([S/G_S]) = [G_S\sslash G_S], \]
the stack of conjugacy classes of $G$.  (Reasoning:  For a perfectoid space $Y\to S$, a $Y$-point of $\In_S([S/G_S])$ is a $G_S$-torsor $\tilde{Y}\to Y$ together with a $G_S$-equivariant automorphism $i\from \tilde{Y}\to \tilde{Y}$.  Such an automorphism arises as $i(y)=f(y).y$, where $f\from \tilde{Y}\to G_S$ is a morphism satisfying $f(gy)=gf(y)g^{-1}$.  The pair $(\tilde{Y},f)$ then constitutes a $Y$-point of $[G_S\sslash G_S]$.)  For its part, the characteristic class $\cc_{[S/G_S]/S}(M)$ lies in $H^0(\In_S([S/G_S]), K_{\In_S([S/G_S])/S})$, and by Example \ref{ExTModG} we have an isomorphism
\[ H^0(\In_S([S/G_S]), K_{\In_S([S/G_S])/S})\isom \Hom_G(C_c(G,\Lambda) \otimes \mathrm{Haar}(G,\Lambda), \Lambda) \]
onto the same module appearing in \eqref{EqTrDist}.

%Now suppose $\pi$ is a smooth admissible representation of $G$ on a $\Lambda$-module.  The {\em trace distribution} of $\pi$ is canonically an element of $\Hom_G(C_c(G,\Lambda)\otimes \mathrm{Haar}(G,\Lambda), \Lambda)$, where $G$ is meant to act on $C_c(G,\Lambda)$ by conjugation.  Namely, the trace distribution of $\pi$ sends $f\otimes \mu$ (where $f\in C_c(T,\Lambda)$ and $\mu\in\mathrm{Haar}(G,\Lambda)$) to the trace of the operator 
%\[ v\mapsto \int_{g\in G} f(g) \pi(g) v \; d\mu(g). \]
%On the other hand we have the object $[\pi]$ of $D_{\et}([S/G_S],\Lambda)$.  Its characteristic class $\cc_{[S/G_S]/S}$ lies in 
%\begin{equation}
%\label{EqDualizingSheafForGModG}
%H^0([G_S\sslash G_S],K_{[G_S\sslash G_S]/S})
%\isom \Hom_G(C_c(G,\Lambda)\otimes \mathrm{Haar}(G,\Lambda), \Lambda) 
%\end{equation}
%(this being a special case of \eqref{EqDualizingSheafForTModG}), where again $G$ acts on $C_c(G,\Lambda)$ by conjugation.  Note that if $G$ is unimodular and a Haar measure is chosen, then the module in \eqref{EqDualizingComplexForGModG} becomes isomorphic to $\Dist(G,\Lambda)^G$, the module of conjugation-invariant distributions on $G$.

\begin{pro}  Let $G$ be a locally pro-$p$ group satisfying the hypotheses of Example \ref{ExLocallyProPGroup}.  Let $M$ be an admissible object of $D_{\et}([S/G_S],\Lambda)$. Then
\[ \cc_{[S/G_S]/S}(M) = \trdist(M). \]
\end{pro}

\begin{proof}  The characteristic class of $M$ is the categorical trace of the identity on the object $\mf{X}=([S/G_S],M)$, which is $\ev_{\mf{X}}\circ \coev_{\mf{X}}$.  This equals the image of the identity map through the left side of the following commutative diagram:
\[
\xymatrix{
H^0(\RHom(M,M)) \ar[rr]^{\isom} \ar[d]_{\coev_{\mf{X}}} && \End_G M \ar[d]^{\coev_{M}} \\
H^0(\Delta_f^!(\D M \boxtimes M)) \ar[rr]_{\isom} \ar[d]_{\ev_{\mf{X}}} && \Hom_{G\times G} (C_c(G,\Lambda)\otimes \mathrm{Haar}(G,\Lambda),M^\vee \boxtimes M)  \ar[d]^{\ev_M} \\
H^0(K_{\In_S([S/G_S]/S}) \ar[rr]_{\isom} && \Hom_G(C_c(G,\Lambda)\otimes \mathrm{Haar}(G,\Lambda), \Lambda) 
}
\]
Whereas on the right side of the diagram, the map labeled $\coev_M$ carries the identity to the integration map $I$ described in the proof of Proposition \ref{ProAdmissible}, and then $\ev_M$ carries $I$ onto $\trdist(M)$ by definition of the latter.
\end{proof}

\subsection{A K\"unneth theorem for characteristic classes}

The goal of this section is to prove the compatibility of the categorical trace with fiber products, to get an analogue of the relation $\tr(A\otimes B)=\tr(A) \tr(B)$ for square matrices.  Again, throughout this section we fix a decent base v-stack $S$.

As an example of what we will do, suppose $X_1,X_2\to S$ are two fine morphisms of decent v-stacks, and suppose $A_i\in D_{\et}(X_i,\Lambda)$ is ULA over $S$ for $i=1,2$.  Then $A_1\boxtimes_S A_2$ is ULA over $S$, so we may define the characteristic class $\cc_{X_1\times_S X_2/S}(A_1\boxtimes_S A_2)$ in $H^0(\In_S(X_1\times_S X_2),K_{\In_S(X_1\times_S X_2)})$.  We have an isomorphism $\In_S(X_1\times_S X_2)\isom \In_S(X_1)\times_S \In_S(X_2)$, and therefore a K\"unneth map 
\begin{eqnarray*}
\kappa_S \from K_{\In_S(X_1)/S} \otimes K_{\In_S(X_2)/S} \to K_{\In_S(X_1\times_S X_2)/S}
\end{eqnarray*}
which we notate as $\mu_1\otimes\mu_2\mapsto \mu_1\boxtimes_S \mu_2$ for global sections $\mu_i$ of $K_{\In_S(X_i)/S}$.  Then it is straightforward to show (and a corollary of Theorem \ref{ThmKunnethForCharacteristicClass} below) that 
\[ \cc_{X_1/S}(A_1)\boxtimes_S \cc_{X_2/S}(A_2)=\cc_{X_1\times_S X_2/S}(A_1\boxtimes_S A_2).\]

For our applications we need a more general result involving fiber products over bases other than $S$.  First, we need a modification of the above K\"unneth map in a general setting.

\begin{dfn}[Modified K\"unneth map]  Let $U\to T$ be a cohomologically smooth morphism of decent $S$-v-stacks. Suppose we are given a 2-commutative diagram of decent $S$-v-stacks
\[ 
\xymatrix{
Y_i \ar[r]^{f_i} \ar[d]_{g_i} & X_i \ar[d] \\
U \ar[r] & T
}
\]
for $i=1,2$, such that $f_1$ and $f_2$ are fine.   Let $f\from Y_1\times_U Y_2\to X_1\times_T X_2$ and $g\from Y_1\times_U Y_2 \to U$ be the induced product maps.  Let $A_i$ be an object of $D_{\et}(X_i,\Lambda)$ for $i=1,2$.  We define a map
\[ \kappa_{U/T}\from f_1^!A_1 \boxtimes_U f_2^!A_2 \to f^!(A_1\boxtimes_T A_2) \otimes g^*K_{U/T} \]
as follows.  There is a cartesian diagram
\[
\xymatrix{
Y_1\times_U Y_2 \ar[r]^{\Delta_{U/T}'} \ar[d]_g & Y_1\times_T Y_2 \ar[d]^{g_1\times_T g_2} \\
U \ar[r]_{\Delta_{U/T}} & U\times_T U 
}
\]
The map $\kappa_{U/T}$ is defined as the composition
\begin{eqnarray*}
&&f_1^!A_1\boxtimes_U f_2^!A_2 \\
&\isom& (\Delta_{U/T}')^*(f_1^!A_1 \boxtimes_T f_2^!A_2) \\
&\to& (\Delta_{U/T}')^*(f_1\times_T f_2)^!(A_1\boxtimes_T A_2) \\
&\stackrel{\eqref{EqInverseDualizingComplexes}}{\isom}& (\Delta_{U/T}')^*(f_1\times_T f_2)^!(A_1\boxtimes_T A_2) \otimes g^*(\Delta_{U/T})^!\Lambda_T \otimes g^*K_{U/T} \\
&\to& (\Delta_{U/T}')^*(f_1\times_T f_2)^! (A_1\boxtimes_T A_2) \otimes (\Delta_{U/T}')^!\Lambda_{Y_1\times_T Y_2} \otimes g^*K_{U/T} \\
&\to& (\Delta_{U/T}')^!(f_1\times_T f_2)^!(A_1\boxtimes_T A_2) \otimes g^*K_{U/T} \\
&\isom& f^!(A_1\boxtimes_T A_2) \otimes g^*K_{U/T}.
\end{eqnarray*}
\end{dfn}

In particular, the case $X_1=X_2=T=S$ yields a map 
\begin{equation}
\label{EqModifiedKunneth}
 \kappa_{U/S} \from K_{Y_1/S} \boxtimes_U K_{Y_2/S} \to K_{Y_1\times_U Y_2/S} \otimes g^*K_{U/S}.
 \end{equation}

We can now introduce the setup of the main theorem of this section.  We consider bases $T\to S$ satisfying two hypotheses:  (1) $T\to S$ is cohomologically smooth, and (2) $\Delta_{T/S}\from T\to T\times_S T$ is cohomologically smooth.  These are satisfied for instance when $T=[S/G]$, where $G$ is a cohomologically smooth locally spatial group diamond over $S$.  Considering the diagram
\[
\xymatrix{
\In_S(T) \ar[r] \ar[d] & T \ar[d]^{\Delta_{T/S}} & \\
T \ar[r]_{\Delta_{T/S}} & T\times_S T\ar[r] \ar[d] & T \ar[d] \\
& T \ar[r] & S
}
\]
in which both squares are cartesian, we see that $\In_S(T)\to S$ is also cohomologically smooth.   Furthermore we can trivialize the dualizing complex $K_{\In_S(T)/S}$.

\begin{lem}  \label{LemTrivialDualizingSheaf} Let $T$ be a decent $S$-v-stack such that the structure map $\pi\from T\to S$ and diagonal $\Delta_{T/S} \from T\to T\times_S T$ are both cohomologically smooth.   Then the object $\mf{T}=(T,\Lambda_T)\in\CoCorr_S$ is dualizable, and its characteristic class $\cc_{T/S}(\Lambda_T)$, considered as a morphism $\Lambda_{\In_S(T)} \to K_{\In_S(T)/S}$, is an isomorphism.
\end{lem}

In the case $T=[S/G]$, where $G$ is a cohomologically smooth locally spatial group diamond over $S$, the inertia stack is $\In_S(T)=[G\sslash G]$, the stack of conjugacy classes of $G$.  This is a cohomologically smooth stack of dimension 0, so perhaps it is unsurprising that it has trivial dualizing complex; the lemma states that the trivialization is in fact canonical.

\begin{proof} Let $\pr_1,\pr_2\from T\times_S T\to T$ be the projection morphisms. The morphism in $D_{\et}(T,\Lambda)$ associated with the cohomological correspondence $\coev_{\mf{T}}\from$ $(S,\Lambda_S)\to (T\times_S T, \pr_2^*K_{T/S})$ is an isomorphism:
\[\pi^*\Lambda_S\isom\Delta_{T/S}^!\pr_1^!\pi^*\Lambda_S \stackrel{\alpha}{\isom} \Delta_{T/S}^!\pr_2^*\pi^!\Lambda_S =  \Delta_{T/S}^!\pr_2^*K_{T/S}. \]
Here we the use cohomological smoothness of $\pi$ to get the isomorphism $\alpha$, as in Theorem \ref{ThmSmoothBaseChange}.   The morphism $\Delta_{T/S}^*\pr_2^*K_{T/S}\to \pi^!\Lambda_S$ associated with $\ev_{\mf{T}}\from (T\times_S T,\pr_2^*K_{T/S})\to (S,\Lambda_S)$ is also an isomorphism (this is true without any hypotheses on $\pi$ or $\Delta_{T/S}$).   Referring to the diagram
\[
\xymatrix{
& & \In_S(T) \ar[dl]_{h} \ar[dr]^{h} && \\
& T \ar[dl]_{\pi} \ar[dr]^{\Delta_{T/S}} && T \ar[dl]_{\Delta_{T/S}} \ar[dr]^{\pi}  & \\
S && T\times_S T && S
}
\]
we may describe the characteristic class $\cc_{T/S}(\Lambda_T)$ as the composition
\begin{eqnarray*}
h^*\pi^*\Lambda_S &\isom&  h^*\Delta_{T/S}^!\pr_2^*K_{T/S}\\&\stackrel{\beta}{\isom}& h^!\Delta_{T/S}^*\pr_2^* K_{T/S}\isom h^!K_{T/S}\isom K_{\In_S(T)/S},
\end{eqnarray*}
which is an isomorphism.  Here we once again applied Theorem \ref{ThmSmoothBaseChange}, using the cohomological smoothness of $\Delta_{T/S}$. 
\end{proof}

Now suppose $p_i\from X_i\to T$ ($i=1,2$) is a fine morphism of decent $S$-v-stacks, with fiber product $p\from X_1\times_T X_2\to T$.  Then there is an isomorphism
\[ \In_S(X_1\times_T X_2) \isom \In_S(X_1)\times_{\In_S(T)} \In_S(X_2), \]
and therefore by the discussion above we have a modified K\"unneth map 
\[ \kappa_{\In_S(T)/S}\from K_{\In_S(X_1)/S} \boxtimes_{\In_S(T)} K_{\In_S(X_2)/S} \to K_{\In_S(X_1\times_T X_2)/S} \otimes \In(p)^*K_{\In_S(T)/S}. \]
Using the trivialization $\Lambda_{\In_S(T)/S}\isom K_{\In_S(T)/S}$ from Lemma \ref{LemTrivialDualizingSheaf}, we obtain a map
\begin{equation}
\label{EqBoxtimesOnDistributions}
K_{\In_S(X_1)/S} \boxtimes_{\In_S(T)} K_{\In_S(X_2)/S} \to K_{\In_S(X_1\times_T X_2)/S}, 
\end{equation}
which on global sections we notate as $\mu_1\otimes\mu_2\mapsto \mu_1\boxtimes_{\In_S(T)} \mu_2$.  

Finally we can state the main theorem of the section.

\begin{thm} \label{ThmKunnethForCharacteristicClass} Let $T$ be a decent $S$-v-stack such that the structure map $\pi\from T\to S$ and the diagonal $\Delta_{T/S}\from T\to T\times_S T$ are both cohomologically smooth. Let $X_1,X_2\to T$ be two fine morphisms of decent v-stacks (so also the induced morphisms $X_1,X_2 \to S$ are fine), and let $A_i\in D_{\et}(X_i,\Lambda)$ be a sheaf which is ULA over $S$.    Then $A_1\boxtimes_T A_2$ is ULA over $S$, and
\[ \cc_{X_1/S}(A_1)\boxtimes_{\In_S(T)} \cc_{X_2/S}(A_2)
= \cc_{X_1\times_T X_2/S}(A_1\boxtimes_T A_2). \]
\end{thm}

In order to prove Theorem \ref{ThmKunnethForCharacteristicClass}, we need to enhance the category $\CoCorr_S$ to include data coming from a smooth base v-stack, which we allow to vary.   At a first pass, one might think that such a category would have objects $(X\to T,A)$, where $T\to S$ is cohomologically smooth, $X\to T$ is a morphism of v-stacks, and $A\in D_{\et}(X,\Lambda)$.  The morphisms $(X\to T,A)\to (X'\to T',A')$ would be pairs $(q^\natural,\alpha)$, where $q^\natural =(p,q,p')$ is a morphism of correspondences as in a 2-commutative diagram:
\begin{equation}
\label{EqBasedCorrespondence}
\xymatrix{
X \ar[d]_{p} & C\ar[l]_{c} \ar[d]_{q} \ar[r]^{c} & X' \ar[d]^{p'}\\
T & U \ar[l]_{u} \ar[r]^{u'} & T'
}
\end{equation} 
and $\alpha\from c^*A\to (c')^!A'$ is a morphism.  We assume that $u'$ is cohomologically smooth. Given pairs $\mf{X}_1=(X_1\to T,A_1)$ and $\mf{X}_2=(X_2\to T,A_2)$ with common base $T$, we could then define $\mf{X}_1\boxtimes_T \mf{X}_2 = (X_1\times_T X_2\to T,A_1\boxtimes_T A_2)$.  

However, it turns out that this definition is not functorial in $\mf{X}_1$ and $\mf{X}_2$.  That is, given a morphism $e\from T\to T'$ in $\Corr_S$ and morphisms $f_i\from \mf{X}_i\to \mf{X}_i'$ lying over $e$, one cannot in general define a map $f_1\otimes_e f_2\from \mf{X}_1\boxtimes_T \mf{X}_2\to \mf{X}_1'\boxtimes_{T'} \mf{X}_2'$.  Essentially, this is due to the appearance of the invertible sheaf $K_{U/S}$ appearing in the modified K\"unneth map \eqref{EqModifiedKunneth}.  

To obtain a functorial definition of $\mf{X}_1\boxtimes_T \mf{X}_2$, we need to define an enhancement of the category which keeps track of an invertible sheaf living on the base.

\begin{dfn}[The category of based cohomological correspondences] \label{DefBCoCorr} 
We define a symmetric monoidal 2-category $\BCoCorr_S$.  The objects of $\BCoCorr_S$ are triples $(X\to T,A,B)$, where $X\to T$ is a (fine) morphism of decent $S$-v-stacks whose structure maps to $S$ are fine, where $A$ is an object of $D_{\et}(X,\Lambda)$, and where $B$ is an {\em invertible} object of $D_{\et}(T,\Lambda)$.   We assume that the structure map $T\to S$ is cohomologically smooth.  

Given objects $\mf{X}=(X\to T,A,B)$ and $\mf{X}'=(X'\to T',A',B')$, an object of the category $\Hom_{\BCoCorr_S}(\mf{X},\mf{X}')$ is a triple $(q^\natural,\alpha,\beta)$.  The first element in the triple is a morphism $q^\natural=(p,q,p')$ between correspondences as in \eqref{EqBasedCorrespondence}, where $u'$ is cohomologically smooth.   The second element is a morphism
$\alpha\from c^*A \to (c')^!A'$, and the third is an {\em isomorphism} $\beta\from u^*B\isomto (u')^! B'$.  Compositions of morphisms are defined similarly as in Definition \ref{DefCoCorr}.  (Note that the cohomological smoothness of $u'$ is preserved under composition of correspondences.)  Given objects $(X\to T,A,B)$ and $(X'\to T',A',B')$, and morphisms between them represented by $C\to X\times_U X'$ (lying over $U\to T\times_S T'$) and $D\to X\times_V X'$ (lying over $V\to T\times_S T'$), a 2-morphism is an equivalence class of 2-commutative diagrams as in \eqref{EqTwoMorphism}, together with a similar one involving morphisms $U\to V$.  

The monoidal structure on $\BCoCorr_S$ is defined by 
\[ (X\to T,A,B)\otimes (X'\to T',A',B')=(X\times_S X'\to T\times_S T',A\boxtimes_S A',B\boxtimes_S B'). \]
The unit object is $(S \overset{\mathrm{id}}{\to} S,\Lambda_S,\Lambda_S)$.   

Finally, we introduce the obvious monoidal functors $\mc{B},\mc{S}\from \BCoCorr_S\to \CoCorr_S$, with $\mc{B}(X\to T,A,B)=(T,B)$ (the {\em base}) and $\mc{S}(X\to T,A,B)=(X,A)$ (the {\em source}).
\end{dfn}

So far the objects $A$ and $B$ in Definition \ref{DefBCoCorr} have nothing to do with each other.  They only begin to interact when we talk about fiber products of objects of $\BCoCorr_S$ over common bases.   Given objects $\mf{X}_i=(X_i\stackrel{p_i}{\to} T,A_i,B)$ for $i=1,2$ with common base $\mf{T}=(T,B)$, we define
\[ \mf{X}_1\boxtimes_{\mf{T}} \mf{X}_2 = (X_1\times_T X_2\stackrel{p}{\to} T, (A_1\boxtimes_T A_2)\otimes p^*B^{-1},B). \]  
We claim that $\boxtimes_{\mf{T}}$ defines a monoidal functor 
\begin{equation}
\label{EqFiberProductFunctor}
\BCoCorr_S \times_{\mc{B},\CoCorr_S} \BCoCorr_S \to \BCoCorr_S.
\end{equation}
A morphism in the category $\BCoCorr_S \times_{\mc{B},\CoCorr_S} \BCoCorr_S$ is a morphism $e\from \mf{T}\to \mf{T}'$ in $\CoCorr_S$ together with a pair of morphisms $\mf{f}_i\from \mf{X}_i\to \mf{X}_i'$ ($i=1,2$) lying over $e$.  We may represent this state of affairs with a diagram 
\begin{equation}
\label{EqMorphismOfCorrespondences}
\xymatrix{
X_i \ar[d]_{p_i} & C_i \ar[l]_{c_i} \ar[d]_{q_i} \ar[r]^{c_i'} & X'_i \ar[d]^{p'_i} \\
T & U \ar[l]_{u} \ar[r]^{u'} & T'
}
\end{equation}
for $i=1,2$, with $u'$ cohomologically smooth, together with morphisms $\alpha_i\from c_i^*A_i\to (c_i')^!A_i'$ for $i=1,2$ and an isomorphism $\beta\from u^*B\isomto (u')^! B'$.   Taking fiber products over the base correspondence, we obtain a morphism of correspondences $q^\natural=(p,q,p')$ fitting into a diagram 
\begin{equation}
\xymatrix{
X_1\times_T X_2 \ar[d]_{p} & C_1\times_U C_2\ar[d]_{q} \ar[l]_c \ar[r]^{c'} & X_1' \times_{T'} X_2' \ar[d]^{p'} \\
T & U \ar[l]^u \ar[r]_{u'} & T' 
}
\end{equation}
The required morphism
\begin{equation}
\label{EqFiberProductOfCorrespondences}
c^*\left((A_1\boxtimes_T A_2)\otimes p^*B^{-1}\right) \to (c')^!\left((A_1'\boxtimes_{T'} A_2') \otimes (p')^*(B')^{-1}\right)
\end{equation}
is defined as the composition
\begin{eqnarray*}
c^*\left((A_1\boxtimes_T A_2)\otimes p^*B^{-1}\right) &\isom&
(c_1^*A_1 \boxtimes_U c_2^* A_2)\otimes q^*u^*B^{-1} \\
&\stackrel{(\alpha_1\boxtimes_U\alpha_2)\otimes \beta}{\to}&
(c_1')^!A_1' \boxtimes_U (c_2')^!A_2' \otimes q^*((u')^!B')^{-1} \\
&\stackrel{\kappa_{U/T'}}{\to}&
(c')^!(A_1' \boxtimes_{T'} A_2') \otimes q^*(K_{U/T'}\otimes ((u')^!B')^{-1}) \\
&\isom& (c')^!(A_1' \boxtimes_{T'} A_2') \otimes q^*(u')^*(B')^{-1} \\
&\isom& (c')^!\left((A_1'\boxtimes_{T'} A_2') \otimes (p')^*(B')^{-1}\right),
\end{eqnarray*}
where in the last step we used Lemma \ref{LemInvertibleSheaf}.  

To completely justify that \eqref{EqFiberProductFunctor} is a functor, one must also produce a 2-isomorphism
\begin{equation}
\label{EqBoxtimesIsFunctorial}
(\mf{f}_1'\boxtimes_{e'} \mf{f}_2') \circ (\mf{f}_1\boxtimes_e \mf{f}_2) \isom (\mf{f}_1'\circ \mf{f}_1)\boxtimes_{e'\circ e} (\mf{f}_2'\circ \mf{f}_2) 
\end{equation}
whenever all compositions are defined.  Furthermore one must also show that \eqref{EqFiberProductFunctor} is a monoidal functor; that is, we have an isomorphism
\begin{equation}
\label{EqBoxtimesIsTensorFunctor}
 (\mf{X}_1 \boxtimes_{\mf{T}} \mf{X}_2) \otimes (\mf{X}_1'\boxtimes_{\mf{T}'} \mf{X}_2')
\isom (\mf{X}_1\otimes \mf{X}_1')\boxtimes_{\mf{T}\otimes \mf{T}'} (\mf{X}_2\otimes\mf{X}_2'). 
\end{equation}
The details are straightforward but tedious.

We can now prove Theorem \ref{ThmKunnethForCharacteristicClass}.   Let $X_1,X_2\to T$ be two morphisms satisfying the assumptions of that theorem, and let $A_i\in D_{\et}(X_i,\Lambda)$ be two sheaves which are ULA over $S$.  Assume that the structure map $\pi\from T\to S$ and the diagonal $\Delta_{T/S}$ are both cohomologically smooth. 

Let $\mf{X}_i=(X_i\to T, A_i,\Lambda_T)\in \BCoCorr_S$ for $i=1,2$, so that $\mc{B}(\mf{X}_1)=\mc{B}(\mf{X}_2)=\mf{T}=(T,\Lambda_T)$.  Then $\mf{X}_i$ is dualizable, with dual $\mf{X}_i^\vee=(X_i\to T,\D_{X_i/S}A_i,K_{T/S})$, as witnessed by $\coev_{\mf{X}_i}\from 1_{\BCoCorr_S}\to \mf{X}_i\otimes \mf{X}_i^\vee$ and $\ev_{\mf{X}_i}\from \mf{X}_i^\vee \otimes\mf{X}_i \to 1_{\BCoCorr_S}$.   Note that $\mc{B}(\coev_{\mf{X}_i})=\coev_{\mf{T}}$ and $\mc{S}(\coev_{\mf{X}_i})=\coev_{\mc{S}(\mf{X}_i)}$, and similarly for $\ev$.   Then the categorical trace of $1_{\mf{X}_i}$ is $\tr(1_{\mf{X}_i})=\ev_{\mf{X}_i}\circ \coev_{\mf{X}_i}$, so that $\mc{S}(\tr(1_{\mf{X}_i}))=\tr(1_{\mc{S}(\mf{X}_i)})=\cc_{X_i/S}(A_i)$. 

Now consider $\mf{X}=\mf{X}_1\boxtimes_{\mf{T}} \mf{X}_2=(X_1\boxtimes_T X_2\to T, A_1\boxtimes_T A_2,\Lambda_T)$.  Define an object $\mf{X}^\vee = \mf{X}_1^\vee\boxtimes_{\mf{T}^\vee} \mf{X}_2^\vee$, and define morphisms $\coev_{\mf{X}}$ and $\ev_{\mf{X}}$ via the diagrams
\[
\xymatrix{
1_{\BCoCorr_S} \ar[rrrr]^{\coev_{\mf{X}_1}\boxtimes_{\coev_{\mf{T}}}\coev_{\mf{X}_2}\;\;\;\;} 
 \ar[drrrr]_{\coev_{\mf{X}}} &&&& (\mf{X}_1\otimes\mf{X}_1^\vee)\boxtimes_{\mf{T}\otimes\mf{T}^\vee} (\mf{X}_2\otimes\mf{X}_2^\vee) \ar[d]^{\cong \eqref{EqBoxtimesIsTensorFunctor}} \\
&&&&  \mf{X}\otimes \mf{X}^\vee
 }
 \]
and
\[
\xymatrix{
\mf{X}\otimes\mf{X}^\vee \ar[d]_{\eqref{EqBoxtimesIsTensorFunctor}\cong} \ar[drrrr]^{\ev_{\mf{X}}} &&&&\\
(\mf{X}_1^\vee \otimes \mf{X}_1)\boxtimes_{\mf{T}\otimes\mf{T}^\vee} (\mf{X}_2^\vee\otimes\mf{X}_2) \ar[rrrr]_{\ev_{\mf{X}_1}\boxtimes_{\ev_{\mf{T}}} \ev_{\mf{X}_2}} &&&& 1_{\BCoCorr_S} 
}
\]
Then $\mf{X}^\vee$, $\coev_{\mf{X}}$ and $\ev_{\mf{X}}$ witness the dualizability of $\mf{X}$.  It follows that $\mc{S}(\mf{X})=(X_1\times_T X_2,A_1\boxtimes_T A_2)$ is dualizable, so that $A_1\boxtimes_T A_2$ is ULA over $S$.   Now consider $\tr(1_{\mf{X}})=\ev_{\mf{X}}\circ \coev_{\mf{X}}$, an endomorphism of $1_{\BCoCorr_S}$.   On the one hand, $\mc{S}(\tr(1_{\mf{X}}))=\tr(1_{\mc{S}(\mf{X})})\in \End 1_{\CoCorr_S}$ is the data of the inertia stack $\In_S(X_1\times_T X_2)$ together with the characteristic class $\cc_{X_1\times_T X_2/S}(A)\in H^0(\In_S(X_1\times_T X_2),K_{\In_S(X_1\times_T X_2)/S})$.  On the other hand, \eqref{EqBoxtimesIsFunctorial} gives a 2-isomorphism
\[ \tr(1_{\mf{X}}) \isom (\ev_{\mf{X}_1}\boxtimes_{\ev_{\mf{T}}} \ev_{\mf{X}_2}) \circ (\coev_{\mf{X}_1}\boxtimes_{\coev_{\mf{T}}} \coev_{\mf{X}_2}) 
\isom \tr(1_{\mf{X_1}}) \boxtimes_{\tr(1_{\mf{T}})} \tr(1_{\mf{X}_2}).  \]
The source of this morphism is the correspondence $\In_S(X_1)\times_{\In_S(T_1)} \In_S(X_2)\to S\times_S S \isom S$ together with a global section of $K_{\In_S(X_1\times_T X_2)/S}$.  Reviewing the definition of $\boxtimes$ for morphisms in $\BCoCorr_S$ as in \eqref{EqFiberProductOfCorrespondences}, we see that this section is the image of $\cc_{X_1/S}(A_1)\otimes \cc_{X_2/S}(A_2)$ under 
\begin{eqnarray*}
 K_{\In_S(X_1)/S}\boxtimes_{\In_S(T)} K_{\In_S(X_2)/S} &\stackrel{\kappa_{\In_S(T)/S}}{\to}&
K_{\In_S(X_1\times_T X_2)/S} \otimes \In(p)^*K_{\In_S(T)} \\
&\isom& K_{\In_S(X_1\times_T X_2)/S}, 
\end{eqnarray*}
where the last isomorphism is induced from the inverse to 
$\cc_{T/S}(\Lambda_T)\from \Lambda_{\In_S(T)}  \to K_{\In_S(T)/S}$.  The result is exactly $\cc_{X_1/S}(A_1)\boxtimes_{\In_S(T)} \cc_{X_2/S}(A_2)$ as defined in Theorem \ref{ThmKunnethForCharacteristicClass}.

\subsection{The case of $[X/G]$ for $G$ smooth}
Let $X$ be a nice diamond over $S$ which is equipped with
an action of a cohomologically smooth $S$-group diamond $G$.   Let $\alpha\from X\times_S G\to X$ be the action map $(x,g)\mapsto g(x)$.   Let $Y=[X/G]$ be the stack quotient; this is a decent $S$-v-stack whose structure map to $S$ is fine.  The point of this section is to compare two contexts for the Lefschetz-Verdier trace formula: one for the identity correspondence on $[X/G]$, and the other for the morphism $g\from X\to X$ for an individual $g\in G(S)$.  

Let $A\in D_{\mathrm{\acute{e}t}}(Y,\Lambda)$
be ULA over $S$.  Then the pair $(Y,A)$ is dualizable in $\CoCorr_S$,  and we obtain
a characteristic class 
\[
\mathrm{cc}_{Y/S}(A)\in H^{0}(\mathrm{In}_{S}(Y),K_{\mathrm{In}_{S}(Y)/S}).
\]
On the other hand,  the pullback $A_X$ of $A$ along $X\to Y$ is also ULA over $S$ (because $G\to S$ is cohomologically smooth).  For each element $g\in G(S)$,  we have an isomorphism $u_g\from A_X\to g^*A_X$ lying over $g\from X\to X$. The pair $(g,u_g)$ constitutes an endomorphism of the dualizable object $(X,A_X)$ in $\CoCorr_S$, so we may define the categorical trace $\tr(g,u_g)\in H^0(\Fix(g),K_{\Fix(g)/S})$.  Here $\Fix(g)=X\times_{g,X\times_S X,\Delta_{X/S}} X$ is the fixed-point locus of $g$ on $X$.  The object is to show how the $\tr(g,u_g)$ can be derived from $\mathrm{cc}_{Y/S}(A)$.

First we give a concrete presentation of $\mathrm{In}_{S}(Y)$.   Define a correspondence $c$ on $X$ by
\begin{eqnarray*}
c=\pr_X\times_S \alpha:X\times_S G&\to& X\times_S X\\
(x,g)&\mapsto& (x,g(x)).
\end{eqnarray*}
Then the fixed-point locus $\Fix(c)\subset X\times_S G$ is $G$-stable for the $G$-action on $X\times_S G$ given by $h(x,g)=(h(x),hgh^{-1})$,  and then $\In_S(Y)\isom [\Fix(c)/G]$.  With respect to this isomorphism, the canonical map $p\from \In_S(Y)\to \In_S([S/G])\isom [G\sslash G]$ is the quotient by $G$ of the projection map $\Fix(c)\to G$.  

The $G$-equivariance of $A\vert_X$ may be expressed an isomorphism $u\from  A\vert_{X\times G} \to \alpha^*A\vert_X$.  This is not a cohomological correspondence in general (as $\alpha^*\neq \alpha^!$).   To obtain a cohomological correspondence on non-stacky objects,  we work over the base $G$.  Let $X_G=X\times_S G$,  and consider the correspondence $\tilde{c}$ defined by the diagram of diamonds over $G$:
\[
\xymatrix{ 
& X_G \ar[dl]_{\id_{X_G}} \ar[dr]^{\tilde{\alpha}\from (x,g)\mapsto (g(x),g) } & \\
X_G && X_G
}
\]
By design,  the fiber of this correspondence over $g\in G(S)$ is automorphism $g\from X\to X$.  
Moreover, there is a natural isomorphism $\Fix(c)\isom\Fix(\tilde{c})$, and the fiber of $\mathrm{Fix}(\tilde{c})$ over any $g\in G(S)$
is exactly $\mathrm{Fix}(g)$.  
The $G$-equivariance of $A\vert_{X}$ is encoded by an isomorphism $\tilde{u}\from A\vert_{X_G}\to \tilde{\alpha}^*A\vert_{X_G}$.  Since $\tilde{\alpha}$ is an isomorphism, we have $\tilde{\alpha}^*\isom\tilde{\alpha}^!$,  and therefore the pair $(\tilde{c},\tilde{u})$ constitutes an endomorphism of the dualizable object $(X_G,A\vert_{X_G})$ of $\CoCorr_G$.  The categorical trace of $(\tilde{c},\tilde{u})$ is an element
\[\tr(\tilde{c},\tilde{u}) \in H^0(\Fix(\tilde{c}),K_{\Fix(\tilde{c})/G}). \]
This is the ``universal local term'' for the action of $G$ on $X$,  in the sense that for any $g\in G(S)$, the restriction map
\[ H^0(\Fix(\tilde{c}),K_{\Fix(\tilde{c})/G}) \to H^0(\Fix(g), K_{\Fix(g)/S}) \]
carries $\tr(\tilde{c},\tilde{u})$ onto $\tr(g,u_g)$.   

We want to compare the characteristic class $\mathrm{cc}_{Y/S}(A)$
with the universal local term $\tr(\tilde{c},\tilde{u})$.
To do this, we first observe that from the Cartesian square
\[
\xymatrix{\mathrm{Fix}(\tilde{c})\ar[d]_{q}\ar[r] & G\ar[d]\\
\mathrm{In}_{S}(Y)\ar[r] & [G\sslash G]
}
\]
we obtain a canonical map $q^{\ast}K_{\mathrm{In}_{S}(Y)/[G\sslash G]}\to K_{\mathrm{Fix}(\tilde{c})/G}$,
and thus a canonical pullback map
\begin{equation}
\label{EqPullbackMap}
q^*\from H^{0}(\mathrm{In}_{S}(Y),K_{\mathrm{In}_{S}(Y)/[G\sslash G]})\to H^{0}(\mathrm{Fix}(\tilde{c}),K_{\mathrm{Fix}(\tilde{c})/G}).
\end{equation}
Next,  Lemma \ref{LemTrivialDualizingSheaf} applied to $T=[S/G]$ shows that $\cc_{T/S}(\Lambda_T)$ is an isomorphism $\Lambda_{[G\sslash G]} \isomto K_{[G\sslash G]/S}$.  This induces an isomorphism
\begin{equation}
\label{EqIdentificationOfDualizingSheaves}
K_{\In_S(Y)/[G\sslash G]} \isom p^!\Lambda_{[G\sslash G]} \isomto p^!K_{[G\sslash G]/S} \isom 
K_{\In_S(Y)/S}. 
\end{equation}
 
Combining \eqref{EqPullbackMap} and \eqref{EqIdentificationOfDualizingSheaves}, we obtain a canonical map
\begin{equation}
\label{EqDefinitionOfIota}
\iota:H^{0}(\mathrm{In}_{S}(Y),K_{\mathrm{In}_{S}(Y)/S})\to H^{0}(\mathrm{Fix}(\tilde{c}),K_{\mathrm{Fix}(\tilde{c})/G}).
\end{equation}
The main result of this section is the following theorem.
\begin{thm}
\label{ThmCharClassOnXModG}
Notation and assumptions as above, we have an equality
\[
\iota\left(\mathrm{cc}_{Y/S}(A)\right)=\mathrm{tr}_{\tilde{c}}(\tilde{u},A|_{X_{G}}).
\]
\end{thm}

\begin{proof}  We restate the theorem in the language of based cohomological correspondences.  The main players are
\begin{itemize}
\item $T=([S/G],\Lambda_{[S/G]})$, a dualizable object of $\CoCorr_S$.
\item $\mf{Y}=(Y\to [S/G],A,\Lambda_{[S/G]})$, a dualizable object of $\BCoCorr_S$ with base $T$.
\item $\mf{X}_G=(X_G,A_{X_G})$, a dualizable object of $\CoCorr_G$ with base $1_{\CoCorr_G}$.
\item $\alpha\in \End \mf{X}_G$, the endomorphism described by the pair $(\tilde{c},\tilde{u})$.
\end{itemize}
We would like to relate $\tr(\id_{\mf{Y}})$ to $\tr(\alpha)$.  The idea is to promote $\alpha$ to an endomorphism of based cohomological correspondences which lies over $\id_{T}$.  To this end we introduce some more objects in $\BCoCorr_S$:
\begin{itemize}
\item $\mf{G}=(G\to S,\Lambda_G,\Lambda_S)$ with base $1_{\CoCorr_S}$,
\item $\mf{G}'=(G\to [S/G]\times_S [S/G],\Lambda_G,\Lambda_{[S/G]}\boxtimes_S K_{[S/G]/S})$, with base $T\otimes T^\vee$, where the morphism $G\to [S/G]\times_S [S/G]$ is defined as the trivial $G\times_S G$-torsor $G_{G\times_S G}=G\times_S G\times_S G$.
\end{itemize}
We also define morphisms in $\BCoCorr_S$:
\begin{itemize}
\item $\coev_G \from \mf{G}\to \mf{G}'$, which has base $\coev_T$ and source $1_{(G,\Lambda_G)}$,
\item $\ev_G\from \mf{G}'\to \mf{G}$, which has base $\ev_T$ and source $1_{(G,\Lambda_G)}$.
\item An automorphism $\alpha_0\in \Aut \mf{G}'$ lying over the identity on both base and source, coming from a 2-isomorphism of $G\to [S/G]^2$ corresponding to the automorphism of the trivial torsor $G_{G\times G}\to G_{G\times G}$ defined by $(x,g,h)\mapsto (x,gx,h)$.  
\end{itemize}
Now observe that $\tr_G:=\ev_G\circ\alpha_0\circ\coev_G$ is an endomorphism of $\mf{G}$ with base $\tr(\id_T)$ and source $\id_{\mc{S}(\mf{G})}$, whose underlying based correspondence is shown in the diagram:
\[
\xymatrix{
G \ar[d] & G \ar[l]_{\id_G} \ar[d] \ar[r]^{\id_G} & G \ar[d] \\
S & [G\sslash G] \ar[l] \ar[r] & S
}
\]
where the central horizontal arrow sends $g$ to its own conjugacy class.  (If we had omitted $\alpha_0$ from the definition of $\tr_G$, the central horizontal arrow would send everything to the identity of $G$.)  

Recall that $\mc{S}\from \BCoCorr_S\to \CoCorr_S$ takes a based cohomological corresponce onto its source.  Let $\mc{F}\from \CoCorr_G\to\CoCorr_S$ be the functor which forgets the base $G$.  We have a diagram in $\CoCorr_S$:
\[
\xymatrix{
\mc{S}(1_{\BCoCorr_S} \boxtimes_{1_{\CoCorr_S}}\mf{G}) 
\ar[d]_{\mc{S}(\coev_{\mf{Y}}\boxtimes_{\coev_T} \coev_G)} 
\ar[rr]^{\sim} 
&&
\mc{F}(1_{\CoCorr_G})
\ar[d]^{\mc{F}(\coev_{\mf{X}_G})} 
\\
\mc{S}((\mf{Y}\otimes\mf{Y}^\vee) \boxtimes_{T\otimes T^\vee} \mf{G}')
\ar[d]_{\mc{S}(\id_{\mf{Y}\otimes \mf{Y}^\vee}  \boxtimes_{\id_{T\otimes T^\vee}} \alpha_0)} 
\ar[rr]^{\sim}
&&
\mc{F}(\mf{X}_G \otimes \mf{X}_G^\vee)
\ar[d]^{\mc{F}(\alpha)} 
\\
\mc{S}((\mf{Y}\otimes\mf{Y}^\vee) \boxtimes_{T\otimes T^\vee} \mf{G}') 
\ar[d]_{\mc{S}(\ev_{\mf{Y}} \boxtimes_{\ev_T} \ev_G )}
\ar[rr]^{\sim}
&&
\mc{F}(\mf{X}_G\otimes \mf{X}_G^\vee) 
\ar[d]^{\mc{F}(\ev_{\mf{X}_G})}
\\
\mc{S}(1_{\BCoCorr_S} \boxtimes_{1_{\CoCorr_S}} \mf{G})
\ar[rr]^{\sim}
&&
\mc{F}(1_{\CoCorr_G})
}
\]
where all squares are filled in with 2-isomorpisms.  The functoriality property of $\boxtimes$ from \eqref{EqBoxtimesIsFunctorial} now gives a 2-isomorphism
\begin{equation}
\label{EqCompareCCToTrace}
 \mc{S}\left(\cc_{Y/S}(A)\boxtimes_{\cc_{[S/G]/S}(\Lambda_{[S/G]})}  \tr_G(\alpha_0) \right)\isomto \mc{F}(\tr(\alpha)). 
 \end{equation}
The isomorphism of v-stacks implicit in \eqref{EqCompareCCToTrace} is expressed by the fact that we have a cartesian diagram:
\[
\xymatrix{
\Fix(\tilde{c}) \ar[r] \ar[d] & G \ar[d] \\
\In_S(Y) \ar[r] & [G\sslash G] 
}
\]
On the level of cohomology classes, \eqref{EqCompareCCToTrace} tells us that $\tr(\alpha)$, considered as an element of $H^0(\Fix(\tilde{c}),K_{\Fix(\tilde{c})/S})$ can be derived from $\cc_{Y/S}(A)$ in the manner described by the theorem.
\end{proof}

We conclude the section with a remark about isolated fixed points.  Assume there exists a conjugacy-invariant open subset $U\subset G$ whose elements act on $X$ with only isolated fixed points.  (This is the case for the action of the positive loop group on the affine Grassmannian. We study that scenario in the next section.)   Write $\In_S([X/G])_U$ for the pullback of $[U\sslash G]$ under $\In_S([X/G])\to [G\sslash G]$.  Then $\In_S([X/G])_U\to [U\sslash G]$ is \'etale over $[U\sslash G]$;  as such we have a canonical trivialization $K_{\In_S([X/G])_U/S}\isom \Lambda_{\In_S([X/G])_U}$.  Therefore the restriction over $U$ of the characteristic class of $A$ is an element
\[ cc_{[X/G]/S}(A)_U\in H^0(\In_S([X/G])_U,\Lambda); \]
that is, it is a continuous function on the space of pairs $(x,g)\in X\times_S U$ with $g(x)=x$.  Theorem \ref{ThmCharClassOnXModG} implies that this function is $(x,g)\mapsto \loc_x(g,A)$.

%\documentclass[11pt]{article}
%\usepackage{macros, xcolor}

%\begin{document}
%\maketitle

\section{Local terms on the $B_{\mathrm{dR}}$-affine Grassmannian}
\label{SectionLocalTerms}
The goal of this chapter is to explicitly compute certain local terms
on the $B_{\dR}$-affine Grassmannian, in terms of the geometric Satake equivalence.

\subsection{The main result}

To explain the main result, let us fix some notation. Let $F/\mathbf{Q}_{p}$
be a finite extension with residue field $\mathbf{F}_q$.  Let $C$ be the completion of an algebraic closure of $F$.   Let $G/F$ be a connected reductive group, and let $\mathrm{Gr}_{G}=LG/L^{+}G$ be the associated
$B_{\mathrm{dR}}$-affine Grassmannian over $\Spd C$.    We explain the notation: for a perfectoid $C$-algebra $R$, we have the loop group $LG(R)=B_{\dR}(R)$ and its positive subgroup $L^+G(R)=B_{\dR}^+(R)$.   Then $\Gr_G$ is an ind-spatial diamond admitting an action of $L^+G$ and in particular its subgroup $G(F)$.  For a cocharacter $\mu$ of $G_{\overline{F}}$, we let $\Gr_{G,\leq \mu}$ be the corresponding closed Schubert cell;  this is a proper diamond.  
Finally, define the {\em local Hecke stack} by
\[ \Hecke_G^{\loc} = [L^+G\backslash \Gr_G ] = [L^+G \backslash LG/L^+G].\]
We remark that there are versions of these objects living over $\Spd F$, but we will not need these for our results.

Fix a coefficient ring $\Lambda\in\left\{ \mathbf{Z}/\ell^{n}\mathbf{Z}[\sqrt{q}],\mathbf{Z}_{\ell}[\sqrt{q}\right\} $.
The {\em Satake category}
\[ \Sat_G(\Lambda) \subset D_{\et}(\Hecke^{\loc}_{G,C}, \Lambda) \]
is the subcategory of objects which are perverse, $\Lambda$-flat, and ULA over $\Spd C$ \cite[Definition I.6.2]{FarguesScholze}.   It is a symmetric monoidal category under the convolution product.

%JW -- I wanted to promote this into a theorem
\begin{thm}[{\cite[Theorem I.6.3]{FarguesScholze}}] 
\label{ThmSatakeEquivalence} There is an equivalence of symmetric monoidal categories:
\begin{eqnarray*}
\Rep_{\hat{G}}(\Lambda) &\isomto& \Sat_G(\Lambda) \\
V &\mapsto& \mc{S}_V 
\end{eqnarray*}
where $\hat{G}$ is the Langlands dual group (considered over $\Lambda$),  and $\Rep_{\hat{G}}(\Lambda)$ is the category of representations of $\hat{G}$ on finite projective $\Lambda$-modules.
\end{thm}
We continue to write $\mc{S}_V$ for the pullback of this object along the quotient $\Gr_G\to [L^+G \backslash \Gr_G]$. 

Our next order of business is to determine, for $g\in G(F)_{\sr}$, the fixed point locus $\Gr_G^g$.  The answer is the same regardless of which sort of affine Grassmannian we consider (classical, Witt vector, $B_{\dR}$), as the following proposition shows.

\begin{pro} \label{PropositionFixedPointsOnGr} Let $K^+$ be a discrete valuation ring with algebraically closed residue field $k$ and fraction field $K$.  Let $G$ be a reductive group over $K^+$.  Let $g\in G(K^+)$ be an element whose image in $G(k)$ is strongly regular, and let $T=\Cent(g,G)$.  The inclusion $T\subset G$ induces a bijection
\[ T(K)/T(K^+)\isom (G(K)/G(K^+))^g,\]
so that the fixed point locus of $g$ may be identified with $X_*(T)$.  

Consequently if $\Gr_G$ is any incarnation of the affine Grassmannian, then $\Gr_{G,\leq\mu}^g$ is finite over its base with underlying set $X_*(T)_{\leq \mu}$.
\end{pro}

\begin{proof} 
Let $\mc{B}$ be the (reduced) Bruhat-Tits building of the split reductive group $G_K$ over the discretely valued field $K$.  Thus $\mc{B}$ is a locally finite simplicial complex admitting an action of $LG=G(K)$.  We will identify the $LG$-set $\Gr_G$ with a piece of this building. 

By \cite[5.1.40]{BT2} there exists a hyperspecial point $\bar o \in \mc{B}$ corresponding to $L^+G=G(K^+)$.   The point $\bar o$ can be characterized by \cite[4.6.29]{BT2} as the unique fixed point of $L^+G$.  Let $\mc{B}^\tx{ext}$ be the extended Bruhat-Tits building of $G_K$.  Recall that $\mc{B}^\tx{ext}=\mc{B} \times X_*(A_G)_\R$, where $A_G$ is the connected center of $G$. The group $LG$ acts on $X_*(A_G)_\R$ via the isomorphism $X_*(A_G)_\R \to X_*(A_G')_\R$, where $A_G'$ is the maximal abelian quotient of $G$. Let $o=(\bar o,z)$ be any point in $\mc{B}^\tx{ext}$ lying over $\bar o$. Then $L^+G$ can be characterized as the full stabilizer of $o$ in $G(K)$:  It is clear that $L^+G$ stabilizes $o$, and the reverse inclusion follows from the Cartan decomposition $LG = L^+G \cdot X_*(T) \cdot L^+G$ (which relies on $\bar o$ being hyperspecial) and the fact that $X_*(T)$ acts on the apartment of $T$ in $\mc{B}^\tx{ext}$ by translations. It follows that the action of $LG$ on $\mc{B}^\tx{ext}$ provides an $LG$-equivariant bijection from $\Gr_G$ to the orbit of $LG$ through $o$.

Now suppose $x\in \Gr_G$ is fixed by a strongly regular element $g\in L^+T_{\sr}$.  Then its image in $\mc{B}^\tx{ext}$ is a $g$-fixed point belonging to the orbit of $o$, and we can write $x=ho$ for some $h\in LG$.  For every root $\alpha\from T\to \Gm$, the element $\alpha(g)$ does not lie in the kernel of $L^+\Gm\to \Gm$.  According to \cite[3.6.1]{TitsCorvallis} the image of $x$ in $\mc{B}$ belongs to the apartment $\mc{A}$ of $T$. At the same time, $g\in L^+G$ also fixes $\bar o$, so for the same reason $\bar o \in \mc{A}$. Thus $\bar o$ belongs to both apartments $\mc{A}$ and $h^{-1}\mc{A}$. Since $L^+G$ acts transitively on the apartments containing $\bar o$ \cite[4.6.28]{BT2}, we can multiply $h$ on the right by an element of $L^+G$ to ensure that $h^{-1}\mc{A}=\mc{A}$. By \cite[7.4.10]{BT1} we then have $h \in L^+N(T,G)$. Since $\bar o$ is hyperspecial, every Weyl reflection is realized in $L^+G$ and hence we may again modify $h$ on the right to achieve $h \in LT$. We see now that $x=ho$ is fixed by all of $LT$ and that furthermore the coset $x=hL^+G$ is the image of the coset $hL^+T$.
\end{proof}

Proposition \ref{PropositionFixedPointsOnGr} shows that if we fix a split maximal torus $\hat{T}\subset\hat{G}$,
there is a natural finite-to-one map
\begin{align*}
\mathrm{Gr}_{G}^{g} & \to X_{+}^{\ast}(\hat{T})\\
x & \mapsto\nu_{x}.
\end{align*}
Note that $\nu_{x}$ simply records which open Schubert cell of $\mathrm{Gr}_{G}$
contains the point $x$. 

Now, for any $V\in\mathrm{Rep}(\hat{G})$ and any $x\in\mathrm{Gr}_{G}^{g}$, there is an associated local term $\mathrm{loc}_{x}(g,\mathcal{S}_{V})\in\Lambda$. The main result of this chapter
is the following theorem, giving an explicit computation of these local terms.
\begin{thm}\label{ThmLocalTermsGrassmannian}
Let $V\in\mathrm{Rep}(\hat{G})$ be an object of the Satake category,
and let $g\in G(F)_{\mathrm{sr}}$ be a strongly regular semisimple
element. Then for any $x\in\mathrm{Gr}_{G}^{g}$, there is an equality in $\Lambda$:
\[
\mathrm{loc}_{x}(g,\mathcal{S}_{V})=(-1)^{\left\langle 2\rho,\nu_{x}\right\rangle }\mathrm{rank}_{\Lambda}V[\nu_{x}].
\]
\end{thm}

Note that since $V$ is (by hypothesis) a finite projective $\Lambda$-module,
and tori are reductive in the strongest sense, the weight space $V[\nu_{x}]$
is a finite projective $\Lambda$-module, so the right-hand side of
this equality is well-defined.

Due to the highly inexplicit nature of local terms, the proof of Theorem
\ref{ThmLocalTermsGrassmannian} is rather indirect. Indeed, we would be able to give a simple
proof of Theorem \ref{ThmLocalTermsGrassmannian} \emph{if} we knew the equality between ``true''
and ``naive'' local terms on $\mathrm{Gr}_{G}$. Unfortunately,
this equality seems to be a very difficult problem. Even for schemes,
the problem of comparing true and naive local terms was only settled
very recently by Varshavsky. Instead, our strategy reduces the computation
of the local terms in Theorem \ref{ThmLocalTermsGrassmannian} to an analogous computation on the
Witt vector affine Grassmannian, where a global-to-local argument
can be pushed through. The key theme in the proof is the idea that
\emph{local terms are constant in families.}

For our applications, the following restatement of the main results of this section in terms of characteristic classes on the quotient $[\Gr_{G,\leq \mu}/L_m^+G]$ will be useful.  

\begin{thm} \label{ThmCharClassOfLocalHeckeStack}  Let $V$ be such that $\mc{S}_V$ is supported on some Schubert cell $\Gr_{G,\leq \mu}$. Choose some large $m$ such that the $L^+ G$-action on this cell factors through the quotient $L_m^+G$, and set $X=[\Gr_{G,\leq \mu}/L_m^+G]$.  

Then the set of connected components of $\In_S(X)_{\sr}$ may be identified with $X_*^+(T)_{\leq \mu}/W$, and the dualizing complex of $\In_S(X)_{\sr}$ has a canonical trivialization.   With respect to those identifications,  the restriction of $\cc_{X/S}(\mc{S}_V)$ to $\In_S(X)_{\sr}$ is the function sending $\lambda\in X_*(T)$ to $(-1)^{\class{2\rho_G,\lambda}}\rank_\Lambda V[\lambda]$.  
\end{thm}

\begin{proof} The first claim is proved in \S \ref{SubsectionIndependence} below. Since $L_m^+G$ is a cohomologically smooth group diamond, Theorem \ref{ThmCharClassOnXModG} applies to the quotient $X=[\Gr_{G,\leq\mu}/L_m^+G]$.  The remark following the proof of that theorem applies to the locus $L_m^+G_{\sr}$, so that we may relate $\cc_{X/S}(\mc{S}_V)$ to the local terms $\loc_x(g,\mc{S}_V)$.  The latter have been computed by Theorem \ref{ThmLocalTermsGrassmannian}.
\end{proof}

\subsection{Strategy of proof}

In this section, we reduce Theorem \ref{ThmLocalTermsGrassmannian} to four auxiliary propositions
stated below. The proofs of these propositions will occupy the remainder
of this chapter.

As a preliminary observation, note that all of the objects appearing in Theorem \ref{ThmLocalTermsGrassmannian} depend on $G$ only through its base change to $\overline{F}$, so we may enlarge $F$ whenever convenient in the argument. In particular, we can and do assume that $G$ admits a split reductive model $\mathcal{G} / \mathcal{O}_F$, and that $\mathcal{G}(\mathcal{O}_F)$ contains elements of finite prime-to-$p$ order with strongly regular semisimple image in $\mathcal{G}(\mathbf{F}_q)$.

Now we begin the argument. First, we show that the local terms appearing in Theorem \ref{ThmLocalTermsGrassmannian} are essentially
independent of $g$.
\begin{pro}\label{PropIndependenceOfg}
In the notation and setup of Theorem \ref{ThmLocalTermsGrassmannian}, $\mathrm{loc}_{x}(g,\mathcal{S}_{V})$
depends on $g$ and $x$ only through the cocharacter $\nu_{x}$.
More precisely, if $g,g'\in G(F)_{\mathrm{sr}}$ are two strongly
regular semisimple elements and $x\in\mathrm{Gr}_{G}^{g}$ resp. $x'\in\mathrm{Gr}_{G}^{g'}$
are fixed points such that $\nu_{x}=\nu_{x'}$, then
\[
\mathrm{loc}_{x}(g,\mathcal{S}_{V})=\mathrm{loc}_{x'}(g',\mathcal{S}_{V}).
\]
\end{pro}

Next, we are going to degenerate from characteristic zero into characteristic
$p$. For this, fix a split reductive model $\mathcal{G}/\mathcal{O}_{F}$
of $G$, and let $\mathrm{Gr}_{\mathcal{G}}$ be the associated Beilinson-Drinfeld
affine Grassmannian over $S=\Spd \OO_C$.
Recall that this is a small v-sheaf which interpolates between the
$B_{\mathrm{dR}}$-affine Grassmannian $\mathrm{Gr}_{G}$ and the
Witt vector affine Grassmannian $\mathrm{Gr}_{\mathcal{G}}^{W}$,
in the sense that we have a commutative diagram
\[
\xymatrix{(\mathrm{Gr}_{\mathcal{G}}^{W})^{\lozenge}\ar[d]\ar[r]^{i} & \mathrm{Gr}_{\mathcal{G}}\ar[d] & \mathrm{Gr}_{G}\ar[d]\ar[l]_{j}\\
s=\mathrm{Spd}\,\overline{\mathbf{F}_{q}}\ar[r] & S & \eta=\mathrm{Spd} C \ar[l]
}
\]
with cartesian squares. We will crucially use the fact that all of
these Grassmannians satisfy compatible forms of geometric Satake,
in the sense that there are natural monoidal functors
\[
\xymatrix{ & D_{\mathrm{\acute{e}t}}(\mathrm{Gr}_{G},\Lambda)\\
\mathrm{Rep}(\hat{G})\ar[r]\ar[ur]\ar[dr] & D_{\mathrm{\acute{e}t}}(\mathrm{Gr}_{\mathcal{G}},\Lambda)\ar[u]^{j^{\ast}}\ar[d]_{i^{\ast}}\\
 & D_{\mathrm{\acute{e}t}}((\mathrm{Gr}_{\mathcal{G}}^{W})^{\lozenge},\Lambda)
}
\]
such that the vertical arrows are equivalences of categories on the
essential images of $\mathrm{Rep}(\hat{G})$.
\begin{pro}\label{PropFixedPointsBD}
Let $g\in\mathcal{G}(\mathcal{O}_{F})$ be an element such that $\overline{g}\in\mathcal{G}(\mathbf{F}_q)$
is strongly regular semisimple. Then $\mathcal{T}=\mathrm{Cent}(g,\mathcal{G})$
is a maximal torus, and there is a natural isomorphism
\[
\mathrm{Gr}_{\mathcal{G}}^{g}\cong\mathrm{Gr}_{\mathcal{T}}\cong X_{\ast}(\mathcal{T})_S.
\]
\end{pro}

In particular, if $\beta\simeq S\subset\mathrm{Gr}_{\mathcal{G}}^{g}$
is any connected component, then $\beta_{\eta}$ and $\beta_{s}$
are isolated fixed points for the $g$-action in the generic and special
fiber, respectively.
\begin{pro}\label{PropDegeneratingLocalTerms}
Let $g\in\mathcal{G}(\mathcal{O}_{F})$ be an element such that $\overline{g}\in\mathcal{G}(\mathbf{F}_q)$
is strongly regular semisimple. Then
for any $V\in\mathrm{Rep}(\hat{G})$ and any connected component $\beta\subset\mathrm{Gr}_{\mathcal{G}}^{g}$,
we have the equality
\[
\mathrm{loc}_{\beta_{\eta}}(g,j^{\ast}\mathcal{S}_{V})=\mathrm{loc}_{\beta_{s}}(g,i^{\ast}\mathcal{S}_{V})
\]
of local terms.
\end{pro}

Finally, we compute the local terms on the Witt vector affine Grassmannian
by a direct argument. 
\begin{pro}\label{PropLocalTermsWitt}
Let $g\in\mathcal{G}(\mathcal{O}_{F})$ be an element with finite prime-to-$p$ order such that $\overline{g}\in\mathcal{G}(\mathbf{F}_q)$
is strongly regular semisimple. Then
for any $x\in\mathrm{Gr}_{\mathcal{G}}^{W,g}$ and any $V\in\mathrm{Rep}(\hat{G})$,
\[
\mathrm{loc}_{x}(g,\mathcal{S}_{V})=(-1)^{\left\langle 2\rho,\nu_{x}\right\rangle }\mathrm{rank}_{\Lambda}V[\nu_{x}]
\]
where $\nu_{x}\in X_{+}^{\ast}(\hat{T})$ is as before.
\end{pro}

\subsection{Local terms and base change}

In this section we prove two key technical results, namely that formation
of local terms commutes with any base change, and with passage from
perfect schemes to v-sheaves.

In order to fix notation, we briefly recall the key definitions concerning
local terms; we apologize for the overlap with Chapter \ref{SectionLefschetzVerdier}. Let $S$
be a small v-sheaf, which will be our base object. Let $f:X\to S$
be a map of v-sheaves representable in nice diamonds. Consider a correspondence
$c=(c_1,c_2)\from C\to X\times_{S}X$ given by a map of v-sheaves
representable in nice diamonds. This gives rise to a cartesian diagram
\[
\xymatrix{\mathrm{Fix}(c)\ar[d]_{c'}\ar[r]^{\Delta'} & C\ar[d]^{c=(c_{1},c_{2})}\\
X\ar[r]_{\Delta} & X\times_{S}X
}
\]
of small v-sheaves. We will sometimes assume that $c_{1}$ is proper
and that $\mathrm{Fix}(c)$ is a disjoint union of open-closed subspaces
which are proper over $S$. These conditions will hold e.g. if $f$
and $c$ are proper.

Let $\mathcal{F}\in D_{\mathrm{\acute{e}t}}(X,\Lambda)$ be an $f$-ULA
object. Recall that a cohomological correspondence over $c$ is a
map $u:Rc_{2!}c_{1}^{\ast}\mathcal{F}\to\mathcal{F}$, i.e. an element
$u\in\mathrm{Hom}(c_{1}^{\ast}\mathcal{F},Rc_{2}^{!}\mathcal{F})$.
If $c_{1}$ is proper, then applying $Rf_{!}$ induces an endomorphism
\begin{align*}
Rf_{!}u:Rf_{!}\mathcal{F} & \to Rf_{!}Rc_{1\ast}c_{1}^{\ast}\mathcal{F}=Rf_{!}Rc_{1!}c_{1}^{\ast}\mathcal{F}\\
 & \;\;\;\cong Rf_{!}Rc_{2!}c_{1}^{\ast}\mathcal{F}\overset{u}{\to}Rf_{!}\mathcal{F}
\end{align*}
of $Rf_{!}\mathcal{F}$. On the other hand, there is a natural map
\[
\mathrm{Hom}(c_{1}^{\ast}\mathcal{F},Rc_{2}^{!}\mathcal{F})\to H^{0}(\mathrm{Fix}(c),K_{\mathrm{Fix}(c)/S}),
\]
cf. the discussion immediately before Definition 4.3.5, and we write
$\mathrm{tr}_{c}(u,\mathcal{F})\in H^{0}(\mathrm{Fix}(c),K_{\mathrm{Fix}(c)/S})$
for the image of $u$ under this map. 

If $\beta\subset\mathrm{Fix}(c)$ is a closed-open subspace with \emph{proper}
structure map $g:\beta\to S$, then $H^{0}(\beta,K_{\beta/S})=H^{0}(\beta,Rg^{!}\Lambda)$
is canonically a direct summand of $H^{0}(\mathrm{Fix}(c),K_{\mathrm{Fix}(c)/S})$,
and we can further consider the image of $\mathrm{tr}_{c}(u,\mathcal{F})$
under the map
\begin{align*}
H^{0}(\mathrm{Fix}(c),K_{\mathrm{Fix}(c)/S}) & \to H^{0}(\beta,Rg^{!}\Lambda)\cong H^{0}(S,Rg_{\ast}Rg^{!}\Lambda)\\
 & \;\;\;\;\;\cong H^{0}(S,Rg_{!}Rg^{!}\Lambda)\to H^{0}(S,\Lambda).
\end{align*}
By definition, this is the local term $\mathrm{loc}_{\beta}(u,\mathcal{F})$.
In most situations we care about, $S$ is connected, so $H^{0}(S,\Lambda)=\Lambda$
and we simply regard $\mathrm{loc}_{\beta}(u,\mathcal{F})$ as an
element of $\Lambda$. Note that local terms are additive, in the
sense that if $\beta=\beta_{1}\coprod\beta_{2}$, then $\mathrm{loc}_{\beta}(u,\mathcal{F})=\mathrm{loc}_{\beta_{1}}(u,\mathcal{F})+\mathrm{loc}_{\beta_{2}}(u,\mathcal{F})$.
If $S$ is a geometric point, $f$ and $c$ are proper, and $\pi_{0}(\mathrm{Fix}(c))$
is a discrete (and therefore finite) set, the usual Lefschetz trace
formula holds, and says that
\[
\mathrm{tr}(u|R\Gamma_{c}(X,\mathcal{F}))=\sum_{\beta\in\pi_{0}(\mathrm{Fix}(c))}\mathrm{loc}_{\beta}(u,\mathcal{F}).
\]

We need to understand how local terms interact with base change on
$S$. More precisely, assume we are given a morphism $a:T\to S$.
Then all objects and morphisms above naturally base change to objects
over $T$. Note that $\mathrm{Fix}(c)_{T}=\mathrm{Fix}(c_{T})$. We
write $a_{X}:X_{T}\to X$, $a_{C}:C_{T}\to C$, etc. for the base
changes of $a$. We naturally get a cohomological correspondence $u_{T}$
on $\mathcal{F}_{T}=a_{X}^{\ast}\mathcal{F}$ over $c_{T}$ by taking
the image of $u$ under the map
\begin{align*}
\mathrm{Hom}(c_{1}^{\ast}\mathcal{F},Rc_{2}^{!}\mathcal{F}) & \to\mathrm{Hom}(a_{C}^{\ast}c_{1}^{\ast}\mathcal{F},a_{C}^{\ast}Rc_{2}^{!}\mathcal{F})\cong\mathrm{Hom}(c_{1,T}^{\ast}a_{X}^{\ast}\mathcal{F},a_{C}^{\ast}Rc_{2}^{!}\mathcal{F})\\
 & \;\;\;\;\to\mathrm{Hom}(c_{1,T}^{\ast}a_{X}^{\ast}\mathcal{F},Rc_{2,T}^{!}a_{X}^{\ast}\mathcal{F}).
\end{align*}
The final arrow here is induced by the canonical map $a_{C}^{\ast}Rc_{2}^{!}\mathcal{F}\to Rc_{2,T}^{!}a_{X}^{\ast}\mathcal{F}$.
This map is a special case of the natural transformation $\beta_{f,g}:\tilde{f}^{\ast}Rg^{!}\to R\tilde{g}^{!}f^{\ast}$
which exists for any cartesian diagram
\[
\xymatrix{X'\ar[r]^{\tilde{f}}\ar[d]_{\tilde{g}} & Y'\ar[d]^{g}\\
X\ar[r]_{f} & Y
}
\]
with $g$ representable in nice diamonds. The transformation in question
is adjoint to the map $R\tilde{g}_{!}\tilde{f}^{\ast}Rg^{!}\cong f^{\ast}Rg_{!}Rg^{!}\to f^{\ast}$
(it is also\emph{ }adjoint to the map $Rg^{!}\to Rg^{!}Rf_{\ast}f^{\ast}\cong R\tilde{f}_{\ast}R\tilde{g}^{!}f^{\ast}$).

In this setup, the next proposition says that formation
of local terms commutes with base change along $T\to S$.
\begin{pro}\label{PropLocalTermsBaseChange}
For any given $\beta\subset\mathrm{Fix}(c)$ as above, the natural
map
\[
H^{0}(S,\Lambda)\to H^{0}(T,\Lambda)
\]
sends $\mathrm{loc}_{\beta}(u,\mathcal{F})$ to $\mathrm{loc}_{\beta_{T}}(u_{T},\mathcal{F}_{T})$.
In particular, if $S$ and $T$ are connected, then $\mathrm{loc}_{\beta}(u,\mathcal{F})=\mathrm{loc}_{\beta_{T}}(u_{T},\mathcal{F}_{T})$
as elements of $\Lambda$.
\end{pro}

\begin{proof}
By a straightforward argument, this reduces to showing that there
is a natural map
\[
H^{0}(\mathrm{Fix}(c),K_{\mathrm{Fix}(c)/S})\to H^{0}(\mathrm{Fix}(c)_{T},K_{\mathrm{Fix}(c)_{T}/T})
\]
compatible with the map $H^{0}(S,\Lambda)\to H^{0}(T,\Lambda)$, and
sending $\mathrm{tr}_{c}(u,\mathcal{F})$ to $\mathrm{tr}_{c_{T}}(u_{T},\mathcal{F}_{T})$.
To obtain the map itself, apply $H^{0}(\mathrm{Fix}(c),-)$ to the
composition
\begin{align*}
K_{\mathrm{Fix}(c)/S} & =R(f\circ c')^{!}\Lambda\to Ra_{\mathrm{Fix}(c)\ast}a_{\mathrm{Fix}(c)}^{\ast}R(f\circ c')^{!}\Lambda\\
 & \;\overset{\beta_{a,f\circ c'}}{\longrightarrow}Ra_{\mathrm{Fix}(c)\ast}R(f_{T}\circ c'_{T})^{!}a^{\ast}\Lambda=Ra_{\mathrm{Fix}(c)\ast}K_{\mathrm{Fix}(c)_{T}/T}.
\end{align*}
The claim about the relation between $\mathrm{tr}_{c}$ and $\mathrm{tr}_{c_{T}}$
now follows from the fact that the base change functor $\mathrm{CoCorr}_{S}\to\mathrm{CoCorr}_{T}$
is symmetric monoidal, and therefore preserves dualizable objects
and traces of endomorphisms thereof.
\end{proof}
We will also need to compare local terms associated with perfect schemes
and with v-sheaves. More precisely, fix a perfect field $k/\mathbf{F}_{p}$,
and let $\mathrm{PSch}_{k}$ be the category of perfect schemes over
$k$. There is a natural functor $X\mapsto X^{\lozenge}$ from $\mathrm{PSch}_{k}$
to small v-sheaves over $\mathrm{Spd}\,k$, characterized by $(\mathrm{Spec}\,R)^{\lozenge}(A,A^{+})=\mathrm{Hom}_{k}(R,A)$.
Said differently, $X^{\lozenge}$ sends $\Spec R$ to $\Spa(R,R^{+})^{\lozenge}$
where $R^{+}$ is the integral closure of $k$ in $R$. This functor
commutes with finite limits. Moreover, if $f:X\to Y$ is separated
and perfectly of finite type, then $f^{\lozenge}$ is representable
in locally spatial diamonds and compactifiable with finite dim.trg.
By \cite[\S 27]{ScholzeEtaleCohomology},  for any $X$ there is a fully faithful symmetric
monoidal functor $c_{X}^{\ast}:D_{\mathrm{\acute{e}t}}(X,\Lambda)\to D_{\mathrm{\acute{e}t}}(X^{\lozenge},\Lambda)$
compatible with $f^{\ast}$ and $Rf_{!}$ in the evident senses. Moreover,
one has canonical natural transformations
\[
c_{X}^{\ast}R\mathscr{H}\mathrm{om}(-,-)\to R\mathscr{H}\mathrm{om}(c_{X}^{\ast}-,c_{X}^{\ast}-)
\]
and $c_{X}^{\ast}Rf^{!}\to Rf^{\lozenge!}c_{Y}^{\ast}$ for $f$ separated
and perfectly of finite type.

Now, let $\mathrm{PSch}_{k}^{\mathrm{ft}}$ be the full subcategory
of schemes separated and perfectly of finite type over $k$. Fix $X\in\mathrm{PSch}_{k}^{\mathrm{ft}}$
with structure map $f:X\to\Spec k$, and let $c:C\to X\times_{k}X$
be a correspondence in $\mathrm{PSch}_{k}^{\mathrm{ft}}$ such that
$c_{1}$ and $f\circ c'$ are perfectly proper. Let $\mathcal{F}\in D_{\mathrm{\acute{e}t}}(X,\Lambda)$
be an $f$-ULA object equipped with a cohomological correspondence
$u$ lying over $c$, so we get local terms $\mathrm{loc}_{\beta}(u,\mathcal{F})\in H^{0}(\mathrm{Spec}k,\Lambda)=\Lambda$
by the schematic version of the recipe recalled above.

On the other hand, applying $(-)^{\lozenge}$ and using commutation
with finite limits, we get a correspondence $c^{\lozenge}:C^{\lozenge}\to X^{\lozenge}\times_{\mathrm{Spd}\,k}X^{\lozenge}$
of v-sheaves over $S=\mathrm{Spd}\,k$ with $\mathrm{Fix}(c)^{\lozenge}=\mathrm{Fix}(c^{\lozenge})$,
satisfying all of our assumptions from above. Moreover, $u$ naturally
induces a cohomological correspondence $u^{\lozenge}$ on $c_{X}^{\ast}\mathcal{F}$
lying over $c^{\lozenge}$, by taking the image of $u$ under the
natural map
\begin{align*}
\mathrm{Hom}(c_{1}^{\ast}\mathcal{F},Rc_{2}^{!}\mathcal{F}) & \to\mathrm{Hom}(c_{C}^{\ast}c_{1}^{\ast}\mathcal{F},c_{C}^{\ast}Rc_{2}^{!}\mathcal{F})\cong\mathrm{Hom}(c_{1}^{\lozenge\ast}c_{X}^{\ast}\mathcal{F},c_{C}^{\ast}Rc_{2}^{!}\mathcal{F})\\
 & \;\;\;\;\to\mathrm{Hom}(c_{1}^{\lozenge\ast}c_{X}^{\ast}\mathcal{F},Rc_{2}^{\lozenge!}c_{X}^{\ast}\mathcal{F}).
\end{align*}

\begin{pro}\label{PropAlgebraizingLocalTerms}
Maintain the previous setup and notation. Then $c_{X}^{\ast}\mathcal{F}$
is $f^{\lozenge}$-ULA, and for any open-closed $\beta\subset\mathrm{Fix}(c)$,
we have an equality
\[
\mathrm{loc}_{\beta}(u,\mathcal{F})=\mathrm{loc}_{\beta^{\lozenge}}(u^{\lozenge},c_{X}^{\ast}\mathcal{F})
\]
of local terms.
\end{pro}

\begin{proof}
This is formally identical to the proof of Proposition \ref{PropLocalTermsBaseChange}, using
the fact that $(-)^{\lozenge}$ induces a symmetric monoidal functor
on the appropriate categories of cohomological correspondences. 
\end{proof}

\subsection{Independence of $g$}
\label{SubsectionIndependence}
In this section we prove Proposition \ref{PropIndependenceOfg}. In this section only, we
set $S=\Spd C$.

Fix $V$ as in the proposition. Decomposing $V$ into isotypic summands
for the action of $Z(\hat{G})^{\circ}$, we can assume that $\mathcal{S}_{V}$
is supported on a single connected component of $\mathrm{Gr}_{G}$.
We can then pick some $\mu$ such that $\mathcal{S}_{V}$ is supported
on the Schubert cell $\mathrm{Gr}_{G,\leq\mu}$. Choose some large
$m$ such that the $L^{+}G$ action on $\mathrm{Gr}_{G,\leq\mu}$
factors over the truncated loop group $L_{m}^{+}G$. The sheaf $\mathcal{S}_{V}$
is naturally the pullback of a sheaf again denoted $\mathcal{S}_{V}$
on the quotient stack $X=[\mathrm{Gr}_{G,\leq\mu}/L_{m}^{+}G]$, so
we can consider the characteristic class $\mathrm{cc}_{X/S}(\mathcal{S}_{V})$.

To analyze this class, we need to understand the inertia stack of
$X$. For this, we need some notation. Let $L_{m}^{+}G_{\mathrm{sr}}$
be the preimage of the strongly regular semisimple locus $G_{\mathrm{sr}}\subset G$
under the theta map $L_{m}^{+}G\to G$. Pick any maximal torus $T\subset G$
with Weyl group $W$, and set $L_{m}^{+}T_{\mathrm{sr}}=L_{m}^{+}T\cap L_{m}^{+}G_{\mathrm{sr}}$.
\begin{pro}
\label{PropInSX}
1. The open substack 
\[
\mathrm{In}_{S}([S/L_{m}^{+}G])_{\mathrm{sr}}=[L_{m}^{+}G_{\mathrm{sr}}\stslash L_{m}^{+}G]\subset\mathrm{In}_{S}([S/L_{m}^{+}G])
\]
is canonically identified with $[L_{m}^{+}T_{\mathrm{sr}}/(W\ltimes L_{m}^{+}T)]$
via the natural map.

2. The open substack 
\[
\mathrm{In}_{S}(X)_{\mathrm{sr}}=\mathrm{In}_{S}(X)\times_{\mathrm{In}_{S}([S/L^{+}G])}[L_{m}^{+}G_{\mathrm{sr}}\stslash L_{m}^{+}G]\subset\mathrm{In}_{S}(X)
\]
is canonically identified with $X_{\ast}(T)_{\leq\mu}\times^{W}[L_{m}^{+}T_{\mathrm{sr}}\stslash L_{m}^{+}T]$,
such that the natural map $\mathrm{In}_{S}(X)_{\mathrm{sr}}\to\mathrm{In}_{S}([S/L_{m}^{+}G])_{\mathrm{sr}}$
coincides via the identification in part 1. with the evident projection
onto $[L_{m}^{+}T_{\mathrm{sr}}/(W\ltimes L_{m}^{+}T)]$.
\end{pro}

\begin{proof}  The idea behind (1) is that any $g\in L^+G_{\sr}$ is conjugate to an element of $L^+T_{\sr}$, which is well-defined up to the action of the normalizer of this group, which is $W\ltimes L^+T$.

For (2), we observe that an object of $\In_S(X)_{\sr}$ is a pair $(x,g)$, where $g\in L_m^+G_{\sr}$ fixes $x\in \Gr_{G,\leq\mu}$;  the automorphisms of this object are $L^+_mG$.  The $g$ can be conjugated to lie in $L^+T_{\sr}$, and then by Proposition \ref{PropositionFixedPointsOnGr},  the $x$ can be identified with an element of $X_*(T)_{\leq\mu}$, which is well-defined up to an element of $W$.  
\end{proof}
\begin{cor}
There is a natural isomorphism
\[
H^{0}(\mathrm{In}_{S}(X)_{\mathrm{sr}},K_{\mathrm{In}_{S}(X)_{\mathrm{sr}}/S})\cong C(X_{\ast}(T)_{\leq\mu},\Lambda)^{W}
\]
which sends $\cc_{X/S}(\mc{S}_V)$ to the function sending $\lambda\in X_*(T)_{\leq \mu}$ to $\loc_{x_\lambda}(g,\mc{S}_V)$, where $x_\lambda\in \Gr_{G,\leq \mu}$ is the $T$-fixed point corresponding to $\lambda$, and $g\in L_m^+T_{\sr}$.  In particular, $\loc_{x_\lambda}(g,\mc{S}_V)$ does not depend on the choice of $g\in L_m^+T_{\sr}$.  
\end{cor}

\begin{proof}  Combine Theorem \ref{ThmCharClassOnXModG} with the description of $\In_S(X)_{\sr}$ from Proposition \ref{PropInSX}.
\end{proof}

\subsection{Degeneration to characteristic $p$}

In this section we prove Propositions \ref{PropFixedPointsBD} and \ref{PropDegeneratingLocalTerms}. 

\begin{proof}[Proof of Proposition \ref{PropFixedPointsBD}] The isomorphism $\mathrm{Gr}_{\mathcal{T}}\cong {X_{\ast}(\mathcal{T})}_S$
is \cite[Proposition 21.3.1]{ScholzeLectures}.   There is an evident map $f: \mathrm{Gr}_{\mathcal{T}}\to\mathrm{Gr}_{\mathcal{\mathcal{G}}}^{g}$, and it remains to see that $f$ is an isomorphism. For this, we first note that $f$ is a closed immersion. This follows from the observation the source and target of $f$ are both closed subfunctors of $\mathrm{Gr}_{\mathcal{G}}$. For the source, this follows from \cite[Proposition 20.3.7]{ScholzeLectures}, while for the target this follows from the fact that $\mathrm{Gr}_{\mathcal{G}} \to S$ is separated.

Since $f$ is a closed immersion, it is both qcqs and specializing.   By \cite[Lemma 12.5]{ScholzeEtaleCohomology}
 it is enough to check that $f$ is a bijection on rank one geometric points. This can be checked separately
on the generic and special fibers.   Both cases are handled by Proposition \ref{PropositionFixedPointsOnGr}.
\end{proof}

\begin{proof}[Proof of Proposition \ref{PropDegeneratingLocalTerms}] By two applications of (the connected case of) Proposition \ref{PropLocalTermsBaseChange}, applied to the maps $\eta \to S$ and $s \to S$, we get equalities
\[
\mathrm{loc}_{\beta_{\eta}}(g,j^{\ast}\mathcal{S}_{V})= \mathrm{loc}_{\beta}(g,\mathcal{S}_{V})
= \mathrm{loc}_{\beta_{s}}(g,i^{\ast}\mathcal{S}_{V}),
\]
and the result follows.
\end{proof}

\subsection{Local terms on the Witt vector affine Grassmannian}
\begin{proof}[Proof of Proposition \ref{PropLocalTermsWitt}]
Fix $g$ and $V$ as in the statement, and let $\mathcal{T}\subset\mathcal{G}$
be the connected centralizer of $g$. For every $\nu\in X_{\ast}(\mathcal{T})$,
let $S_{\nu}\subset\mathrm{Gr}_{\mathcal{G}}^{W}$ be the associated
semi-infinite orbit, with closure $\overline{S_{\nu}}=\cup_{\nu'\leq\nu}S_{\nu}$. 

Let $X\subset\mathrm{Gr}_{\mathcal{G}}^{W}$ be a finite union of
closed Schubert cells containing the support of $\mathcal{S}_{V}$,
so $X$ is a perfectly projective $k$-scheme by the results in \cite{BhattScholzeWitt}.
Write $X_{\nu}=X\cap S_{\nu}$, $X_{\leq\nu}=X\cap\overline{S_{\nu}}$,
and $\partial X_{\leq\nu}=X_{\leq\nu}\smallsetminus X_{\nu}$. Note
that all of these spaces are stable under $g$, and in fact under
$\mathcal{T}$. Note also that each $X_{\nu}$ contains a \emph{unique}
$g$-fixed point $x_{\nu}$.

\begin{pro} \label{PropEulerCharOfSemiInfiniteOrbit} The compactly supported Euler characteristic of $\mc{S}_V$ on $X_{\nu}$ is
\[ \chi_c(X_{\nu},\mc{S}_V) = (-1)^{\left\langle 2\rho,\nu\right\rangle }\rank V[\nu]. \]
\end{pro}

\begin{proof} This is a consequence of the integral-coefficients version of the geometric Satake equivalence for the Witt vector affine Grassmannian given in \cite{YuIntegralGeometricSatake}.  There it is shown (Proposition 4.2) that $H^d_c(X_\nu, \mc{S}_V)$ is zero unless $d=\class{2\rho,\nu}$, and in that degree it corresponds exactly to the $\nu$-weight functor in the Satake category.
\end{proof}

Any $t\in\mathcal{T}$ must act trivially on $H^d_c(X_\nu,\mc{S}_V)$, so $\chi_{c}(U,\mathcal{S}_{V})$
coincides with the trace of $g$ on $R\Gamma_{c}(X_\nu,\mathcal{S}_{V})$.  The same is true for $X_{\leq \nu}$ and $\partial X_{\leq \nu}$.  We compute:
\begin{align*}
(-1)^{\left\langle 2\rho,\nu\right\rangle }\rank V[\nu] & =\chi_{c}(X_{\nu},\mathcal{S}_{V})\\
 & =\chi_{c}(X_{\leq\nu},\mathcal{S}_{V})-\chi_{c}(\partial X_{\leq\nu},\mathcal{S}_{V})\\
 & =\mathrm{tr}(g|R\Gamma_{c}(X_{\leq\nu},\mathcal{S}_{V}))-\mathrm{tr}(g|R\Gamma_{c}(\partial X_{\leq\nu},\mathcal{S}_{V}))\\
 & =\sum_{\nu'\leq\nu}\mathrm{loc}_{x_{\nu'}}(g,\mathcal{S}_{V}|_{X_{\leq\nu}})-\sum_{\nu'\leq\nu,\nu'\neq\nu}\mathrm{loc}_{x_{\nu'}}(g,\mathcal{S}_{V}|_{\partial X_{\leq\nu}})\\
 & =\sum_{\nu'\leq\nu}\mathrm{loc}_{x_{\nu'}}(g,\mathcal{S}_{V})-\sum_{\nu'\leq\nu,\nu'\neq\nu}\mathrm{loc}_{x_{\nu'}}(g,\mathcal{S}_{V})\\
 & =\mathrm{loc}_{x_{\nu}}(g,\mathcal{S}_{V}).
\end{align*}
The penultimate equality is the key technical fact, and follows from
Proposition \ref{PropDeperfection} below together with the assumptions on $g$.
\end{proof}

\begin{pro}\label{PropDeperfection}
Let $k/\mathbf{F}_{p}$ be an algebraically closed field, and let $X$ be a perfectly
finite type $k$-scheme with an automorphism $g:X\to X$ of finite
prime-to-p order. Let $A\in D_{c}^{b}(X,\mathbf{Z}_{\ell})$
be an object equipped with a morphism $u:g^{\ast}A\to A$.
Then for every isolated $g$-fixed point $x$, the true local term
$\mathrm{loc}_{x}(g,A)$ equals the naive local term $\mathrm{tr}(g|A_{x})$. 

In particular, if $Z\subset X$ is a g-stable closed subscheme, then
$\mathrm{loc}_{x}(g,A)=\mathrm{loc}_{x}(g,A|_{Z})$.
\end{pro}

\begin{proof}
With the word ``perfectly'' deleted,  this is a recent result of
Varshavsky \cite{VarshavskyLocalTerms} (combine Theorem 4.10(b) and Corollary 5.4(b)).  We will reduce to Varshavsky's result by deperfecting. 

Precisely, since $\mathrm{loc}_{x}(g,A)$ is insensitive
to replacing $X,A$ by $U,A|U$ for $U\subset X$
any $g$-invariant open neighborhood of $x$, we can assume that $X$
is affine, so $X=\Spec R$ with $R$ perfectly of finite type.
Let $R_{0}\subset R$ be a finite type $k$-algebra with $R_{0}^{\mathrm{perf}}=R$,
and let $R_1\subset R$ be the $k$-algebra generated by $g^{i}R_{0}$
for all $1\leq i\leq\ord(g)$. Then $R_1\subset R$ is a finite-type
$k$-algebra stable under $g$, with $R_1^{\mathrm{perf}}=A$, so $X_{1}=\Spec R_1$
is a deperfection of $X$ equipped with an automorphism $g_{1}$ deperfecting
$g$; since $X\to X_{1}$ is a homeomorphism, there is a unique $g_{1}$-fixed
point $x_{1}$ under $x$. Next, $g^{\ast}A \to A$
deperfects uniquely to a complex $A_{1}$ on $X_1$ equipped with
a map $g_{1}^{\ast}A_{1}\to A_{1}$, using the
equivalence of categories $D(X_{\mathrm{et}},\Lambda)\isom D(X_{1,\et},\Lambda)$.
Finally, we compute that
\[
\mathrm{loc}_{x}(g,A)=\mathrm{loc}_{x_{1}}(g_{1},A_{1})=\mathrm{tr}(g_{1}|A_{1,x_{1}})=\mathrm{tr}(g|A_{x})
\]
where the first equality is formal nonsense (the six functors on $k$-varieties
and on perfectly finite type $k$-schemes are compatible under perfection),
the second equality is Varshavsky's theorem, and the third equality
is trivial.
\end{proof}

%!TEX root = FixedPoints2022.tex
\section{Application to the Hecke stacks}
\label{SectionProofOfMainTheorem}

In this final chapter we prove Theorem \ref{TheoremMain}, by applying the technology of the Lefschetz-Verdier trace formula to the Hecke stacks over $\Bun_G$.  

\subsection{$\Bun_G$, the local and global Hecke stacks, and their relation to shtuka spaces}
Let $F/\Q_p$ be a finite extension, and let $G/F$ be a connected reductive group.   Let $k$ be an algebraically closed perfectoid field containing the residue field of $F$.

For an algebraically closed perfectoid field $C/k$, there is a bijection \cite{FarguesGBundles}
\[ b \mapsto \E^b \]
between Kottwitz' set $B(G)$ and isomorphism classes of $G$-bundles on the Fargues-Fontaine curve $X_C$.  Therefore the moduli stack of $G$-bundles is some geometric version of the set $B(G)$.

\begin{dfn}[{\cite[Definition III.0.1 and Theorem III.0.2]{FarguesScholze}}] Let $\Bun_G$ be the v-stack which assigns to a perfectoid space $S/k$ the groupoid of $G$-bundles on $X_S$.  Given a class $b\in B(G)$, let $i_b\from \Bun_G^b\to \Bun_G$ be the locally closed substack classifying $G$-bundles which are isomorphic to $\E^b$ at every geometric point.
\end{dfn}

Then $\Bun_G$ is a cohomologically smooth Artin v-stack over $\Spd k$ \cite[Theorem I.4.1(vii)]{FarguesScholze}.  Central to its study are the substacks $\Bun_G^b$.  For each $b\in B(G)$ we have an isomorphism $\Bun_G^b\isom [\Spd k/\tilde{G}_b]$, where 
\[\tilde{G}_b=\underline{\Aut}\; \E_b\]
is a group diamond over $\Spd k$.   This fits in an exact sequence of group diamonds over $\Spd k$:
\[ 0 \to \tilde{G}_b^\circ \to \tilde{G}_b \to G_b(F)_{\Spd k} \to 0\]
Here, the neutral component $\tilde{G}_b^\circ\subset \tilde{G}_b$ is a cohomologically smooth group diamond over $\Spd k$, and $G_b$ is the automorphism group of the isocrystal $b$.  The group $G_b$ is an inner form of a Levi subgroup of the quasisplit inner form of $G$.   If $b$ is basic, then $i_b$ is an open immersion, and $\tilde{G}_b=G_b(F)_{\Spd k}$.

We next recall the Hecke correspondence on $\Bun_G$ and its relation to the local shtuka spaces $\Sht_{G,b,\mu}$.   Since our main result on the cohomology of local Shimura varieties does not concern the action of a Weil group, all objects in this discussion will live over the base $S=\Spd C$, where $C$ is an algebraically closed perfectoid field containing $F$, whose residue field contains $k$. In particular we have $\Bun_{G,C}=\Bun_G \times_{\Spd k} S$.   If $T$ is a perfectoid space over $C$, the Fargues-Fontaine curve $X_T$ comes equipped with a degree 1 Cartier divisor $D_T$, corresponding to the untilt $T$ of $T^\flat$.   

We introduce now a diagram of v-stacks over $\Spd C$ containing both local and global Hecke correspondences:
\begin{equation}
\label{EqHeckeCorrespondenceDiagram}
\xymatrix{
\Bun_{G,C} \ar[d] & \Hecke_{G,C} \ar[r]^{h_1} \ar[d] \ar[l]_{h_2} & \Bun_{G,C} \ar[d] \\
\Bun_{G,C}^{\loc} & \Hecke^{\loc}_{G,C} \ar[r]_{h_1^{\loc}} \ar[l]^{h_2^{\loc}} & \Bun_{G,C}^{\loc}
}
\end{equation}
We explain below the objects and morphisms appearing in \eqref{EqHeckeCorrespondenceDiagram}.   Let $T=\Spa(R,R^+)$ be an affinoid perfectoid space over $\Spa C$.
\begin{itemize}
\item The $T$-points of the stack $\Hecke_{G,C}$ classify triples $(\mc{E}_1,\mc{E}_2,f)$, where $\mc{E}_1$ and $\mc{E}_2$ are $G$-bundles on $X_T$,  and 
\[ f\from \mc{E}_1\vert_{X_T\backslash D_T} \isom \mc{E}_2\vert_{X_T\backslash D_T} \]
is an isomorphism which is meromorphic along $D_T$.
\item The morphism $h_i$ sends a triple as above to $\mc{E}_i$ for $i=1,2$.
\item The $T$-points of $\Bun_{G,C}^{\loc}$ classify $G$-bundles on $\Spec B_{\dR}^+(R)$, this being the completion of $X_T$ along $D_T$.   Such $G$-bundles are v-locally trivial on $T$, so that we have an isomorphism 
\[ \Bun_{G,C}^{\loc} \isom [\Spd C / L^+G], \]
where $L^+G = G(B_{\dR}^+)$ is the positive loop group.
\item The $T$-points of $\Hecke^{\loc}_{G,C}$ classify triples $(\mc{E}_1,\mc{E}_2,f)$, where $\mc{E}_1$ and $\mc{E}_2$ are $G$-bundles on $\Spec B_{\dR}^+(R)$, and $f$ is an isomorphism between their restrictions to $\Spec B_{\dR}(R)$,  meromorphic along $D_T$.   We have an isomorphism
\[ \Hecke_{G,C}^{\loc} \isom [L^+G \backslash LG / L^+G ],\]
where $LG=G(B_{\dR})$ is the full loop group.   Put another way, we have the $B_{\dR}$-affine Grassmannian $\Gr_{G,C}=LG/L^+G$, and then $\Hecke_{G,C}^{\loc}=[L^+G \backslash \Gr_{G,C}]$.  
\item The morphism $h_i^{\loc}$ sends such a triple to $\mc{E}_i$ for $i=1,2$.
\item The vertical maps send an object to its completion along $D_T$ in the evident manner.
\end{itemize}
The squares in \eqref{EqHeckeCorrespondenceDiagram} are cartesian, by Beauville-Laszlo gluing.

It is a basic fact that $\Bun_{G,C}$ is a decent v-stack, and the structure map $\Bun_{G,C} \to S=\Spd C$ is fine. With some care, it is possible to ``truncate'' some of the other objects appearing in \eqref{EqHeckeCorrespondenceDiagram} to obtain decent $S$-v-stacks with fine structure maps to $S$.  In particular, let $\mu$ be a dominant cocharacter of $G$, and let $\Hecke_{G,\leq\mu,C}$ be the substack of $\Hecke_{G,\mu}$ consisting of triples $(\mc{E}_1,\mc{E}_2,f)$ where the meromorphy of $f$ is fiberwise bounded by $\mu$.  Then $\Hecke_{G,\leq\mu,C}$ is decent and the maps to $\Bun_{G,C}$ induced by restricting $h_1$ and $h_2$ are fine. We may define $\Hecke_{G,\leq\mu,C}^{\loc}$ analogously; this is isomorphic to $[\Gr_{G,\leq \mu,C}/L^+G]$, where $\Gr_{G,\leq \mu,C}$ is the bounded Grassmannian.  This is not quite a decent v-stack. However, if we instead form the quotient $[\Gr_{G,\leq \mu,C}/L^{+}_{m}G]$ for some sufficiently large truncation as in Theorem \ref{ThmCharClassOfLocalHeckeStack}, we do obtain a decent $S$-v-stack with fine structure map. This is sufficient for our purposes. 
% By \cite[Proposition 20.2.3]{ScholzeWeinstein}, $\Gr_{\leq \mu,C}$ is a spatial diamond, proper over $\Spd C$.  
%
%The action of $L^+G$ on $\Gr_{G,\leq \mu,C}$ factors through a quotient $L^+_mG$ for some integer $m\gg 0$;  the latter object is a cohomologically smooth group diamond.  We have a cartesian diagram:
%\begin{equation}
%\label{EqHeckeCorrespondenceDiagramTruncated}
%\xymatrix{
%\Hecke_{G,\leq \mu,C} \ar[r]^{h_1} \ar[d] & \Bun_{G,C} \ar[d] \\
%\Hecke^{\loc}_{G,\leq \mu,m,C} \ar[r]_{h_1^{\loc}}& \Bun_{G,C,m}^{\loc}
%}
%\end{equation}
%which shows that $\Hecke_{G,\leq\mu,C}$ is an $S$-Artin v-stack.  

%Now let $\Lambda$ be a ring which is $n$-torsion for some $n$ prime to $p$.
%The {\em Satake category}
%\[ \Sat_G(\Lambda) \subset D_{\et}(\Hecke^{\loc}_{G,C}, \Lambda) \]
%is the subcategory of objects which are perverse, $\Lambda$-flat, and ULA over $\Spd C$ \cite[Definition I.6.2]{FarguesScholze}.   It is a symmetric monoidal category under the convolution product.
%
%\begin{thm}[{\cite[Theorem I.6.3]{FarguesScholze}}] 
%\label{ThmSatakeEquivalence} There is an equivalence of symmetric monoidal categories:
%\begin{eqnarray*}
%\Rep_{\hat{G}}(\Lambda) &\isomto& \Sat_G(\Lambda) \\
%V &\mapsto& \mc{S}_V 
%\end{eqnarray*}
%where $\hat{G}$ is the Langlands dual group, (considered over $\Lambda$),  and $\Rep_{\hat{G}}(\Lambda)$ is the category of representations of $\hat{G}$ on finite projective $\Lambda$-modules.
%\end{thm}

Now let $b\in B(G,\mu)$ be basic.   We explain the relation between Hecke stacks and local shtuka spaces. It will be helpful to refer to the commutative diagram of stacks
\begin{equation}
\label{BigHeckeDiagram}
\xymatrix{
\Hecke^{b,1}_{G,\leq \mu,C} \ar@/_2.5pc/[dd]_{h_1^{b,1}} \ar@/^2pc/[rr]^{h_2^{b,1}} \ar[r]^{i_b''} \ar[d]_{i_1''} &
\Hecke^{\ast,1}_{G,\leq \mu,C} \ar[r]^{h_2^{\ast,1}} \ar[d]_{i_1'} &
\Bun_{G,C}^1 \ar[d]_{i_1} \\
\Hecke^{b,\ast}_{G,\leq\mu,C} \ar[r]_{i_b'} \ar[d]_{h_1^{b,\ast}} &
\Hecke_{G,\leq\mu,C} \ar[d]_{h_1} \ar[r]_{h_2} & \Bun_{G,C} \\
\Bun_{G,C}^b \ar[r]_{i_b}  & \Bun_{G,C}
}
\end{equation}
in which all squares are cartesian, the morphisms labeled with $i$ are open immersions, and the morphisms $h_1$ and $h_2$ are proper. 

The top row of \eqref{BigHeckeDiagram} can be described via the diagram:
\[
\xymatrix{
[\Gr_{G,\leq\mu,C}^{1,\adm}/\underline{G(F)}] \ar[r] \ar[d]_{\isom} &
[\Gr_{G,\leq\mu}^1/\underline{G(F)}] \ar[r] \ar[d]_{\isom} &
[\point/\underline{G(F)}] \ar[d]_{\isom}\\
\Hecke_{\leq\mu,C}^{b,1} \ar[r]_{i_b''} &
\Hecke_{\leq\mu,C}^{\ast,1} \ar[r]_{h_2^{\ast,1}} &
\Bun_{G,C}^1.
}
\]
Explanation: $\Gr_{G,\leq\mu,C}^1$ assigns to $T=\Spa(R,R^+)$ the set of pairs $(\mc{E},f)$, where $\mc{E}$ is a $G$-bundle on $X_S$, and $f\from \mc{E}^1\vert_{X_T\backslash D_T}\isom \mc{E}\vert_{X_T\backslash D_T}$ is an isomorphism,  which is bounded by $\mu$ along $D_T$.  
The bundle $\mc{E}^1$ can be canonically trivialized over $\Spa B_{\dR}^+(R)$, and in so doing we obtain an isomorphism $\Gr^1_{G,\leq\mu,C}\isom \Gr_{G,\leq\mu,C}$.   Within $\Gr_{G,\leq\mu}^1$ we have the open locus $\Gr_{G,\leq\mu}^{1,\adm}$, consisting of those pairs $(\mc{E},\gamma)$, where $\mc{E}$ is everywhere isomorphic to $\mc{E}^b$.

Similarly the leftmost column of \eqref{BigHeckeDiagram} can be described via the diagram:
\[
\xymatrix{
[\Gr_{G,\leq\mu,C}^{b,\adm}/\underline{G_b(F)}] \ar[r]^{\isom} \ar[d]
& \Hecke_{G,\leq\mu,C}^{b,1} \ar[d]^{i_1''} \\
[\Gr_{G,\leq\mu,C}^b/\underline{J_b(F)}] \ar[r]^{\isom} \ar[d] &
\Hecke_{G,\leq\mu,C}^{b,\ast} \ar[d]^{h_1^{b,\ast}} \\
[\Spd C/\underline{G_b(F)}] \ar[r]_{\isom} &
\Bun_{G,C}^b.
}
\]
Explanation: $\Gr_{G,\leq\mu,C}^b$ assigns to $T=\Spa(R,R^+)$ the set of pairs $(\mc{E},f)$, where $\mc{E}$ is a $G$-bundle, and $f\from \mc{E}\vert_{X_T\backslash D_T}\isom \mc{E}^b\vert_{X_T\backslash D_T}$ is an isomorphism, which is bounded by $\mu$ along $D_T$.  We have an isomorphism $\Gr_{G,\leq\mu,C}^b\isom \Gr_{G,\leq -\mu,C}$.  Within $\Gr_{G,\leq\mu,C}^b$, we have the open {\em admissible locus} $\Gr_{G,\leq\mu,C}^{b,\adm}$ consisting of pairs $(\mc{E},f)$, where $\mc{E}$ is everywhere isomorphic to $\E^1$.

The moduli space of local shtukas $\Sht_{G,b,\mu,C}$ appears as the fiber product:
\begin{equation}
\xymatrix{
\Sht_{G,b,\mu,C} \ar[rr] \ar[d] &&  \Spd C \ar[d] \\
\Hecke^{b,1}_{G,\leq\mu,C} \ar[rr]_{h_1^{b,1}\times h_2^{b,1}} && \Bun_{G,C}^b\times \Bun_{G,C}^1  
}
\end{equation}
where the right vertical morphism corresponds to $\E^b\times\E^1$.   This is evident from the definition of $\Sht_{G,b,\mu}$:  its $S$-points are morphisms $f\from \E^1_{X_S\backslash D_S}\isom \E^b_{X_S\backslash D_S}$ which are bounded by $\mu$ on $D_S$.  %Similarly, for any compact open subgroup $K\subset G(F)$, the finite-level shtuka moduli space $\Sht_{(G,b,\mu),C}/\underline{K}$ appears as a pullback by the map $j_K\from [\Spd C/\underline{K}] \to [\Spd C/\underline{G(F)}]$.

We also have the period morphisms:
\begin{equation}
\label{EqPeriodMorphisms}
\xymatrix{
& \Sht_{G,b,\mu,C} \ar[dl]_{\pi_1} \ar[dr]^{\pi_2} & \\
\Gr^b_{G,\leq \mu,C} && \Gr^1_{G,\leq \mu,C} 
}
\end{equation}
The morphism $\pi_1$ is a $G_b(F)_S$-equivariant $G(F)_S$-torsor over the admissible locus $\Gr^{b,\adm}_{G,\leq\mu,C}$.  Similarly, $\pi_2$ is a $G(F)_S$-equivariant $G_b(F)_S$-torsor over the admissible locus $\Gr^{1,\adm}_{G,\leq\mu,C}$.  

\subsection{The inertia stack of the Hecke stack;  admissibility of elliptic fixed points}

We continue to put $S=\Spd C$.  Here we investigate the inertia stack $\In_S(\Hecke_{G,\leq \mu,C})$, or at least the part of it lying over the strongly regular locus in $\In_S([S/G(F)_S])\isom [G(F)_S\stslash G(F)_S]$.  

It will help to introduce some notation.  Suppose $\mc{E}$ is a $G$-bundle on $X_C$ equipped with a trivialization over the completion at $\infty=D_C$.  Let $T\subset G$ be a maximal torus.  We have seen in Proposition \ref{PropositionFixedPointsOnGr} that there is a bijection $\lambda\mapsto L_\lambda$ between $X_*(T)$ and the set of $T$-fixed points of $\Gr_{G}$.  Given $\lambda\in X_*(T)$, we let $\mc{E}[\lambda]$ be the modification of $\mc{E}$ corresponding to $L_\lambda$. 

\begin{lem}[{\cite[Lemma 3.5.5]{CaraianiScholze}, see also \cite[\S2.2]{ChenFarguesShen}, but note that we use the opposite convention concerning Schubert cells}] \label{LemmaBehaviorOfKappa} Let $\E$ be a $G$-bundle on $X_{C}$ equipped with a trivialization at $\infty$, let $T\subset G$ be a maximal torus, let $\lambda\in X_*(T)$ be a cocharacter, and let $\hat{\lambda}\in X^*(\hat{T})$ be the corresponding character.  In the group $X^*(Z(\hat{G})^{\Gamma})$ we have
\[\kappa(\E[\lambda])=\kappa(\E)+\hat{\lambda}\vert_{Z(\hat{G})^\Gamma}.\]
\end{lem}

\begin{pro} 
\label{PropositionDescriptionOfFixedPointsInM}
Suppose a pair $(g,g')\in G(F)_{\sr}\times G_b(F)_{\sr}$ fixes a point $x\in \Sht_{G,b, \mu}(C)$.  Let $T=\Cent(g,G)$ and $T'=\Cent(g',G_b)$.  
Then $\pi_1(x)\in \Gr_{G,\leq-\mu}^{g'}$ and $\pi_2(x)\in \Gr_{G,\leq\mu}^g$ correspond to cocharacters $\lambda'\in X_*(T')_{\leq -\mu}$ and $\lambda\in X_*(T)_{\leq \mu}$, respectively.  

There exists $y\in G(\breve{F})$ such that $\ad y$ is an $F$-rational isomorphism $T\to T'$, which carries $g$ to $g'$ and $\lambda$ onto $-\lambda'$.  The invariant $\inv[b](g,g')\in B(T)\isom X_*(T)_{\Gamma}$ agrees with the image of $\lambda$ under $X_*(T)\to X_*(T)_{\Gamma}$.   Therefore $(g,g',\lambda)$ lies in $\Rel_{b,\mu}$.   
\end{pro}

\begin{proof}  The point $x$ corresponds to an isomorphism $\gamma\from \E^1[\lambda]\to \E^b$,  and also to an isomorphism $\gamma'\from \E^1\to \E^b[\lambda']$.  Each of these interlaces the action of $g$ with $g'$, and furthermore $\gamma=\gamma'$ away from $\infty$.    Trivializing $\E^1$ and $\E^b$ away from $\infty$, we see that $g$ and $g'$ become conjugate over the ring $B_e=H^0(X_C\backslash \set{\infty},\OO_{X_C})$,  which implies they are conjugate over $\overline{F}$ and (by Lemma \ref{LemmaSteinberg}) they are even conjugate over $\breve{F}$.  Let $y\in G(\breve{F})$ be an element such that $(\ad y)(g) = g'$.  Then $\ad y$ is a $\breve{F}$-rational isomorphism $T\to T'$ which carries $g$ onto $g'$.  In fact, since there is only one such isomorphism, we can conclude that $\ad y$ is an $F$-rational isomorphism $T\to T'$.    Let $\lambda_0=(\ad y^{-1})(\lambda')\in X_*(T)$.  

Let $b_0=y^{-1}by^{\sigma}$.   Then (cf. Definition \ref{dfn:inv}) we have $b_0\in T(F)$.  The element $y$ induces isomorphisms $y\from \E^{b_0}\to \E^b$ and $y\from \E^{b_0}[\lambda_0]\to \E^b[\lambda']$.  Then the isomorphism $y^{-1}\gamma'\from \E^1\to \E^{b_0}[\lambda_0]$ descends to an isomorphism of $T$-bundles;  comparing this with the isomorphism $\gamma y^{-1}\from \E^1[\lambda]\to \E^{b_0}$ shows that $\lambda_0=-\lambda$.   In light of the isomorphism of $T$-bundles $\E^1[\lambda]\isom \E^{b_0}$,  Lemma \ref{LemmaBehaviorOfKappa} implies that the identity $\kappa(\E^{b_0})=\lambda$ holds in $B(T)$.  But also $\kappa(\E^{b_0})$ is the class of $[b_0]$ in $B(T)$, which is $\inv[b](g,g')$ by definition.
\end{proof}

Proposition \ref{PropositionDescriptionOfFixedPointsInM} shows that if $(g,g')\in G(F)_{\sr}\times G_b(F)_{\sr}$ fixes a point of $\Sht_{G,b,\mu}$, then $g$ and $g'$ are related.  However the converse may fail: if a pair of related strongly regular elements $(g,g')$ is given, it is not necessarily true that $(g,g')$ fixes a point of $\Sht_{G,b, \mu}$. Indeed, a necessary condition for this is that the action of $g'$ on $\Gr_{G,\leq -\mu}^b$ has a fixed point in the admissible locus, and this is not automatic.

This converse result is always true, however, if $g$ (or equivalently, $g'$) is an elliptic element.

\begin{thm}
\label{TheoremEllipticFixedPointsAreAdmissible}  Let $g\in G(F)_{\elli}$.  Then the fixed points of $g$ acting on $\Gr_{G,\leq\mu}^1$ lie in the admissible locus $\Gr_{G,\leq\mu}^{1,\adm}$.   Similarly, if $g'\in G_b(F)_{\elli}$, then the fixed points of $g'$ acting on $\Gr_{G,\leq -\mu}^b$ lie in the admissible locus $\Gr_{G,\leq -\mu}^{b,\adm}$.  
\end{thm}

\begin{proof}  We prove the first statement; the second is similar.  Let $g\in G(F)_{\elli}$, and let $T=\Cent(g,G)$ be the elliptic maximal torus containing $g$. Suppose we are given a $g$-fixed point $x\in \Gr_{G,\leq\mu}(C)$.  Then $x$ corresponds to a cocharacter $\lambda\in X_*(T)$, which in turn corresponds to a modification $\E^1[\lambda]$ of the trivial $G$-bundle $\E^1$.  We wish to show that $\E^1[\lambda]\isom \E^b$.  First we will show that it is semistable.

Let $b'\in G(\breve{F})$ be an element whose class in $B(G)$ corresponds to the isomorphism class of $\E^1[\lambda]$.  We wish to show that $b'$ is basic.  We have the algebraic group $G_{b'}/F$, which is a priori an inner form of a Levi subgroup $M^*$ of $G^*$, where $G^*$ is the quasi-split inner form of $G$.  Showing that $b'$ is basic is equivalent to showing that $M^*=G^*$.

We have an isomorphism $\gamma\from \E^1[\lambda]\isom \E^{b'}$.   The action of $g\in T(F)$ on $\E^1$ extends to an action on $\E^1[\lambda]$, which can be transported via $\gamma$ to obtain an automorphism $g'\in \tilde{G}_{b'}(C)=\Aut \E_{b'}$.  Let $\overline{g}'$ be the image of $g'$ under the projection $\tilde{G}_{b'}(C)\to G_{b'}(F)$.

The $G$-bundles $\E^1$ and $\E^{b'}$ may be trivialized over $\Spec B_{\dR}^+(C)$.   In doing so, we obtain embeddings of $G(F)=\Aut \E^1$ and $\tilde{G}_{b'}(C)=\Aut \E^{b'}$ into $G(B_{\dR}^+(C))$;  we denote both of these by $h\mapsto h_\infty$.  We also have the isomorphism $\gamma_\infty$ between $\E^1$ and $\E^{b'}$ over $\Spec B_{\dR}(C)$;  we may identify $\gamma_\infty$ with an element of $G(B_{\dR}(C))$, and then $g'_\infty =\gamma_\infty g_\infty \gamma_\infty^{-1}$ holds in $G(B_{\dR}(C))$.

%Recall from \S\ref{SectionVectorBundlesGBundles} that both $\E_b[\nu]$ and $\E_{b'}$ come equipped with a trivialization over $B_{\dR}^+$. The stalks of $g'$, $g$, and $\gamma$ at the point $\infty\from \Spec C\to X$ are thus naturally identified with elements of $G(B_{\dR}^+)$. Among them, $g'_\infty\in P(B_{\dR}^+)$, where $P\subset G$ is a parabolic subgroup with Levi factor $J_{b'}$.  Let $\overline{g}'_\infty$ be the image of $g_\infty$ under $P(B_{\dR}^+)\to J_{b'}(B_{\dR}^+)$.   By Lemma \ref{AutomorphismsOfGBundles}, $\overline{g}'_\infty\in J_{b'}(F)$.

The element $\bar g'_\infty$ is conjugate to $g'_\infty$,
%\cite{lem:hum}
so $\bar g'_\infty$ is conjugate to $g_\infty$ in $G(B_{\dR})$. Since $g$ and $\overline{g}'$ are both regular semisimple $\bar F$-points of $G$, being conjugate in $G(B_{\dR})$ is the same as being conjugate in $G(\bar F)$. Their centralizers, being $F$-rational tori, are thus isomorphic over $F$. Thus $G_{b'}$ contains a maximal torus that is elliptic for $G$. Elliptic maximal tori transfer across inner forms \cite[\S10]{Kot86}, which means that the Levi subgroup $M^* \subset G^*$ of which $G_{b'}$ is an inner form contains a maximal torus that is elliptic for $G^*$. Therefore $M^*=G^*$.

We have shown that $\E^1[\lambda]\isom \E_{b'}$ is semistable, implying that $\Aut \E^{b'}=G_{b'}(F)$ and that $g'\in G_{b'}(F)$.  Lemma \ref{LemmaBehaviorOfKappa} shows that $\kappa([b'])$ equals the image of $\lambda$ in $\pi_1(G)_{\Gamma}$;  this is the same as the image of $\mu$, which in turn is the same as $\kappa([b])$ because $b\in B(G,\mu)$.   Since $b'$ is basic, we have $[b']=[b]$ by \cite[Proposition 5.6]{KottwitzIsocrystals}.
\end{proof}

Recall the locally profinite set $\Rel_{b,\leq \mu}$ from Definition \ref{DfnRelB}.  This is the set of conjugacy classes of triples $(g,g',\lambda)$, where $g\in G(F)$ and $g'\in G_b(F)$ are related strongly regular elements, and $\lambda$ is a cocharacter of $T=\Cent(g,G)$, bounded by $\mu$, such that $\kappa(\inv[b](g,g'))$ agrees with the image of $\lambda$ in $X_*(T)_{\Gamma}$.   Let $\Rel_{b,\leq \mu,\elli}$ be the subset where $g$ (equivalently, $g'$) is elliptic. 

Theorem \ref{TheoremEllipticFixedPointsAreAdmissible} has the following corollary.   For a v-stack $X$, we write $\abs{X}$ for the underlying topological space.

\begin{cor} \label{CorIdentificationOfInHecke} Let $\In_S(\Hecke^{b,\ast}_{G,\leq\mu,S})_{\elli}$ be the preimage under $\In_S(h_1)$ of $\In_S(\Bun^b_{G,S})_{\elli}$.  Similarly let $\In_S(\Hecke^{\ast,1}_{G,\leq\mu,S})_{\elli}$ be the preimage under $\In_S(h_2)$ of $\In_S(\Bun^1_{G,S})_{\elli}$.  Then 
\[ \In_S(\Hecke^{b,\ast}_{G,\leq\mu,S})_{\elli}=
\In_S(\Hecke^{\ast,1}_{G,\leq\mu,S})_{\elli}=
\In_S(\Hecke^{b,1}_{G,\leq\mu,S})_{\elli}. \]
There is a homeomorphism $\abs{\In_S(\Hecke_{G,b,\leq \mu,S}^{b,1})_{\elli}}\isom \Rel_{b,\mu,\elli}$.
\end{cor}

\begin{proof} The first claim is just the statement that fixed points of elliptic elements on $\Gr^b_G$ and $\Gr^1_G$ are admissible.  For the second claim: since $\Hecke^{\ast,1}_{G,\leq\mu,S} \isom [\Gr_{G,\leq\mu,S}/G(F)_S]$, we can think of $\abs{\In_S(\Hecke_{G,b,\leq \mu}^{b,1})_{\elli}}$ as the set of conjugacy classes of pairs $(g,\lambda)$, where $g\in G(F)_{\elli}$ and $\lambda\in X_*(T)_{\leq \mu}$, where $T=\Cent(g,G)$.  We have an isomorphism $\E^1[\lambda]\isom \E^b$.  The element $g\in G(F)\isom \Aut \E^1$ determines an element $g'\in G_b(F)\isom \Aut\E^b$, up to conjugacy.  By Proposition \ref{PropositionDescriptionOfFixedPointsInM}, the triple $(g,g',\lambda)$ determines an element of $\Rel_{b,\mu,\elli}$.  Conversely, given such a triple $(g,g',\lambda)$, the pair $(g,\lambda)$ determines an element $g''\in G_b(F)$ as we have just argued, but then $g'$ and $g''$ are conjugate by Remark \ref{RmkAtMostOne}.  
\end{proof}

\subsection{Transfer of distributions from $G_b$ to $G$} \label{sub:6.3}

We continue to let $b$ be a basic element of $B(G)$. Let $\Lambda$ be a ring in which $p$ is invertible. Recall the Hecke transfer map
\[ T^{G_b\to G}_{b,\mu}\from C(G_b(F)_{\sr} \stslash G_b(F),\Lambda) \to C(G(F)_{\sr} \stslash G(F),\Lambda) \]
from \ref{DefinitionTbmu}.  As promised, we can now promote this to a transfer of distributions, at least after restriction to elliptic loci (and assuming, as we have been doing all along, that the $\Lambda$-valued Haar measures on $G(F)$ and $G_b(F)$ are chosen compatibly).  

Recall the period morphisms:
\[ \Gr^b_{G,\leq \mu,C} \stackrel{\pi_1}{\leftarrow} \Sht_{G,b,\mu,C}  \stackrel{\pi_2}{\to} \Gr^1_{G,\leq\mu,C}, \]
in which $\pi_1$ is a $G(F)_S$-torsor over its image, and $\pi_2$ is a $G_b(F)_S$-torsor over its image.
Consider the action map on $\Sht_{G,b,\mu,C}$:
\[ \alpha_{\Sht}\from G(F)_S \times G_b(F)_S \times \Sht_{G,b,\mu,C} \to \Sht_{G,b,\mu,C} \]
and also those on the period domains:
\begin{eqnarray*}
\alpha_1\from G(F)_S\times \Gr^1_{G,\leq\mu,C} &\to& \Gr^1_{G,\leq\mu,C} \\
\alpha_b\from G_b(F)_S\times \Gr^b_{G,\leq\mu,C} &\to& \Gr^b_{G,\leq\mu,C} 
\end{eqnarray*}
For $?\in\set{\Sht,1,b}$ we can define the elliptic fixed-point locus $\Fix(\alpha_?)_{\elli}$ of the corresponding action map, consisting of pairs $(g,x)$ with $g$ elliptic and $g.x=x$;  let us think of each $\Fix(\alpha_?)$ as a locally profinite set.  For instance, $\Fix(\alpha_1)_{\elli}$ is the set of pairs $(g,\lambda)$, where $g\in G(F)_{\elli}$, and $\lambda\in X_*(T_g)$ ($T_g=\Cent(g,G)$) is bounded by $\mu$.  These fit into a diagram
\begin{equation}
\label{EqFixLocusDiagram}
\xymatrix{
\Fix(\alpha_b)_{\elli} \ar[d]_{q_1} & \Fix(\alpha_{\Sht})_{\elli} \ar[l]_{p_1} \ar[r]^{p_2} & \Fix(\alpha_1)_{\elli} \ar[d]^{q_2} \\
G_b(F)_{\elli} && G(F)_{\elli}
}
\end{equation}
of locally profinite sets, in which $p_1$ is a $G_b(F)$-equivariant $G(F)$-torsor, $p_2$ is a $G(F)$-equivariant $G_b(F)$-torsor, and $q_1$ and $q_2$ are finite \'etale.  (The maps $p_i$ are surjective by Theorem \ref{TheoremEllipticFixedPointsAreAdmissible}.)  Furthermore, let us observe that for $(g,g',x)\in \Fix(\alpha_{\Sht})$, the image of $x$ in $\Gr_G^1(C)^g$ may be identified with a cocharacter $\lambda\in X_*(T)$ of $T=\Cent(g,G)$, and then the triple $(g,g',\lambda)$ lies in $\Rel_{b,\mu}$ by Proposition \ref{PropositionDescriptionOfFixedPointsInM}.  A key observation is that we have a diagram of stacks in locally profinite sets, in which both squares are cartesian:
\begin{equation}
\label{EqLiftOfRelCorrespondence}
\xymatrix{
[G_b(F)_{\elli}\sslash G_b(F)] \ar[d] & [\Fix(\alpha_{\Sht})_{\elli}/(G_b(F)\times G(F))] \ar[l] \ar[d] \ar[r] & [G(F)_{\elli}\sslash G(F)] \ar[d] \\
G_b(F)_{\elli}\sslash G_b(F) & \Rel_{b,\mu,\elli} \ar[r] \ar[l] & G(F)_{\elli}\sslash G(F) 
}
\end{equation}
Thus, at least over the elliptic locus, we have promoted a correspondence between sets of conjugacy classes to a correspondence between stacks of conjugacy classes.  Formally, this is exactly what is required to promote our transfer of functions to a transfer of distributions.   

\begin{lem}\label{LemmaDistributions}  Let $H$ be a locally pro-$p$ group,  and let $\Lambda$ be a commutative ring in which $p$ is invertible.  Choose a $\Lambda$-valued Haar measure on $H$.  Let $h\from \tilde{T}\to T$ be an $H$-torsor in locally profinite sets. The integration-along-fibers map
$C_c(\tilde{T},\Lambda) \to C_c(T,\Lambda)$ induces an isomorphism of $C(T,\Lambda)$-modules
\[ h_*\from C_c(\tilde{T},\Lambda)_H\to C_c(T,\Lambda) \]
and, dually, an isomorphism of $C(T,\Lambda)$-modules
\[ h_*\from \Dist(\tilde{T},\Lambda)^H \to \Dist(T,\Lambda). \]
\end{lem}

\begin{proof}
The $C(T,\Lambda)$-modules $C_c(T,\Lambda)$ and $C_c(\tilde T,\Lambda)$ are smooth in the sense of Definition \ref{dfn:modlps}.   Therefore by Lemma \ref{lem:sheaflps} the statement is local on $T$, so we may assume that the torsor  $\tilde{T}=T\times H$ is split. Then $C_c(\tilde T,\Lambda)_H=C_c(T,\Lambda) \otimes_\Lambda C_c(H,\Lambda)_H$.   The integration map $C_c(H,\Lambda)_H \to \Lambda$ is an isomorphism, so that $C_c(\tilde{T},\Lambda)_H\isom C_c(T,\Lambda)$.
\end{proof}

Recall from \S \ref{SectionInteractionWithOrbitalIntegrals} that we have chosen compatible Haar measures on $G(F)$ and $G_b(F)$.

\begin{dfn} \label{DfnTTilde} 
% Choose $\Lambda$-valued Haar measures $dg$ and $dg'$ on $G(F)$ and on $G_b(F)$. Then, 
With notation as in \eqref{EqFixLocusDiagram}, we define a $\Lambda$-linear map 
\begin{equation}
\label{EquationGtoJ}
\tilde{T}_{b,\mu}^{G\to G_b} \from C_c(G(F)_{\elli},\Lambda)_{G(F)} \to C_c(G_b(F)_{\elli},\Lambda)_{G_b(F)} 
\end{equation}
as the composition
\begin{eqnarray*}
C_c(G(F)_{\elli},\Lambda)_{G(F)} 
&\stackrel{q_2^*}{\to}&
C_c(\Fix(\alpha_1)_{\elli},\Lambda)_{G(F)} \\
&\stackrel{(p_2)_*^{-1}}{\isomto}&
C_c(\Fix(\alpha_{\Sht})_{\elli},\Lambda)_{G(F)\times G_b(F)} \\
&\stackrel{(p_1)_*}{\isomto}& C_c(\Fix(\alpha_b)_{\elli},\Lambda)_{G_b(F)} \\
&\stackrel{\cdot K_\mu}{\to}& C_c(\Fix(\alpha_b)_{\elli},\Lambda)_{G_b(F)} \\
&\stackrel{q_{1*}}{\to}&
C_c(G_b(F)_{\elli},\Lambda)_{G_b(F)},
\end{eqnarray*}
where $q_2^*$ means pullback, $q_{1*}$ means pushforward ({\em i.e.}, sum over fibers), the isomorphisms $(p_i)_*$ are induced by our choices of Haar measures as in Lemma \ref{LemmaDistributions}, and finally $K_\mu\in C(\Fix(\alpha_b),\Lambda)^{G_b(F)}$ is the function $(g',\lambda')\mapsto (-1)^d\rank V_\mu^\vee[\lambda']$, where $d=\class{\mu,2\rho_G}$.
\end{dfn}

% Again, we emphasize that this definition depends on the specific choice of Haar measures $dg$ and $dg'$; rescaling these choices will rescale $\tilde{T}_{b,\mu}^{G\to G_b}$.

\begin{pro}\label{PropositionLiftOfPhi} Assume that $\Lambda = \overline{\mathbf{Q}_{\ell}}$.
Let $\phi\in C_c(G(F)_{\elli},\Lambda)$, and let $\phi'\in C_c(G_b(F)_{\elli},\Lambda)$ be any lift of $\tilde{T}^{G\to G_b}_{b,\mu}\phi$.  Then the orbital integrals of $\phi$ and $\phi'$ are related by
\[ \phi_{G_b}' = T_{b,\mu}^{G\to G_b}\phi_G. \]
\end{pro}

\begin{proof}   For $g'\in G_b(F)_{\elli}$, with centralizer $T'$, we have:
\begin{eqnarray*}
&& \phi'_{G_b}(g') \\
&=&\int_{h'\in G_b(F)/T'(F)} \phi'(h'g'(h')^{-1})\;dh'\\
&=&
(-1)^d \sum_{\lambda'\in X_*(T')_{\leq -\mu}} \rank V_\mu^\vee[\lambda']
\int_{h'\in \frac{G_b(F)}{T'(F)}} [(p_1)_*(p_2)_*^{-1}q_2^*\phi](h'.(g',\lambda')) dh\;dh'. 
\end{eqnarray*}
Let $T$ be a transfer of the elliptic torus $T'$ to $G$.  Since $\Fix(\alpha_{\Sht})_{\elli}\to \Fix(\alpha_b)_{\elli}$ is a $G(F)$-torsor,  we may choose for each $\lambda'$ a lift $y_{\lambda'}=(g_{\lambda'},g',x_{\lambda'})$ of $(g',\lambda')$ to $\Fix(\alpha_{\Sht})$ with $g_{\lambda'}\in T(F)$.   Then $\phi'_{G_b}(g')$ equals
\[ (-1)^d\sum_{\lambda'\in X_*(T')_{\leq -\mu}}\rank V_\mu^\vee[\lambda']
\int_{h'\in G_b(F)/T'(F)} \int_{h\in G(F)} [(p_2)_*^{-1}q_2^*\phi]
\left((h,h').(y_{\lambda'})\right) \; dh\;dh'
\]
 We rewrite the inner integral as a nested integral; so that our expression for $\phi'_{G_b}(g')$ equals:
\begin{eqnarray*}
&&(-1)^d\sum_{\lambda'} \rank V_\mu^{\vee}[\lambda']
\int_{h'\in G_b(F)/T'(F)} \int_{h\in G(F)/T(F)} \int_{t\in T(F)} [(p_2)_*^{-1}q_2^*\phi]\left((ht^{-1},h'). y_{\lambda'}\right) \;dt\;dh\;dh' \\
&=&(-1)^d
\sum_{\lambda'} \rank V_\mu^{\vee}[\lambda']\int_{h'\in G_b(F)/T'(F)} \int_{h\in G(F)/T(F)} \int_{t'\in T'(F)} [(p_2)_*^{-1}q_2^*\phi]\left((h,h't').y_{\lambda'}\right) \;dt'\;dh\;dh'
\end{eqnarray*}
Here we have used Proposition \ref{PropositionDescriptionOfFixedPointsInM}:  
there is an isomorphism $\iota\from t\mapsto t'$ between $T(F)$ and $T'(F)$ satisfying $(t,t').y_{\lambda'}=y_{\lambda'}$.  This induces a bijection $\lambda\mapsto \lambda'=-\iota_*\lambda$ between $X_*(T)_{\leq\mu}$ and $X_*(T')_{\leq-\mu}$.   Given $\lambda\in X_*(T)_{\leq \mu}$, we let $g_{\lambda}=g_{\lambda'}$.  Then by Proposition \ref{PropositionDescriptionOfFixedPointsInM}, the preimage of $g'$ in $\Rel_{b,\mu}$ is exactly $\set{(g_{\lambda},g,\lambda)}_{\lambda\in X_*(T)_{\leq\mu}}$.  

Noting that $\rank V_\mu[\lambda]=\rank V_{\mu}^{\vee}[\lambda']$,  we exchange the order of the first two integrals above to obtain
\begin{eqnarray*}
\phi'_{G_b}(g')
&=&
(-1)^d\sum_{\lambda\in X_*(T)_{\leq\mu}} \rank V_{\mu}[\lambda]
\int_{h\in G(F)/T(F)}\int_{h'\in G_b(F)} [(p_2)_*^{-1}q_2^*\phi]\left((h,h')\cdot y_{\lambda'}\right)\;dh'\;dh
\\
&=& (-1)^d\sum_{\lambda\in X_*(T)_{\leq\mu}} \rank V_{\mu}[\lambda']\int_{h\in G(F)/T(F)}
\phi(hg_\lambda h^{-1})\;dh\\
&=&(-1)^d\sum_{\lambda\in X_*(T)_{\leq\mu}} \rank V_\mu[\lambda] \phi_G(g_{\lambda}) \\
&=& [T_{b,\mu}^{G\to G_b}\phi_G](g').
\end{eqnarray*}
\end{proof}

\begin{dfn} 
\label{DefinitionTransferOfDistributions} 
Let 
\[ \mc{T}_{b,\mu}^{G_b\to G} \from \Dist(G_b(F)_{\elli},\Lambda)^{G_b(F)} \to
\Dist(G(F)_{\elli},\Lambda)^{G(F)} \]
be the $\Lambda$-linear dual of $\tilde{T}_{b,\mu}^{G\to G_b}$.  % Again, note that this map depends on the Haar measures $dg$ and $dg'$.  
\end{dfn}

\begin{pro} \label{PropCompatibility}  Assume that $\Lambda = \overline{\mathbf{Q}_{\ell}}$. 
%and that the Haar measures $dg$ and $dg'$ are compatible in the sense of \S \ref{SectionInteractionWithOrbitalIntegrals}.   
Then the transfer of distributions $\mc{T}^{G_b\to G}_{b,\mu}$ extends the transfer of functions $T^{G_b\to G}_{b,\mu}$ from Definition \ref{DefinitionTbmu}.  
\end{pro}

\begin{proof} Let $f\in C(G_b(F)_{\elli},\Lambda)^{G_b(F)}$ be a conjugation-invariant function.  Let $\phi\in C_c(G(F)_{\elli},\Lambda)$,  and let $\phi'\in C_c(G_b(F)_{\elli},\Lambda)$ be a lift of $\tilde{T}^{G\to G_b}_{b,\mu}\phi$.  
Using the Weyl integration formula \eqref{EqWeylIntegration}, Lemma \ref{LemAdjointTransfers}, and Proposition \ref{PropositionLiftOfPhi} we compute
\begin{eqnarray*}
\int_{g\in G(F)_{\elli}} \phi(g)\mc{T}_{b,\mu}^{G_b\to G}(f\;dg')
&=&
\int_{g'\in G_b(F)_{\elli}} f(g')\phi'(g')\; dg' \\
&=&
\class{f,\phi'_{G_b}}_{G_b}\\
&=& \class{f,T_{b,\mu}^{G\to G_b}\phi_G}_{G_b} \\
&=& \class{T_{b,\mu}^{G_b\to G} f, \phi_G}_{G}\\
&=&
\int_{g\in G(F)_{\elli}} \phi(g)(T_{b,\mu}^{G_b\to G} f)\; dg, 
\end{eqnarray*}
so that 
\[ \mc{T}_{b,\mu}^{G_b\to G}(f\; dg') = T_{b,\mu}^{G_b\to G}(f) \; dg\]
as desired.
\end{proof}

\subsection{Hecke operators on $\Bun_G$ and the cohomology of shtuka spaces}

We are finally ready to reap our rewards. For the remainder of this chapter, we fix a prime $\ell \neq p$, and write $\Lambda$ for a $\mathbf{Z}_\ell$-algebra. Let $\hat{G}$ be the Langlands dual group over $\mathbf{Z}_{\ell}$.

We begin by quickly reviewing the results of \cite{FarguesScholze} on the categories $D_{\lis}(\Bun_G,\Lambda)$ and $D_{\lis}(\Bun_G^{b},\Lambda)$, and the action of Hecke operators on $D_{\lis}(\Bun_G,\Lambda)$.

The first key fact is that for any $b \in B(G)$, there is a natural equivalence of categories
\begin{equation}
\label{EqDerCatOfBunGb}
 D(G_b(F),\Lambda) \isom D_{\lis}(\Bun_G^b,\Lambda) 
 \end{equation}
\cite[Theorem I.5.1]{FarguesScholze}.  For a complex $\rho$ of smooth representations of $G_b(F)$, we will slightly abusively also write $\rho$ for the corresponding object of $D_{\lis}(\Bun_G^b,\Lambda)$.  

Next, recall that there is a notion of ULA objects in $D_{\lis}(\Bun_G,\Lambda)$. These admit the following concrete characterization. 

\begin{thm}[{\cite[Theorem I.5.1(v)]{FarguesScholze}}] \label{ThmULASheavesOnBunG} The following are equivalent for an object $A\in D_{\lis}(\Bun_G,\Lambda)$.  
\begin{enumerate}
\item $A$ is ULA over $\Spd k$.
\item For all $b\in B(G)$, the restriction $i_b^*A$, considered as an object of $D(G_b(F),\Lambda)$ via \eqref{EqDerCatOfBunGb}, is admissible in the sense that $(i_b^*A)^K$ is a perfect complex for all pro-$p$ open subgroups $K\subset G_b(F)$.  
\end{enumerate}
Moreover, ULA objects are preserved under Verdier duality $\D=\D_{\Bun_G/\Spd k}$, and satisfy Verdier biduality.
\end{thm}

\begin{cor} Let $b\in B(G)$, and let $\rho$ be an admissible complex in $D(G_b(F),\Lambda)$.  The objects $(i_b)_*\rho$ and $(i_b)_!\rho$ of $D_{\lis}(\Bun_G,\Lambda)$ are ULA over $\Spd k$.  
\end{cor}

\begin{proof} The object $(i_b)_!\rho$ is ULA by the criterion in Theorem \ref{ThmULASheavesOnBunG}.  Using Verdier duality (P4.) we have $\D((i_b)_!\rho^\vee)\isom (i_b)_* \rho$, so that $(i_b)_* \rho$ is also ULA.
\end{proof}

Next, recall that any object $V$ of $\Rep_{\hat{G}}(\Lambda)$ gives rise to a Hecke operator $T_V$, which is an endofunctor of $D_{\lis}(\Bun_{G,C},\Lambda)$. When $\Lambda$ is a torsion ring, there is a natural equivalence $D_{\lis}(\Bun_{G,C},\Lambda) \cong D_{\et}(\Bun_{G,C},\Lambda)$, and the operator $T_V$ is defined concretely as the operation
\begin{eqnarray*}
T_V\from  D_{\et}(\Bun_{G,C},\Lambda) &\to& D_{\et}(\Bun_{G,C}, \Lambda) \\
\mc{F} &\mapsto& h_{2!}(h_1^* \mc{F}\otimes \mc{S}_V).
\end{eqnarray*}
Here $\mc{S}_V\in D_{\et}(\Hecke_{G,C},\Lambda)$ is pulled back from the object $\mc{S}_V\in D_{\et}(\Hecke_{G,C}^{\loc},\Lambda)$ corresponding to $V$ under the Satake equivalence (Theorem \ref{ThmSatakeEquivalence}).  

%\begin{rmk}  \label{RmkSwitcheroo}  There is an involution $\sw$ on  $\Hecke_{G,C}^{\loc}\isom [L^+G\backslash LG/L^+G]$ induced from $g\mapsto g^{-1}$.  (In other words, $\sw$ sends a modification of vector bundles to its inverse.)  Then $\sw$ induces an involution $\sw^*$ on the Satake category, and hence on the dual group $\hat{G}$.  This involution is identified up to conjugacy as the Chevalley involution of $\hat{G}$ (see \cite[Proposition VI.12.1]{FarguesScholze}, which refers to this as the Cartan involution), which is $t\mapsto t^{-1}$ on $\hat{T}$.  One must take care in referring to pullbacks of $\mc{S}_V$ to $\Gr_{G,C}$.  The pullback along $\Gr_{G,C}\to [L^+G\backslash \Gr_{G,C}]\isom \Hecke_{G,C}^{\loc}$ is what we continue to call $\mc{S}_V$.  With this convention, the pullback along $\Gr_{G,C}\to [\Gr_{G,C}/L^+G]\isom \Hecke_{G,C}^{\loc}$ is $\mc{S}_{\sw^*(V)}$.
%\end{rmk}

\begin{thm}[{\cite[Theorem IX.0.1]{FarguesScholze}}]  \label{ThmHeckePreservesULA}The Hecke operators preserve the subcategories of ULA and compact objects in $D_{\lis}(\Bun_{G,C},\Lambda)$. For any $V$, $T_V$ has left and right adjoint given by $T_{V^\vee}$, where $V^\vee$ is the dual representation of $\hat{G}$. The actions of Hecke operators are compatible with extension of scalars along any ring map $\Lambda \to \Lambda'$.
\end{thm}

Next, we explain the relation between the Hecke operators $T_V$ and the cohomology of local shtuka spaces.  Let $\mu$ be a dominant cocharacter of $G$, and let $V_\mu \in \Rep(\hat{G})$ be the associated Weyl module. For any $\mathbf{Z}_{\ell}$-algebra $\Lambda$, we write $V_{\mu,\Lambda} \in \Rep(\hat{G}_{\Lambda})$ for the base change of $V_{\mu}$. Let $\mc{S}_{\mu}=\mc{S}_{V_{\mu}}$ be the corresponding object in the Satake category with $\mathbf{Z}_{\ell}$-coefficients; similarly, if $\Lambda$ is a torsion ring we write $\mc{S}_{\mu,\Lambda}=\mc{S}_{V_{\mu},\Lambda}$ for the corresponding object with $\Lambda$-coefficients. We will slightly abuse notation by using the same notations for the pullbacks of $\mc{S}_{\mu}$ and $\mc{S}_{\mu,\Lambda}$ to various other v-stacks, including $\Gr_{G,\leq \mu,C}$ and $\Sht_{G,b,\mu,C}$ (along the period morphism $\pi_1$ from \eqref{EqPeriodMorphisms}).

\begin{lem}  
\label{LemCohomologyOfShtukasModK}
Let $\Lambda$ be a $\mathbf{Z}_\ell$-algebra, let $K\subset G(F)$ be an open compact subgroup, and let $I_{K,\Lambda}={\cInd}_K^{G(F)} \Lambda$, where $\cInd$ is compactly supported induction.  Then there is a natural isomorphism
\[ R\Gamma_c(\Sht_{G,b,\mu,K,C},\mc{S}_{\mu,\Lambda}) \isom i_b^\ast T_{V_{\mu,\Lambda}} (i_1)_! I_{K,\Lambda} \] in $D(G_b(F),\Lambda)$.
\end{lem}

\begin{proof} 
When $\Lambda$ is a torsion ring, we can give the following direct argument. The global sections of $i_b^\ast T_{V_{\mu,\Lambda}} (i_1)_!I_{K,\Lambda}$ over $\Spd C\to \Bun_G^b$ are
\[
R\Gamma(\Spd C, i_b^\ast h_{2!} (h_1^*(i_1)_! I_{K,\Lambda} \otimes_{\Lambda} \mc{S}_{\mu,\Lambda})) \isom R\Gamma_c(\Gr^{b,1}, I_{K,\Lambda} \vert_{\Gr^{b,1}} \otimes_{\Lambda} \mc{S}_{\mu,\Lambda}). \]
Now use $I_{K,\Lambda} \isom (j_K)_! \Lambda$ along with proper base change to get the result.

The general case follows from the proof of \cite[Proposition IX.3.2]{FarguesScholze}.
\end{proof}

Recall that when $\rho$ is any smooth $G_b(F)$-representation with $\overline{\Q_\ell}$-coefficients,  we defined an object
\[ R\Gamma(G,b,\mu)[\rho] \isom \varinjlim_K R\Hom_{G_b(F)}(R\Gamma_c(\Sht_{(G,b,\mu),C}/K,\mc{S}_{\mu}) \otimes \overline{\mathbf{Q}_{\ell}}, \rho), \]
in $D(G(F),\overline{\mathbf{Q}_{\ell}})$, cf. Definition \ref{dfn:stcoh}. The association $\rho \mapsto R\Gamma(G,b,\mu)[\rho]$ clearly extends to a functor $D(G_b(F),\overline{\mathbf{Q}_{\ell}}) \to D(G(F),\overline{\mathbf{Q}_{\ell}})$. Our next goal is to give an alternative approach to this construction, which is valid for more general coefficient rings and which makes the finiteness properties of this construction transparent.

\begin{pro}\label{PropHInTermsOfHecke}  Let $\rho$ be any object of $D(G_b(F),\overline{\mathbf{Q}_{\ell}})$.  Then there is a natural isomorphism 
\[ R\Gamma(G,b,\mu)[\rho] \isom i_1^*T_{{V^{\vee}_{\mu,\overline{\mathbf{Q}_{\ell}}}}} (i_b)_*\rho \] in $D(G(F),\overline{\mathbf{Q}_{\ell}})$.

If $\rho$ is admissible, then so is $R\Gamma(G,b,\mu)[\rho]$. If $\rho$ is of finite length, then so is $R\Gamma(G,b,\mu)[\rho]$.
\end{pro}

\begin{proof}  Let $K\subset G(F)$ be a compact open subgroup.  Using 
Lemma \ref{LemCohomologyOfShtukasModK}, various adjunctions, and the compatibility of Hecke operators with extension of scalars, we have
\begin{eqnarray*}
R\Hom_{G_b(F)}(R\Gamma_c(\Sht_{G,b,\mu,K,C},\mc{S}_V) \otimes \overline{\mathbf{Q}_{\ell}},\rho) 
&\isom&
R\Hom_{G_b(F)}(i_b^\ast T_{V_{\mu}} (i_1)_! {\cInd}_K^{G(F)} \mathbf{Z}_{\ell} \otimes \overline{\mathbf{Q}_{\ell}}, \rho) \\
& \isom & R\Hom_{G_b(F)}(i_b^\ast T_{V_{\mu,\overline{\mathbf{Q}_{\ell}}}} (i_1)_! {\cInd}_{K}^{G(F)} \overline{\mathbf{Q}_{\ell}}, \rho) \\
&\isom& (i_1^* T_{{V^{\vee}_{\mu,\overline{\mathbf{Q}_{\ell}}}}}  (i_b)_*\rho)^K. 
\end{eqnarray*}
Taking the colimit over $K$ gives the first claim. The claim about preservation of admissibility now follows from Theorem \ref{ThmULASheavesOnBunG} combined with Theorem \ref{ThmHeckePreservesULA}. For the final claim, fix some $\rho$ of finite length. Note that $i_1^*T_{{V^{\vee}_{\mu,\overline{\mathbf{Q}_{\ell}}}}} (i_b)_*\rho$ is the smooth dual of $i_1^*T_{{V^{\vee}_{\mu,\overline{\mathbf{Q}_{\ell}}}}} (i_b)_! \rho^{\vee}$, so it's enough to show that $i_1^*T_{{V^{\vee}_{\mu,\overline{\mathbf{Q}_{\ell}}}}} (i_b)_! \rho^{\vee}$ is of finite length. But finite length is equivalent to being both compact and admissible, so we conclude by observing that the operation $i_1^*T_{{V^{\vee}_{\mu,\overline{\mathbf{Q}_{\ell}}}}} (i_b)_!(-)$ preserves compact objects.
\end{proof}

\begin{dfn} For any $\mathbf{Z}_{\ell}$-algebra $\Lambda$, we write \[ R\Gamma(G,b,\mu)[-]: D(G_b(F),\Lambda) \to D(G(F),\Lambda) \] for the functor $i_1^* T_{{V^{\vee}_{\mu,\Lambda}}}  (i_b)_*(-)$.
\end{dfn}

By the previous discussion, this functor is compatible with extension of scalars along any map $\Lambda \to \Lambda'$, and preserves admissible objects. Moreover, if $\Lambda$ is Artinian and $\rho \in D(G_b(F),\Lambda)$ is admissible of finite length, then $R\Gamma(G,b,\mu)[\rho]$ is also of finite length, by the same argument as in the proof of Proposition \ref{PropHInTermsOfHecke}.

We now come to the technical heart of this paper. Choose $\mathbf{Z}_{\ell}$-valued Haar measures on $G(F)$ and $G_b(F)$, compatibly as in \S\ref{SectionInteractionWithOrbitalIntegrals}. These induce $\Lambda$-valued Haar measures on the same groups compatibly with varying $\Lambda$.  Then for any $\Lambda$, any admissible representation $\pi$ of $G(F)$ with coefficients in $\Lambda$ has a corresponding $\Lambda$-valued trace distribution $\trdist(\pi)$, and similarly for $G_b(F)$.  Recall also that we defined a transfer of $\Lambda$-valued distributions $\mc{T}_{b,\mu}^{G_b\to G}$, Definition \ref{DefinitionTransferOfDistributions}.
%, which also depends on these choices of Haar measures.

\begin{pro} \label{PropTraceDist} Let $\Lambda$ be any torsion $\mathbf{Z}_{\ell}$-algebra, and let
 $\rho$ be any admissible representation of $G_b(F)$ with coefficients in $\Lambda$.  Then have an equality
\[ \trdist R\Gamma (G,b,\mu)[\rho]_{\elli}
= \mc{T}_{b,\mu}^{G_b\to G} \trdist(\rho)_{\elli} \]
 in $\Dist(G(F)_{\elli},\Lambda)^{G(F)}$.
\end{pro}

%As a sanity check, note that all dependencies on the choices of Haar measures are compatible: rescaling $dg$ and $dg'$ by $C$ and $C'$, we see that $\trdist R\Gamma (G,b,\mu)[\rho]$ scales by $C$, $\trdist(\rho)$ scales by $C'$, and $\mc{T}_{b,\mu}^{G_b\to G}$ scales by $C/C'$.

\begin{proof} 
In the following proof, we set $S=\Spd C$ and $V=V_{\mu,\Lambda}$ for brevity.

We have an isomorphism
\[ H^0(\In_S(\Bun_{G,C}^1),K_{\In_S(\Bun_{G,C}^1)/S})\isom \Dist(G(F),\Lambda)^{G(F)},\] 
and similarly for $G_b(F)$.   With respect those isomorphisms,  the left side of the desired equality is the characteristic class
\[  \cc_{\Bun_{G,C}^1/S}\left(i_1^*T_{V^\vee}(i_b)_*\rho \right)\] restricted to $\Dist(G(F)_{\elli},\Lambda)^{G(F)}$.

For the remainder of the proof we introduce the abbreviations $B=\Bun_{G,C}$ and $H=\Hecke_{G,\leq\mu,C}$ and $\In=\In_S$.   Let us also use a subscript ``ell'' to mean restriction to the appropriate elliptic locus.  Taking inertia stacks in \eqref{BigHeckeDiagram}, we obtain a commutative diagram
\begin{equation}
\label{EqBigInertiaDiagram}
\xymatrix{
& &
\In(H^{\ast,1})_{\elli} \ar[r]^{\In(h_2^{\ast,1})_{\elli}} \ar[d]_{j_1'} \ar[ddll]_{\id}
&
\In(B^1)_{\elli} \ar[d]^{j_1} \\
& & 
\In(H^{\ast,1}) \ar[r]_{\In(h_2^{\ast,1})} \ar[d]_{\In(i_1')} 
&
\In(B^1) \ar[d]^{\In(i_1)} \\
\In(H^{b,\ast})_{\elli} \ar[r]_{j_b'} \ar[d]_{\In(h_1^{b,\ast})_{\elli}} 
&
\In(H^{b,\ast}) \ar[r]_{\In(i_b')} \ar[d]_{\In(h_1^{b,\ast})} 
&
\In(H) \ar[r]_{\In(h_2)} \ar[d]_{\In(h_1)} 
&
\In(B) \\
\In(B^b)_{\elli} \ar[r]_{j_b} 
&
\In(B^b) \ar[r]_{\In(i_b)} 
&
\In(B)
&
}
\end{equation}
in which all squares are cartesian, and the morphism labeled $\id$ is the equality from Corollary \ref{CorIdentificationOfInHecke}.   The characteristic class in question is 
\begin{eqnarray*}
\label{EqInvocationOfTraceFormula}
 j_1^*\cc_{B^1/S}(i_1^*T_{V^\vee}(i_b)_*\rho )
&\stackrel{\text{Lem. }\ref{LemRestrictionOfCC}}{=}& j_1^*\In(i_1)^*\cc_{B/S}\left((h_2)_!(h_1^*(i_b)_*\rho \otimes \mc{S}_{V^\vee}\right) \\
&\stackrel{\text{Cor.} \ref{CorLefschetzVerdier}}{=}&j_1^*\In(i_1)^*\In(h_2)_*\cc_{H/S}(h_1^*(i_b)_*\rho \otimes \mc{S}_{V^\vee})\\
&=&(\In(h_2^{\ast,1})_{\elli})_*(j_1')^*\In(i_1')^*\cc_{H/S}(h_1^*(i_b)_*\rho \otimes\mc{S}_{V^\vee})\\
&=&(\In(h_2^{\ast,1})_{\elli})_*(j_b')^*\In(i_b')^*\cc_{H/S}(h_1^*(i_b)_*\rho \otimes\mc{S}_{V^\vee})\\
&\stackrel{\text{Lem. }\ref{LemRestrictionOfCC}}{=}& 
(\In(h_2^{\ast,1})_{\elli})_*(j_b')^*\cc_{H^{b,\ast}/S}((i_b')^*h_1^*(i_b)_*\rho \otimes\mc{S}_{V^\vee}) \\
&=&(\In(h_2^{\ast,1})_{\elli})_*(j_b')^*
\cc_{H^{b,\ast}/S}
((h_1^{b,\ast})^*\rho \otimes \mc{S}_{V^\vee}).
\end{eqnarray*}
Noting that $H^{b,\ast}\isom [\Gr^b_{G,\leq-\mu,C}/G_b(F)_S]$, we have a cartesian diagram of decent $S$-v-stacks:
\begin{equation}
\label{EqDiagramForHBStar}
\xymatrix{
H^{b,\ast} \ar[r] \ar[d] &  H^{\loc}_m \ar[d] \\
B^b \ar[r] & B_m^{\loc} 
}
\end{equation}
Here $B^b\isom [S/G_b(F)_S]$, $H^{\loc}_m\isom [\Gr_{G,\leq-\mu,C}/L^+_mG]$, and $B_m^{\loc}\isom [S/L^+_mG]$;  the $m$ here is chosen large enough so that the action of $L^+G$ on $\Gr_{G,\leq-\mu,C}$ factors through the quotient $L^+_mG$.    Through this, we can identify $(h_1^{b,\ast})^*\rho \otimes \mc{S}_{V^\vee}$ with $\rho \boxtimes_{B^{\loc}_m} \mc{S}_{V^\vee}$.    It is at this point we apply Theorem \ref{ThmKunnethForCharacteristicClass}, valid because the base $B_m^{\loc}=[S/L^+_mG]$ satisfies the hypotheses of Lemma \ref{LemTrivialDualizingSheaf}.  We get
\begin{eqnarray*}
 j_1^*\cc_{B^1/S}(i_1^*T_{V^\vee}(i_b)_*\rho)
&=& (\In(h_2^{\ast,1})_{\elli})_*(j_b')^*\left(\cc_{B^b/S}\rho \boxtimes_{\In(B^{\loc}_m)} \cc_{H_m^{\loc}/S}(\mc{S}_{V^\vee}) \right). \\
&=& (\In(h_2^{\ast,1})_{\elli})_*\left(\cc_{B^b/S}\rho_{\elli} \boxtimes_{\In(B^{\loc}_m)_{\sr}} \cc_{H_m^{\loc}/S}(\mc{S}_{V^\vee})_{\sr}\right) 
\end{eqnarray*}

Considering the diagram
\[
\xymatrix{
H^0(\In(B^b)_{\elli},K_{\In(B^b)/S})
\ar[d]_{-\boxtimes_{\In(B_m^{\loc})_{\sr}} \cc_{H_m^{\loc}/S}(\mc{S}_{V^\vee})_{\sr}}
\ar[rr]^{\sim}
&&
\Dist(G_b(F)_{\elli},\Lambda)^{G_b(F)} \ar[d]^{q_1^*(-)\otimes (-1)^{\class{2\rho_G,\mu}}\rank V^\vee[-]} \\
H^0(\In(H^{b,*})_{\elli},
K_{\In(H^{b,*})/S}) \ar[rr]^{\sim} \ar[d]_{=}
& & \Dist(\Fix(\alpha_b)_{\elli},\Lambda)^{G_b(F)} \ar[d]^{(p_1)_*^{-1}} \\
H^0(\In(H^{b,1})_{\elli},K_{\In(H^{b,1})/S})\ar[d]_{=} \ar[rr]^{\sim} &&
\Dist(\Fix(\alpha_{\Sht})_{\elli},\Lambda)^{G(F)\times G_b(F)}\ar[d]^{(p_2)_*} \\
H^0(\In(H^{\ast,1})_{\elli},K_{\In(H^{\ast,1})/S}) 
\ar[d]_{\In(h_2^{\ast,1})_*}
\ar[rr]^{\sim} && 
\Dist(\Fix(\alpha_1)_{\elli},\Lambda)^{G(F)} \ar[d]^{(q_2)_*} \\
H^0(\In(B^1)_{\elli},K_{\In(B^1)/S}) \ar[rr]_{\sim}
&&
\Dist(G(F)_{\elli},\Lambda)^{G(F)}
}
\]
our characteristic class is the image of $\cc_{B^b/S}(\rho)_{\elli}$ under the composite vertical map on the left.  The diagram is commutative;  the hardest thing to check is the commutativity of the top square, which follows from Proposition \ref{ThmCCOfRhoTimesSV} below. The composition along the right column is $\mc{T}^{G_b\to G}_{b,\mu}$, giving us the desired equality of distributions.  
\end{proof}

It remains to justify one step in this computation. Maintain the notation and assumptions of the previous theorem. The v-stack $H^{b,\ast}$ can be expressed as a fiber product as in \eqref{EqDiagramForHBStar};  we have the ULA object $\rho \boxtimes_{B_m^{\loc}} \mc{S}_{V^\vee}$, whose characteristic class can be calculated using Theorem \ref{ThmKunnethForCharacteristicClass}.

Let
\[ \alpha_b \from G_b(F)_S\times \Gr_{G,\leq-\mu} \to \Gr_{G,\leq-\mu} \]
be the action map, so that we have an isomorphism 
\[ \In_S(H^{b,\ast}) \isom [\Fix(\alpha_b)/G_b(F)_S]. \]
Let $\Fix(\alpha_b)_{\sr}$ be the open subset lying over $G_b(F)_{\sr}$ (and use the same convention for other objects);  then $\Fix(\alpha_b)_{\sr}$ is a locally profinite set, which is finite over $G_b(F)_{\sr}$ with fibers $X_*(T)_{\leq -\mu}$.  

\begin{pro}\label{ThmCCOfRhoTimesSV} The characteristic class
\[ \cc_{H/S}(\rho \boxtimes \mc{S}_{V^\vee})_{\sr}\in H^0(\In_S(H)_{\sr},K_{\In_S(T)/S})\isom \Dist(\Fix(\alpha_b)_{\sr},\Lambda)^{G_b(F)} \]
equals the image of $\trdist(\rho) \otimes (-1)^{\class{2\rho_G,-}}\rank V^{\vee}[-]$ under the evident map
\begin{equation}
\label{EqMultiplyDistributionByFunction}   \Dist(G_b(F)_{\sr},\Lambda)^{G_b(F)}\otimes C(X_*(T)_{\leq -\mu},\Lambda)^W \to \Dist(\Fix(\alpha_b)_{\sr},\Lambda)^{G_b(F)}.
\end{equation}
\end{pro}

\begin{proof} Take inertia stacks in \eqref{EqDiagramForHBStar},  and restrict to the strongly regular locus in $\In_S[S/L^+_mG]$ to obtain a cartesian diagram
\[
\xymatrix{
[\Fix(\alpha_b)_{\sr}/G_b(F)_S] \ar[rr] \ar[d] && [G_{b,\sr}(F)_S\stslash G_b(F)_S]  \ar[d] \\
X_*(T)_{\leq \mu}\times^W [L_m^+T_{\sr}\stslash T_{\sr}] \ar[rr]&& 
[L_m^+T_{\sr}/(W\ltimes L_m^+T)] 
}
\]
The K\"unneth map \eqref{EqBoxtimesOnDistributions} in this situation reduces to \eqref{EqMultiplyDistributionByFunction} on the level of global sections.  The result now follows from Theorems \ref{ThmKunnethForCharacteristicClass} and \ref{ThmCharClassOfLocalHeckeStack}.
\end{proof}

This formally implies the following theorem.
\begin{thm} \label{TheoremCalculationOfCharacter} Let $\rho$ be any finite length admissible $G_b(F)$-representation with $\overline{\mathbf{Q}_{\ell}}$-coefficients. Assume that $\rho$ admits a $\overline{\mathbf{Z}_{\ell}}$-lattice in the sense of Definition \ref{DefLattice}.  Then for all $\phi \in C_c(G(F)_{\elli},\overline{\mathbf{Q}_{\ell}})$, the equality \[ \tr (\phi | \Mant_{b,\mu}(\rho)) = \left[\mc{T}_{b,\mu}^{G_b\to G} (\trdist \rho)\right](\phi) \] holds.
\end{thm}

Recall that by definition, $\mc{T}_{b,\mu}^{G_b\to G} (\trdist\rho)(\phi)$ depends only on $(\trdist\rho)_{\elli}$.

\begin{proof} Fix $\rho$ and $\phi$ as in the theorem, and fix a $\overline{\mathbf{Z}_{\ell}}$-lattice $\rho^{\circ} \subset \rho$. After rescaling, we may also assume $\phi$ is valued in $\overline{\mathbf{Z}_{\ell}}$. It is clear from the definitions that $\tr (\phi | \Mant_{b,\mu}(\rho)) = \tr (\phi | R\Gamma(G,b,\mu)[\rho^{\circ}])$ and \[ \left[\mc{T}_{b,\mu}^{G_b\to G} (\trdist \rho)\right](\phi) = \left[\mc{T}_{b,\mu}^{G_b\to G} (\trdist \rho^{\circ})\right](\phi). \]

For all $n\geq 1$, set $\rho^{\circ}_{n} = \rho^{\circ} \otimes \mathbf{Z}/\ell^n$, and write $\phi_n \in C_c(G(F)_{\elli},\overline{\mathbf{Z}_{\ell}}/\ell^n)$ for the obvious reductions of $\phi$. Applying Proposition \ref{PropTraceDist} with $\Lambda = \overline{\mathbf{Z}_{\ell}}/\ell^n$, we get equalities \[ \tr (\phi_n | R\Gamma(G,b,\mu)[\rho^{\circ}_n]) = \left[\mc{T}_{b,\mu}^{G_b\to G} (\trdist \rho^{\circ}_n)\right](\phi_n) \]
for all $n\geq 1$. The result now follows by taking the inverse limit over $n$.
\end{proof}

\subsection{Proof of Theorem \ref{TheoremMain}}

We are finally ready to prove the main theorem of the paper, which we restate for the convenience of the reader.

\begin{thm} \label{TheoremMainRedux}
Assume the refined local Langlands correspondence \cite[Conjecture G]{KalethaLocalLanglands}. Let $\phi\from W_F \times \tx{SL}_2 \to {^LG}$ be a discrete Langlands parameter with coefficients in $\overline{\Q}_\ell$, and let $\rho\in \Pi_\phi(G_b)$ be a member of its $L$-packet.  After ignoring the action of $W_E$, we have an equality
 \[\Mant_{b,\mu}(\rho)=\sum_{\pi\in\Pi_{\phi}(G)} \left[\dim \Hom_{S_\phi}(\delta_{\pi,\rho},r_\mu)\right]\pi + \mathrm{err} \]
in $\Groth(G(F))$, where $\mathrm{err} \in \Groth(G(F))$ is a virtual representation whose character vanishes on the locus of elliptic elements of $G(F)$.

If the packet $\Pi_{\phi}(G)$ consists entirely of supercuspidal representations and the semisimple $L$-parameter $\varphi_{\rho}$ associated with $\rho$ as in \cite[\S I.9.6]{FarguesScholze} is supercuspidal, then in fact $\mathrm{err}=0$.
\end{thm}

%
%\begin{thm} \label{TheoremCalculationOfCharacter}
%Let $\rho$ be a finite length admissible $G_b(F)$-representation, and assume that $\rho$ contains an $G_b(F)$-invariant $\overline{\Z}_\ell$-lattice. Let $(\tr\rho)_{\elli}\in \Dist(G_b(F)_{\elli},\overline{\Q}_{\ell})$ be the trace distribution of $G_b(F)_{\elli}$ on $\rho$.   Then the $H^i(G,b,\mu)[\rho]$ are finite-length admissible representations of $G(F)$, and the trace distribution of $G(F)_{\elli}$ on the Euler characteristic $H^*(G,b,\mu)[\rho]$ equals $T_{b,\mu}^{G_b\to G} (\tr\rho)_{\elli}$.
%\end{thm}
%
%In the next subsection, we will deduce from this a more general result, Theorem \ref{TheoremCalculationOfCharacterNoLattice}, which doesn't require the assumption that $\rho$ admits a lattice.
%
%Combining Theorem \ref{TheoremCalculationOfCharacterNoLattice} with Theorem \ref{TheoremEndoscopicTraceRelation} shows that
%the trace distribution of $H^*(G,b,\mu)[\rho]$ on $G(F)_{\elli}$ agrees with the trace distribution of 
%\[ (-1)^d \sum_{\pi \in \Pi_\phi(G_b)} \dim \Hom_{S_\phi}(\delta_{\pi,\rho},r_\mu). \]
%
%This implies all of Theorem \ref{TheoremMain} except for the final claim regarding the vanishing of the error term.  This final claim is proved in Theorem \ref{NoError} below.

The main ingredient in the proof is the following extension of Theorem \ref{TheoremCalculationOfCharacter} to its natural level of generality.

\begin{thm} \label{TheoremCalculationOfCharacterNoLattice} Let $\rho$ be any finite length admissible $G_b(F)$-representation with $\overline{\mathbf{Q}_{\ell}}$-coefficients.  Then for all $\phi \in C_c(G(F)_{\elli},\overline{\mathbf{Q}_{\ell}})$, the equality \[ \tr (\phi | \Mant_{b,\mu}(\rho)) = \left[\mc{T}_{b,\mu}^{G_b\to G} (\trdist \rho)\right](\phi) \] holds.

In particular, the virtual character of $\Mant_{b,\mu}(\rho)$ restricted to $G(F)_{\elli}$ is equal to $T_{b,\mu}^{G_b \to G}(\Theta_\rho)$.
\end{thm}

We will formally deduce this from Theorem \ref{TheoremCalculationOfCharacter} by a continuity argument.

For the proof of this theorem, it will be convenient to use the language of Grothendieck groups. In particular, by the finiteness results mentioned above, $\Mant_{b,\mu}(-)$ can be regarded as a group homomorphism $\Mant_{b,\mu}(-): \Groth(G_b(F)) \to \Groth(G(F))$. Recall that 
%(after choosing a Haar measure) 
any element $\phi \in C_c(G(F),\overline{\mathbf{Q}_\ell})$ defines a linear form $\tr(\phi | -): \Groth(G(F)) \to \overline{\mathbf{Q}_\ell}$. By definition, a linear form $f: \Groth(G(F)) \to \overline{\mathbf{Q}_\ell}$ is a \emph{trace form} if it can be written as $\tr( \phi | -)$ for some $\phi \in C_c(G(F),\overline{\mathbf{Q}_\ell})$. The key ingredient in the proof of Theorem \ref{TheoremCalculationOfCharacterNoLattice} is the following result, which roughly says that $\Mant_{b,\mu}(\rho)$ is a continuous function of $\rho$.

\begin{thm} \label{TheoremContinuous} For any $\phi \in C_c(G(F),\overline{\mathbf{Q}_\ell})$, the linear form \[ \tr(\phi | \Mant_{b,\mu}(-)) : \Groth(G_b(F)) \to \overline{\mathbf{Q}_{\ell}} \] is a trace form.
\end{thm}

With future applications in mind, we'll actually prove the following refined form of this theorem which also accounts for the Weil group action.

\begin{thm}
For any fixed $\phi\in C_{c}(G(F),\overline{\mathbf{Q}_{\ell}})$
and $w\in W_{E}$, the linear form $\Groth(G_{b}(F))\to\overline{\mathbf{Q}_{\ell}}$
defined by 
\[
\rho\mapsto\mathrm{tr}(\phi\times w | R\Gamma(G,b,\mu)[\rho])
\]
is a trace form.
\end{thm}

In the classical setting of Rapoport-Zink spaces, this was conjectured by Taylor, cf. \cite[Conjecture 8.3]{ShinEL}. Taking $w=1$, we deduce Theorem \ref{TheoremContinuous}.

\begin{proof}
For any reductive group $H/F$, the trace Paley-Wiener theorem of Bernstein-Deligne-Kazhdan, \cite{BDK},
characterizes trace forms among all linear forms on $\Groth(H(F))$ by the
following two conditions:

1. There is some open compact subgroup $K\subset H(F)$
such that $f(\pi)\neq0$ only if $\pi^{K}\ne0$.

2. For any parabolic $P=MU\subset H$ and any irreducible smooth $M(F)$-representation
$\sigma$, $f(i_{M}^{H}(\sigma\psi))$ is an algebraic
function of $\psi$, where $\psi$ varies over the unramified characters of $M(F)$. Here $i_{M}^{H}(-)$ denotes normalizes parabolic induction.

We'll prove the theorem by showing that the linear form $\mathrm{tr}(\phi\times w | R\Gamma(G,b,\mu)[-])$ satisfies the conditions
of the trace Paley-Wiener theorem, applied to the group $H=G_b$.

\textbf{Verification of Condition 1. }Fix a pro-$p$ open compact
subgroup $K\subset G(F)$ such that $\phi$ is bi-$K$-invariant.
If $\mathrm{tr}(\phi\times w| R\Gamma(G,b,\mu)[\rho])\neq0$,
then $(i_{\mathbf{1}}^{\ast}T_{V^{\vee}_{\mu,\overline{\mathbf{Q}_{\ell}}}}i_{b \ast }\rho)^{K}\neq0$.
Therefore, it suffices to see that there is some open compact $K'\subset G_{b}(F)$
such that $(i_{\mathbf{1}}^{\ast}T_{V^{\vee}_{\mu,\overline{\mathbf{Q}_{\ell}}}}i_{b \ast}\rho)^{K}\neq0$
only if $\rho^{K'}\neq0$. For this, write
\begin{align*}
(i_{\mathbf{1}}^{\ast}T_{V^{\vee}_{\mu,\overline{\mathbf{Q}_{\ell}}}}i_{b \ast}\rho)^{K} & \cong R\mathrm{Hom}(i_{\mathbf{1}!}\cInd_{K}^{G(F)}\overline{\mathbf{Q}_{\ell}},T_{V^{\vee}_{\mu,\overline{\mathbf{Q}_{\ell}}}}i_{b \ast}\rho)\\
 & \cong R\mathrm{Hom}(T_{V_{\mu,\overline{\mathbf{Q}_{\ell}}}}i_{\mathbf{1}!}\cInd_{K}^{G(F)}\overline{\mathbf{Q}_{\ell}},i_{b \ast}\rho)\\
 & \cong R\mathrm{Hom}(i_{b}^{\ast}T_{V_{\mu,\overline{\mathbf{Q}_{\ell}}}}i_{\mathbf{1}!}\cInd_{K}^{G(F)}\overline{\mathbf{Q}_{\ell}},\rho).
\end{align*}

But now $i_{b}^{\ast}T_{V_{\mu,\overline{\mathbf{Q}_{\ell}}}}i_{\mathbf{1}!}\cInd_{K}^{G(F)}\overline{\mathbf{Q}_{\ell}}$
is compact, and hence supported on only finitely many Bernstein components
for $G_{b}(F)$. This shows that the irreducible
$\rho$'s with $(i_{\mathbf{1}}^{\ast}T_{V^{\vee}_{\mu,\overline{\mathbf{Q}_{\ell}}}}i_{b \ast}\rho)^{K}\neq0$
are supported on finitely many Bernstein components for $G_{b}(F)$.
Quite generally, if $\Theta$ is any finite union of Bernstein components
for $G_{b}(F)$, we can choose an open compact subgroup
$K'\subset G_{b}(F)$ such that $\rho^{K'}\neq0$ if
$\rho$ is supported on $\Theta$. This gives the result.

\textbf{Verification of Condition 2.} Fix $P=MU\subset G_{b}$ and
$\sigma$ as in Condition 2. Let $X=\mathrm{Spec}R$ be the smooth affine
algebraic variety over $\overline{\mathbf{Q}_{\ell}}$ parametrizing
unramified characters of $M(F)$. Let $\boldsymbol{\psi}:M(F)\to R^{\times}$
be the universal character, and form $\Pi=i_{M}^{G_{b}}(\sigma\boldsymbol{\psi})$.
This is an admissible smooth $R[G_{b}(F)]$-module
interpolating the parabolic inductions $i_{M}^{G_{b}}(\sigma\psi)$
over varying unramified characters $\psi$ in the evident sense.\footnote{To see that $\Pi$ is admissible in our slightly nonstandard sense, observe first that $\Pi^K$ is finitely generated as an $R$-module for all pro-$p$ open compact subgroups $K \subset G_b(F)$, since $P(F) \backslash G_b(F) / K$ is finite. But $R$ is a smooth $\overline{\mathbf{Q}_{\ell}}$-algebra, so any finitely generated $R$-module is automatically a perfect complex, giving the desired result.} Since $\Pi$ is admissible, the pushforward
 $i_{b \ast}\Pi\in D_{\mathrm{lis}}(\mathrm{Bun}_{G},R)$ is ULA.
Since Hecke operators preserve ULA complexes, we deduce that $i_{\mathbf{1}}^{\ast}T_{V^{\vee}_{\mu,R}}i_{b \ast}\Pi\in D(G(F),R)^{BW_{E}}$
is an admissible complex of smooth $R[G(F)]$-modules
with $W_{E}$-action, which interpolates the individual complexes
\[ R\Gamma(G,b,\mu)[i_{M}^{G_{b}}(\sigma\psi)]=i_{\mathbf{1}}^{\ast}T_{V^{\vee}_{\mu,\overline{\mathbf{Q}_{\ell}}}}i_{b \ast}i_{M}^{G_{b}}(\sigma\psi)\]
in the evident sense. 

Now fix a pro-$p$ open compact subgroup $K \subset G(F)$ such that $\phi$ is bi-$K$-invariant, so $\phi\times w$
defines an endomorphism of the perfect complex \[ (i_{\mathbf{1}}^{\ast}T_{V^{\vee}_{\mu,R}}i_{b \ast }\Pi)^{K}\in\mathrm{Perf}(R). \]
Let $f\in R$ be the trace of this endomorphism. Unwinding definitions,
we see that for any unramified character $\psi:M(F)\to\overline{\mathbf{Q}_{\ell}}^{\times}$
with associated point $x_{\psi}\in X(\overline{\mathbf{Q}_{\ell}})$,
there is an equality
\[
f(x_{\psi})=\mathrm{tr}\left(\phi\times w| R\Gamma(G,b,\mu)[i_{M}^{G_{b}}(\sigma\psi)]\right).
\]
This shows that $\mathrm{tr}(\phi\times w|R\Gamma(G,b,\mu)[i_{M}^{G_{b}}(\sigma\psi)])$
is an algebraic function of $\psi$, as desired.
\end{proof}

Let us say a subset $S\subset \mathrm{Irr}_{\overline{\mathbf{Q}_{\ell}}}(G(F))$
is dense if any trace form on $\Groth(G(F))$ which vanishes on $S$ vanishes identically.
For instance, the Langlands classification implies that (for any fixed
choice of isomorphism $\mathbf{C}\simeq\overline{\mathbf{Q}_{\ell}}$)
tempered representations are dense, cf. \cite[Theorem 0]{KazhdanCuspidal}. 

\begin{lem} \label{IntegralRepresentationsDense}
The subset of irreducible representations $\pi\in\mathrm{Irr}_{\overline{\mathbf{Q}_{\ell}}}(G(F))$
admitting $\overline{\mathbf{Z}_{\ell}}$-lattices is dense.
\end{lem}

It seems reasonable to think of this lemma as an $\ell$-adic analogue of the density of tempered
representations.

\begin{proof} Let $f$ be a trace form, and assume that $f(\tau)=0$ for every $\tau \in\mathrm{Irr}_{\overline{\mathbf{Q}_{\ell}}}(G(F))$ admitting a $\overline{\mathbf{Z}_{\ell}}$-lattice.

By Proposition \ref{PropLanglands}, it's enough to show that $f( i_{M}^{G}(\sigma \psi))=0$ for any parabolic $P=MU \subset G$, any unramified character $\psi$ of $M(F)$, and any $\sigma \in\mathrm{Irr}_{\overline{\mathbf{Q}_{\ell}}}(M(F))$ admitting a $\overline{\mathbf{Z}_{\ell}}$-lattice. Fix $P$ and $\sigma$, and consider the function $g$ on unramified characters of $M(F)$ sending $\psi$ to $f \left( i_{M}^{G}(\sigma \psi) \right)$. By the easy direction of the trace Paley-Wiener theorem, $g$ is a regular function on the variety of unramified characters of $M(F)$. 

Let us say an unramified character $\psi$ is integral if it takes values in $\overline{\mathbf{Z}_{\ell}}^{\times}$. If $\psi$ is integral, then $\sigma \psi$ admits a $\overline{\mathbf{Z}_{\ell}}$-lattice, and hence also $i_{M}^{G}(\sigma \psi)$ admits a $\overline{\mathbf{Z}_{\ell}}$-lattice. In particular, if $\psi$ is integral and $i_{M}^{G}(\sigma \psi)$ is irreducible, then $g(\psi)=0$ by combining these observations with our assumption on $f$. Now integral characters are Zariski-dense in the variety of unramified characters of $M(F)$, and the subset $T$ of integral characters such that $i_{M}^{G}(\sigma \psi)$ is irreducible is also Zariski-dense (use \cite[Theorem 5.1]{Dat}). Since $g(\psi)=0$ for all $\psi \in T$, we deduce that $g \equiv 0$, so in particular \[ 0 = g(\psi) = f \left( i_{M}^{G}(\sigma \psi) \right) \] for all $\psi$. This gives the result. 
\end{proof}

\begin{proof}[Proof of Theorem \ref{TheoremCalculationOfCharacterNoLattice}] Fix $\phi$ as in the statement of the theorem, and consider the linear form \[f(-)=\tr(\phi | \Mant_{b,\mu}(-)) - \left[\mc{T}_{b,\mu}^{G_b\to G} (\trdist - )\right](\phi) \] on $\Groth(G_b(F))$.  By Theorem \ref{TheoremCalculationOfCharacter}, we know that $f(\rho)=0$ if $\rho$ admits a lattice. We need to show that $f$ vanishes identically. 

The key observation is that $f$ is a trace form. Indeed, $\tr(\phi | \Mant_{b,\mu}(-))$ is a trace form by Theorem \ref{TheoremContinuous}. Moreover, $\left[\mc{T}_{b,\mu}^{G_b\to G} (\trdist - )\right](\phi)$ is a trace form, since we can rewrite $\left[\mc{T}_{b,\mu}^{G_b\to G} (\trdist \rho )\right](\phi)$ as the trace of $\tilde{T}_{b,\mu}^{G\to G_b}(\phi) \in C_c(G_b(F)_{\elli},\overline{\mathbf{Q}_{\ell}})_{G_b(F)}$ acting on $\rho$. Thus $f$ is a difference of trace forms, and hence a trace form. Since $f(\rho)=0$ for any $\rho$ admitting a lattice, Lemma \ref{IntegralRepresentationsDense} now implies the desired result.

For the final claim about virtual characters, choose compatible $\overline{\mathbf{Q}_{\ell}}$-valued Haar measures $dg$ and $dg'$ on $G(F)$ and $G_b(F)$. Fix some $\rho$, and let $\Xi \in C(G(F)_{\sr} \stslash G(F),\overline{\mathbf{Q}_{\ell}})$ be the virtual character of $\Mant_{b,\mu}(\rho)$. Pick any $\phi \in C_c(G(F)_{\elli},\overline{\mathbf{Q}_{\ell}})$. Then \[\tr(\phi | \Mant_{b,\mu}(\rho)) = \int_{G(F)} \Xi (g) \phi(g) dg\] by definition. On the other hand, \[ \left[\mc{T}_{b,\mu}^{G_b\to G} (\trdist \rho )\right](\phi) =  \int_{G(F)} T_{b,\mu}^{G_b \to G}(\Theta_{\rho})(g) \phi(g) dg\] by compatibility of the Haar measures and Proposition \ref{PropCompatibility}. Combining these observations, we get an equality \[ \int_{G(F)} T_{b,\mu}^{G_b \to G}(\Theta_{\rho})(g) \phi(g) dg = \int_{G(F)} \Xi (g) \phi(g) dg\] for any $\phi \in C_c(G(F)_{\elli},\overline{\mathbf{Q}_{\ell}})$.  The result now follows by varying $\phi$.
\end{proof}

\begin{proof}[Proof of Theorem \ref{TheoremMainRedux}] The claimed equality in $\Groth(G(F))$ is an immediate consequence of Theorem \ref{TheoremCalculationOfCharacterNoLattice} and Theorem \ref{TheoremEndoscopicTraceRelation}.

For the claim regarding the error term, consider the virtual representation \[ \mathrm{err}=\Mant_{b,\mu}(\rho)- \sum_{\pi\in \Pi_\phi(G)} \dim \Hom_{S_\phi}(\delta_{\pi,\rho},r_\mu)\pi. \] By the first half of the theorem, we know that $\mathrm{err}$ is non-elliptic. By Theorem \ref{ThmNonelliptic}, it thus suffices to show that $\mathrm{err}$ is a virtual sum of supercuspidal representations. Since the packet $\Pi_{\phi}(G)$ is supercuspidal by assumption, we're reduced to showing that $\Mant_{b,\mu}(\rho)$ is a virtual supercuspidal representation. By definition, this is the Grothendieck class of the complex $A = i_{\mathbf{1}}^{\ast} T_{V^{\vee}_{\mu,\overline{\mathbf{Q}_{\ell}}}} i_{b \ast} \rho \in D(G(F),\overline{\mathbf{Q}_{\ell}})$, so we need to see that any irreducible $\tau$ occurring in the Jordan-H\"older series of $H^\ast (A)$ is supercuspidal. Since $\varphi_{\tau} = \varphi_{\rho}$ by the commutation of Hecke operators with excursion operators, the claim now follows from the assumption on $\varphi_{\rho}$ and \cite[Theorem I.9.6.viii]{FarguesScholze}.
\end{proof}

\subsection{Application to inner forms of $\GL_n$}

We give an application to the local Langlands correspondence. Recall that for any $G/F$, any $b\in B(G)$ and any $\tau \in \mathrm{Irr}(G_b(F))$, the construction in \cite[Proposition I.9.1]{FarguesScholze} (applied to $A= i_{b!} \tau$) gives rise to a semisimple $L$-parameter $\varphi_{\tau}:W_F \to \phantom{}^L G(\overline{\mathbf{Q}_{\ell}})$ associated with $\tau$. This construction is canonical and satisfies a long list of desirable properties \cite[Theorem I.9.6]{FarguesScholze}. However, it is a highly nontrivial problem to compare this construction with ``previously known" realizations of the local Langlands correspondence.

\begin{thm} Let $G$ be any inner form of $\mathrm{GL}_n/F$, and let $\pi$ be an irreducible smooth representation of $G(F)$. Then the $L$-parameter $\varphi_\pi$ associated with $\pi$ as in \cite[\S I.9]{FarguesScholze} agrees with the usual semisimplified $L$-parameter attached to $\pi$.
\end{thm}

\begin{proof}
By \cite[Theorem I.9.6.viii]{FarguesScholze}, we can assume $\pi$ is supercuspidal. Pick some basic $b$ with $G_b = \mathrm{GL}_{n}/F$, and let $\rho \in \mathrm{Irr}(G_b(F))$ be the Jacquet-Langlands transfer of $\pi$ \cite{DKV}, so the (usual) semisimple $L$-parameters of $\rho$ and $\pi$ agree. By \cite[Theorem 1.9.6.viii-ix]{FarguesScholze}, we know that $\varphi_\rho$ agrees with the usual semisimple $L$-parameter of $\rho$. To prove the theorem, it thus suffices to show that $\varphi_\pi = \varphi_\rho$.

Pick any $\mu$ such that $b \in B(G,\mu)$. By Theorem \ref{TheoremCalculationOfCharacterNoLattice} and the usual character relation characterizing the Jacquet-Langlands correspondence, we have an equality $\Mant_{b,\mu}(\rho) = \dim V_{\mu} \cdot \pi + e$ in $\Groth(G(F))$, where $e$ is a non-elliptic virtual representation. Since $\pi$ is supercuspidal, this implies that $\pi$ occurs as a subquotient of some cohomology group of the complex $A = i_{\mathbf{1}}^{\ast} T_{V^{\vee}_{\mu,\overline{\mathbf{Q}_{\ell}}}} i_{b \ast} \rho \in D(G(F),\overline{\mathbf{Q}_{\ell}})$. But Hecke operators commute with excursion operators, so $\varphi_{\tau} = \varphi_{\rho}$ for any irreducible $\tau$ occurring in the Jordan-Holder series of $H^\ast (A)$. 
\end{proof}

\appendix
%!TEX root = FixedPoints2022.tex
\section{Endoscopy}
\label{LLCReview}

\subsection{Endoscopic character relations}
\label{CharacterRelations}

We recall here the endoscopic character identities, which are part of the refined local Langlands correspondence, following the formulation of \cite[\S5.4]{KalRI}, also recalled in \cite[\S4.2]{KalethaLocalLanglands}. They play a key role in the proof of Theorem \ref{TheoremEndoscopicTraceRelation}. We recall the notation established before the statement of that theorem.

\begin{itemize}
\item $F/\Q_p$ is a finite extension, $F^\tx{nr}/F$ a maximal unramified extension.
\item $G$ is a connected reductive group defined over $F$.
\item $G^*$ is a quasi-split connected reductive group defined over $F$.
\item $\Psi$ is a $G^*$-conjugacy class of inner twists  $\psi\from G^*\to G$.
\item $\bar z_\sigma=\psi^{-1}\sigma(\psi)\in G^*_{\tx{ad}}$, so that $\bar z\in Z^1(F,G^*_{\tx{ad}})$.
\item $z\in Z^1(u\to W,Z(G^*)\to G^*)$ is a lift of $\bar z$.
\item $b\in G(F^{\tx{nr}})$ is a decent basic element.
\item $G_b$ is the corresponding inner form of $G$.
\item $\xi\from G_{F^{\tx{nr}}}\to G_{b,F^{\tx{nr}}}$ is the identity map.
\item $z_b \in Z^1(u\to W,Z(G)\to G)$ and $g \in G(\ol{\breve F})$ satisfy \eqref{eq:zb}.
\item $\mf{w}$ is a Whittaker datum for $G^*$.
\item $\phi\from W_F \times \mathrm{SL}_2 \to {^LG}$ is a discrete $L$-parameter.
\item $S_\phi=\Cent(\phi,\hat{G})$.
\item $S_\phi^+$ is the group defined in Definition \ref{DefinitionOfSphiplus}.
\item $\lambda_z$ resp. $\lambda_{z_b}$ the image of the class of $z$ resp. $z_b$ under the isomorphism $H^1(u \to W,Z(G^*) \to G^*) \to \pi_0(Z(\hat{\bar G})^+)^*$.
\end{itemize}

% Thus $G$ is a connected reductive group defined over a finite extension $F$ of $\Q_p$, $b \in G(F^\tx{nr})$ is a decent basic element, $G_b$ the corresponding inner form. We fix an inner twising $\psi : G^* \to G$ with $G^*$ quasi-split, let $\bar z_\sigma = \psi^{-1}\sigma(\psi) \in G^*_\tx{ad}$ and let $z \in Z^1(u \to W,Z(G^*) \to G^*)$ be a lift of $\bar z \in Z^1(F,G^*_\tx{ad})$. Let $z_b \in Z^1(u \to W,Z(G) \to G)$ be the image of $b$ under \eqref{eq:rivsbg} and let $\xi : G_{F^\tx{nr}} \to J_{b,F^\tx{nr}}$ denote the identity map. Thus $(\psi,z) : G^* \to G$ and $(\xi\circ\psi,\psi^{-1}(z_b)\cdot z) : G^* \to G_b$ are rigid inner twists. Let $\mf{w}$ be a Whittaker datum for $G^*$. Given a discrete parameter $\phi : W_F \to {^LG}$,

Recall that $\tx{Ad}(g) : G_{z_b} \to G_b$ is an $F$-isomorphism. We will use it to identify the two groups, and drop $g$ from the notation. We will use the letter $g$ for a different purpose below.
 
Associated to $\phi$ are the $L$-packets $\Pi_\phi(G)$ and $\Pi_\phi(G_{z_b})$ and the bijections
\[
\Pi_{\phi}(G) \to \Irr(\pi_0(S_{\phi}^+),\lambda_{z}),\qquad
\Pi_{\phi}(G_{z_b}) \to \Irr(\pi_0(S_{\phi}^+),\lambda_{z}+\lambda_{z_b})
\]
denoted by $\pi \mapsto \tau_{z,\mf{w},\pi}$ and $\rho \mapsto \tau_{z,\mf{w},\rho}$.

We now choose  a semi-simple element $s \in S_\phi$ and an element $\dot s \in S_\phi^+$ which lifts $s$. Let $e(G)$ and $e(G_{z_b})$ be the Kottwitz signs of the groups $G$ and $G_{z_b}$, as defined in \cite{Kot83}. Of course $e(G_{z_b})=e(G_b)$. Consider the virtual characters
\[ e(G)\sum_{\pi \in \Pi_\phi(G)} \tr\tau_{z,\mf{w},\pi}(\dot s) \cdot \Theta_\pi\qquad\tx{and}\qquad e(G_{z_b})\sum_{\rho \in \Pi_\phi(G_{z_b})} \tr\tau_{z,\mf{w},\rho}(\dot s) \cdot \Theta_\rho.\]
The endoscopic character identities are equations which relate these two virtual characters to virtual characters on an endoscopic group $H_1$. From the pair $(\phi,\dot s)$ one obtains a refined elliptic endoscopic datum
\begin{equation} \label{eq:red}
\mf{\dot e}=(H,\mc{H},\dot s,\eta)
\end{equation}
in the sense of \cite[\S5.3]{KalRI} as follows. Let $\hat H=\tx{Cent}(s,\hat G)^\circ$. The image of $\phi$ is contained in $\tx{Cent}(s,\hat G)$, which in turns acts by conjugation on its connected component $\hat H$. This gives a homomorphism $W_F \to \tx{Aut}(\hat H)$. Letting $\Psi_0(\hat H)$ be the based root datum of $\hat H$ \cite[\S1.1]{Kot84} and $\Psi_0^\vee(\hat H)$ its dual, we obtain the homomorphism
\[ W_F \to \tx{Aut}(\hat H) \to \tx{Out}(\hat H) = \tx{Aut}(\Psi_0(\hat H)) = \tx{Aut}(\Psi_0(\hat H)^\vee). \]
Since the target is finite, this homomorphism extends to $\Gamma_F$ and we obtain a based root datum with Galois action, hence a quasi-split connected reductive group $H$ defined over $F$. Its dual group is by construction equal to $\hat H$. We let $\mc{H}=\hat H \cdot \phi(W_F)$, noting that the right factor normalizes the left so their product $\mc{H}$ is a subgroup of $^LG$. Finally, we let $\eta : \mc{H} \to {^LG}$ be the natural inclusion. Note that by construction $\phi$ takes image in $\mc{H}$, i.e. it factors through $\eta$.

We can realize the $L$-group of $H$ as $^LH = \hat H \rtimes W_F$, but we caution the reader that $W_F$ does not act on $\hat H$ via the map $W_F \to \tx{Aut}(\hat H)$ given by $\phi$ as above. Rather, we have to modify this action to ensure that it preserves a pinning of $\hat H$. More precisely, after fixing an arbitrary pinning of $\hat H$ we obtain a splitting $\tx{Out}(\hat H) \to \tx{Aut}(\hat H)$ of the projection $\tx{Aut}(\hat H) \to \tx{Out}(\hat H)$ and the action of $W_F$ on $\hat H$ we use to form $^LH$ is given by composing the above map $W_F \to \tx{Out}(\hat H)$ with this splitting.

Both $^LH$ and $\mc{H}$ are thus extensions of $W_F$ by $\hat H$, but they need not be isomorphic. If they are, we fix arbitrarily an isomorphism $\eta_1 : \mc{H} \to {^LH}$ of extensions. Then $\phi^s = \eta_1 \circ \phi$ is a discrete parameter for $H$.

In the general case we need to introduce a $z$-pair $\mf{z}=(H_1,\eta_1)$ as in \cite[\S2]{KS99}. It consists of a $z$-extension $H_1 \to H$ (recall this means that $H_1$ has a simply connected derived subgroup and the kernel of $H_1 \to H$ is an induced torus) and $\eta_1 : \mc{H} \to {^LH_1}$ is an $L$-embedding that extends the natural embedding $\hat H \to \hat H_1$. As is shown in \cite[\S2.2]{KS99}, such a $z$-pair always exists. Again we set $\phi^s = \eta_1 \circ \phi$ and obtain a discrete parameter for $H_1$. In the situation where an isomorphism $\eta_1 : \mc{H} \to {^LH}$ does exist, we will allows ourselves to take $H=H_1$ and so regard $\mf{z}=(H,\eta_1)$ as a $z$-pair, even though in general $H$ will not have a simply connected derived subgroup.

The virtual character on $H_1$ that the above virtual characters on $G$ and $G_{z_b}$ are to be related to is
\[ S\Theta_{\phi^s} := \sum_{\pi^s \in \Pi_{\phi^s}(H_1)} \tx{dim}(\tau_{\pi^s})\Theta_{\pi^s}. \]
Here $\pi^s \mapsto \tau_{\pi^s}$ is a bijection $\Pi_{\phi^s}(H_1) \to \tx{Irr}(\pi_0(\tx{Cent}(\phi^s,\hat H_1)/Z(\hat H_1)^\Gamma))$ determined by an arbitrary choice of Whittaker datum for $H_1$. The argument in the proof of Lemma \ref{lem:deltaindep} shows the independence of $\tx{dim}(\tau_{\pi^s})$ of the choice of a Whittaker datum for $H_1$.

The relationship between the virtual characters on $G$, $G_{z_b}$, and $H_1$, is expressed in terms of the Langlands-Shelstad transfer factor $\Delta_\tx{abs}'[\mf{\dot e},\mf{z},\mf{w},(\psi,z)]$ for the pair of groups $(H_1,G)$ and the corresponding Langlands-Shelstad transfer factor $\Delta_\tx{abs}'[\mf{\dot e},\mf{z},\mf{w},(\xi\circ\psi,\psi^{-1}(z_b)\cdot z)]$ for the pair of groups $(H_1,G_{z_b})$, both of which are defined by \cite[(5.10)]{KalRI}. We will abbreviate both of them to just $\Delta$. It is a simple consequence of the Weyl integration formula that the character relation \cite[(5.11)]{KalRI} can be restated in terms of character functions (rather than character distributions) as
\begin{equation}
\label{FormulaForThetaPi}
e(G)\sum_{\pi\in \Pi_\phi(G)} \tr\tau_{z,\mf{w},\pi}(\dot s)\Theta_\pi(g)=
\sum_{h_1\in H_1(F)/\tx{st.}} \Delta(h_1,g)S\Theta_{\phi^s}(h_1)
\end{equation}
for any strongly regular semi-simple element $g \in G(F)$. The sum on the right runs over stable conjugacy classes of strongly regular semi-simple elements of $H_1(F)$. We also have the analogous identity for $G_{z_b}$:
\begin{equation}
\label{FormulaForThetaRho}
 e(G_{z_b})\sum_{\rho\in \Pi_\phi(G_{z_b})} \tr\tau_{z,\mf{w},\rho}(\dot s)\Theta_\rho(g')=
\sum_{h_1\in H_1(F)/\tx{st.}} \Delta(h_1,g')S\Theta_{\phi^s}(h_1).
\end{equation}
For the purposes of this paper, we are only interested in the right hand sides of these two equations as a bridge between their left-hand sides. Essential for this bridge is a certain compatibility between the transfer factors appearing on both right-hand sides:

\begin{lem} \label{lem:RelationBetweenTransferFactors}
\begin{equation}
\label{RelationBetweenTransferFactors}
 \Delta(h_1,g') = \Delta(h_1,g) \cdot \<\tx{inv}[b](g,g'),s^\natural_{h,g}\>.
\end{equation}
\end{lem}
We need to explain the second factor. Given maximal tori $T_H \subset H$ and $T \subset G$, there is a notion of an admissible isomorphism $T_H \to T$, for which we refer the reader to \cite[\S1.3]{KalethaLocalLanglands}. Two strongly regular semi-simple elements $h\in H(\Q_p)$ and $g\in G(\Q_p)$ are called \emph{related} if there exists an admissible isomorphism $T_h \to T_g$ between their centralizers mapping $h$ to $g$. If such an isomorphism exists, it is unique, and in particular defined over $F$, and shall be called $\varphi_{h,g}$. An element $h_1 \in H_1(F)$ is called related to $g \in G(F)$ if and only if its image $h \in H(F)$ is so. Since $g$ and $g'$ are stably conjugate, an element $h_1 \in H_1(F)$ is related to $g$ if and only if it is related to $g'$. If that is not the case, both $\Delta(h_1,g')$ and $\Delta(h_1,g)$ are zero and \eqref{RelationBetweenTransferFactors} is trivially true. Thus assume that $h_1$ is related to both $g$ and $g'$. Let $s^\natural \in S_\phi$ be the image of $\dot s$ under \eqref{eq:rivsbgllc}. Note that $s^\natural \in s \cdot Z(\hat G)^{\circ,\Gamma}$ and hence the preimage of $s^\natural$ under $\eta$ belongs to $Z(\hat H)^\Gamma$, which in turns embeds naturally into $\hat T_h^\Gamma$. Using the admissible isomorphism $\varphi_{h,g}$ we transport $s^\natural$ into $\hat T_g^\Gamma$ and denote it by $s^\natural_{h,g}$. It is then paired with $\tx{inv}[b](g,g')$ via the isomorphism $B(T_g) \cong X^*(\hat T_g^\Gamma)$ of \cite[\S2.4]{KottwitzIsocrystals}.
\begin{proof}
For every finite subgroup $Z \subset Z(G) \subset T_g$ one obtains from $\varphi_{h,g}$ an isomorphism $T_h/\varphi_{h,g}^{-1}(Z) \to T_g/Z$. Using the subgroups $Z_n$ from \S\ref{sub:geninner} we form the quotients $T_{h,n}=T_h/\varphi_{h,g}^{-1}(Z_n)$ and $T_{g,n}=T_g/Z_n$. From $\varphi_{h,g}$ we obtain an isomorphism
\[ \hat{\bar T_h} \to \hat{\bar T_g} \]
between the limits over $n$ of the tori dual to $T_{h,n}$ and $T_{g,n}$. Let $\dot s_{h,g} \in [\hat{\bar T_g}]^+$ be the image of $\dot s$ under this isomorphism. Let $\tx{inv}[z_b](g,g') \in H^1(u \to W,Z(G) \to T_g)$ be the invariant defined in \cite[\S5.1]{KalRI}. If we replace $\<\tx{inv}[b](g,g'),s^\natural_{h,g}\>$ by $\<\tx{inv}[z_b](g,g'),\dot s_{h,g}\>$ then the lemma follows immediately from the defining formula \cite[(5.10)]{KalRI} of the transfer factors. The lemma follows from the equality $\<\tx{inv}[b](g,g'),s^\natural_{h,g}\> = \<\tx{inv}[z_b](g,g'),\dot s_{h,g}\>$ proved in \cite[\S4.2]{KalRIBG}.
\end{proof}

\subsection{The Kottwitz sign}
\label{SectionKottwitzSign}

We will give a formula for the Kottwitz sign $e(G)$ in terms of the dual group $\hat G$. Fix a quasi-split inner form $G^*$ and an inner twisting $\psi\from G^* \to G$. Let $h \in H^1(\Gamma,G^*_\tx{ad})$ be the class of $\sigma \mapsto \psi^{-1}\sigma(\psi)$. Via the Kottwitz homomorphism \cite[Theorem 1.2]{Kot86} the class $h$ corresponds to a character $\nu \in X^*(Z(\hat G_\tx{sc})^\Gamma)$.

Choose an arbitrary Borel pair $(\hat T_\tx{sc},\hat B_\tx{sc})$ of $\hat G_\tx{sc}$ and let $2\rho \in X_*(\hat T_\tx{sc})$ be the sum of the $\hat B_\tx{sc}$-positive coroots. The restriction map $X^*(\hat T_\tx{sc}) \to X^*(Z(\hat G_\tx{sc}))$ is surjective and we can lift $\nu$ to $\dot\nu \in X^*(\hat T_\tx{sc})$ and form $\<2\rho,\dot\nu\> \in \Z$. A different lift $\dot\nu$ would differ by an element of $X^*(\hat T_\tx{ad})$, and since $\rho \in X_*(\hat T_\tx{ad})$ we see that the image of $\<2\rho,\dot\nu\>$ in $\Z/2\Z$ is independent of the choice of lift $\dot\nu$. We thus write $\<2\rho,\nu\> \in \Z/2\Z$. Since any two Borel pairs in $\hat G_\tx{sc}$ are conjugate $\<2\rho,\nu\>$ does not depend on the choice of $(\hat T_\tx{sc},\hat B_\tx{sc})$.

\begin{lem} \label{lem:kotsign}
\[ e(G) = (-1)^{\<2\rho,\nu\>}. \]
\end{lem}

\begin{proof}
We fix $\Gamma$-invariant Borel pairs $(T_\tx{ad},B_\tx{ad})$ in $G^*_\tx{ad}$ and $(\hat T_\tx{sc},\hat B_\tx{sc})$ in $\hat G_\tx{sc}$. Then we have the identification $X^*(T_\tx{ad})=X_*(\hat T_\tx{sc})$. Let $(T_\tx{sc},B_\tx{sc})$ be the preimage in $G^*_\tx{sc}$ of $(T_\tx{ad},B_\tx{ad})$.

By definition the Kottwitz sign is the image of $h$ under
\[ \xymatrix{ H^1(\Gamma,G^*_\tx{ad})\ar[r]^-\delta&H^2(\Gamma,Z(G^*_\tx{sc}))\ar[r]^-\rho&H^2(\Gamma,\{\pm 1\})\ar[r]&\{\pm 1\}, }\]
where $\rho \in X^*(T_\tx{sc})$ is half the sum of the $B_\tx{sc}$-positive roots and its restriction to $Z(G^*_\tx{sc})$ is independent of the choice of $(T_\tx{ad},B_\tx{ad})$. By functoriality of the Tate-Nakayama pairing this is the same as pairing $\delta h \in H^2(\Gamma,Z(G^*_\tx{sc}))$ with $\rho \in H^0(\Gamma,X^*(Z(G^*_\tx{sc})))$. The canonical pairing $X^*(T_\tx{ad}) \otimes X^*(\hat T_\tx{sc}) \to \Z$ induces the perfect pairing $X^*(T_\tx{sc})/X^*(T_\tx{ad}) \otimes X^*(\hat T_\tx{sc})/X^*(\hat T_\tx{ad}) \to \Q/\Z$ and hence the isomorphism $X^*(Z(G^*_\tx{sc})) \to \tx{Hom}_\Z(X^*(Z(\hat G_\tx{sc})),\Q/\Z) = Z(\hat G_\tx{sc})$, where the last equality uses the exponential map. Under this isomorphism $\rho \in X^*(Z(G^*_\tx{sc}))^\Gamma$ maps to the element $(-1)^{2\rho} \in Z(\hat G_\tx{sc})^\Gamma$ obtained by mapping $(-1) \in \C^\times$ under $2\rho \in X^*(T_\tx{ad})=X_*(\hat T_\tx{sc})$. The lemma now follows from \cite[Lemma 1.8]{Kot86}.
\end{proof}

\section{Elementary Lemmas}

\subsection{Homological algebra}
\label{SectionHomologicalAlgebra}

\begin{lem} \label{lem:reflexivemodulesnaive}
Let $R$ be a discrete valuation ring with maximal ideal $\mf{m}$. Let $\kappa=R/\mf{m}$ be the residue field, and let $\Lambda=R/\mf{m}^k$ for some $k>0$. For a $\Lambda$-module $M$ we have the dual module $M^*=\tx{Hom}_\Lambda(M,\Lambda)$ and the natural morphisms $M \to M^{**}$ and $(M^* \otimes M) \to (M \otimes M^*)^*$.

The morphism $M \to M^{**}$ is an isomorphism if and only if $M$ is finitely generated.
\end{lem}
\begin{proof}
For the ``if'' direction of the first point we note that the structure theorem for $R$-modules implies that a finitely generated $\Lambda$-module is a direct sum of finitely many cyclic $\Lambda$-modules, and each cyclic $\Lambda$-module is isomorphic to its own double dual.

Conversely assume that $M\to M^{**}$ is an isomorphism. We induct on $k$.  If $k=1$, then $\Lambda$ is a field, and this is well-known.  For general $k$ we consider $N=M/\mf{m} M$.  The ring $\Lambda$ is an Artinian serial ring, and hence it is injective as a module over itself.  Thus the dualization functor is exact, and we get a commutative diagram
\begin{equation}
\label{EquationSnakeLemma}
\xymatrix{
0\ar[r] & \mf{m} M \ar[r] \ar[d] & M \ar[r] \ar[d]_{\isom} & N \ar[r] \ar[d] & 0\\
0 \ar[r] & (\mf{m} M)^{**} \ar[r] & M^{**} \ar[r] & N^{**} \ar[r] & 0,
}
\end{equation}
which shows that the right-most vertical map is surjective and the left-most vertical map is injective.

We have an isomorphism of $\Lambda$-modules $\mf{m}^{m-1}\Lambda \to \kappa$, from which we obtain
\[ N^*=\Hom_{\Lambda}(N,\Lambda)=
\Hom_{\Lambda}(N,\mf{m}^{m-1}\Lambda)
\isom \Hom_{\kappa}(N,\kappa). \]
Thus $N^{**}$ is also the double dual of $N$ in the category of $\kappa$-vector spaces, and it is easy to check that the right-most vertical map in \eqref{EquationSnakeLemma} is the canonical map in that category. Thus, this map is an isomorphism, and $N$ is fintely generated as a $\kappa$-vector space.

By the Snake Lemma, the left-most vertical arrow in \eqref{EquationSnakeLemma} is an isomorphism. We can apply the inductive hypothesis to the $(\Lambda/\mf{m}^{m-1})$-module $\mf{m} M$ and conclude that it is finitely generated. Thus so is $M$.

\end{proof}

\begin{lem} \label{lem:reflexivemodules} Let $\Lambda$ be an arbitrary ring, and let $D(\Lambda)$ be the derived category of $\Lambda$-modules.  For an object $M$ of $D(\Lambda)$, let $\D M=\Rhom(M,\Lambda[0])$.
\begin{enumerate}
\item Assume that $\Lambda=R/\mf{m}^k$ for a discrete valuation ring $R$ with maximal ideal $\mf{m}$.   Then the natural morphism $M\to \D\D M$ is an isomorphism if and only if each $H^i(M)$ is finitely generated.
\item For general $\Lambda$, the following are equivalent:
\begin{enumerate}
\item The natural maps $M\to \D\D M$ and $\D M\otimes M\to \D(M\otimes \D M)$ are isomorphisms.
\item The natural map $M\otimes \D M\to \Rhom(M,M)$ is an isomorphism.
\item $M$ is strongly dualizable;  that is, for any object $N$, $N\otimes \D M\to \Rhom(M,N)$ is an isomorphism.
\item $M$ is a compact object;  that is, the functor $N\mapsto \Rhom(M,N)$ commutes with colimits.
\item $M$ is a perfect complex; that is, $M$ is isomorphic to a bounded complex of finitely generated projective $\Lambda$-modules.
\end{enumerate}
\end{enumerate}
(Throughout, the $\otimes$ means derived tensor product.)
\end{lem}
\begin{proof}
For the first statement, the self-injectivity of $\Lambda$ implies that $H^i(\D M)\isom H^{-i}(M)^*$, so that $H^i(\D \D M)\isom H^i(M)^{**}$.  Therefore $M\to \D\D M$ is an isomorphism if and only if each $H^i(M)\to H^i(M)^{**}$ is an isomorphism.  By Lemma \ref{lem:reflexivemodulesnaive}, this is equivalent to each $H^i(M)$ being finitely generated.

We now turn to the second statement.   For (a)$\implies$(b), assume that $M\to \D \D M$ and $\D M\otimes M\to \D(M\otimes \D M)$ are isomorphisms.  Then $\Rhom(M,M)\isom \Rhom(M,\D \D M)\isom \Rhom(M\otimes \D M, \Lambda)\isom \D(M\otimes \D M)\isom \D M\otimes M$.

For (b)$\implies$(c), the identity map on $M$ induces a morphism $\varepsilon\from \Lambda[0]\to \Rhom(M,M)\isomto M\otimes \D M$ (the coevaluation map).  The required inverse to $N\otimes \D M\to \Rhom(M,N)$ is
\[ \Rhom(M,N)\stackrel{\text{id}\otimes \varepsilon}{\to} \Rhom(M,N)\otimes M\otimes \D M \to N\otimes \D M.\]

For (c)$\implies$(d), we use the fact that $\otimes$ commutes with colimits.

For (d)$\implies$(e), we use the fact that compact objects of $D(\Lambda)$ are perfect \cite[Tag 07LT]{Stacks}.

Finally, for (e) implies (a), we can write $M$ as a bounded complex of finitely generated projective $\Lambda$-modules.  Then duals and derived tensor products can be computed on the level of chain complexes.  We are reduced to showing, for finitely generated projective $\Lambda$-modules $A$ and $B$, that $A\to A^{**}$ and $A^*\otimes B\to (A\otimes B^*)^*$ are isomorphisms.   After localizing on $\Lambda$, we may assume that $A$ and $B$ are free of finite rank (since duals commute over direct sums), where these statements are easy to check.
\end{proof}

We thank Bhargav Bhatt for helping us with the above proof.

\subsection{Sheaves on locally profinite sets} \label{Sheaveslpfsets}

Let $S$ be a locally profinite set and $\Lambda$ a discrete ring. We have the ring $C(S,\Lambda)$ of locally constant functions on $S$, and the non-unital ring $C_c(S,\Lambda)$ of locally constant compactly supported functions on $S$. For each compact open subset $U \subset S$ let $\tb{1}_U$ denote the characteristic function. Then $C(U,\Lambda)$ is a principal ideal of both $C_c(S,\Lambda)$ and $C(S,\Lambda)$ generated by $\tb{1}_U$. Multiplication by $\tb{1}_U$ is a homomorphism $C(S,\Lambda) \to C(U,\Lambda)$ of rings with unity. In this way every $C(U,\Lambda)$-module becomes a $C(S,\Lambda)$-module.

\begin{dfn} \label{dfn:modlps}
We call a $C(S,\Lambda)$-module $M$ 
\begin{enumerate}
    \item \emph{smooth} if it satisfies the following equivalent conditions
    \begin{enumerate}
        \item The multiplication map $M \otimes_{C(S,\Lambda)}C_c(S,\Lambda) \to M$ is an isomorphism.
	\item The natural map $\varinjlim (\tb{1}_U \cdot M) \to M$ is an isomorphism, where the colimit runs over the open compact subsets $U \subset S$ and the transition map $\tb{1}_U\cdot M \to \tb{1}_V\cdot M$ for $U \subset V$ is given by the natural inclusion.
    \end{enumerate}
    \item \emph{complete} if the natural map $M \to \varprojlim_U (\tb{1}_U \cdot M)$ is an isomorphism, where again $U$ runs over the open compact subsets of $S$ and the transition map $\tb{1}_U\cdot M \to \tb{1}_V\cdot M$ for $V \subset U$ is multiplication by $\tb{1}_V$.
\end{enumerate}
\end{dfn}

\begin{lem}
Let $V \subset S$ be compact open and let $M$ be any $C(S,\Lambda)$-module. Then 
\begin{enumerate}
    \item $\tb{1}_{V}\cdot M$ is a submodule of $\varinjlim_U (\tb{1}_U \cdot M)$ and equals $\tb{1}_V\cdot \varinjlim (\tb{1}_U\cdot M)$.
    \item $\tb{1}_{V}\cdot M$ is a submodule of $\varprojlim_U (\tb{1}_U \cdot M)$ and equals $\tb{1}_V\cdot \varprojlim (\tb{1}_U\cdot M)$.
\end{enumerate}
\end{lem}

\begin{lem}
\begin{enumerate}
	\item The functor $M \mapsto M^s:=\varinjlim (\tb{1}_U \cdot M)$ is a projector onto the category of smooth modules.
	\item The functor $M \mapsto M^c:=\varprojlim (\tb{1}_U \cdot M)$ is a projector onto the category of complete modules.
	\item The two functors give mutually inverse equivalences of categories between the categories of smooth and complete modules.
\end{enumerate}
\end{lem}

Let $\mc{B}$ the set of open compact subsets of $S$. Then $\mc{B}$ is a basis for the topology of $S$ and is closed under taking finite intersections and finite unions. Restriction gives an equivalence between the category of sheaves on $S$ and the category of sheaves on $\mc{B}$. Define $R(U)=C(U,\Lambda)$. This is a sheaf of rings on $S$.

Let $\mc{F}$ be an $R$-module sheaf on $S$. For $U \in \mc{B}$ we extend the $R(U)$-module structure on $\mc{F}(U)$ to a $C(S,\Lambda)$-module structure as remarked above. Then the restriction map $\mc{F}(S) \to \mc{F}(U)$ becomes a morphism of $C(S,\Lambda)$-modules.

\begin{lem} \label{lem:flabby}
\begin{enumerate}
	\item For any $U \in \mc{B}$ the restriction map $\mc{F}(S) \to \mc{F}(U)$ is surjective and its restriction to $\tb{1}_U \cdot \mc{F}(S)$ is an isomorphism $\tb{1}_U \cdot \mc{F}(S) \to \mc{F}(U)$.
	\item We have $\mc{F}(S) = \varprojlim_U \mc{F}(U)$, where the transition maps are the restriction maps.
\end{enumerate}
\end{lem}

Let $M$ be an $C(S,\Lambda)$-module. Let $\mc{F}_M(U)=R(U)M=\tb{1}_UM$. This is a $C(S,\Lambda)$-submodule of $M$. Given $V,U \in \mc{B}$ with $V \subset U$ we have the map $\mc{F}_M(U) \to \mc{F}_M(V)$ defined by multiplication by $\tb{1}_V$. In this way $\mc{F}_M$ becomes an $R$-module sheaf.

Let $f : M \to N$ be a morphism of $C(S,\Lambda)$-modules. We define for each $U$ the morphism $f_U : \mc{F}_M(U) \to \mc{F}_N(U)$ simply by restricting $f$ to $\mc{F}_M(U)$. One checks immediately that $(f_U)_U$ is a morphism of sheaves of $R$-modules. Therefore we obtain a functor from the category of $C(S,\Lambda)$-modules to the category of sheaves of $R$-modules.

Given a sheaf $\mc{F}$ on $S$ we can define the smooth module $M_\mc{F}^s$ and the complete module $M_\mc{F}^c$ by
\[ M_\mc{F}^s = \varinjlim_U \mc{F}(U)\qquad M_\mc{F}^c = \varprojlim_U \mc{F}(U)\]
where the limit is taken over the restriction maps, and the colimit is taken over their sections given by Lemma \ref{lem:flabby}, and in both cases $U$ runs over $\mc{B}$. Conversely, given any $C(S,\Lambda)$-module $M$ we have the sheaf $\mc{F}_M$.

\begin{lem} \label{lem:sheaflps}
These functors give mutually inverse equivalences of categories from the category of smooth (resp. complete) $C(S,\Lambda)$-modules to the category of $R$-module sheaves. These equivalences commute with the equivalence between the categories of smooth and complete modules. Furthermore $\mc{F}_M(S)=M^c$.
\end{lem}

\section{Some representation theory}

Let $G$ be a reductive group over a finite extension $F/\mathbf{Q}_p$. For a parabolic subgroup $P$ of $G$ we write $i_P^G$ for the functor of normalized parabolic induction and $r_G^P$ for the normalized Jacquet module functor.

Fix a minimal parabolic $P_0=M_0 U_0$. A parabolic subgroup $P$ is called standard if it contains $P_0$. There is a unique Levi factor $M$ of $P$ that contains $M_0$, and conversely $M$ determines $P$. In that situation we may write $i_M^G$ and $r_G^M$ in place of $i_P^G$ and $r_G^P$.

\subsection{Non-elliptic representations}

Recall that a finite-length (virtual) $G(F)$-representation is \emph{non-elliptic} if its Harish-Chandra character vanishes on all elliptic elements. Our goal in this section is the following result.

\begin{thm}\label{ThmNonelliptic} Let $\pi \in \Groth(G(F))$ be any finite-length virtual $G(F)$-representation with $\mathbf{C}$-coefficients, or with $\overline{\mathbf{Q}_{\ell}}$-coefficients. Then $\pi$ is non-elliptic if and only if it can be expressed as a $\mathbf{Q}$-linear virtual combination of representations induced from proper parabolic subgroups of $G$.
\end{thm}

When $G(F)$ has compact center, this is (a weaker version of) a classical result of Kazhdan \cite{KazhdanCuspidal}. The general statement seems to be well-known to experts, but we were unable to find an explicit formulation in the literature.

\begin{proof}It suffices to treat the case of complex coefficients. Parabolic inductions are non-elliptic by van Dijk's formula \cite{vanDijk}, so the ``if" direction is clear. We will deduce the ``only if" direction from \cite{DatK0}; in what follows we freely use various notations from loc. cit., in particular writing $\mathscr{R}(G)$ for the Grothendieck group of finite length smooth $\mathbf{C}$-representations of $G(F)$. 

Suppose that $\pi \in \mathscr{R}(G)$ is non-elliptic. Following the notation of \cite{DatK0}, pick any $f \in \overline{\mathscr{H}}^{d(G)}(G)$. Then all regular semisimple non-elliptic orbital integrals of $f$ vanish by \cite[Theorem 3.2.iii]{DatK0}, so $\tr (f | \pi) = 0$ by our assumption on $\pi$ and the Weyl integration formula. Therefore $\pi \in  \overline{\mathscr{H}}^{d(G)} (G)^{\perp}$, so $\pi \in \mathscr{R}_{\mathbf{C}_{d(G)}}(G)$ by \cite[Theorem 3.2.ii]{DatK0}. Now applying \cite[Proposition 2.5.i]{DatK0} to the Hopf system $\mathscr{A}(-) = \mathscr{R}(-) \otimes \mathbf{Q}$ with $d=d(G)$, we see that $\pi \in \mathscr{R}(G) \otimes \mathbf{Q}$ is annihilated by the operator $1-\sum_{d(M)>d(G)}c_d(M) i_{M}^{G} r_{G}^{M}$ for some rational numbers $c_d(M)$. Therefore \[\pi = \sum_{d(M)>d(G)}c_d(M) i_{M}^{G} r_{G}^{M}(\pi),\] and the right-hand side is a $\mathbf{Q}$-linear virtual combination of proper parabolic inductions, giving the result.
\end{proof}

\subsection{Integral representations and parabolic inductions}
 Fix a prime $\ell \neq p$. As usual, let $\Groth(G(F))$ be the Grothendieck group of finite-length smooth $\overline{\mathbf{Q}_{\ell}}$-representations of $G(F)$.

\begin{dfn} \label{DefLattice} Let $\pi$ be an admissible smooth $\overline{\mathbf{Q}_{\ell}}$-representation of $G(F)$. We say $\pi$ admits a $\overline{\mathbf{Z}_{\ell}}$-lattice if there exists an admissible smooth $\ell$-torsion-free $\overline{\mathbf{Z}_{\ell}}[G(F)]$-module $L$ together with an isomorphism $L[1/\ell] \simeq \pi$.
\end{dfn}

Recall that our convention on the meaning of admissible is slightly non-standard, so in particular any such $L$ has the property that $L^{K}$ is a finite free $\overline{\mathbf{Z}_{\ell}}$-module for all open compact pro-$p$ subgroups $K \subset G(F)$, and whence $L$ is a free $\overline{\mathbf{Z}_{\ell}}$-module. The existence of a $\overline{\mathbf{Z}_{\ell}}$-lattice in our sense implies, but is strictly stronger than, the existence of a ``$\overline{\mathbf{Z}_{\ell}}G(F)$-r\'eseau" in the sense of \cite{VignerasBook}. Note also that if $\pi$ is a finite-length admissible representation admitting a $\overline{\mathbf{Z}_{\ell}}$-lattice, then any such lattice is finitely generated as a $\overline{\mathbf{Z}_{\ell}}[G(F)]$-module by \cite{VignerasWhittaker}.

The goal of this section is to prove the following result.
\begin{pro}\label{PropLanglands} The group $\Groth(G(F))$ is generated by representations of the form $ i_{M}^{G}(\sigma \otimes \psi)$, where $i_{M}^{G}(-)$ is the normalized parabolic induction functor associated with a standard Levi subgroup $M$, $\psi$ is an unramified character of $M(F)$, and $\sigma$ is an irreducible admissible $\overline{\mathbf{Q}_{\ell}}$-representation of $M(F)$ admitting a $\overline{\mathbf{Z}_{\ell}}$-lattice.
\end{pro}

We will deduce this from Dat's theory of $\nu$-tempered representations \cite{Dat}. In particular, we will apply the theory from \cite{Dat} with $\mathcal{K} = \overline{\mathbf{Q}_{\ell}}$ or with $\mathcal{K}=E \subset \overline{\mathbf{Q}_{\ell}}$ a finite extension of $\mathbf{Q}_{\ell}$, equipped with the usual norms, so $\nu$ is a positive multiple of the usual $\ell$-adic valuation.

\begin{lem} \label{TemperedLattice} Let $\pi$ be any irreducible smooth $\overline{\mathbf{Q}_{\ell}}$-representation of $G(F)$. If $\pi$ is $\nu$-tempered, then $\pi$ admits a $\overline{\mathbf{Z}_{\ell}}$-lattice.
\end{lem}
\begin{proof} Suppose given $\pi$ as in the lemma. By \cite[II.4.7]{VignerasBook}, we may choose some $E$ and some admissible $E$-representation $\pi_E$ together with an isomorphism $\pi_E \otimes_E \overline{\mathbf{Q}_{\ell}} = \pi$. By definition, $\pi$ is $\nu$-tempered if and only if $\pi_E$ is $\nu$-tempered, \cite[Lemma 3.3]{Dat}. Since $\pi_E$ is $\nu$-tempered, it admits an $\mathcal{O}_E$-lattice $L$ by \cite[Proposition 6.3]{Dat}. Then $L \otimes_{\mathcal{O}_E} \overline{\mathbf{Z}_{\ell}}$ is the desired $\overline{\mathbf{Z}_{\ell}}$-lattice in $\pi$.
\end{proof}

We will now freely use all the notation and results of \cite[ \S 2-3]{Dat}, with $\mathcal{K} = \overline{\mathbf{Q}_{\ell}}$. A triple $(M,\sigma,\psi)$ consisting of a standard Levi subgroup $M \subset G$, a $\nu$-tempered irreducible representation $\sigma$ of $M(F)$, and an unramified character $\psi$ of $M(F)$ with $-\nu(\psi) \in (\mf{a}_P)^{*,+}$, is called a Langlands triple. The corresponding representation $i_P^G(\sigma\otimes\psi)$ has a unique irreducible quotient, which we will denote by $j_P^G(\sigma\otimes\psi)$. Every irreducible smooth representation $\pi$ of $G(F)$ is isomorphic to $j_P^G(\sigma\otimes\psi)$ for a (essentially) unique Langlands triple, cf. \cite[Theorem 3.11]{Dat}. The uniqueness of the triple $(M,\sigma,\psi)$ with a given irreducible quotient $\pi$ allows us to index the representation $i_P^G(\sigma\otimes\psi)$ by $\pi$. We shall write $I(\pi)$ for this representation, and refer to it as the \emph{standard representation} associated with $\pi$. Note that there is a natural surjection $I(\pi) \to \pi$.

On the other hand, by \cite[Theorem 3.11.ii]{Dat}, $\lambda_{\pi}:= -\nu(\psi) \in \mf{a}_{M_0}^{\ast}$ is also a well-defined invariant of $\pi$. Note that $\pi$ is $\nu$-tempered if and only if $\lambda_{\pi}=0$, and that $M$ can be read off from $\lambda_{\pi}$. The following key lemma is the analogue of \cite[Lemma XI.2.13]{BW} in our setting.

\begin{lem} \label{keyinequality} Let $\pi$ be any irreducible representation. Write $\pi = j_P^G(\sigma\otimes\psi)$, and let $\pi'$ be any nonzero irreducible subquotient of $I(\pi) = i_{M}^{G}(\sigma \otimes \psi)$. Then $\lambda_{\pi'} \leq \lambda_{\pi}$ in the usual partial ordering on $\mf{a}_{M_0}^{\ast}$, and $\lambda_{\pi'} < \lambda_{\pi}$ if $\pi'$ is a subquotient of $\tx{ker}(I(\pi) \to \pi)$.
\end{lem}

\begin{proof} After twisting, we may assume $\pi$ and $\pi'$ have integral central characters. Write $\pi' = j_Q^G(\sigma'\otimes\psi')$ for some Langlands triple $(L,\sigma',\psi')$. By the proof of \cite[Theorem 3.11.i]{Dat}, $\lambda_{\pi'}$ occurs in $-\nu(\mc{E}(A_L,r^{\ol{Q}}_{G}(\pi')))$, so the result now follows from the subsequent proposition.
\end{proof}

\begin{pro} \label{pro:stdrep_exp}
Let $MU=P$ and $LN=Q$ be standard parabolic subgroups of $G$. Let $\sigma$ be a $\nu$-tempered irreducible representation of $M(F)$, and $\psi : M(F) \to \overline{\mathbf{Q}_{\ell}}^\times$ an unramified character with $\mu=-\nu(\psi) \in (\mf{a}_P^G)^{*,+}$. Let $\pi'$ be a subquotient of $i_P^G(\sigma\otimes\psi)$ and let $\mu' \in -\nu(\mc{E}(A_L,r^{\ol{Q}}_G(\pi')))$. Then
\begin{enumerate}
	\item $\mu' \leq \mu$.
	\item If $\pi'$ is a subquotient of $\tx{ker}(i_P^G(\sigma\otimes\psi) \to j_P^G(\sigma\otimes\psi))$, then $\mu' < \mu$.
\end{enumerate}
\end{pro}
\begin{proof}
The exponents of $\pi'$ are a subset of the exponents of $r^{\ol{Q}}_{G}(i_P^G(\sigma\otimes\psi))$. These were analyzed in the proof of \cite[Lemma 3.7]{Dat}, where it was shown that if $\psi'$ is such an exponent and $\mu'=-\nu(\psi')$ is such an exponent, then 
\[ \mu-\mu' = \mu-[\mu]_L^G + [\mu-w^{-1}\mu]_L^G - \ol{{^+}(\mf{a}_{\ol{Q}}^G)^*},	\]
for some $w \in W_M \lmod W/ W_L$. It was moreover shown that $\mu-[\mu]_L^G$ and $[\mu-w^{-1}\mu]_L^G$ belong to $\ol{^+(\mf{a}_M^G)^*}$, which shows $\mu-\mu' \geq 0$, hence (1).

For (2) we may replace $\pi'$ with $\tx{ker}(i_P^G(\sigma\otimes\psi) \to j_P^G(\sigma\otimes\psi))$, since the exponents of the former are again a subset of the exponents of the latter. We assume by way of contradiction that $\mu=\mu'$. We have $\mu \in (\mf{a}_P^G)^{*,+}$ and $\mu' \in (\mf{a}_L^G)^*$, so $\mu' = \mu$ implies that the intersection $(\mf{a}_P^G)^* \cap (\mf{a}_L^G)^*$ is non-empty. Since $(\mf{a}_P^G)^{*,+}$ is an open subset of $(\mf{a}_M^G)^*$, we see that $(\mf{a}_M^G)^* \subset (\mf{a}_L^G)^*$, hence $Q \subset P$. According to the formula $r^{\ol{Q}}_{G}(\pi') = r^{\ol{Q}\cap M}_{M}(r^{\ol{P}}_{G}(\pi'))$, $\mu' \in -\nu(\mc{E}(A_M,r^{\ol{P}}_{G}(\pi'))$. 

By construction of the Langlands quotient $j_P^G(\sigma\otimes\psi)$ we have the exact sequence
\[ 0 \to \pi' \to i_P^G(\sigma\otimes\psi) \to i_{\ol{P}}^G(\sigma\otimes\psi), \]
where the map between the two parabolic inductions is the intertwining operator $J_{\ol{P},P}$ of \cite[Lemma 3.7]{Dat}. We recall that this intertwining operator was obtained via Frobenius reciprocity from the unique (up to scalar) element of $\tx{Hom}_M(r^{\ol{P}}_G(i_P^G(\sigma\otimes\psi)),\sigma\otimes\psi)$. This element is the unique retraction of the natural embedding of $\sigma\otimes\psi$ into $r^{\ol{P}}_G(i_P^G(\sigma\otimes\psi))$.

We can describe this element in a slightly different way that is more suitable for our purposes. The representation $r^{\ol{P}}_G(i_P^G(\sigma\otimes\psi))$ has a filtration indexed by elements of $W_M \lmod W_G / W_M$ (strictly speaking, one has to choose a total order that refines the Bruhat order), and the  natural embedding of $\sigma\otimes\psi$ into $r^{\ol{P}}_G(i_P^G(\sigma\otimes\psi))$ identifies $\sigma\otimes\psi$ with the beginning part of this filtration, indexed by $w=1$. It is shown in equation (3.9) of the proof of \cite[Lemma 3.7]{Dat} that for any exponent $\psi''$ of a subqoutient corresponding to $w \neq 1$, $\mu''=-\nu(\psi'')$ satisfies $\mu'' < \mu$. On the other hand, all exponents of $\sigma\otimes\psi$ have image $\mu$ under $-\nu$. Therefore, the retraction $r^{\ol{P}}_G(i_P^G(\sigma\otimes\psi)) \to \sigma\otimes\psi$ is simply the projection onto the $\mu$-direct summand of the the exponent decomposition of $r^{\ol{P}}_G(i_P^G(\sigma\otimes\psi))$. 

Applying $r^{\ol{P}}_G$ to the above displayed exact sequence we obtain the exact sequence
\[ 0 \to r^{\ol{P}}_G(\pi') \to r^{\ol{P}}_G(i_P^G(\sigma\otimes\psi)) \to r^{\ol{P}}_G(i_{\ol{P}}^G(\sigma\otimes\psi)). \]
Therefore $r^{\ol{P}}_G(\pi')$, being the kernel of $r^{\ol{P}}_G(i_P^G(\sigma\otimes\psi)) \to r^{\ol{P}}_G(i_{\ol{P}}^G(\sigma\otimes\psi))$, is contained in the kernel of the composition of this map with the evaluation-at-$1$ map $r^{\ol{P}}_G(i_{\ol{P}}^G(\sigma\otimes\psi)) \to \sigma\otimes\psi$. But that composition is, by construction of $J_{\ol{P},P}$ via Frobenius reciprocity, equal to the projection $r^{\ol{P}}_G(i_P^G(\sigma\otimes\psi)) \to (r^{\ol{P}}_G(i_P^G(\sigma\otimes\psi)))_\mu$. Thus the exponents of the kernel of that projection are those exponents of $r^{\ol{P}}_G(i_P^G(\sigma\otimes\psi))$ whose image $\mu''$ under $-\nu$ is not equal to $\mu$. By what was said in the previous paragraph, these satisfy $\mu'' < \mu$.
\end{proof}

\begin{proof}[Proof of Proposition \ref{PropLanglands}]
Fix a point $\theta$ in the Bernstein variety for $G(F)$, and let $\Irr(G(F))_{\theta} \subset \Irr(G(F))$ be the finite set of irreducible representations with cuspidal support $\theta$. Let 
\[ \Groth(G(F))_{\theta} \subset \Groth(G(F)) \] 
be the subgroup generated by $\Irr(G(F))_{\theta}$, so 
\[ \Groth(G(F))=\oplus_{\theta} \Groth(G(F))_{\theta}.\] 
By Lemma \ref{TemperedLattice}, it suffices to prove that $\Groth(G(F))_{\theta}$ is generated by representations $i_{M}^{G}(\sigma \otimes \psi)$ for Langlands triples $(M,\sigma,\psi)$, i.e. by standard representations. Note that $\pi \in \Irr(G(F))_{\theta}$ implies $I(\pi) \in \Groth(G(F))_{\theta}$, cf. \cite[Proposition 2.4]{BDK}. We will prove the finer result that the standard representations $I(\pi), \pi \in \Irr(G(F))_{\theta}$ give a basis for $\Groth(G(F))_{\theta}$.

Set \[ S = \{ \lambda_{\pi}, \pi \in \Irr(G(F))_{\theta} \} \subset \mf{a}_{M_0}^{\ast}. \] Note that $S$ is a finite set, and inherits a natural partial order from the partial order on $\mf{a}_{M_0}^{\ast}$. Pick any $\pi \in \Irr(G(F))_{\theta}$. If $\lambda_{\pi}$ is minimal in $S$, then the natural map $I(\pi) \to \pi$ is an isomorphism by Lemma \ref{keyinequality}. In general, if $\lambda_{\pi}$ is not minimal in $S$, then by Lemma \ref{keyinequality} and induction on $S$, we may assume that $\tx{ker}(I(\pi) \to \pi)$ is a $\mathbf{Z}$-linear combination of standard representations $I(\pi'), \pi' \in \Irr(G(F))_{\theta}$. Then also $\pi = I(\pi) - \tx{ker}(I(\pi) \to \pi)$ is a $\mathbf{Z}$-linear combination of standard representations, giving the desired result.
\end{proof}

\bibliographystyle{amsalpha}
\bibliography{bibfile}

\end{document}